\documentclass[draft]{memo-l}

\DeclareFontFamily{OT1}{pzc}{}
\DeclareFontShape{OT1}{pzc}{m}{it}{<-> s * [1.10] pzcmi7t}{}
\DeclareMathAlphabet{\mathpzc}{OT1}{pzc}{m}{it}

\usepackage{amsmath,amssymb,amscd,latexsym,mathrsfs,epic,wasysym,latexsym,tikz,mathrsfs,cite,hyperref}
\usepackage{stmaryrd}
\usepackage{pb-diagram}
\usepackage[matrix,arrow]{xy}
\usepackage{commath}
\usepackage{tikz-cd}
\usepackage{bbm}
\usepackage{ytableau}
\usepackage{mathdots}
\usepackage{imakeidx} 
\usepackage{mathtools}
\usepackage{sansmath}
\usepackage{upgreek}

\makeatletter \@addtoreset{equation}{chapter}

\makeatother

\numberwithin{section}{chapter}
\numberwithin{equation}{section}

\numberwithin{equation}{section}

\newtheorem{Proposition}[equation]{Proposition}
\newtheorem{Lemma}[equation]{Lemma}
\newtheorem{Theorem}[equation]{Theorem}
\newtheorem{Corollary}[equation]{Corollary}
\newtheorem{Inductive Assumption}[equation]{Inductive Assumption}
\theoremstyle{definition}  

\newtheorem{Remark}[equation]{Remark}

\newtheorem{Example}[equation]{Example}

\makeatletter \@addtoreset{equation}{chapter}

\makeatother

\numberwithin{section}{chapter}
\numberwithin{equation}{section}

\let\<\langle
\let\>\rangle

\newcommand\Comment[2][\relax]{\space\par\medskip\noindent%
   \fbox{\begin{minipage}{\textwidth}\textbf{Comment\ifx\relax#1\else---#1\fi}\newline%
        #2\end{minipage}}\medskip
}

\def\bi{\text{\boldmath$i$}}
\def\bj{\text{\boldmath$j$}}

\def\bk{\text{\boldmath$k$}}
\def\bl{\text{\boldmath$l$}}

\def\bs{\text{\boldmath$s$}}

\def\bt{\text{\boldmath$t$}}

\def\br{\text{\boldmath$r$}}
\def\b1{\text{\boldmath$1$}}

\def\ggw{\text{\boldmath$l$}}
\def\GGW{{\mathsf{GGW}}}

\def\lgath{g}
\def\lgathz{\mathsf{g}}
\def\gath{h}

\def\udl{k}

\def\bu{\text{\boldmath$u$}}

\def\m{\mathfrak{m}}

\def\pmod#1{\text{ }(\text{\rm mod } #1)\,}

\newcommand{\Hom}{\operatorname{Hom}}
\newcommand{\Ext}{\operatorname{Ext}}

\newcommand{\End}{\operatorname{End}}

\newcommand{\im}{\operatorname{im}}
\newcommand{\id}{\operatorname{id}}
\def\sgn{\mathtt{sgn}}

\newcommand{\cha}{\operatorname{char}}

\newcommand{\bideg}{\operatorname{bideg}}

\newcommand{\cus}{{\operatorname{cus}}}

\newcommand{\Z}{\mathbb{Z}}
\newcommand{\N}{\mathbb{N}}
\newcommand{\K}{\mathbb{K}}
\newcommand{\F}{\mathbb{F}}

\newcommand{\0}{{\bar 0}}
\renewcommand{\1}{{\bar 1}}
\def\eps{{\varepsilon}}
\def\phi{{\varphi}}

\newcommand{\zC}{{\mathsf{C}}}
\newcommand{\zF}{{\mathsf{F}}}

\newcommand{\zc}{{\mathsf{c}}}
\newcommand{\zv}{{\mathsf{v}}}
\newcommand{\zz}{{\mathsf{z}}}
\newcommand{\ze}{{\mathsf{e}}}
\newcommand{\zs}{{\mathsf{s}}}
\newcommand{\za}{{\mathsf{a}}}
\newcommand{\zb}{{\mathsf{b}}}
\newcommand{\zu}{{\mathsf{u}}}
\newcommand{\zw}{{\mathsf{w}}}
\newcommand{\zE}{{\mathsf{E}}}
\newcommand{\zB}{{\mathsf{B}}}
\newcommand{\zm}{{\mathsf{m}}}
\newcommand{\zP}{{\mathsf{P}}}

\newcommand{\zV}{{\mathsf{V}}}
\newcommand{\zW}{{\mathsf{W}}}

\newcommand{\zg}{{\mathsf{g}}}
\newcommand{\zh}{{\mathsf{h}}}

\newcommand{\zA}{{\mathsf{A}}}

\newcommand{\funF}{{\mathpzc F}}
\newcommand{\funG}{{\mathpzc G}}

\newcommand{\di}{{\operatorname{div}}}

\newcommand{\super}{\mathop{\tt s}\nolimits}
\newcommand{\swr}{\wr_{\super}}

\newcommand{\GGI}{{\mathrm{GGI}}}
\newcommand{\GGR}{{\mathrm{GGR}}}

\newcommand{\GGIS}{{\mathsf{GGI}}}
\newcommand{\GGRS}{{\mathsf{GGR}}}

\newcommand{\ga}{\gamma}
\newcommand{\Ga}{\Gamma}
\newcommand{\la}{\lambda}
\newcommand{\La}{\Lambda}
\newcommand{\al}{\alpha}
\newcommand{\be}{\beta}

\def\Si{\mathfrak{S}}
\newcommand{\si}{\sigma}
\newcommand{\om}{\omega}
\newcommand{\Om}{\Omega}

\newcommand{\de}{\delta}
\newcommand{\De}{\Delta}
\newcommand{\ka}{\kappa}

\newcommand{\Ker}{\operatorname{Ker}}

\newcommand{\Irr}{{\mathrm {Irr}}}

\def\id{\mathop{\mathrm {id}}\nolimits}
\renewcommand{\Im}{\operatorname{Im}}
\newcommand{\Ind}{{\mathrm{Ind}}}
\newcommand{\Coind}{{\mathrm{Coind}}}

\newcommand{\funh}{{\mathpzc{h}}}
\newcommand{\funf}{{\mathpzc{f}}}
\newcommand{\fung}{{\mathpzc{g}}}
\newcommand{\funq}{{\mathpzc{q}}}
\newcommand{\funQ}{{\mathpzc{Q}}}

\newcommand{\funa}{{\mathpzc{a}}}
\newcommand{\funb}{{\mathpzc{b}}}

\newcommand{\Mor}{{\mathrm {Mor}}}

\def\rank{\mathop{\mathrm{ rank}}\nolimits}

\newcommand{\M}{M}
\newcommand{\G}{G}
\newcommand{\LL}{L}
\newcommand{\Y}{Y}

\newcommand{\zL}{{\mathsf{L}}}
\newcommand{\zN}{{\mathsf{N}}}
\newcommand{\zM}{{\mathsf{M}}}

\newcommand{\Sy}{S}

\newcommand{\Di}{\Ga}
\newcommand{\Diz}{{\mathsf{\Ga}}}

\newcommand{\Res}{{\mathrm{Res}}}
\newcommand{\Ann}{{\mathrm {Ann}}}
\newcommand{\C}{{\mathbb C}}
\newcommand{\Q}{{\mathbb Q}}
\newcommand{\R}{{\mathbb R}}

\newcommand{\IS}{{\bar C}}
\newcommand{\zIS}{{\bar \zC}}

\newcommand{\D}{{\mathscr D}}

\newcommand{\EC}{{\mathcal E}}

\renewcommand{\mod}{\bmod \,}

\def\K{\mathbb K}

\newcommand{\Zig}{{\mathsf{A}}}

\newcommand{\ttB}{{\tt B}}

\newcommand{\ttP}{{\tt P}}
\newcommand{\ttX}{{\tt X}}
\newcommand{\ttY}{{\tt Y}}
\newcommand{\ttZ}{{\tt Z}}

\newcommand{\ttx}{{\mathtt x}}
\newcommand{\tty}{{\mathtt y}}
\newcommand{\ttz}{{\mathtt z}}

\def\ttB{{\mathtt B}}
\def\ttb{{\mathtt b}}
\def\ttT{{\mathtt T}}

\def\col{{\operatorname{col}}}

\def\g{{\mathfrak g}}

\def\Par{{\mathcal P}}

\def\um{{\underline{m}}}
\def\un{{\underline{n}}}

\def\ud{{\underline{d}}}

\def\b{\mathfrak{b}}
\def\k{\Bbbk}

\newcommand{\ggi}{{f}}
\newcommand{\ggis}{\mathsf{f}}

\def\Comp{\Uplambda}
\def\EC{\Uplambda_0}

\def\spa{\operatorname{span}}
\def\height{{\operatorname{ht}}}

\def\wt{{\operatorname{wt}}}

\def\op{{\mathrm{op}}}
\def\sop{{\mathrm{sop}}}
\def\re{{\mathrm{re}}}
\def\im{{\mathrm{im}\,}}
\def\onto{{\twoheadrightarrow}}
\def\into{{\hookrightarrow}}
\def\Mod#1{#1\!\operatorname{-Mod}}
\def\mod#1{#1\!\operatorname{-mod}}

\renewcommand\O{\mathcal O}

\def\iso{\stackrel{\sim}{\longrightarrow}}

\def\ch{\operatorname{ch}}
\def\lan{\langle}
\def\ran{\rangle}

\def\Stand{\operatorname{Stand}}

\def\Schur{{\operatorname{Schur}}}

\def\ttB{{\mathtt B}}
\def\ttb{{\mathtt b}}

\def\tta{{\mathtt a}}
\def\ttl{{\mathtt l}}
\def\ttT{{\mathtt T}}

\def\bla{\text{\boldmath$\lambda$}}
\def\bmu{\text{\boldmath$\mu$}}
\def\bnu{\text{\boldmath$\nu$}}

{\catcode`\|=\active
  \gdef\set#1{\mathinner{\lbrace\,{\mathcode`\|"8000%
  \let|\midvert #1}\,\rbrace}}
}
\def\midvert{\egroup\mid\bgroup}

\colorlet{darkgreen}{green!50!black}
\tikzset{dots/.style={very thick,loosely dotted},
         greendot/.style={fill,circle,color=darkgreen,inner sep=1.5pt,outer sep=0},
         blackdot/.style={fill,circle,color=black,inner sep=1.1pt,outer sep=0},
         graydot/.style={fill,circle,color=gray,inner sep=1.1pt,outer sep=0},
         reddot/.style={fill,circle,color=red,inner sep=1.1pt,outer sep=0},
         bluedot/.style={fill,circle,color=blue,inner sep=1.1pt,outer sep=0}
}
\def\greendot(#1,#2){\node[greendot] at(#1,#2){}}
\def\blackdot(#1,#2){\node[blackdot] at(#1,#2){}}
\def\graydot(#1,#2){\node[graydot] at(#1,#2){}}
\def\reddot(#1,#2){\node[reddot] at(#1,#2){}}
\def\bluedot(#1,#2){\node[bluedot] at(#1,#2){}}

\newenvironment{braid}{
  \begin{tikzpicture}[baseline=6mm,black,line width=.7pt, scale=0.32,
                      draw/.append style={rounded corners},
                      every node/.append style={font=\fontsize{5}{5}\selectfont}]%
  }{\end{tikzpicture}
}

\def\Grid(#1,#2){
  \draw[very thin,gray,step=2mm] (0,0)grid(#1,#2);
  \draw[very thin,darkgreen,step=10mm] (0,0)grid(#1,#2);
}
\newcommand{\braidbox}[6][0.2]{
	\draw[rounded corners, color=black] (#2-#1,#4) -- (#3+#1,#4) -- (#3+#1,#5) -- (#2-#1,#5) -- cycle;
	\draw ({(#2+#3)/2},{(#4+#5)/2}) node{\footnotesize #6};
}
\newcommand{\redbraidbox}[6][0.2]{
	\draw[rounded corners, color=red] (#2-#1,#4) -- (#3+#1,#4) -- (#3+#1,#5) -- (#2-#1,#5) -- cycle;
	\draw ({(#2+#3)/2},{(#4+#5)/2}) node{\footnotesize #6};
}

\newcommand{\darkgreenbraidbox}[6][0.2]{
	\draw[rounded corners, color=darkgreen] (#2-#1,#4) -- (#3+#1,#4) -- (#3+#1,#5) -- (#2-#1,#5) -- cycle;
	\draw ({(#2+#3)/2},{(#4+#5)/2}) node{\footnotesize #6};
}

\newcommand\Tableau[2][\relax]{
  \begin{tikzpicture}[scale=0.5,draw/.append style={thick,black}]
    \ifx\relax#1\relax%
    \else 
      \foreach\box in {#1} { \filldraw[blue!30]\box+(-.5,-.5)rectangle++(.5,.5); }
    \fi
    \newcount\row\newcount\col
    \row=0
    \foreach \Row in {#2} {
       \col=1
       \foreach\k in \Row {
          \draw(\the\col,\the\row)+(-.5,-.5)rectangle++(.5,.5);
          \draw(\the\col,\the\row)node{\k};
          \global\advance\col by 1
       }
       \global\advance\row by -1
    }
  \end{tikzpicture}
}

\newcommand\YoungDiagram[2][\relax]{
  \begin{tikzpicture}[scale=0.5,draw/.append style={thick,black}]
    \ifx\relax#1\relax%
    \else 
    \foreach\box in {#1} {
      \filldraw[blue!30]\box rectangle ++(1,1);
    }
    \fi
    \newcount\row
    \row=0
    \foreach \col in {#2} {
       \draw(1,\the\row)grid ++(\col,1);
       \global\advance\row by -1
    }
  \end{tikzpicture}
}

\newenvironment{Young}{\begingroup
       \def\vr{\vrule height0.89\hoogte width\dikte depth 0.2\hoogte}
       \def\fbox##1{\vbox{\offinterlineskip
                    \hrule height\dikte
                    \hbox to \breedte{\vr\hfill##1\hfill\vr}
                    \hrule height\dikte}}
       \vbox\bgroup \offinterlineskip \tabskip=-\dikte \lineskip=-\dikte
            \halign\bgroup &\fbox{##\unskip}\unskip  \crcr }
       {\egroup\egroup\endgroup}
\def\diagram#1{\relax\ifmmode\vcenter{\,\begin{Young}#1\end{Young}\,}\else%
              $\vcenter{\,\begin{Young}#1\end{Young}\,}$\fi}

\makeindex

\begin{document}

\frontmatter

\title[Imaginary Schur-Weyl for quiver Hecke superalgebras]{{Imaginary Schur-Weyl duality for quiver Hecke superalgebras}}

\author{\sc Alexander Kleshchev}
\address{Department of Mathematics\\ University of Oregon\\ Eugene\\ OR 97403, USA}
\email{klesh@uoregon.edu}



\thanks{The author was supported by the NSF grant DMS-2346684. }

\begin{abstract}
Quiver Hecke {\em super}algebras $R_{\theta}$ were defined by Kang-Kashiwara-Tsuchioka and used by Hill-Wang and Kang-Kashiwara-Oh to categorify quantum Kac-Moody superalgebras and by Kang-Kashiwara-Oh to categorufy highest weight modules over quantum Kac-Moody superalgebras. In affine type $A_{p-1}^{(2)}$, a Morita equivalence 
established by Kang, Kashiwara and Tsuchioka connects quiver Hecke superalgebras to blocks of double covers of symmetric and alternating groups in (odd) characteristic $p$. These results motivate further  study of representation theory of quiver Hecke superalgebras. 

The irreducible (graded super)modules over quiver Hecke superalgebras can be classified in terms of cuspidal modules, as shown by the author and Livesey.  
To every indivisible positive root $\al$ of type $A_{p-1}^{(2)}$ and a non-negative integer $d$, one associates an explicit quotient $\bar R_{d\al}$ of $R_{d\al}$ called the cuspidal algebra. If the root $\al$ is real, the cuspidal algebra $\bar R_{d\al}$ is well-understood and has a unique irreducible module called a (real) irreducible cuspidal module. But if $\al=\de$, the imaginary null-root, the corresponding imaginary cuspidal algebra $\bar R_{d\de}$ is still rather mysterious.  For example, it has only been known that the number of the isomorphism classes of the irreducible $\bar R_{d\de}$-modules equals the number of the $\ell$-multipartitions of $d$, but there has been no way to canonically associate an irreducible $\bar R_{d\de}$-module to such a multipartiton. The imaginary cuspidal algebra is especially important because of its connections to the RoCK blocks of the double covers of symmetric and alternating groups established by the author and Livesey. 

In this paper we undertake a detailed study of the imaginary cuspidal algebra $\bar R_{d\de}$ and its representation theory. We use the so-called Gelfand-Graev idempotents and subtle degree and parity shifts to construct a (graded) Morita (super)equivalent algebra $\zC(n,d)$ (for any $n\geq d$). The advantage of the algebra  $\zC(n,d)$ is that, unlike $\bar R_{d\de}$, it is non-negatively graded. Moreover, the degree zero component $\zC(n,d)^0$ is shown to be isomorphic to the direct sum of tensor products of $\ell$ copies of the classical Schur algebras. Since the irreducible modules over a non-negatively graded algebra are lifts of the irreducible modules over its degree zero component, this immediately gives the classification (and the description of dimensions/characters, etc.) of the irreducible  $\zC(n,d)$-modules, and hence of the irreducible $\bar R_{d\de}$-modules, in terms of the classical Schur algebras. In particular, this allows us to canonically label these by the $\ell$-multipartitions of $d$. 

The results of this paper will be used crucially in our future work on RoCK blocks of the double covers of symmetric and alternating groups. In particular, this will lead to a `local' description, up to Morita equivalence, of arbitrary defect RoCK blocks of these groups in terms of generalized Schur algebras corresponding to an explicit Brauer tree algebra. In view of the recent results on Brou\'e's Conjecture for these groups, this will provide a `local'  description of an arbitrary block of the double covers of symmetric and alternating groups up to derived equivalence.

\end{abstract}

\maketitle

\setcounter{page}{4}

\tableofcontents

\chapter{Introduction}
\section{Motivation and general description of main results}
\subsection{Motivation}
Quiver Hecke {\em super}algebras were defined by Kang, Kashiwara and Tsuchioka \cite{KKT} and used by Hill and Wang \cite{HW} and Kang, Kashiwara and Oh \cite{KKO,KKOII} to categorify quantum Kac-Moody superalgebras, and by Kang, Kashiwara and Oh \cite{KKO,KKOII} to categorify highest weight modules over quantum Kac-Moody superalgebras. 

In affine type $A_{p-1}^{(2)}$, a Morita equivalence 
established by Kang, Kashiwara and Tsuchioka \cite{KKT} connects cyclotomic quotients of quiver Hecke superalgebras to blocks of double covers of symmetric and alternating groups in (odd) characteristic $p$ and related Hecke superalgebras. This is analogous to the connection between cyclotomic quotients of quiver Hecke algebras (Khovanov-Lauda-Rouquier algebras) and blocks of symmetric groups and Hecke algebras established in \cite{BKyoung,Rou}.

Further tight connections between the imaginary cuspidal representations of quiver Hecke superalgebras and RoCK blocks of the double covers of symmetric and alternating groups were established in \cite{KlLi}, leading eventually to the proof of Brou\'e's Abelian Defect Conjecture for these groups \cite{ELV,BKodd}. 

So it is important to further study representation theory of quiver Hecke superalgebras. 

\subsection{Quiver Hecke superalgebras}
We fix a ground field $\F$. In the main body of the paper we often work over more general ground rings $\k$, but in this introduction we concentrate on the case $\k=\F$. 
We will consider the quiver Hecke superalgebras of Lie type $A_{2\ell}^{(2)}$ over $\F$. Let $Q_+$ be the non-negative part of the root lattice in type $A_{2\ell}^{(2)}$ and $\Phi_+$ be the system of positive roots, see \S\ref{SSLTN} for details. Of particular importance is the imaginary null-root $\de\in\Phi_+$.

To every $\theta\in Q_+$, following \cite{KKT}, we associate a graded $\F$-superalgebra $R_\theta$ called the quiver Hecke superalgebra, see \S\ref{SQHSA}. We will be working in the category $\mod{R_\theta}$ of finitely generated graded supermodules over $R_\theta$. (It is known that the irreducible graded $R_\theta$-supermodules are finite-dimensional and irreducible as usual $R_\theta$-modules, so one might expect that the grading and the  superstructure should not be essential when we consider the irreducible modules but this is not the case.)

\subsection{Cuspidal modules and cuspidal algebras}
The irreducible graded supermodules over quiver Hecke superalgebras were classified in terms of cuspidal irreducible modules in \cite{KlLi}. The theory is reviewed in \S\ref{SSCuspSys}. 
To every indivisible positive root $\al\in\Phi_+$ and $d\in\Z_{\geq 0}$, one associates an explicit quotient $\bar R_{d\al}$ of $R_{d\al}$ called the cuspidal algebra. 
The modules over $\bar R_{d\al}$ inflated to $R_{d\al}$  (as $\al$ and $d$ vary) are called the cuspidal modules. Arbitrary irreducible $R_\theta$-modules can be classified in terms of the cuspidal irreducible modules, see Theorem~\ref{THeadIrr}.

If the positive root $\al$ is real, the corresponding {\em real cuspidal algebra}\, $\bar R_{d\al}$ is well-understood and has a unique irreducible module. In this way, as $\al$ varies over the real positive roots and $d$ varies over positive integers, we obtain the (well-understood) {\em real cuspidal irreducible modules} over quiver Hecke superalgebras. 

But if $\al=\de$, the corresponding {\em imaginary cuspidal algebras} $\bar R_{d\de}$ are still rather mysterious. The  modules over quiver Hecke superalgebras arising by inflation from the imaginary cuspidal algebras are referred to as the {\em imaginary cuspidal modules}. As a consequence of the Kang-Kashawara-Oh categorification theorem, it was proved in \cite{KlLi} that $\bar R_{d\de}$ has exactly $|\Par^J(d)|$ irreducible modules (up to isomorphism), where 
$J:=\{0,1,\dots,\ell-1\}$ 
and $\Par^J(d)$ is the set of $J$-multipartitions of $d$.
However, in \cite{KlLi}, we were unable to canonically label the  irreducible $\bar R_{d\de}$-modules by such multipartitons. 

\subsection{Imaginary cuspidal algebras and classical Schur algebras}
In this paper we undertake a detailed study of the imaginary cuspidal algebra $\bar R_{d\de}$ and its representation theory. 
For any $n\geq d$, we use the so-called Gelfand-Graev idempotents and subtle {\em degree and parity shifts} to construct graded superalgebras $\zC(n,d)$ which are (graded) Morita (super)equivalent to $\bar R_{d\de}$. The advantage of the algebras  $\zC(n,d)$ is that, unlike $\bar R_{d\de}$, they are {\em non-negatively graded}. Moreover, for the degree zero component $\zC(n,d)^0$ we obtain the algebra isomorphism 
$$
\zC(n,d)^0\cong \bigoplus_{d_0+\dots+d_{\ell-1}=d}S(n,d_0)\otimes\dots\otimes S(n,d_{\ell-1}),
$$
where $S(n,d)$ denotes the {\em classical Schur algebra} as in \cite{Green}. 
Since the irreducible modules over the non-negatively graded algebra $\zC(n,d)$ are lifts of the irreducible modules over its degree zero component, this immediately gives the classification (and the description of dimensions/characters, etc.) of the irreducible  $\zC(n,d)$-modules, and hence, by Morita equivalence, of the irreducible $\bar R_{d\de}$-modules, in terms of the classical Schur algebras. In particular, this allows us to canonically label these by the set $\Par^J(d)$.

\subsection{Connection to RoCK blocks of double covers of symmetric groups}
The imaginary cuspidal algebras $\bar R_{d\de}$ are especially important because of their connections to the {\em RoCK blocks} of the double covers of symmetric and alternating groups established by the author and Livesey \cite{KlLi}. In fact, the results of this paper will be used crucially in our forthcoming work \cite{KlRoCK} which gives a `local' description, up to Morita equivalence, of {\em arbitrary defect}\, RoCK blocks of the double covers in terms of generalized Schur algebras corresponding to an explicit Brauer tree algebra. In view of the recent results on Brou\'e's Abelian Defect Conjecture for these groups \cite{ELV,BKodd}, this will provide a `local'  description of an arbitrary block of the double covers of symmetric and alternating groups up to derived equivalence. This will also allow us to describe the decomposition numbers of arbitrary defect RoCK blocks in terms of Littlewood-Richardson coefficients, inverse Kostka polynomials and decomposition numbers of the classical Schur algebras, generalizing \cite{FKM} where the abelian defect case was considered. 

\section{Contents of the paper and statements of the main results}
Let $d\in\Z_{\geq 0}$. We work with the imaginary cuspidal algebra  $\bar R_{d\de}$. Note that in the main body of the paper  we have used the notation $\hat C_d:=\bar R_{d\de}$. 

\subsection{Combinatorial sets}
The following combinatorial sets will be used:
\begin{enumerate}
\item[$\bullet$] the set $\Comp(n,d)$ of all {\em compositions of $d$ with $n$ parts}; the elements of $\Comp(n,d)$ are the tuples $\mu=(\mu_1,\dots,\mu_n)$ of non-negative integers such that $|\mu|:=\mu_1+\dots+\mu_n=d$;

\item[$\bullet$] the set $\Comp(d)=\bigcup_{n\geq 1}\Comp(n,d)$ of all {\em compositions of $d$}; note that the union is not disjoint since we identify the compositions $\mu\in\Comp(n,d)$ and $\nu\in\Comp(m,d)$ for $m>n$ if  $\nu=(\mu_1,\dots,\mu_n,0,\dots,0)$; we have the usual {\em dominance order} $\unlhd$ on the set $\Comp(d)$; 

\item[$\bullet$] the set $\Comp^J(n,d)$ of {\em $J$-multicompositions of $d$ with $n$ parts}; the elements of $\Comp^J(d)$ are the tuples $\bmu=(\mu^{(0)},\dots,\mu^{(\ell-1)})$ of compositions with $n$ parts such that $|\mu^{(0)}|+\dots+|\mu^{(\ell-1)}|=d$. 

\item[$\bullet$] the set $\EC(n,d)\subseteq \Comp(n,d)$ of all {\em essential compositions of $d$ with $n$ parts}, i.e. those  
$\mu\in\Comp(n,d)$ with $\mu_1,\dots,\mu_n\neq 0$;

\item[$\bullet$] the set $\EC(d):=\bigsqcup_{n\geq 1}\EC(n,d)=\bigsqcup_{n= 1}^d\EC(n,d)$ of all {\em essential compositions of $d$};
\item[$\bullet$] the set $\Par(d)$ of all {\em partitions of $d$}, i.e. compositions of $d$ with weakly decreasing parts;

\item[$\bullet$] the set $\Par^J(d)$ of {\em $J$-multipartitions of $d$}; the elements of $\Par^J(d)$ are the tuples $\bmu=(\mu^{(0)},\dots,\mu^{(\ell-1)})$ of partitions with $|\mu^{(0)}|+\dots+|\mu^{(\ell-1)}|=d$. 

\item[$\bullet$] the set $\Comp^\col(n,d)$ of all {\em colored compositions of $d$ with $n$ parts}, i.e. all pairs $(\mu,\bj)$ with $\mu\in\Comp(n,d)$ and $\bj=(j_1,\cdots, j_n)\in J^n$ an $n$-tuple of elements of $J$ usually written as $\bj=j_1\cdots j_n$;

\item[$\bullet$] the set $\Comp^\col(d):=\bigcup_{n\geq 1}\Comp^\col(n,d)$ of all {\em colored compositions of $d$};

\item[$\bullet$] the set $\EC^\col(n,d):=\{(\mu,\bj)\in \Comp^\col(n,d)\mid \mu\in\EC(n,d)\}$ of {\em essential colored compositions} of $d$ with $n$ parts;
\item[$\bullet$] the set $\EC^\col(d):=\bigsqcup_{n\geq 1}\EC^\col(n,d)=\bigsqcup_{n= 1}^d\EC^\col(n,d)$ of {\em essential colored compositions of $d$}.

\end{enumerate}

\subsection{Cuspidal parabolic subalgebras}
For $\la=(\la_1,\dots,\la_n)\in\Comp(d)$, we have an idempotent $1_{\la\de}\in \bar R_{d\de}$ and the cuspidal parabolic subalgebra 
$$
\bar R_{\la_1\de}\otimes\cdots\otimes \bar R_{\la_n\de}\cong \bar R_{\la\de}\subseteq 1_{\la\de} \bar R_{d\de}1_{\la\de}, 
$$ 
see (\ref{EHatParabolic}). This yields the functors of parabolic induction and restriction 
\begin{align*}
\Ind_{\la\de}^{d\de}&:\mod{\bar R_{\la\de}}\to\mod{\bar R_{d\de}},\ V\mapsto \bar R_{d\de}1_{\la\de}\otimes_{\bar R_{\la\de}} V,
\\
\Res_{\la\de}^{d\de}&:\mod{\bar R_{d\de}}\to\mod{\bar R_{\la\de}},\ W\mapsto 1_{\la\de}\bar R_{d\de}\otimes_{\bar R_{d\de}} W,
\end{align*}
cf. Lemmas~\ref{LCuspRes}, \ref{LTensImagIsImag}.

\subsection{Gelfand-Graev idempotents}
In this paper we introduce certain remarkable idempotents in  $\bar R_{d\de}$ which we call  {\em Gelfand-Graev idempotents}  (generalizing the class of Gelfand-Graev idempotents used in \cite{KlLi}). 


To every colored composition $(\mu,\bj)\in\Comp^\col(d)$  we associate the explicit Gelfand-Graev idempotent 
$
\ggi^{\mu,\bj}\in \bar R_{d\de},
$ 
see (\ref{EGGIdempotent}). Then the idempotents $\{\ggi^{\mu,\bj}\mid (\mu,\bj)\in \EC^\col(d)\}$ are distinct and orthogonal to each other, hence 
\begin{equation}\label{EIntro_3}
\ggi_d:=\sum_{(\mu,\bj)\in \EC^\col(d)}\ggi^{\mu,\bj}\in \bar R_{d\de}
\end{equation}
is also an idempotent. So we can consider the {\em Gelfand-Graev idempotent truncation}
$$
C_d:=\ggi_d\bar R_{d\de}\ggi_d.
$$

\vspace{2mm}
\noindent
{\bf Theorem A.} 
{\em 
The graded superalgebras $C_d$ and $\bar R_{d\de}$ are graded Morita superequivalent. 
}
\vspace{2mm}

In other words, $\mod{C_d}$ and $\mod{\bar R_{d\de}}$ are equivalent as graded supercategories, see \S\ref{SSMorita}. Theorem A is proved in Corollary~\ref{CEquivFG}.

For a finite dimensional graded $C_d$-supermodule $V$, in view of (\ref{EIntro_3}), we have
$$
V=\bigoplus _{(\mu,\bj)\in \EC^\col(d)}\ggi^{\mu,\bj}V.
$$
We introduce the {\em Gelfand-Graev formal character}\, of a finite dimensional graded $C_d$-supermodule $V$: 
$$
\operatorname{GGCh}(V):=\sum_{(\mu,\bj)\in\EC^\col(d)}(\dim \ggi^{\mu,\bj}V)\cdot (\mu,\bj) \in \Z\cdot\EC^\col(d).
$$
where $\Z\cdot\EC^\col(d)$ is the free $\Z$-module  on the  
basis $\EC^\col(d)$.

\subsection{Parabolic subalgebras and Gelfand-Graev induction}
For a composition $\la=(\la_1,\dots,\la_n)\in\Comp(d)$, we also have the Gelfand-Graev idempotent in the parabolic subalgebra $\bar R_{\la\de}$ 
$$
\ggi_\la=\ggi_{\la_1}\otimes\dots\otimes \ggi_{\la_n}\in 
\bar R_{\la_1\de}\otimes\cdots\otimes \bar R_{\la_n\de}=\bar R_{\la\de}\subseteq 1_{\la\de} \bar R_{d\de}1_{\la\de},
$$
which yields the Gelfand-Graev parabolic subalgebra $$
C_\la=C_{\la_1}\otimes\dots\otimes C_{\la_n}\subseteq \ggi_\la C_d\ggi_\la
$$
and the functors of Gelfand-Graev parabolic induction and restriction 
\begin{align*}
\GGI_{\la}^{d}&:\mod{C_{\la}}\to\mod{C_{d}},\ V\mapsto C_{d}\ggi_{\la}\otimes_{C_{\la}} V,
\\
\GGR_{\la}^{d}&:\mod{C_{d}}\to\mod{C_{\la}},\ W\mapsto \ggi_{\la}C_{d}\otimes_{C_{d}} W.
\end{align*}
Under the Morita equivalence of Theorem A, these correspond to the functors $\Ind_{\la\de}^{d\de}$ and $\Res_{\la\de}^{d\de}$, respectively, see (\ref{EGGIInd}).

For the special composition $\om_d=(1^d)\in\Comp(d)$, the idempotent truncation $\ggi_{\om_d}C_d\ggi_{\om_d}$ was studied in \cite{KlLi}. We have constructed an explicit isomorphism  
\begin{equation}\label{EIntro090924}
\ggi_{\om_d}C_d\ggi_{\om_d}\cong H_d(A_\ell),
\end{equation} 
where $H_d(A_\ell)$ is the {\em affine Brauer tree superalgebra} corresponding to the {\em graded Brauer tree superalgebra $A_\ell$}, see \S\ref{SSAffZig} and Theorem~\ref{TMorIso}. Moreover,  it was even proved that $\ggi_{\om_d}\bar R_{d\de}\ggi_{\om_d}=\ggi_{\om_d}C_{d}\ggi_{\om_d}$ is graded Morita superequivalent to $\bar R_{d\de}$ provided $\cha \F=0$ or $\cha \F>p$. However, in the case $0<\cha \F\leq p$, which is crucial for modular representation theory, the algebras $\ggi_{\om_d}\bar R_{d\de}\ggi_{\om_d}$ and $\bar R_{d\de}$ are {\em not} Morita equivalent. This explains the necessity of using the more general Gelfand-Graev idempotent $\ggi_d$, see Theorem A. 

As a special case of (\ref{EIntro090924}), we get the following explicit description of the Gelfand-Graev truncated imaginary cuspidal algebra in rank $1$:
$$
C_1\cong H_1(A_\ell)\cong A_\ell[z],
$$
where $A_\ell[z]$ is a twisted polynomial superalgebra in the even indeterminate $z$ of degree $4$ with coefficients in $A_\ell$, see \S\ref{SSAffZig}. 
This immediately yields the explicit construction of the irreducible modules over $C_1\cong A_\ell[z]$: 
$$
\Irr(C_1)=\{L_j\mid j\in J\},
$$
where each $L_j$ is $1$-dimensional spanned by a vector $v_j$ and the action of the generators of  $A_\ell[z]$ on $v_j$ is described explicitly, see \S\ref{SSL}.

\subsection{Imaginary tensor spaces}
To understand the irreducible modules over $C_d$ we introduce the {\em colored imaginary tensor spaces} 
$$
\M_{d,j}:=\GGI_{\om_d}^d \LL_j^{\boxtimes d}=C_d\ggi_{\om_d}\otimes_{C_{\om_d}}\LL_j^{\boxtimes d}\qquad(j\in J). 
$$
with generator $v_{d,j}:=\ggi_{\om_d}\otimes v_j^{\otimes d}$.

Fix $j\in J$. For the special essential colored composition $(\mu,\bj)\in\EC^\col(n,d)$ with $j_1=\dots=j_n=j$, we denote 
\begin{equation}\label{E090924_8}
(\mu,j):=(\mu,\bj). 
\end{equation}
We then have
\begin{equation}\label{E090924_6}
M_{d,j}=\bigoplus_{\mu\in\EC(d)}\ggi^{\mu,j}M_{d,j},
\end{equation}
see Lemma~\ref{LMWtSpaces}. 

\vspace{2mm}
\noindent
{\bf Theorem B.} 
{\em 
Let $j\in J$. For every $\mu\in\Par(d)$ there exists a unique up to isomorphism irreducible graded $C_d$-supermodule $L_j(\mu)$ such that $L_j(\mu)$ is a composition factor of $\M_{d,j}$ and 
$$
\operatorname{GGCh}(L_{j}(\mu))=(\mu,j)+\sum_{\nu\in\EC(d),\,\nu\lhd\mu}b_{\nu,j}\cdot (\nu,j)
$$
for some $b_{\nu,j}\in\Z_{\geq 0}$. 
Moreover, $\{L_j(\mu)\mid \mu\in\Par(d)\}$ is a complete and non-redundant set of the composition factors of $M_{d,j}$ up to isomorphism.  
}
\vspace{2mm}

For $\bmu=(\mu^{(0)},\dots,\mu^{(\ell-1)})\in\Par^J(d)$ with $d_j:=|\mu^{(j)}|$ for all $j\in J$, we define
\begin{equation}\label{ELBla}
L(\bmu):=\GGI^{d}_{d_0,\dots,d_{\ell-1}}\big(L_0(\mu^{(0)})\boxtimes\dots\boxtimes L_{\ell-1}(\mu^{(\ell-1)})\big).
\end{equation}

\vspace{2mm}
\noindent
{\bf Theorem C.} 
{\em 
We have that $\{L(\bmu)\mid\bmu\in\Par^J(d)\}$ is a complete and non-redundant set of the irreducible graded $C_d$-supermodules up to isomorphism. 
}
\vspace{2mm}

Theorems B and C are proved in Theorem~\ref{LAmountAmount}.

\subsection{Imaginary Schur-Weyl and Howe dualities}
\label{SSISWHowe}
To prove Theorems B and C and to further study representation theory of $C_d$, we use a Schur-Weyl type duality between $C_d$ and the symmetric group $\Si_d$ on $d$ letters. We then upgrade it to a Howe type duality between $C_d$ and the classical Schur algebra.

Fix $j\in J$. 
The isomorphism (\ref{EIntro090924}) allows us to define the explicit right action of $\Si_d$ on $M_{d,j}$. Let 
$$
\IS_{d,j}:=C_d/\Ann_{C_d}(\M_{d,j}).
$$
The following double centralizer property is proved in  Theorem~\ref{4.5e}(ii) (`$\sop$' stands for the opposite superalgebra, and we always consider $\F\Si_d$ as a graded superalgebra in the trivial way): 

\vspace{2mm}
\noindent
{\bf Theorem D (Imaginary Schur-Weyl Duality).} 
{\em 
Let $j\in J$. Then we have\, $\End_{\IS_{d,j}}(\M_{d,j})^\sop\cong \F\Si_d$  and\, $\End_{\F\Si_d}(\M_{d,j})\cong \IS_{d,j}$. 
}
\vspace{2mm}

We prove in Theorem~\ref{TMdProj} that $M_{d,j}$ is projective as a graded supermodule over $\IS_{d,j}$; however, it is not in general a projective generator. To define a projective generator we use the action of $\Si_d$ on $M_{d,j}$ to define the graded $\IS_{d,j}$-supermodules
$$
\Di_j^\la:=\{m\in \M_{d,j}\mid mg=m \ \text{for all}\ g\in\Si_\la\}\qquad(\la\in\Comp(d))
$$
(the analogues of tensor products of divided powers of the natural module over the general linear group). 
Here $\Si_\la$ denotes the standard parabolic subgroup $\Si_\la:=\Si_{\la_1}\times\dots\Si_{\la_n}\leq \Si_d$ for $\la=(\la_1,\dots,\la_n)\in\Comp(d)$. 
Recall that $S(n,d)$ denotes the classical Schur algebra. 

\vspace{2mm}
\noindent
{\bf Theorem E (Imaginary Howe Duality).} 
{\em 
Let $j\in J$. Then: 
\begin{enumerate}
\item[{\rm (i)}] for any $\la\in\Comp(d)$ we have $\Di_j^\la\cong \IS_{d,j}\ggi^{\la,j}$;
\item[{\rm (ii)}] there is an explicit algebra isomorphism
$$
\textstyle S(n,d)\cong \End_{\IS_{d,j}}\big(\bigoplus_{\la\in\Comp(n,d)}\Di_j^\la\big);
$$
\item[{\rm (iii)}] if $n\geq d$ then $\bigoplus_{\la\in\Comp(n,d)}\Di_j^\la$ is a projective generator for $\IS_{d,j}$. 
\end{enumerate}
}
\vspace{2mm}

Part (i) is proved in Corollary~\ref{CZCIdem}, part (ii) in Theorem~\ref{TSchurEnd3}, and part (iii) in Theorem~\ref{TProjGen}(iii). 

If $n\geq d$, Theorem E yields an  equivalence 
$$
\funb_{n,d,j}:\mod{S(n,d)}\to\mod{\IS_{d,j}}.
$$
It is well-known \cite{Green} that the irreducible modules over $S(n,d)$ are canonically labeled by their highest weights which are elements of $\Par(d)$: 
$$
\{L_{n,d}(\la)\mid \la\in \Par(d)\},
$$ 
Then $L_j(\la)\cong \funb_{n,d,j}(L(\la))$, see Lemma~\ref{LfunbIrr}. 

There is an even more direct connection between $C_d$ and the classical Schur algebra. To explain this, we need to introduce an important regrading of $C_d$.

\subsection{Regrading truncated cuspidal algebras}
For $(\mu,\bj)\in\Comp^\col(n,d)$, set
\begin{align*}
t_{\mu,\bj}&:=d(2\ell+1)+\sum_{s=1}^n\mu_s^2(2j_s-4\ell)
\in\Z,
\\
\eps_{\mu,\bj}&:=\sum_{s=1}^n\mu_sj_s\pmod{2}\in\Z/2\Z,
\end{align*}
and consider the new graded superalgebra 
\begin{equation}\label{EReGradingC}
\zC_d=\bigoplus_{(\la,\bi),(\mu,\bj)\in\EC^\col(d)}
\funQ^{t_{\la,\bi}-t_{\mu,\bj}}\Uppi^{\eps_{\la,\bi}-\eps_{\mu,\bj}} \ggi^{\mu,\bj}C_d\ggi^{\la,\bi}.
\end{equation}
Here $\funQ$ is the degree shift and $\Uppi$ is the parity shift;  thus, $\zC_d$ is the same algebra as $C_d$ but with a different grading ($\Z$-grading) and a different superstructure ($\Z/2$-grading), see  \S\ref{SSRegradingCd} for details.

We denote by $\zc\in\zC_d$ the element corresponding to an  element $c\in C_d$; for example, we have the idempotents $\ggis^{\mu,\bj}$ and $\ggis_\la$ in $\zC_d$ corresponding to the idempotents $\ggi^{\mu,\bj}$ and $\ggi_\la$ in $C_d$, respectively. We also have the parabolic $\zC_\la\subseteq \ggis_\la\zC_d\ggis_\la$. 

The graded superalgebras $C_d$ and $\zC_d$ are graded Morita superequivalent. 
in fact, given a graded $C_d$-supermodule $V$, the $\zC_d$-module corresponding to $V$ under the graded Morita superequivalence of $C_d$ and $\zC_d$ is
$$
\zV:=\bigoplus_{(\la,\bi)\in\EC^\col(d)}^n \funQ^{-t_{\la,\bi}}\Uppi^{-\eps_{\la,\bi}}\ggi^{\la,\bi}V;
$$
thus $\zV=V$ as a vector space but has a different grading and superstructure and $\zc$ acts on $\zV$ in the same way as $c$ acts on $V$ for any $c\in C_d$. The vector of $\zV$ corresponding to $v\in V$ is denoted by $\zv$. In particular, we have the imaginary tensor spaces $\zM_{d,j}$ over $\zC_d$ with generator $\zv_{d,j}$, obtained by regrading of $\M_{d,j}$. The right $\k\Si_d$-module structure is then inherited from that on $\M_{d,j}$, see \S\ref{SSISWHowe}. 
The decomposition (\ref{E090924_6}) becomes
\begin{equation}\label{E090924_7}
\zM_{d,j}=\bigoplus_{\mu\in\EC(d)}\ggis^{\mu,j}\zM_{d,j},
\end{equation}

While the graded superalgebras $C_d$ and $\zC_d$ are graded Morita superequivalent, this subtle `regrading' happens to be very useful. The corresponding regrading is also crucial when dealing with RoCK blocks of double covers of symmetric groups, see \cite{KlRoCK}. The major advantage of $\zC_d$ over $C_d$ is that it is non-negatively graded:

\vspace{2mm}
\noindent
{\bf Theorem F.} 
{\em 
The algebra $\zC_d$ is non-negatively graded. 
}
\vspace{2mm}

Theorem F is proved in Theorem~\ref{TNonNeg}.

\subsection{The elements $\lgathz_{\mu,\bj}$ and $\lgath_{\mu,\bj}$}
A key role in this paper is played by certain explicit elements  $\lgath_{\mu,\bj}\in C_d$ and the corresponding elements $\lgathz_{\mu,\bj}\in\zC_d$ defined in \S\ref{SSUpsilon} for any $(\mu,\bj)\in \Comp^\col(d)$. Recalling the notation (\ref{E090924_8}), in the special case where $j_1=\dots=j_d=j$ we have the elements $\lgath_{\mu,j}\in C_d$ and $\lgathz_{\mu,j}\in\zC_d$. 
The introduction of the elements $\lgathz_{\mu,\bj}$ can be motivated by the following remarkable properties:

\vspace{2mm}
\noindent
{\bf Theorem G.} 
{\em 
Let  $(\mu,\bj)\in \Comp^\col(n,d)$ and $\la\in\Comp(d)$.
\begin{enumerate}
\item[{\rm (i)}] The element $\lgathz_{\mu,\bj}$ is even and of degree $0$. 
\item[{\rm (ii)}] The space  $\ggis^{\mu,\bj}\zC_{\mu}^0\ggis^{\om_d,j_1^{\mu_1}\cdots j_n^{\mu_n}}$ is $1$-dimensional and spanned by $\lgathz_{\mu,\bj}$. 
\item[{\rm (iii)}] 
There is an even degree $0$ isomorphism of right $\k\Si_d$-modules
$$
\ggis^{\la,j} \zM_{d,j}\iso \k_{\Si_\la}\otimes_{\k\Si_\la}\k\Si_d
$$
which maps $\lgathz_{\la,j} \zv_{d,j}$ to  $1\otimes 1_{\Si_d}$, where the right permutation module $\k_{\Si_\la}\otimes_{\k\Si_\la}\k\Si_d$ is purely even and concentrated in degree $0$. 
\end{enumerate}
}
\vspace{2mm}

Note that part (ii) of the theorem, which is proved in Lemma~\ref{LGBasis}, already determines the elements $\ggis^{\mu,\bj}$ uniquely up to a scalar.
Part (i) is proved in Lemma~\ref{LDegreeG}. 
Part (iii) is proved in Proposition~\ref{PzMPerm}.

We now consider the graded superalgebra 
$$
\zC_j(n,d):=\End_{\zC_{d}}\Big(\bigoplus_{\la\in\Comp(n,d)}\zC_{d}\ggis^{\la,j}\Big)^\sop
=\bigoplus_{\la,\mu\in\Comp(n,d)} \ggis^{\mu,j}\zC_{d}\ggis^{\la,j}.
$$
The action of $\zC_d$ on $\zM_{d,j}$ yields the action of $\zC_j(n,d)$ on
$$
\zM_j(n,d):=
\bigoplus_{\la\in\Comp(n,d)} \ggis^{\la,j}\zM_{d,j}.
$$
By Theorem~G, as graded right $\k\Si_d$-supermodules, 
$$
\zM_j(n,d)\simeq\bigoplus_{\la\in\Comp(n,d)}\k_{\Si_\la}\otimes_{\k\Si_\la}\k\Si_d.
$$
So by definition of the classical Schur algebra, 
$
\End_{\k\Si_d}(\zM_j(n,d))= S(n,d).
$
On the other hand, the action of $\zC_j(n,d)$ on $\zM_j(n,d)$ gives the graded superalgebra homomorphism 
$$\phi:\zC_j(n,d)\to \End_{\k\Si_d}(\zM_j(n,d)).$$

\vspace{2mm}
\noindent
{\bf Theorem H.} 
{\em 
Let $j\in J$ and $n\geq d$. Then the restriction of $\phi$ to the degree zero composnent is an algebra isomorphism 
$\zC_j(n,d)^0\iso S(n,d)$. 
}
\vspace{2mm}

Theorem H is proved in Theorem~\ref{T050924_4}.

\subsection{The algebra $\zC(n,d)$} 
For a $J$-multicomposition $\bmu\in\Comp^J(n,d)$ we denote 
$$
\ud(\bmu)=(|\mu^{(0)}|,\dots,|\mu^{(\ell-1)}|);
$$
this is an element of the set $\Comp(J,d)$ of all tuples $\ud=(d_0,\dots,d_{\ell-1})$ of non-negative integers with $d_0+\dots+d_{\ell-1}=d$. 

There is an embedding 
$$\ttb:\Comp^J(n,d)\to \Comp^\col(n\ell,d),\ \bmu\mapsto (\tta(\bmu),\bk^{(n)})
$$
where 
$$\tta(\bmu)=(\mu^{(0)}_1,\dots,\mu^{(0)}_n,
\dots,\mu^{(\ell-1)}_1,\dots, \mu^{(\ell-1)}_n)\quad \text{and}\quad \bk^{(n)}:=0^n1^n\cdots (\ell-1)^n\in J^{n\ell}$$ (informally, read the parts of the compositions $\mu^{(0)},\dots,\mu^{(\ell-1)}$ and record the corresponding colors), see (\ref{EBeta}). 
We will consider the multicompositions in $\La^J(n,d)$ as the colored compositions in $\Comp^\col(n\ell,d)$ via this embedding. In particular, we now have the idempotents $\ggis^\bmu:=\ggis^{\ttb(\bmu)}$ for any $\bmu\in\Comp^J(n,d)$. 

We consider the 
graded superalgebra
$$
\zC(n,d):=\End_{\zC_{d}}\Big(\bigoplus_{\bla\in\Comp^J(n,d)}\zC_{d}\ggis^\bla\Big)^\sop
\simeq\bigoplus_{\bla,\bmu\in\Comp^J(n,d)} \ggis^\bmu\zC_{d}\ggis^\bla.
$$
If $n\geq d$, 
this 
is just a Morita equivalent version of $\zC_{d}$, see Lemma~ \ref{LEXMor}: 

\vspace{2mm}
\noindent
{\bf Theorem I.} 
{\em 
If $n\geq d$, 
then $\zC(n,d)$ and $\zC_{d}$ are graded Morita superequivalent. 
}
\vspace{2mm}

For $\ud\in\Comp(J,d)$, we also have the graded superalgebra 
$$
\zC(n,\ud):=\End_{\zC_{d}}\Big(\bigoplus_{\substack{\bla\in\Comp^J(n,d),\\ \ud(\bla)=\ud}}\zC_{d}\ggis^\bla\Big)^\sop
\simeq\bigoplus_{\substack{\bla,\bmu\in\Comp^J(n,d),\\\ud(\bla)=\ud(\bmu)=\ud}} \ggis^\bmu\zC_{d}\ggis^\bla.
$$

\vspace{2mm}
\noindent
{\bf Theorem J.} 
{\em We have for the degree zero components:
\begin{enumerate}
\item[{\rm (i)}] $\zC(n,d)^0=\bigoplus_{\ud\in\Comp(J,d)}\zC(n,\ud)^0$;
\item[{\rm (ii)}] for each $\ud\in\Comp(J,d)$, we have 
$
\zC(n,\ud)^0\cong \zC_0(n,d_0)^0\otimes\dots\otimes \zC_{\ell-1}(n,d_{\ell-1})^0$;
\end{enumerate}
In particular, taking into account Theorem H, 
$$\zC(n,d)^0\cong \bigoplus_{\ud=(d_0,\dots,d_{\ell-1})\in\Comp(J,d)}S(n,d_0)\otimes\dots\otimes S(n,d_{\ell-1}).$$
}
\vspace{2mm}

Theorem~J is proved in Lemmas~\ref{L050924_3}, \ref{L050924_2} and Corollary~\ref{C100924}

\mainmatter

\chapter{Background and generalities}

\section{Combinatorial and Lie theoretic notation}
\subsection{General conventions}
We will work either over a principal ideal domain $\O$ of characteristic $0$ or over a field $\F$ of arbitrary characteristic. We denote by $\k$ either $\O$ or $\F$. 
We denote $\N:=\Z_{\geq 0}$, $\N_+:=\Z_{> 0}$, and write $\Z/2=\Z/2\Z=\{\0,\1\}$.



Throughout the paper we fix 
$\ell\in\N_+$ and denote 
\begin{align}
\label{EIJ}
p:=2\ell+1,\quad 
I := \{0,1,\dots,\ell\}\quad\text{and}\quad J:=\{0,1,\dots,\ell-1\}.
\end{align}
We identify $I$ with the set of vertices of a certain Dynkin diagram, see \S\ref{SSLTN}. 
For $i\in I$, we define $\|i\|\in\Z/2$ as follows 
\begin{equation}\label{EIParity}
\|i\|:=
\left\{
\begin{array}{ll}
\1 &\hbox{if $i=0$,}\\
\0 &\hbox{otherwise.}
\end{array}
\right.
\end{equation}

Throughout the paper, $q$ is an indeterminate and $\pi$ is an indeterminate satisfying $\pi^2=1$. To be more precise, we will work with the ring 
$
\Z^\pi:=\Z[\pi]/(\pi^2-1),
$
as well as the rings
$\Z^\pi[q,q^{-1}]$ of Laurent polynomials in $q$ and 
$\Z^\pi((q))$ of Laurent series in $q$ with coefficients in $\Z^\pi$. 
We have the $\Z^\pi$-linear involution
\begin{equation}\label{EBarInv}
\bar{}\,:\Z^\pi[q,q^{-1}]\to\Z^\pi[q,q^{-1}],\ q^n\mapsto q^{-n}
\end{equation}
For $n\in \N$ and $k\in\N_+$, we will use the following elements of $\Z[q,q^{-1}]\subseteq\Z^\pi[q,q^{-1}]$:
\begin{equation}\label{EQIntGen}
[n]_{q^k}:=\frac{q^{kn}-q^{-kn}}{q^k-q^{-k}},\quad [n]_{q^k}^!:=[1]_{q^k}[2]_{q^k}\cdots[n]_{q^k}. 
\end{equation}
For $i\in I$, we define 
\begin{equation}\label{EQInt1}
q_i:=
\left\{
\begin{array}{ll}
q &\hbox{if $i=0$,}\\
q^2 &\hbox{if $0<i<\ell$,}\\
q^4 &\hbox{if $i=\ell$.}
\end{array}
\right.
\end{equation}
For $n\in\N$, we define the following elements of $\Z^\pi[q,q^{-1}]$
\begin{equation}\label{EQInt}\begin{split}
[n]_i&:=
\left\{
\begin{array}{ll}
[n]_{q_i} &\hbox{if $i\neq 0$,}\\
\frac{(q\pi)^{n}-q^{-n}}{q\pi-q^{-1}} &\hbox{if $i=0$,}
\end{array}
\right.
\\ 
[n]_i^!&:=[1]_{i}[2]_{i}\cdots[n]_{i}.
\end{split} 
\end{equation}
Note that 
\begin{equation}\label{EBarGood}
\overline{[n]_i}=[n]_i\qquad(\text{for}\ i>0)
\end{equation}
but in general $\overline{[n]_0}\neq[n]_0$. In fact, 
for $n\in\N$, we have 
\begin{equation}\label{E2n!Bar}
\overline{[2n]_0^!}=\pi^n[2n]_0^!.
\end{equation}

For $ d\in \N$, we write elements of $I^d$ as $\bi=(i_1,\dots,i_ d)$ or $\bi=i_1\cdots i_ d$, with $i_1,\dots,i_ d\in I$, and refer to them as {\em words with entries in $I$}. Given words $\bi\in I^ d$ and $\bj\in I^c$, we have the concatenation word $\bi\bj\in I^{d+c}$. Similarly, we have the set of words $J^ d$ with entries in $J$.

For for $a,b\in \Z$, we will consider the {\em integral segments}  
\begin{equation}\label{ESegment}
[a,b]:=\{a,a+1,\dots,b\}.
\end{equation}
Given integral segments $\ttX$ and $\ttX_1,\dots,\ttX_k$, we say that $\ttX$ is an {\em increasing disjoint union} of $\ttX_1,\dots,\ttX_k$ if $\ttX=\ttX_1\sqcup\dots\sqcup \ttX_k$ and $a<b$ whenever $a\in \ttX_r$, $b\in \ttX_s$ and $r<s$.

\subsection{Partitions and compositions}
\label{SSPar}
A {\em composition} is an infinite sequence $\la=(\la_1,\la_2,\dots)$ of non-negative integers which are eventually zero. We write $\varnothing$ for the composition $(0,0,\dots)$. 
When writing compositions, we may collect consecutive equal parts together with a superscript, and omit an infinite tail of $0$'s. 
For example, for $d\in\N_+$, we have the composition 
\begin{equation}\label{EOmd}
\om_d:=(1^d).
\end{equation}

Given compositions $\la=(\la_1,\dots,\la_n)$ and $\mu=(\mu_1,\dots,\mu_m)$, we say that  $\la$ is a {\em refinement} of $\mu$ if there exist 
$0=k_0\leq k_1\leq \dots\leq k_{m-1}\leq k_m=n$ such that 
$\mu_s=\sum_{r=k_{s-1}+1}^{k_s}\la_r$ for all $s=1,\dots,m$. 


Any composition $\la$ has finite sum $|\la|:=\la_1+\la_2+\dots$, and we say that $\la$ is a composition of $|\la|$. For $d\in\N$, we write $\Comp(d)$ for the set of all compositions of $d$. 
The  standard {\em dominance order} on $\Comp(d)$ is denoted $\unlhd$, see \cite[(1.4.6)]{JamesKerber}. For $\la=(\la_1,\la_2,\dots)\in\Comp(d)$ and $k\in\N_+$ we denote 
\begin{equation}\label{EkLa}
k\la:=(k\la_1,k\la_2,\dots)\in\Comp(kd).
\end{equation} 

For $n\in\N_+$ and $d\in \N$, we set 
\begin{align*}
\Comp(n,d)&:=\{\la=(\la_1,\la_2,\dots)\in\Comp(d)\mid \la_k=0
\ \text{for all $k>n$}\}, 
\end{align*}
If $\la\in\Comp(n,d)$ we often write $\la=(\la_1,\dots,\la_n)$. 

For $\la\in\Comp(n,d)$ we have segments
\begin{equation}\label{EPartitionOfComposition}
\ttP^\la_r:=[\la_1+\dots+\la_{r-1}+1,\la_1+\dots+\la_r], 
\qquad(1\leq r\leq n).
\end{equation}
Then $[1,d]$ is an increasing disjoint union of $\ttP^\la_1,\dots, \ttP^\la_n$.

A composition $\la=(\la_1,\la_2,\dots)$ is called {\em essential} if $\la_k=0$ for some $k\in\N_+$ implies that $\la_l=0$ for all $l\geq k$. We denote by $\EC(d)$ the set of all essential compositions of $d$. When we write $\la=(\la_1,\dots,\la_n)$ for $\la\in\EC(d)$, we always suppose that $\la_n>0$, in which case we also write $\la\in\EC(n,d)\subseteq\Comp(n,d)$. 
Note that $\EC(d)=\bigsqcup_{n=1}^d\EC(n,d)$. 
Given $\la\in\EC(n,d)$ and $\mu\in\EC(m,c)$ we have the {\em concatenation} 
$
\la\mu:=(\la_1,\dots,\la_n,\mu_1,\dots,\mu_m)\in\EC(n+m,c+d).
$

Occasionally we will also use the set of compositions
\begin{equation}\label{ELaJ}
\Comp(J,d)=\{\ud=(d_0,\dots,d_{\ell-1})\in\N^\ell\mid d_0+\dots+d_{\ell-1}=d\}.
\end{equation}

A \emph{partition} is a composition whose parts are weakly decreasing. We write 
$\Par(d)$ for the set of partitions of $d$. We set 
\begin{align*}
\Comp_+(n,d)&:=\Comp(n,d)\cap\Par.
\end{align*}
We denote by $\la'$ the {\em transposed partition} of a partition $\la$, see \cite[3.5]{JamesBook}.

Let $\Par^J( d)$ denote the set of all \emph{$J$-multipartitions} of $d$. So the elements of $\Par^J( d)$ are tuples 
$\bla=(\la^{(0)},\dots,\la^{(\ell-1)})$ of partitions satisfying $\sum_{j\in J}|\la^{(j)}|= d$. 
We denote $\Par^J:=\bigsqcup_{ d\in\N}\Par^J( d)$. 
Similarly, we have $J$-multicompositions 
$$
\Comp^J(n,d)=\{(\la^{(0)},\dots,\la^{(\ell-1)})\mid 
\la^{(0)},\dots,\la^{(\ell-1)}\in\Comp(n)\ \,\text{and\, $\textstyle\sum_{j\in J}|\la^{(j)}|=d$}.
\}
$$
We have a bijection
\begin{equation}\label{EBijMultiComp}
\tta:\Comp^J(n,d)
\iso \Comp(n\ell,d),\ 
\bla
\mapsto (\la^{(0)}_1,\dots,\la^{(0)}_n,
\dots,\la^{(\ell-1)}_1,\dots, \la^{(\ell-1)}_n).
\end{equation}

We denote $\Comp^\col(n,d):=\Comp(n,d)\times J^n$. Thus the elements of $\Comp^\col(n,d)$ are pairs $(\la,\bj)$ with $\la\in\Comp(n,d)$ and $\bj\in J^n$. We refer to such pairs as {\em colored compositions (of $d$ of length $n$)}. We set 
$\Comp^\col(d):=\bigsqcup_{n\geq 0}\Comp^\col(n,d).$ 

We also have the sets of  {\em colored essential compositions}: 
\begin{align*} 
\EC^\col(n,d)=\{(\la,\bj)\in\Comp^\col(n,d)\mid \la\in \EC(n,d)\},
\quad
\EC^\col(d)=\bigsqcup_{n=0}^d\EC^\col(n,d).
\end{align*}
For $(\la,\bi)\in \EC^\col(d)$ and $(\mu,\bj)\in \EC^\col(c)$, we have the concatenation
\begin{equation}\label{EConcat}
(\la\mu,\bi\bj)\in\EC^\col(c+d).
\end{equation}

We denote
$$\bk^{(n)}:=0^n1^n\cdots (\ell-1)^n\in J^{n\ell}.$$
We now have the injective map 
\begin{equation}\label{EBeta}
\ttb:\Comp^J(n,d)\to \Comp^\col(n\ell,d),\ \bla\mapsto (\tta(\bla),\bk^{(n)}).
\end{equation}
We sometimes consider multicompositions in $\La^J(n,d)$ as colored compositions in $\Comp^\col(n\ell,d)$ via this injection.

\subsection{Symmetric group}
\label{SSSG}

We denote by $\Si_ d$  the  symmetric group on $d$ letters  
and set
$s_r:=(r,r+1) \in \Si_d$ for $r=1,\dots, d-1$ to be the elementary transpositions in $\Si_d$. 
We denote by $\ttl(w)$ the corresponding {\em length} of $w\in \Si_d$ 
and set $\sgn(w):=(-1)^{\ttl(w)}$. 
We denote by $w_d$ the longest element of $\Si_d$.

We denote by $\leq$ the Bruhat order on the symmetric group $\Si_d$, i.e. for $u,w\in \Si_d$, we have $u\leq w$ if and only if $u=s_{r_{a_1}}\cdots s_{r_{a_b}}$ for some $1\leq a_1<\cdots<a_b\leq l$, where $w=s_{r_1}\cdots s_{r_l}$  is a reduced decomposition for $w$.

For any $\la=(\la_1,\dots,\la_k)\in\Comp( d)$, we have the {\em standard parabolic subgroup} 
$\Si_\la:= \Si_{\la_1}\times\dots\times \Si_{\la_k}\leq \Si_n.$ 
We denote by $w_\la$ the longest element of $\Si_\la$. 
For the composition $\la=(a^b)$ we abbreviate $\Si_{a^b}:=\Si_{(a^b)}\leq \Si_{ab}$.

For $\la,\mu\in\Comp( d)$, we denote by $\D_d^\la$ (resp. ${}^\la\D_d$, resp. ${}^\la\D_d^\mu$) the set of the minimal length coset representatives for $\Si_ d/\Si_\la$ (resp. $\Si_\la\backslash \Si_ d$, resp. $\Si_\la\backslash \Si_ d/\Si_\mu$). It is well-known that ${}^\la\D_d^\mu=\D_d^\mu\cap{}^\la\D_d$. For $\la=(a^b)$ we abbreviate $\D^{a^b}_{ab}:=\D^{(a^b)}_{ab}$, ${}^{a^b}\D_{ab}:={}^{(a^b)}\D_{ab}$, etc. 

Let $\nu$ be yet another composition of $d$. We denote  
$$
\D^\la_\nu:=\D_d^\la\cap \Si_\nu\quad\text{and}\quad {}^\la\D_\nu:={}^\la\D_d\cap \Si_\nu.
$$
These will only be used when $\la$ is a refinement of $\nu$, in which case these are the sets of the shortest coset representatives for $\Si_\nu/\Si_\la$ and $\Si_\la\backslash \Si_\nu$ respectively.

If $\ttX\subseteq [1,d]$ and $w\in\Si_d$, we say that $w$ is {\em increasing on $\ttX$} if $w(a)<w(b)$ for all $a,b\in \ttX$ with $a<b$. In this case, if $\ttY=w(\ttX)$, we say that {\em $w$ maps $\ttX$ increasingly onto $\ttY$} and write 
\begin{equation}\label{ENEArrow}
w:\ttX\nearrow \ttY. 
\end{equation}
Note that for $\la\in\Comp(n,d)$, we have $w\in \D^\la$ if and only if $w$ is increasing on the segments $\ttP^\la_r$ for all $r=1,\dots,n$.

Let $\la,\mu\in\Comp( d)$. If $
 w\Si_\la  w^{-1}=\Si_\mu$ for some $ w\in \Si_ d$, we write  $\mu=: w\la$ and say that $ w$ {\em permutes the parts of $\la$}; in that case we have 
$
 w\Si_\la  w^{-1}=\Si_{ w\mu}. 
$

Let $\la\in\Comp(n, d)$ and $\si\in\Si_n$. Set $\mu:=(\la_{\si^{-1}(1)},\dots,\la_{\si^{-1}(n)})\in\Comp(n,d)$. Let $w_{\la,\si}\in\Si_d$ be the permutation such that
\begin{equation}\label{EWLaSi}
w_{\la,\si}:\ttP^\la_r\nearrow \ttP^\mu_{\si(r)}.
\end{equation} 
Note that 
$w_{\la,\si}\in {}^{\mu}\D_d^\la$ and $w_{\la,\si}\la =\mu$.

Let $\la,\mu\in\Comp( d)$.
If $ w\in {}^\la\D_d^\mu$, then $\Si_\la\cap  w\Si_\mu  w^{-1}$ is a standard parabolic in $\Si_d$. This standard parabolic corresponds to a certain composition of $n$, which we denote $\la\cap  w\mu$. Similarly, $ w^{-1}\Si_\la  w\cap \Si_\mu$ is the  standard parabolic corresponding to a composotion $ w^{-1}\la\cap \mu$. Thus:
\begin{equation}\label{EDJ1}
\Si_\la\cap  w\Si_\mu  w^{-1}=\Si_{\la\cap  w\mu},\quad 
 w^{-1}\Si_\la  w\cap \Si_\mu=\Si_{ w^{-1}\la\cap \mu}\qquad( w\in{}^\la\D_d^\mu).
\end{equation}
Moreover, $ w$ permutes the parts of $ w^{-1}\la\cap \mu$, and $ w( w^{-1}\la\cap \mu)= \la\cap  w\mu$, so
$
 w\Si_{ w^{-1}\la\cap \mu} w^{-1}=\Si_{\la\cap  w\mu}.
$

This can be made more explicit as follows. Let $\la\in\Comp(m,d)$ and $\mu\in\Comp(n,d)$. Let $M(\la,\mu)$ be the set of all matrices $A=(a_{r,s})\in M_{m\times n}(\N)$ such that $\sum_{s=1}^na_{r,s}=\la_r$ for $r=1,\dots,m$ and $\sum_{r=1}^ma_{r,s}=\mu_s$  for $s=1,\dots,n$. Given $A\in M(\la,\mu)$, the segment $[1,d]$ is an increasing disjoint union of the segments 
$$
\ttB^{A}_{1,1},\dots,\ttB^{A}_{m,1},\ttB^{A}_{1,2},\dots,\ttB^{A}_{m,2},\dots,\ttB^{A}_{1,n},\dots,\ttB^{A}_{m,n},
$$
of sizes $|\ttB^{A}_{r,s}|=a_{r,s}$ as well as  
an increasing disjoint union of the segments 
$$
\ttT^{A}_{1,1},\dots,\ttT^{A}_{1,n},\ttT^{A}_{2,1},\dots,\ttT^{A}_{2,n},\dots,\ttT^{A}_{m,1},\dots,\ttT^{A}_{m,n},
$$
of sizes $|\ttT^{A}_{r,s}|=a_{r,s}$. Let $w_A\in\Si_d$ be such that 
$
w_A:\ttB^{A}_{r,s}\nearrow \ttT^{A}_{r,s}$ for all $1\leq r\leq m$ and $1\leq s\leq n$. 
Then $w_A\in {}^\la\D^\mu_d$ and 
\begin{equation}\label{EDLaMuBijMat}
M(\la,\mu)\to {}^\la\D^\mu_d,\ A\mapsto w_A
\end{equation} 
is a bijection. Moreover, 
\begin{align*}
\la\cap w_A\mu=(a_{1,1},\dots,a_{1,n},a_{2,1},\dots,a_{2,n},\dots,a_{m,1},\dots,a_{m,n}),
\\
w_A^{-1}\la\cap \mu=(a_{1,1},\dots,a_{m,1},a_{1,2},\dots,a_{m,2},\dots,a_{1,n},\dots,a_{m,n}).
\end{align*}

Let $d\in\N_+$ and $\la\in\Comp( d)$. 
We define the elements $\ttx_d,\tty_d$ of the group algebra $\k\Si_d$ and the elements $\ttx_\la,\tty_\la$ of the group algebra $\k\Si_\la$ as follows
\begin{equation}\label{Ettx}
\ttx_d=\sum_{g\in \Si_d}g,\quad
\ttx_\la=\sum_{g\in \Si_\la}g,\quad \tty_d=\sum_{g\in \Si_d}\sgn(g)g,\quad \tty_\la=\sum_{g\in \Si_\la}\sgn(g)g. 
\end{equation}

\begin{Lemma} \label{L2.2.1} \cite[Lemma 4.1]{DJHecke}
Let $\la\in\Par(d)$. There exists a unique element $u_\la\in{}^{\la'}\D^\la$ such that $\Si_{\la'}\cap u_\la\Si_\la u_\la^{-1}=\{1\}$. Moreover, $\ttx_{\la'}\k\Si_d\tty_{\la}$ is a free $\k$-module of rank $1$ generated by the element $\ttx_{\la'} u_{\la}\tty_{\la}$. In particular, 
$\ttx_{\la'} u_{\la}\tty_{\la}\neq 0$. 
\end{Lemma}

The group $\Si_d$ acts on $I^d$ by place permutations:  $ g\cdot \bi=i_{ g^{-1}(1)}\cdots i_{ g^{-1}(d)}$ for $ g\in \Si_d$ and $\bi=i_1\cdots i_d\in I^d$. 
If $\la=(\la_1,\dots,\la_k)\in\Comp(d)$ and $\bi^{(1)}\in I^{\la_1},\dots,\bi^{(k)}\in I^{\la_k}$, the words of the form $ g\cdot(\bi^{(1)}\cdots \bi^{(k)})$ with $ g\in \D_d^\la$ are called {\em shuffles} of $\bi^{(1)},\dots, \bi^{(k)}$.

\subsection{Lie theory}
\label{SSLTN}
Let $\g$ be the Kac-Moody Lie algebra of type $A_{2\ell}^{(2)}$ (over $\C)$,  
see \cite[Chapter 4]{Kac}. The Dynkin diagram of $\g$ has vertices labeled by $I$:
\vspace{4mm} 
$$
{\begin{picture}(330, 15)%
\put(6,5){\circle{4}}%
\put(101,2.45){$<$}%
\put(12,2.45){$<$}%
\put(236,2.42){$<$}%
\put(25, 5){\circle{4}}%
\put(44, 5){\circle{4}}%
\put(8, 4){\line(1, 0){15.5}}%
\put(8, 6){\line(1, 0){15.5}}%
\put(27, 5){\line(1, 0){15}}%
\put(46, 5){\line(1, 0){1}}%
\put(49, 5){\line(1, 0){1}}%
\put(52, 5){\line(1, 0){1}}%
\put(55, 5){\line(1, 0){1}}%
\put(58, 5){\line(1, 0){1}}%
\put(61, 5){\line(1, 0){1}}%
\put(64, 5){\line(1, 0){1}}%
\put(67, 5){\line(1, 0){1}}%
\put(70, 5){\line(1, 0){1}}%
\put(73, 5){\line(1, 0){1}}%
\put(76, 5){\circle{4}}%
\put(78, 5){\line(1, 0){15}}%
\put(95, 5){\circle{4}}%
\put(114,5){\circle{4}}%
\put(97, 4){\line(1, 0){15.5}}%
\put(97, 6){\line(1, 0){15.5}}%
\put(6, 11){\makebox(0, 0)[b]{$_{0}$}}%
\put(25, 11){\makebox(0, 0)[b]{$_{1}$}}%
\put(44, 11){\makebox(0, 0)[b]{$_{{2}}$}}%
\put(75, 11){\makebox(0, 0)[b]{$_{{\ell-2}}$}}%
\put(96, 11){\makebox(0, 0)[b]{$_{{\ell-1}}$}}%
\put(114, 11){\makebox(0, 0)[b]{$_{{\ell}}$}}%
\put(230,5){\circle{4}}%
\put(249,5){\circle{4}}%
\put(231.3,3.2){\line(1,0){16.6}}%
\put(232,4.4){\line(1,0){15.2}}%
\put(232,5.6){\line(1,0){15.2}}%
\put(231.3,6.8){\line(1,0){16.6}}%
\put(230, 11){\makebox(0, 0)[b]{$_{{0}}$}}%
\put(249, 11){\makebox(0, 0)[b]{$_{{1}}$}}%
\put(290,2){\makebox(0,0)[b]{\quad if $\ell = 1$.}}%
\put(175,2){\makebox(0,0)[b]{if $\ell \geq 2$, \qquad and}}%
\end{picture}}
$$

\vspace{2mm}
\noindent
We denote by $P$ the {\em weight lattice} of $\g$. 
We have the set $\{\alpha_i \:|\:i \in I\}\subset P$ of {\em simple roots} of $\g$ and the set $\{h_i\:|\:i \in I\}\subset P^*$ of the {\em simple coroots} of $\g$. The {\em Cartan matrix} $\left(\langle h_i, \al_j\rangle\right)_{0\leq i,j\leq \ell}$ is:  
\vspace{2mm} 
$$
\left(
\begin{matrix}
2 & -2 & 0 & \cdots & 0 & 0 & 0 \\
-1 & 2 & -1 & \cdots & 0 & 0 & 0 \\
0 & -1 & 2 & \cdots & 0 & 0 & 0 \\
 & & & \ddots & & & \\
0 & 0 & 0 & \dots & 2 & -1& 0 \\
0 & 0 & 0 & \dots & -1 & 2& -2 \\
0 & 0 & 0 & \dots & 0 & -1& 2 \\
\end{matrix}
\right)
\quad
\text{if\ \  $\ell\geq 2$,\ \  and}
\quad
\left(
\begin{matrix}
2 & -4 \\
-1 & 2
\end{matrix}
\right)
\quad \text{if\ \  $\ell=1$.}
$$

\vspace{4mm} 
We denote by $Q$ the sublattice of $P$ generated by the simple roots and set 
$$Q_+:=\big\{\sum_{i\in I}m_i\al_i\mid m_i\in\N\  \text{for all $i\in I$}\big\}\subset Q.
$$
For $\theta=\sum_{i\in I}m_i\al_i\in Q_+$, its {\em height} is  
$
\height(\theta):=\sum_{i\in I}m_i
$ 
and its 
{\em parity} is   
\begin{equation}\label{EThetaParity}
\|\theta\|:=m_0\pmod{2}.
\end{equation} 
We denote 
$$
\delta := \sum_{i=0}^{\ell - 1} 2 \alpha_i + \alpha_\ell\in Q_+.
$$ 
Note that $\height(\de)=2\ell+1=p$.

We denote by $(.|.)$ a normalized invariant form on $P$ whose Gram matrix with respect to the linearly independent set $\al_0,\al_1,\dots,\al_\ell$ is: 
\vspace{3mm} 
$$
\left(
\begin{matrix}
2 & -2 & 0 & \cdots & 0 & 0 & 0 \\
-2 & 4 & -2 & \cdots & 0 & 0 & 0 \\
0 & -2 & 4 & \cdots & 0 & 0 & 0 \\
 & & & \ddots & & & \\
0 & 0 & 0 & \dots & 4 & -2& 0 \\
0 & 0 & 0 & \dots & -2 & 4& -4 \\
0 & 0 & 0 & \dots & 0 & -4& 8 \\
\end{matrix}
\right)
\quad
\text{if $\ell\geq 2$, and}
\quad
\left(
\begin{matrix}
2 & -4 \\
-4 & 8
\end{matrix}
\right)
\quad \text{if $\ell=1$.}
$$

\vspace{3mm} 
\noindent
In particular, 
recalling the notation (\ref{EQInt1}), we have $
q_i=q^{(\al_i|\al_i)/2}$.

For $\bi=i_1\cdots i_n\in I^n$, we denote 
\begin{equation}\label{EWtI}
\wt(\bi):=\al_{i_1}+\dots+\al_{i_n}\in Q_+.
\end{equation}
For $\theta\in Q_+$ of height $n$, we set
$$I^\theta:=\{\bi=i_1\cdots i_n\in I^n\mid \wt(\bi)=\theta\}.$$ 
We define $I^{\theta}_{\di}$ to be the set of all expressions of the form
$i_1^{(m_1)} \cdots i_r^{(m_r)}$ with
$m_1,\dots,m_r\in \N_+$, $i_1,\dots,i_r\in I$
 and $m_1 \al_{i_1} + \cdots + m_r \al_{i_r} = \theta$ (note we do not insists that $i_k\neq i_{k+1}$). We refer to such expressions as {\em divided power words}. We identify $I^\theta$ with the subset of $I^\theta_\di$ which consists of all expressions as above with all $m_k=1$. We use the same notation for concatenation of divided power words as for concatenation of words.
For $\bi=i_1^{(m_1)} \cdots i_r^{(m_r)}\in I^{\theta}_{\di}$, we denote 
\begin{equation}\label{EHatI}
\hat\bi := i_1^{m_1}\cdots i_r^{m_r}\in I^{\theta}.
\end{equation}

As in \cite[\S5]{Kac}, the set $\Phi$ of {\em roots} of $\g$ is a disjoint union of the set $\Phi^\im=\{n\de\mid n\in \Z\}$ of {\em imaginary roots} and the set $\Phi^\re$ of {\em real roots}. Also $\Phi=\Phi_+\sqcup-\Phi_+$, where $\Phi_+$ is the set of {\em positive roots}. We denote $\Phi_+^\im=\Phi^\im\cap\Phi_+$ and $\Phi_+^\re=\Phi^\re\cap\Phi_+$.

Following \cite{BKT}, we define a {\em convex preorder} on $\Phi_+$ as a preorder $\preceq$ on $\Phi_+$ such that the following three conditions hold for all $\be,\ga\in\Phi_+$:
\begin{eqnarray}
\label{EPO1}
&\be\preceq\ga \ \text{or}\ \ga\preceq \be;
\\
\label{EPO2}
&\text{if $\be\preceq \ga$ and $\be+\ga\in\Phi_+$, then $\be\preceq\be+\ga\preceq\ga$};
\\
&\label{EPO3}
\text{$\be\preceq\ga$ and $\ga\preceq\be$ if and only if $\be$ and $\ga$ are proportional}.
\end{eqnarray}
We write $\be\prec\ga$ if $\be\preceq\ga$ but $\ga\not\preceq\be$. 
By (\ref{EPO3}), $\be\preceq\ga$ and $\ga\preceq\be$ happens for $\be\neq\ga$ if and only if both $\be$ and $\ga$ are imaginary. 
In particular, the set $\Psi$ from (\ref{EPsi}) is totally ordered with respect to $\preceq$. 
The set of real roots splits into two disjoint infinite sets
$\Phi^\re_{\succ\de}:=\{\be\in \Phi_+^\re\mid \be\succ\de\}$ and $\Phi^\re_{\prec\de}:=\{\be\in \Phi_+^\re\mid \be\prec\de\}.$ The following is noted in \cite[\S3.3a,\,(Con3)]{KlLi}:

\begin{Lemma} \label{LCon3} 
Let $\be\in\Phi_+^\im$, and $\be=\sum_{a=1}^b \ga_a$ for some positive roots $\ga_a$. If either $\ga_a\preceq \be$ for all $a=1,\dots,b$ or $\ga_a\succeq \be$ for all $a=1,\dots,b$, then all $\ga_a$ are imaginary. 
\end{Lemma}

Let $\Phi'$ be the root system of type $C_\ell$ whose Dynkin diagram is obtained by dropping the simple root $\al_0$ from our type $A_{2\ell}^{(2)}$ Dynkin diagram. Then $\Phi'=\Phi'_{\text{s}}\sqcup \Phi'_{\text{l}}$ where $\Phi'_{\text{s}}=\{\al\in\Phi'\mid (\al|\al)=4\}$ and $\Phi'_{\text{l}}=\{\al\in\Phi'\mid (\al|\al)=8\}$. 
By \cite[\S6]{Kac}, we have $\Phi^\re=\Phi^\re_{\text{s}}\sqcup \Phi^\re_{\text{m}}\sqcup \Phi^\re_{\text{l}}$ for 
\begin{align}
\Phi^\re_{\text{s}}&= \{(\al+(2n-1)\de)/2 \mid \al\in\Phi'_l,\, n\in\Z\},
\label{EPhiS}
\\
\Phi^\re_{\text{m}}&=\{\al+n\de \mid \al\in\Phi'_s,\, n\in\Z\},
\label{EPhiM}
\\ 
\Phi^\re_{\text{l}}&=\{\al+2n\de \mid \al\in\Phi'_l,\, n\in\Z\}.
\label{EPhiL}
\end{align}
The set of positive roots is then 
$\Phi_+=\Phi_+^\im\sqcup\Phi_+^\re,$  
where $\Phi_+^\im=\{n\de\mid n\in\N_+\},$ while $\Phi_+^\re$ consists of the roots in $\Phi'_+$ together with the roots in (\ref{EPhiS})--(\ref{EPhiL}) with $n\in\N_+$, cf. \cite[Proposition 6.3]{Kac}.


\begin{Example} \label{ExConPr} 
{\rm 
Following \cite[Example 3.5]{McN2}, consider $\R$ as a vector space over $\Q$ and let $\chi:\Q\Phi\to \R$ be an injective linear map. Then it is easy to see that the following defines a convex preorder on $\Phi_+$: 
$$\be\preceq\ga \quad \iff\quad \chi(\be)/\height(\be)\leq \chi(\ga)/\height(\ga).
$$ 
Let  $\tilde\al=2\al_1+\dots+2\al_{l-1}+\al_\ell\in\Phi'$ be the highest root in $\Phi'$. Denote by $\Phi'_{\sharp}$ the set of all roots in 
$\Phi'$ which are non-negative linear combinations of $\al_1,\dots,\al_{\ell-1},-\tilde\al$ (this is a non-standard choice of a positive root system in $\Phi'$). The linear function $\chi:\Q\Phi\to \R$ is determined by a choice of $\Q$-linearly independent real numbers $\chi(\al_0),\dots,\chi(\al_{\ell-1}),\chi(\al_\ell)$. These can be chosen so that $\chi(\al_1),\dots,\chi(\al_{\ell-1}),\chi(-\tilde\al)$ are positive and $\chi(\de)$ is very close to zero. 
For the corresponding convex preorder $\preceq$ we have 
\begin{equation}\label{ESharp}
\begin{split}
\Phi^\re_{\succ\de}=\{\be\in\Phi_+\mid\text{$\be$ is of the form $r\al+s\de$ with $r\in\{1/2,1\}$, $\al\in\Phi'_{\sharp}$}\},
\\
\Phi^\re_{\prec\de}=\{\be\in\Phi_+\mid\text{$\be$ is of the form $-r\al+s\de$ with $r\in\{1/2,1\}$, $\al\in\Phi'_{\sharp}$}\}.
\end{split}
\end{equation}
These convex preorders are especially important since they are related to RoCK blocks of spin double covers of symmetric groups, see \cite{KlLi}. Starting with Section~\ref{SGG}, we will always work with such a convex preorder. 
}
\end{Example}

The set of {\em indivisible positive roots} is:
\begin{equation}\label{EPsi}
\Psi:=\Phi_+^\re\cup\{\de\}.
\end{equation}

Let $\theta\in Q_+$. A {\em root partition of $\theta$} is a pair $(\um,\bmu)$, where $\um=(m_\be)_{\be\in\Psi}$ is a tuple of non-negative integers such that $\sum_{\be\in\Psi}m_\be\be=\theta$, and $\bmu\in\Par^J(m_\de)$ is a $J$-multipartition of $m_\de$, see \S\ref{SSPar}. Denote by $\Par(\theta)$ the set of all root partitions of~$\theta$.

\section{Graded superalgebras}\label{SSBasicRep} 

\subsection{Graded superspaces and superalgebras}
By a {\em graded superspace} we understand a $\k$-module with $\k$-module decomposition $V=\bigoplus_{n\in\Z,\,\eps\in\Z/2}V^n_\eps$. 
We refer to the elements of $V^n_\eps$ as (homogeneous) elements of {\em bidegree} $(n,\eps)$. We also refer to $n$ as {\em degree} and $\eps$ as {\em parity} and write for $v\in V^n_\eps$:
$$
\bideg(v)=(n,\eps),\ \deg(v)=n,\ |v|=\eps.
$$
For $n\in\Z$ and $\eps\in\Z/2$, we use the notation $V^n:= V^n_{\0}\oplus V^n_{\1}$, $V_\eps=\bigoplus_{n\in \Z}V^n_\eps$, and 
\begin{equation}\label{EA>0}
\quad V^{>n}:=\bigoplus_{m>n}V^{m}, \quad
 V^{\geq n}:=\bigoplus_{m\geq n}V^{m},\quad
V^{<n}:=\bigoplus_{m<n}V^{m}.
\end{equation}
A {\em subsuperspace} $W\subseteq V$ is a $\k$-submodule such that $W=\bigoplus_{n,\eps}(W\cap V_{n,\eps})$.

If each $V^n_\eps$ is a free $\k$-module of finite rank and $V^n=0$ for $n\ll0$, the {\em graded superdimension of $V$} is defined to be 
$$
\dim_{q,\pi}V:=\sum_{n\in\Z,\,\eps\in\Z/2}(\rank_\k V^n_\eps)q^n \pi^\eps\in\Z^\pi((q)).
$$

The {\em graded dual superspace} is $V^*:= \bigoplus_{n\in\Z,\,\eps\in\Z/2}(V^n_\eps)^*$ where the elements of $(V^n_\eps)^*=\Hom_\k(V^n_\eps,\k)$ are considered as elements of bidegree $(-n,\eps)$. Recalling the bar-involution from (\ref{EBarInv}), we have that if  $V$ is a free $\k$-module of finite rank then
\begin{equation}\label{EDimDual}
\dim_{q,\pi}V^*=\overline{\dim_{q,\pi}V}.
\end{equation}

Let $V,W$ be graded superspaces. The tensor product $V\otimes W$ is considered as a graded superspace via $\bideg(v\otimes w)=(\deg(v)+\deg(w),|v|+|w|)$ for all homogeneous $v\in V$ and $w\in W$.
For $m\in\Z$ and $\de\in\Z/2$, a bidegree $(m,\de)$ (homogeneous) linear map $f:V\to W$ is a $\k$-linear map satisfying $f(V^n_\eps)\subseteq W^{n+m}_{\eps+\de}$ for all $n,\eps$. We denote the $\k$-module of all bidegree $(m,\de)$  linear maps  from $V$ to $W$ by $\Hom_\k(V,W)^m_\de$, and set 
$$
\Hom_\k(V,W):=\bigoplus_{m\in\Z,\,\de\in\Z/2}\Hom_\k(V,W)^m_\de.
$$

Let $V$ be a graded superspace and $d\in\N_+$.
The symmetric group $\Si_d$ acts on $V^{\otimes d}$ via
\begin{equation}\label{EWSignAction}
{}^{w}(v_1\otimes\dots\otimes v_d):=(-1)^{[w;v_1,\dots,v_d]}v_{w^{-1}(1)}\otimes\dots\otimes v_{w^{-1}(d)},
\end{equation}
where
$$
[w;v_1,\dots,v_d]:=\sum_{1\leq a<c\leq d,\, w^{-1}(a)>w^{-1}(c)}|v_a||v_c|.
$$

A graded ($\k$-){\em superalgebra} is a graded superspace $A=\bigoplus_{n\in \Z,\,\eps\in\Z/2}A^n_\eps$ which is a unital algebra such that $A^n_\eps A_m^\de\subseteq A_{n+m}^{\eps+\de}$. The {\em opposite superalgebra}\, $A^\sop$ is the same as $A$ as a $\k$-module but has multiplication $a*b=(-1)^{|a||b|}ba$. 

Let $A,B$ be graded superalgebras. An {\em isomorphism} $f:A \to B$ is an algebra isomorphism which is homogeneous of bidegree $(0,\0)$. We write $A\cong B$ if $A$ and $B$ are isomorphic in this sense. The tensor product $A\otimes B$ is considered as a graded superalgebra via 
$
(a\otimes b)(a'\otimes b')=(-1)^{|b| |a'|}aa'\otimes bb'
$
for all homogeneous $a,a'\in A$ and $b,b'\in B$. 

Let $A$ be a superalgebra and $d\in\N_+$. 
For $1\leq t\leq d$, $1\leq r<s\leq d$, we have maps 
\begin{align*}
\iota^{(d)}_{t}&:A\to A^{\otimes d},\ a\mapsto 1_A^{\otimes (t-1)}\otimes a\otimes 1_A^{\otimes(d-t)},\\
\iota^{(d)}_{r,s}&:A\otimes A\to A^{\otimes d},\ a\otimes b\mapsto 1_A^{\otimes (r-1)}\otimes a\otimes 1_A^{\otimes(s-r-1)}\otimes b\otimes 1_A^{\otimes(d-s)}.
\end{align*}
For $x\in A$ and $y\in A\otimes A$, we sometimes denote 
\begin{equation}\label{EInsertion}
x_{r}:= \iota^{(d)}_r(x)\in A^{\otimes d},\quad y_{r,s}:=\iota^{(d)}_{r,s}(y)\in A^{\otimes d}.
\end{equation}

Given a graded superalgebra $A$, we consider the {\em wreath superproduct} $A\swr \Si_d$, where $A\swr \Si_d=A^{\otimes d}\otimes \F\Si_d$ as graded superspaces, with the group algebra $\F\Si_d$ always considered as a graded superalgebra in the trivial way, i.e. concentrated in bidegree $(0,\0)$. To define the algebra structure, we identify $A^{\otimes d}$ and $\F\Si_d$ with subspaces of $A^{\otimes d}\otimes \F\Si_d$ in the obvious way, and postulate that $A^{\otimes d}$, $\F\Si_d$ are subalgebras of  $A\swr \Si_d$; then recalling (\ref{EWSignAction}),  we also  postulate that 
\begin{align*}
w\,(a_1\otimes\dots\otimes a_d)={}^w(a_1\otimes\dots\otimes a_d)\,w\qquad(w\in\Si_d,\ a_1,\dots, a_d\in A).
\end{align*}

\subsection{Graded supermodules} Let $A$ be a graded superalgebra. 
A {\em graded $A$-supermodule} $V$ is an $A$-module which is also a graded superspace such that $A^n_\eps V^{m}_\de\subseteq V^{n+m}_{\eps+\de}$ for all $n,m,\eps,\de$. 
A graded subsupermodule $W\subseteq V$ is a graded subsuperspace which is also an $A$-submodule.

Let $V,W$ be graded $A$-supermodules. A 
bidegree $(m,\de)$ (homogeneous) graded $A$-supermodule homomorphism from $V$ to $W$ is $\k$-linear map $f:V\to W$ of a bidegree $(m,\de)$ satisfying $f(av)=(-1)^{\de|a|}af(v)$ for all (homogeneous) $a\in A,v\in V$. We denote by $\Hom_A(V,W)^{m}_{\de}$ the $\k$-module of all bidegree $(m,\de)$ graded $A$-supermodule homomorphism from $V$ to $W$, and set 
$$
\Hom_A(V,W):=\bigoplus_{m\in\Z,\,\de\in\Z/2}\Hom_A(V,W)^{m}_{\de}.
$$
We refer to the elements of $\Hom_A(V,W)$ as 
graded $A$-supermodule homomorphisms from $V$ to $W$. 

For a graded superalgebra $A$, we denote by 
$\Mod{A}$ the category of all 
 graded $A$-supermodules and all graded $A$-supermodule homomorphisms. 
An isomorphism in $\Mod{A}$ will be denoted $\cong$. On the other hand, a homogeneous bidegree $(0,\0)$ isomorphism in $\Mod{A}$ will be denoted $\simeq$. 
In this paper, we assume that all graded superalgebras $A$ are noetherian (as graded superalgebras). 
For such $A$, we usually work in the full subcategory 
$\mod{A}$ of $\Mod{A}$ which consists of all 
{\em finitely generated}\, graded $A$-supermodules. 

In fact, $\mod{A}$ is a graded $(\funQ,\Uppi)$-supercategory in the sense of \cite[Definition 6.4]{BE}, with $\Uppi$ the parity change functor and $\funQ$ the degree shift functor. To be more precise, if $V=\bigoplus_{n,\eps} V^n_\eps\in \mod{A}$, we have $(\Uppi V)^n_\eps=V^n_{\eps+\1}$ with the new action $a\cdot v=(-1)^{|a|}av$, and $(\funQ V)^n_\eps=V^{n-1}_\eps$ (with the old action). The following notation will be used. Suppose that $f=\sum_{n\in\Z,\eps\in\Z/2}a_{n,\eps}q^n\pi^\eps\in\Z^\pi[q,q^{-1}]$ be a Laurent polynomial with the coefficients $a_{n,\eps}\in\N$. Given $V\in\mod{A}$, we interpret $V^{\oplus f}$ as follows:
\begin{equation}\label{EFancyDirectSum}
V^{\oplus f}:=\bigoplus_{n\in\Z,\eps\in\Z/2}\funQ^n\Uppi^\eps V^{\oplus a_{n,\eps}}.
\end{equation}
If $g$ is another such Laurent polynomial then
\begin{equation}\label{EFG}
(V^{\oplus f})^{\oplus g}\simeq V^{\oplus fg}.
\end{equation}

A graded $A$-supermodule $V$ is {\em irreducible}\, if it has exactly two graded subsupermodules: $0$ and $V$. In general, an irreducible graded $A$-supermodule $V$ may or may not be irreducible when considered as a usual $A$-module. If it is, $V$ is called {\em absolutely irreducible}. (The irreducible graded supermodules considered in this paper will all be absolutely irreducible, see Lemma~\ref{LTypeM}.)  
We denote by $\Irr(A)$ a complete and non-redundant set of irreducible graded $A$-supermodules, i.e. every irreducible graded $A$-supermodule is isomorphic to some member of $\Irr(A)$ and $V\not\cong W$ for any any two distinct  $V,W\in\Irr(A)$.

Let $A,B$ be graded superalgebras. Given a graded $A$-supermodule $V$ and a graded $B$-supermodule $W$, we have the graded $(A\otimes B)$-supermodule $V\boxtimes W$ with the action 
$$
(a\otimes b)(v\otimes w)=(-1)^{|b||v|}(av\otimes bw)\qquad(a\in A,\, b\in B,\, v\in V,\, w\in W).
$$

A graded {\em $(A,B)$-bisupermodule}  $V$ is an $(A,B)$-bimodule which is also a graded superspace such that $A^n_\eps V^{m}_{\de}\subseteq V^{n+m}_{\eps+\de}$ and $V^{m}_{\de}B^n_\eps\subseteq V^{n+m}_{\eps+\de}$ for all $n,m,\eps,\de$. 
A homomorphism of graded $(A,B)$-bisupermodules is defined similarly to a homomorphism of graded $A$-supermodules. In particular, a bidegree $(m,\de)$ homomorphism of graded $(A,B)$-bisupermodules $f:V\to W$ satisfies $f(avb)=(-1)^{\de|a|}af(v)b$ for all (homogeneous) $a\in A,\,v\in V,\,b\in B$. In fact, we can identify the notions of a graded $(A,B)$-bisupermodule and a graded $(A \otimes B^{\sop})$-supermodule and then a homomorphism of graded $(A,B)$-bisupermodules is the same thing as a homomorphism of graded $(A\otimes B^\sop)$-supermodules, see e.g. \cite[\S2.2b]{KlLi}. 
Just like for graded supermodules, we use the notation $\cong$ to indicate an isomorphism of graded bisupermodules, and  $\simeq$ to indicate a homogeneous bidegree $(0,\0)$ isomorphism of graded bisupermodules.  

\begin{Example} \label{ExEndSopBim}
{\rm 
Given a graded $A$-supermodule $M$, we have the graded superalgebra $\End_A(M)^\sop$, and $M$ is naturally a graded $(A,\End_A(M)^\sop)$-bisupermodule with $f\in \End_A(M)^\sop=\End_A(M)$ acting via
$m\cdot f=(-1)^{|f||m|}f(m)$. 
}
\end{Example}

\subsection{Induction and restriction for graded supermodules}
Let $A$ be a graded $\k$-superalgebra with a bidegree $(0,\0)$ idempotent $e$. Suppose that $B$ is a graded subsuperalgebra of $eAe$. Then $eA$ is a graded $(B,A)$-bisupermodule and $Ae$ is a graded $(A,B)$-bisupermodule, so we have the usual functors of restriction, induction and coinduction
\begin{align}
\label{EResGeneral}
\Res^A_B&: \Mod{A}\to\Mod{B},\ V\mapsto eA\otimes_A V,
\\
\label{EIndGeneral}
\Ind^A_B&: \Mod{B}\to\Mod{A},\ W\mapsto Ae\otimes_B W,
\\
\label{ECoindGeneral}
\Coind^A_B&: \Mod{B}\to\Mod{A},\ W\mapsto \Hom_B(eA,W).
\end{align}
The functor $\Ind_{B}^A$ is left adjoint to $\Res_{B}^A$ and the functor $\Coind_{B}^A
$ is right adjoint to $\Res_{B}^A$. Note that the definition of the structure of a graded $A$-supermodule on $\Coind_{B}^A W
= \Hom_{B}(eA, W)$ involves sign: $(a\cdot f)(ea')=(-1)^{|a|(|f|+|a'|)}(ea'a)$ for $a,a'\in A$ and $f\in \Hom_{B}(eA, W)$. 

We will need a criterion for an isomorphism $\Ind_{B}^A V\cong \Coind_{B}^A V$. 

\begin{Lemma} \label{LIndCoindIsoGen}
Let $e,f$ be bidegree $(0,\0)$ idempotents in a graded $\k$-superalgebra $A$ and $B\subseteq eAe,\, C\subseteq fAf$ be graded subsuperalgebras. Let $V\in\Mod{B}$ and $W\in\Mod{C}$. 
We make the following assumptions:
\begin{enumerate}
\item[{\rm (1)}] The right $B$-module $Ae$ has basis $\{b_1,\dots,b_n\}$ and the left $C$-module $fA$ has basis $\{c_1,\dots,c_n\}$.
\item[{\rm (2)}] There is an isomorphism of graded $\k$-superalgebras $i:B\iso C$ and a degree $(0,\0)$-isomorphism $j:V\to W$ of graded $\k$-supermodules such that $j(bv)=i(b)j(v)$ for all $b\in B$ and $v\in V$. 
\item[{\rm (3)}] The graded $(C,B)$-bisupermodule $fAe$ has a graded subsuperbimodule $Y$ and a homogeneous element $\xi$ such that $\xi b\equiv (-1)^{|b||\xi |}i(b)\xi \pmod{Y}$ and the map $C\to (fAe)/Y,\ c\mapsto c\xi +Y$ is an isomorphism of graded left $C$-supermodules. 
\item[{\rm (4)}] There is a permutation $\si\in\Si_n$ such that for $1\leq r,s\leq n$, we have $c_sb_r\equiv \de_{\si(r),s}\xi \pmod{Y}$. 
\end{enumerate}
Then there is an isomorphism of graded $A$-supermodules $$\Ind_B^AV\simeq \funQ^{\deg(\xi )}\Uppi^{|\xi |}\Coind_C^AW.$$
\end{Lemma}
\begin{proof}
We have to construct a graded $A$-supermodule isomorphism $\Ind_B^AV\to \Coind_C^AW$ of bidegree $(-\deg(\xi ),|\xi |)$. This will come from a bidegree $(-\deg(\xi ),|\xi |)$ homomorphism of graded $B$-supermodules $V\to \Res^A_B\Coind_C^AW$ using adjointness of $\Ind_B^A$ and $\Res_B^A$. Note that $\Res^A_B\Coind_C^AW
\simeq e\Hom_{C}(fA, W)
\simeq \Hom_{C}(fAe, W).$
By the assumption (3), every element of $fAe$ can written uniquely in the form $c\xi +y$ for $c\in C$ and $y\in Y$. So for every $v\in V$, we have a graded $C$-supermodule homomorphism
$$
F_v:fAe\to W,\ c\xi +y\mapsto (-1)^{|c||v|+|c||\xi|}c\,j(v).
$$
Since $\bideg(F_v)=(\deg(v)-\deg(\xi ),|v|+|\xi |)$, we have a bidegree $(-\deg(\xi ),|\xi |)$ homomorphism of graded $\k$-supermodules 
$$\al:V\to  \Hom_{C}(fAe, W),\ v\mapsto (-1)^{|v||\xi|}F_v.$$ 
We claim that $\al$ is a homomorphism of graded $B$-supermodules. This means $\al(bv)=(-1)^{|b||\xi|}b\al(v)$ for $b\in B$ and $v\in V$, or $F_{bv}=bF_v$. Now 
\begin{align*}
(b F_v)(c\xi +y&)=(-1)^{|b|(|\xi |+|v|+|c|+|\xi |)}F_v((c\xi +y)b)
\\
&=(-1)^{|b|(|v|+|c|)}F_v(c\xi b+yb)
\\
&=(-1)^{|b|(|v|+|c|)+|\xi ||b|}F_v(ci(b)\xi +yb)
\\
&=(-1)^{|b|(|v|+|c|)+|\xi ||b|+(|c|+|b|)|v|+(|c|+|b|)|\xi|}ci(b)j(v)
\\
&=(-1)^{|b||c|+|c||v|+|c||\xi|}ci(b)j(v)
\\
&=(-1)^{|c||bv|+|c||\xi|}cj(bv)
\\
&=F_{bv}(c\xi +y).
\end{align*}

Now, adjointness of $\Ind_B^A$ and $\Res_B^A$ yields a a bidegree $(-\deg(\xi ),|\xi |)$ homomorphism of graded $A$-supermodules
$$
F:Ae\otimes_B V\to \Hom_{C}(fA, W),\ ae\otimes v\mapsto F_{ae,v},
$$
where $F_{ae,v}$ is defined from 
$$
F_{ae,v}(fa')=(-1)^{|a|(|\xi |+|v|+|a'|)}F_v(fa'ae).
$$

By the assumption (1), we have an isomorphism of graded $\k$-supermodules 
\begin{equation}\label{EIndDec}
Ae\otimes_B V\simeq \bigoplus_{r=1}^n(b_r\otimes V)\cong\bigoplus_{r=1}^n V,
\end{equation}
where the second isomorphism comes from the isomorphisms 
$V\to b_r\otimes V,\ v\mapsto b_r\otimes v$. Similarly, we have an isomorphism of graded $\k$-supermodules 
\begin{equation}\label{ECondDec}
\Hom_{C}(fA, W)\simeq \bigoplus_{s=1}^n\Hom_C(Cc_s,W)\cong \bigoplus_{s=1}^n W,
\end{equation}
with the second isomorphism coming from the isomorphisms 
$$\phi_s:W\to \Hom_C(Cc_s,W),\ w\mapsto (cc_s\mapsto (-1)^{|c||c_s|+|c||v|+|c_s||v|}cw).
$$ 

Now, by the assumption (4), we have 
$$
F_{b_r,v}(c_s)=\pm F_v(c_sb_r)=\pm\de_{\si(r),s}j(v).
$$
Thus $F$ maps the vector $b_r\otimes v$ from the $r$th summand of (\ref{EIndDec}) onto $\pm\phi_{\si(r)}(j(v))$ from the $\si(r)$th summand of (\ref{ECondDec}). Hence $F$ is an isomorphism. 
\end{proof}

\begin{Example} \label{ENatural}
{\rm 
Suppose that the graded superalgebra $B$ in Lemma~\ref{LIndCoindIsoGen} is of the form $X\otimes Y$ for two graded superalgebras $X$ and $Y$, and $C=Y\otimes X$. Moreover, suppose that $V$ is of the form $M\boxtimes N$ for a graded $X$-supermodule $M$ and a graded $Y$-supermodule $N$, and $W=N\boxtimes M$.  Then we can take $i:X\otimes Y\to Y\otimes X,\ x\otimes y\mapsto (-1)^{|x||y|}y\otimes x$ and $j:M\otimes N\to N\otimes M,\ m\otimes n\mapsto (-1)^{|m||n|}n\otimes m$. If all the other assumptions of the lemma are satisfied, we will get an isomorphism 
$$\Ind_{X\otimes Y}^A M\boxtimes N\simeq \Coind_{Y\otimes X}^AN\boxtimes M.$$
Checking the explicit construction of this isomorphism, one can see that the isomorphism is natural in $M$ and $N$. 
}
\end{Example}

\subsection{Graded Morita superequivalence}
\label{SSMorita}
For graded $\k$-superalgebras $A$ and $B$, a {\em graded Morita superequivalence} between $A$ and $B$ is a Morita equivalence between $A$ and $B$ induced by a graded $(A,B)$-bisupermodule $M$ and a graded $(B,A)$-bisupermodule $N$ such that $M\otimes_B N\simeq A$ and $N\otimes_A M\simeq B$. (with both bimodule isomorphisms homogeneous of bidegree $(0,\0)$). 
A standard argument shows  that $A$ and $B$ are graded Morita superequivalent if and only if the graded supercategories $\mod{A}$ and $\mod{B}$ are graded superequivalent, i.e. there are functors $\funF: \mod{A}\to\mod{B}$ and $\funG:\mod{B}\to \mod{A}$ such that $\funF\circ\funG$ and $\funG\circ \funF$ are isomorphic to identities via graded supernatural transformations of bidegrees $(0,\0)$, cf.~\cite[Deinition 1.1(iv)]{BE}. 

\begin{Lemma}\label{lem:idmpt_Mor} {\rm \cite[Lemma 2.2.14]{KlLi}}
Let $A$ be a graded $\F$-superalgebra and $e\in A^0_{\0}$ an idempotent such that $eL\neq 0$ for every irreducible graded $A$-supermodule $L$ (equivalently $AeA=A$). Then the functors $eA \otimes_A -$ and $Ae \otimes_{eAe} -$ induce a graded Morita superequivalence between $A$ and $eAe$.
\end{Lemma}

\begin{Lemma}\label{cor:idmpt_Mor} {\rm \cite[Corollary 2.2.15]{KlLi}}
Let $A$ be a graded $\F$-superalgebra and $e\in A^0_{\0}$ an idempotent. 
If\, $|\Irr(eAe)|\geq |\Irr(A)|<\infty$ then the functors $eA \otimes_A -$ and $Ae \otimes_{eAe} -$ induce a graded Morita superequivalence between $A$ and $eAe$.
\end{Lemma}

From now on until the end of the subsection, we assume that $A$ is a graded 
$\O$-superalgebra with $A^n_\eps$ finitely generated as an $\O$-module for all $n$ and $\eps$. Suppose that $\F=\O/\m$ for a maximal ideal $\m$ in $\O$. We can extend scalars to get a graded $\F$-superalgebra $A_\F:=\F\otimes_\O A$. Given $a\in A$, we denote $a_\F:=1\otimes a\in A_\F$. If $V$ is a graded $A$-supermodule then $V_\F:=\F\otimes_\O V$ is a graded $A_\F$-supermodule (similarly for graded bisupermodules).

If $e\in A^{0}_{\0}$ is an idempotent, we have the functors
\begin{equation}\label{EFunctors}
\begin{split}
Ae\otimes_{eAe}-:\mod{eAe}\to\mod{A},\\ 
eA\otimes_{A}-:\mod{A}\to \mod{eAe}.
\end{split}
\end{equation}
and
\begin{equation}\label{EFunctorsF}
\begin{split}
A_\F e_\F\otimes_{e_\F A_\F e_\F}-:\mod{e_\F A_\F e_\F}\to\mod{A_\F},\\ e_\F A_\F\otimes_{A_\F}-:\mod{A_\F}\to \mod{e_\F A_\F e_\F }.
\end{split}
\end{equation}
Note that $eAe$ is an $\O$-direct summand of $A$, so we can identify the graded $\F$-superalgebra $(eAe)_\F$ with the graded $\F$-superalgebra $e_\F A_\F e_\F$. Now, we can also identify as graded bisupermodules $(Ae)_\F=A_\F e_\F$ and $(eA)_\F=e_\F A_\F$. Moreover, for $V\in \mod{eAe}$ and $W\in \mod{A}$, we have functorial isomorphisms
\begin{equation}
\label{EFunExtendScal}
\begin{split}
(Ae\otimes_{eAe}V)_\F\simeq A_\F e_\F\otimes_{e_\F A_\F e_\F}V_\F,
\\
(eA\otimes_{A}W)_\F\simeq e_\F A_\F \otimes_{A_\F}W_\F.
\end{split}
\end{equation}

\begin{Lemma} \label{LMorExtScal} 
If for all fields\, $\F\in{\mathscr F}=\{\O/\m\mid \text{$\m$ is a maximal ideal of $\O$}\}
$, the functors (\ref{EFunctorsF}) are mutually quasi-inverse equivalences, then so are the functors (\ref{EFunctors}). \end{Lemma}
\begin{proof}
We always have an isomorphism of graded $eAe$-superbimodules $eA\otimes_A Ae\simeq eAe$, so it suffices to prove that the multiplication map $\mu:Ae\otimes_{eAe} eA\to A$ is an isomorphism. In view of \cite[Statement 7]{DR}, it suffices to prove that $\mu$ is surjective, i.e. $AeA=A$, given the assumption that $A_\F e_\F A_\F=A_\F$ for all $\F\in{\mathscr F}$. 
 The short exact sequence of $\O$-modules $0\to AeA\stackrel{\iota}{\to} A\to A/AeA\to 0$ yields an exact sequence of $\F$-vector spaces $(AeA)_\F\stackrel{\iota_\F}{\to} A_\F\to (A/AeA)_\F\to 0$. To see that $A/AeA= 0$ it suffices to prove $(A/AeA)_\F=0$ for all $\F\in{\mathscr F}$, or, equivalently, that $\iota_\F$ is surjective for all $\F\in{\mathscr F}$. But it is easy to see that $\Im \iota_\F=A_\F e_\F A_\F$ for all $\F$, and $A_\F e_\F A_\F=A_\F$ by assumption.
\end{proof}

\begin{Example} \label{ExProgenerator} 
{\rm 
Let $A$ be a graded superalgebra and $\{e_1,\dots,e_n\}$ be (not necessarily distinct) bidegree $(0,\0)$ idempotents of $A$. 
We have a graded superalgebra $\bigoplus_{1\leq i,j\leq n}e_iAe_j$ with multiplication inherited from $A$. In fact, setting 
$$
P:=\bigoplus_{i=1}^n Ae_i,
$$
there the isomorphism of graded superalgebras
\begin{equation}\label{EEndAei}
\bigoplus_{1\leq i,j\leq n}e_iAe_j\iso \End_A(P)^\sop,\ a\mapsto\phi_a,
\end{equation}
where for (a homogeneous) $a\in e_iAe_j$, the endomorphism $\phi_a$ sends the summand $Ae_k$ of $P$ to $0$ if $k\neq i$ and the summand $Ae_i$ to the summand $Ae_j$ as follows:
$$
\phi_a(be_i)=(-1)^{|a||b|}be_ia
$$ 
(for a homogeneous $b\in A$). If $P$ is a projective generator for $A$ (equivalently, if $\sum_{i=1}^nAe_iA=A$) then $A$ is graded Morita superequivalent to $\End_A(P)^\sop$ using the graded $(A,\End_A(P)^\sop)$-superbimodule $M$ being $P$ considered as a graded $(A,\End_A(P)^\sop)$-superbimodule as in Example~\ref{ExEndSopBim}. 
}
\end{Example}

\begin{Example} 
\label{SSRegr}
An easy special case of a graded Morita superequivalence comes from a `regrading' of a graded superalgebra in the following sense. Suppose that $A$ is a graded superalgebra and we are given an orthogonal idempotent decomposition $1_A=\sum_{r=1}^n e_r$ with the idempotents $e_r$ of bidegree $(0,\0)$. Then as graded superspaces $A=\bigoplus_{1\leq r,s\leq n}e_sAe_r$. Given {\em grading supershift parameters}
\begin{equation}\label{EShiftParameters}
t_1,\dots,t_n\in \Z\quad\text{and}\quad \eps_1,\dots,\eps_n\in\Z/2, 
\end{equation}
we can consider the new graded superalgebra 
$$\zA:=\bigoplus_{1\leq r,s\leq n}\funQ^{t_r-t_s}\Uppi^{\eps_r-\eps_s} e_sAe_r,$$ 
which equals $A$ as an algebra but has  different grading and parity. To distinguish between the elements of $A$ and $\zA$, we denote by $\za$ the element of $\zA$ corresponding to $a\in A$. In particular, we have idempotents $\ze_1,\dots,\ze_n\in\zA$. 
We have the graded $(\zA,A)$-bisupermodule
$$X:=\bigoplus_{r=1}^n \funQ^{-t_r}\Uppi^{-\eps_r}e_r A=
\bigoplus_{s=1}^n \funQ^{-t_s}\Uppi^{-\eps_s} \zA\ze_s
,$$
and the graded $(A,\zA)$-bisupermodule 
$$
Y:=\bigoplus_{s=1}^n \funQ^{t_s}\Uppi^{\eps_s} Ae_s=\bigoplus_{r=1}^n \funQ^{t_r}\Uppi^{\eps_r} \ze_r \zA,
$$
inducing a graded Morita superequivalence between $A$ and $\zA$. Under this graded Morita superequivalence, a graded $A$-supermodule $V$, corresponds to the graded $\zA$-supermodule 
\begin{equation}\label{EzVMor}
\zV:=X\otimes_A V\simeq \bigoplus_{r=1}^n \funQ^{-t_r}\Uppi^{-\eps_r}e_rV.
\end{equation}

\end{Example}

\subsection{Schur functors}\label{SSSF}
 Here we adapt the material of \cite[Section~ 3.1]{BDK} to the graded superalgebra setting. Let $A$ be a finite dimensional graded $\F$-superalgebra and $P\in\mod{A}$ be a projective graded $A$-supermodule, i.e. a direct summand of a finite direct sum of modules of the form $\funQ^m\Uppi^\de A$. 
 Let $H = \End_A(P)^\sop$, so 
$P$ can be considered as a graded $(A,P)$-bisupermodule with the right action $p\cdot h=(-1)^{|p||h|}h(p)$ for all $p\in P$ and $h\in H$. (In all applications later on we will always have $H$ purely even, so the issue of signs will not arise, but we do not need to make this assupmtion in this subsection.) 
 
Define the functors: 
\begin{align*}
 \funf&:=\Hom_{A}(P,?):\mod{A}\to \mod{H},
 \\
 \fung&:=P\otimes_H?:\mod{H}\to \mod{A}.
 \end{align*}
Then $\funf$ is exact, and $\fung$ is left adjoint to $\funf$, the adjunction isomorphism being 
\begin{align*}
\Hom_A(P\otimes_H W,V)&\iso\Hom_H(W,\Hom_A(P,V)),\\ \phi&\mapsto \big(w\mapsto(p\mapsto(-1)^{|p||w|}\phi(p\otimes w))\big).
\end{align*}

Given an $A$-module $V$, let $O_P(V )$ denote the largest graded subsupermodule $V'$ of $V$  such that $\Hom_A(P,V')=0$. Let $O^P(V)$ denote the graded subsupermodule of $V$ generated by the images of all $A$-homomorphisms from $P$ to $V$. 
Any graded $A$-supermodule homomorphism $V \to W$ sends $O_P(V )$ into $O_P(W )$ and $O^P(V)$ into $O^P(W)$, so we can view $O_P$ and $O^P$ as functors $\mod{A}\to\mod{A}$. Finally, any graded $A$-supermodule homomorphism $V \to W$ induces a well-defined graded $A$-supermodule homomorphism $V/O_P (V ) \to W/O_P (W)$. We thus obtain an exact functor 
$$\funq : \mod{A} \to \mod{A}$$ defined on objects by $V \mapsto V/O_P(V)$.

\begin{Lemma} \label{L3.1a}
The functors $\funf\circ \fung$ and $\funf\circ \funq\circ\fung$ are both isomorphic to the identity.
\end{Lemma}
\begin{proof}
The proof given in \cite[3.1a]{BDK} goes through using the modified version of \cite[Proposition 20.10]{AF}, in which the isomorphism $$\Hom_A(P,P)\otimes_H W\iso \Hom_A(P,P\otimes_H W)$$
involves a sign, namely $h\otimes w\mapsto (p\mapsto (-1)^{|w||p|}h(p)\otimes w)$. 
\end{proof}

\begin{Lemma} \label{L3.1c} 
If\, $V, V' \in \mod{A}$ satisfy\, $O^P(V) = V$ and\, $O_P(V') = 0$, then 
$$
\Hom_{A}(V , V' ) \simeq \Hom_H(\funf( V ) , \funf( V' ) ).
$$
\end{Lemma}
\begin{proof}
The proof given in  \cite[3.1c]{BDK} goes through using the properties of the modified version of the map 
$$
\om:\fung\circ\funf(V)=P\otimes_H\Hom_A(P,V)\to V, \ p\otimes\phi\mapsto (-1)^{|\phi||p|}\phi(p),
$$
cf. \cite[3.1b]{BDK}.
\end{proof}

\begin{Theorem} \label{T3.1d}
The functors $\funf$ and\, $\funq \circ\fung$ induce mutually quasi-inverse graded superequivalences between $\mod{H}$ and the full graded subsupercategory 
of $\mod{A}$ consisting of all\, $V \in\mod{A}$ such that $O_P(V)=0$ and\, $O^P(V)=V$.
\end{Theorem}
\begin{proof}
The proof of \cite[3.1d]{BDK} goes through. 
\end{proof}

As a consequence we have the following relation between the irreducible modules, see \cite[3.1e]{BDK}:

\begin{Lemma} \label{L3.1e} 
Let $\{L_r \mid r \in R\}$ be a complete set of non-isomorphic irreducible graded $A$-supermodules appearing in the head of $P$. For all $r\in R$, set $D_r := \funf(L_r)$. Then, $\{D_r \mid r\in R\}=\Irr(H)$, and $\funq \circ \fung(D_r) \cong L_r$ for all $r$.
\end{Lemma}

Finally, we have the following explicit description of the effect of the composite functor $\funq \circ \fung$ on left ideals of $H$, see \cite[3.1f]{BDK}:

\begin{Lemma} \label{L3.1f}
If every composition factor of the socle of $P$ also appears in its head then for any graded left superideal $J$ of $H$, we have $\funq \circ \fung(J) \simeq PJ$.
\end{Lemma}

\chapter{Some important (super)algebras}
\section{Classical Schur algebras}

Algebras and modules appearing in this section will be considered as graded superalgebras and graded supermodules in the trivial way, i.e. concentrated in bidegree $(0,\0)$. 

\subsection{Permutation modules over symmetric groups}
\label{SSPermMod}

The trivial $\k \Si_\la$-module is denoted $\k_{\Si_\la}$, and the sign $\k\Si_\la$-module is denoted $\sgn_{\Si_\la}$. Sometimes we use the {\em right} trivial $\k \Si_\la$-module, which is again denoted $\k_{\Si_\la}$. 
We have the {\em permutation module} and the {\em sign permutation module} 
\begin{equation}\label{EPermSignPerm}
M^\la:=\Ind_{\Si_\la}^{\Si_d}\k_{\Si_\la}\simeq \k\Si_d\ttx_\la
\quad\text{and}\quad 
N^\la:=\Ind_{\Si_\la}^{\Si_d}\sgn_{\Si_\la}\simeq \k\Si_d\tty_\la
\end{equation}
with generators $m^\la:=1_{\Si_d}\otimes 1_\k$ and $n^\la:=1_{\Si_d}\otimes 1_\k$ and bases 
$\{gm^\la \mid g\in \D_d^\la\}$ and $\{gn^\la \mid g\in \D_d^\la\}$, respectively. 

For $\la,\mu\in\Comp( d)$ and $g\in\Si_d$, define the elements 
\begin{equation}\label{EX}
X_{\mu,\la}^g:=\sum_{w\in \Si_\mu g\Si_\la\cap\,{}^\mu\D}w
\qquad\text{and}\qquad
Y_{\mu,\la}^g:=\sum_{w\in \Si_\mu g\Si_\la\cap\,{}^\mu\D}\sgn(w)w
\end{equation}
of $\k\Si_d$. Then
\begin{equation}
\label{ECleverX}
\sum_{w\in\Si_\mu g\Si_\la}w=\ttx_\mu\cdot X_{\mu,\la}^g=\Big(\sum_{w\in \Si_\mu g\Si_\la\cap\D^\la}w\Big)\cdot\ttx_\la,
\end{equation}
\begin{equation}
\label{ECleverY}
\sum_{w\in\Si_\mu g\Si_\la}\sgn(w)w=\tty_\mu Y_{\mu,\la}^g
=\Big(\sum_{w\in \Si_\mu g\Si_\la\cap\D^\la}\sgn(w)w\Big)\cdot\tty_\la.
\end{equation}
So the right multiplication with the element
$
X_{\mu,\la}^g
$
defines a homomorphism
\begin{equation}\label{EXi}
\xi_{\la,\mu}^g:M^\mu\simeq \k\Si_d\ttx_\mu\to \k\Si_d\ttx_\la
\simeq M^\la,\ m^\mu\mapsto \sum_{w\in \Si_\mu g\Si_\la\cap\D^\la}wm^\la,
\end{equation}
and the right multiplication with the element
$
Y_{\mu,\la}^g$
defines a homomorphism
\begin{equation}\label{EEta}
\eta_{\la,\mu}^g:N^\mu\simeq \k\Si_d\tty_\mu\to \k\Si_d\tty_\la
\simeq N^\la,\ n^\mu\mapsto \sum_{w\in \Si_\mu g\Si_\la\cap\D^\la}\sgn(w)wn^\la
\end{equation}

The following is a well-known consequence of the Mackey Theorem:

\begin{Lemma} \label{LHomMM}
We have that 
$\{\xi_{\la,\mu}^g\mid g\in {}^\la\D_d^\mu\}$ is a basis of $\Hom_{\k\Si_d}(M^\mu,M^\la)$ and $\{\eta_{\la,\mu}^g\mid g\in {}^\la\D_d^\mu\}$ is a basis of $\Hom_{\k\Si_d}(N^\mu,N^\la)$.
\end{Lemma}

\subsection{Classical Schur algebras}
\label{SSSchur}
For general information on Schur algebras we refer the reader to \cite{Green}. The facts which are needed here are also collected in a compact form (and properly referenced) in \cite[\S1]{BDK}.  We recall some facts for reader's convenience. 
Let $d,n\in\N_+$. 
The corresponding {\em classical Schur algebra} is 
$$
S(n,d):=\End_{\k\Si_d}\big(\textstyle\bigoplus_{\nu\in\Comp(n,d)}M^\nu\big).
$$
($S(n,0)$ is interpreted as $\k$.)  
Recalling the homomorphism $\xi_{\la,\mu}^g:M^\mu\to M^\la$  from (\ref{EXi}), we denote its trivial extension to an endomorphism of $\bigoplus_{\nu\in\Comp(n,d)}M^\nu$ again by $\xi_{\la,\mu}^g$. Then Lemma~\ref{LHomMM} implies:

\begin{Lemma} \label{LSchurBasis}
The algebra $S(n,d)$ has basis $\{\xi^g_{\la,\mu}\mid \la,\mu\in\Comp(n,d),\, g\in {}^\la\D_d^\mu\}$. 
\end{Lemma} 


We can also use the right permutation modules $\k_{\Si_\nu}\otimes_{\k\Si_\nu}\k\Si_d\simeq\ttx_\nu(\k\Si_d)$ instead of the left permutation modules $M^\la$. Considering the anti-automorphism 
$i:\k\Si_d\to\k\Si_d,\ \sum_{w\in \Si_d}c_ww\mapsto \sum_{w\in \Si_d}c_ww^{-1},
$
note that $i(\ttx_\nu)=\ttx_\nu$ for all $\nu$. So there is an algebra isomorphism
\begin{equation}\label{E050924}
S(n,d)\iso \End_{\k\Si_d}\Big(\bigoplus_{\nu\in\Comp(n,d)}\ttx_\nu(\k\Si_d)\Big),
\end{equation}
where $\xi^g_{\mu,\nu}\in S(n,d)$ maps onto the endomorphism which is $0$ on the summands $\ttx_\nu(\k\Si_d)$ with $\nu\neq \la$ and the summand $\ttx_\la(\k\Si_d)$ is mapped to the summand $\ttx_\mu(\k\Si_d)$ by the right multiplication with $$i(X^g_{\mu,\la})=\sum_{w\in \Si_\mu g\Si_\la\cap\,{}^\mu\D}w^{-1}=\sum_{w\in \Si_\la g^{-1}\Si_\mu\cap\D^\mu}w,
$$
where we have used (\ref{EX}) and (\ref{ECleverX}).

We can also use the sign permutation modules $N^\nu$ instead of the permutation modules $M^\nu$. Recalling the homomorphism $\eta_{\la,\mu}^g:N^\mu\to N^\la$  from (\ref{EEta}), we denote its trivial extension to an endomorphism of $\bigoplus_{\nu\in\Comp(n,d)}N^\nu$ again by $\eta_{\la,\mu}^g$. 

\begin{Lemma} \label{LSchurSigned}
We have an isomorphism 
$$S(n,d)\simeq \End_{\k\Si_d}\big(\textstyle\bigoplus_{\nu\in\Comp(n,d)}N^\nu\big)
$$
under which each basis element $\xi_{\la,\mu}^g$ of $S(n,d)$ with $g\in {}^\la\D_d^\mu$ corresponds to $\eta_{\la,\mu}^g$. 
\end{Lemma} 

The case $q=1$ of \cite[Theorem 1.11]{DJ} yields (see also \cite[\S2.7]{Green}):

\begin{Lemma} \label{LSchurAnti} 
There is an anti-automorphism $\tau$ of $S(n,d)$ which maps each $\xi^g_{\la,\mu}$ to\, $\xi^{g^{-1}}_{\mu,\la}$. 
\end{Lemma}

Let $W$ be a (left) $S(n,d)$-module. 
Using the anti-automorphism $\tau$ from the lemma, we define the {\em contravariant dual}\, $W^\tau$ of $W$ to be the (left) $S(n,d)$-module $W^*:=\Hom_\k(W,\k)$ with action defined by $(s\cdot f)(w) = f(\tau(s)w)$ for all $s \in S(n,d),\, w \in W,\, 
f \in W^*$. 
We also define the right $S(n,d)$-module $\tilde W$ to be $W$ as a $\k$-module, with right action defined by $ws = \tau(s)w$ for $w \in M,\, s \in S(n,d)$.

Let $k\in\N_+$ and $(d_1,\dots,d_k)\in\Comp(k,d)$. As in \cite[(1.3.1)]{BDK}, there is an explicit algebra map 
$$S(n,d) \to S(n,d_1)\otimes\dots\otimes S(n,d_k).$$ 
So, given $S(n,d_t)$-modules $W_t$ for $t=1,\dots,k$, we can view the tensor product 
\begin{equation}\label{ESchurTensProd}
W_1\otimes \dots\otimes W_k 
\end{equation}
as an $S(n,d)$-module.

We have the orthogonal idempotents 
\begin{equation}\label{EIdSchur}
\xi_\la:=\xi^1_{\la,\la}\in S(n,d) \qquad(\la\in \Comp(n,d))
\end{equation}
summing to the identity in $S(n,d)$. 
So for any $S(n,d)$-module $W$, 
we have the {\em weight space decomposition} $
W=\bigoplus_{\la\in \Comp(n,d)} \xi_\la W.
$ 
Recall $\om_d$ from (\ref{EOmd}).

\begin{Lemma} \label{LKappa} \cite[(6.1d)]{Green}
If $n\geq d$ then there is an algebra isomorphism 
$$\kappa:\k\Si_d\to \xi_{\om_d}S(n,d)\xi_{\om_d},\ g\mapsto \xi^g_{\om_d,\om_d}.
$$
\end{Lemma}

For $N\geq n$ we define
$
\Comp(N,n;d):=\{\la\in\Comp(N,d)\mid\la_{n+1}=\dots=\la_N=0\}
$
and consider the idempotent
$$
\xi(N,n;d):=\sum_{\la\in\Comp(N,n;d)}\xi_\la\in S(N,d).
$$ 
Then there is a natural isomorphism 
\begin{equation}\label{ENnIso}
S(n,d)\iso \xi(N,n;d)S(N,d)\xi(N,n;d), 
\end{equation}
see \cite[\S6.5]{Green}. 
As a special case of (\ref{EResGeneral}) and (\ref{EIndGeneral}), we get the functors 
\begin{equation}\label{EFunctorsNn}
\Res^{S(N,d)}_{S(n,d)} \quad\text{and}\quad 
\Ind^{S(N,d)}_{S(n,d)}.
\end{equation}

Let $\ud=(d_1,\dots,d_k)\in\Comp(k,d)$ and $\un=(n_1,\dots,n_k)\in\Comp(k,n)$. We denote
$$
S(\un,\ud):=S(n_1,d_1)\otimes\dots\otimes S(n_k,d_k).
$$
For $t=1,\dots,k+1$, letting $m_t:=\sum_{s=1}^{t-1}n_s$, 
denote
\begin{equation}\label{EUNUD}
\Comp(\un,\ud):=\{\la\in\Comp(n,d)\mid \sum_{s=m_t+1}^{m_{t+1}}\la_s=d_t\ \text{for all}\ t=1,\dots,k\}.
\end{equation}
and define the idempotent 
\begin{equation}\label{EIdUnUd}
\xi(\un,\ud):=\sum_{\la\in\Comp(\un,\ud)}\xi_\la\in S(n,d).
\end{equation}
Then we can consider $S(\un,\ud)$ as an explicit (unital) subalgebra of the idempotent truncation $\xi(\un,\ud)S(n,d)\xi(\un,\ud)$, referred to as a {\em Levi subalgebra}, see \cite[(1.3g)]{BDK}:
\begin{equation}\label{ELeviEmb}
S(\un,\ud)\subseteq \xi(\un,\ud)S(n,d)\xi(\un,\ud).
\end{equation}
As a special case of (\ref{EResGeneral}) and (\ref{EIndGeneral}), we get the functors 
\begin{equation}\label{EResIndLevi}
\Res^{S(n,d)}_{S(\un,\ud)} \quad\text{and}\quad 
\Ind^{S(n,d)}_{S(\un,\ud)}. 
\end{equation}

\subsection{Representation theory of Schur algebras} \label{SSSchurRep}
Throughout this subsection we assume that $\k=\F$. 

The irreducible $S(n,d)$-modules are parametrized by the elements of  $\Comp_+(n,d)$. We write $L_{n,d}(\la)$ for the irreducible $S(n,d)$-module corresponding to $\la\in \Comp_+(n,d)$. In particular, $L_{n,d}(\la)$ has highest weight $\la$, i.e. $\xi_\la L_{n,d}(\la)\neq 0$ and $\xi_\mu L_{n,d}(\la)= 0$ for $\mu\in \Comp(n,d)$ unless $\mu\unlhd\la$. We have for the contravariant duals $L_{n,d}(\la)^\tau\simeq L_{n,d}(\la)$ for all $\la\in  \Comp_+(n,d)$. 

It is well-known that $S(n,d)$ is a quasi-hereditary algebra with weight poset $(\Comp_+(n,d),\unlhd)$. 
In particular, associated to $\la\in \Comp_+(n,d)$ we have the standard module $\De_{n,d}(\la)$ (resp. costandard module $\nabla_{n,d}(\la)$) such that $\De_{n,d}(\la)$ (resp. $\nabla_{n,d}(\la)$) has simple head (resp. socle) isomorphic to $L_{n,d}(\la)$, and all other composition factors are of the form $L_{n,d}(\mu)$ with $\mu\lhd\la$. We have  $\De_{n,d}(\la)^\tau\simeq \nabla_{n,d}(\la)$
for all $\la\in  \Comp_+(n,d)$. 

Recalling the notation $\tilde W$ from the previous subsection, we also have the right modules $\tilde L_{n,d}(\la), \tilde\De_{n,d}(\la)$ and $\tilde\nabla_{n,d}(\la)$ for each $\la\in \Comp_+(n,d)$. The following general fact on quasi-hereditary algebras is well-known, cf. \cite[Theorem~7.1(iii)]{GreenComb}:

\begin{Lemma} \label{1.2e} 
As	an $(S(n,d),S(n,d))$-bimodule, 
$S(n,d)$	has	a	filtration	 with factors	isomorphic to $\De_{n,d}(\la) \otimes \tilde\De_{n,d}(\la)$, each appearing once for each $\la\in \Comp_+(n,d)$ and ordered in any way refining the dominance order on partitions so that factors corresponding to more dominant $\la$ appear lower in the filtration.
\end{Lemma}

A finite dimensional $S(n,d)$-module $W$ has a {\em standard} (resp.  {\em costandard}) filtration if $W$  has a filtration $0 = W_0 \subset W_1 \subset \dots\subset W_r=W$ such that each sub-quotient $W_s/W_{s-1}$ is isomorphic to a standard (resp. costandard) module. 
The following is well-known:

\begin{Lemma} \label{1.3c} 
Let $(d_1,\dots,d_k)\in\Comp(d)$, and let $W_t$ be a finite dimensional $S(n,d_t)$-module  for $t=1,\dots,k$. 
If each module $W_t$ has a standard (resp. costandard) filtration 
then so does the $S(n,d)$-module $W_1 \otimes  \dots\otimes W_k$. 
\end{Lemma}

Let $\la\in\Comp(n,d)$. We define the $S(n,d)$-modules 
\begin{align*}
S^\la_{n,d}&:=\nabla_{n,d}((\la_1))\otimes\dots\otimes \nabla_{n,d}((\la_n)),
\\
\Di^\la_{n,d}&:=\De_{n,d}((\la_1))\otimes\dots\otimes \De_{n,d}((\la_n)).
\end{align*}
Assuming $d\leq n$ and recalling the notation (\ref{EOmd}), we also define the $S(n,d)$-module
\begin{equation}\label{EExtDef}
\La^\la_{n,d}:=L_{n,d}(\om_{\la_1})\otimes\dots\otimes L_{n,d}(\om_{\la_n}).
\end{equation}
By Lemma~\ref{1.3c}, the module $S^\la_{n,d}$ has a costandard filtration, and the module $\Di^\la_{n,d}$ has a standard filtration. Moreover, since for any $l\leq n$ we have $L_{n,l}(\om_{l})=\De_{n,l}(\om_{l})=\nabla_{n,l}(\om_{l})$, the module $\La^\la_{n,d}$ has both a standard and a costandard filtration, i.e. is a {\em tilting module}. Also, using the formal characters it is easy to establish the following well-known fact:

\begin{Lemma} \label{LLaCompfactors}
Let $d\leq n$ and $\la\in\Comp_+(n,d)$. Then, interpreting the transposed partition $\la'$ as an element of $\Comp_+(n,d)$, we have that $L_{n,d}(\la')$ appears as a composition factor of $\La^\la_{n,d}$ with multiplicity $1$, and $L_{n,d}(\mu)$ is a composition factor of $\La^\la_{n,d}$ only if $\mu\unlhd \la'$. 
\end{Lemma}

Recall the isomorphism $\kappa$ from Lemma~\ref{LKappa}. The following is well-known, see for example {\rm \cite[(1.3b)]{BDK}}.

\begin{Lemma} \label{1.3b} 
Let $\la  \in \Comp(n,d)$. We have isomorphisms of $S(n,d)$-modules: 
\begin{enumerate}
\item[{\rm (i)}] 
  $S(n,d)\xi_\la\simeq \Di^\la_{n,d}$.  
\item[{\rm (ii)}] If $n \geq d$, then  $S(n,d)\kappa(\tty_\la)\simeq \La^\la_{n,d}$.
\end{enumerate}
\end{Lemma}

\begin{Lemma} \label{1.3d} {\rm \cite[(1.3d)]{BDK}} Let $n\geq d$, $\la\in \Comp_+(n,d)$, and consider the transposed partition $\la'$ as an element of $\Comp_+(n,d)$. Then 
$\dim\Hom_{S(n,d)}(\Di^\la_{n,d}, \La^{\la'}_{n,d})=1,$ 
and the image of any non-zero  homomorphism $\Di^\la_{n,d}\to \La^{\la'}_{n,d}$ is isomorphic to $\De_{n,d}(\la)$.
\end{Lemma}

\begin{Lemma} \label{LNnEquiv} {\rm \cite[\S6.5]{Green}} 
Let $N\geq n\geq d$. Then the functors (\ref{EFunctorsNn}) yield a Morita equivalence between $S(N,d)$ and $S(n,d)$. Moreover, considering any $\la\in\Comp_+(n,d)$ as an element of $\Comp_+(N,d)$, we have 
$
\Res^{S(N,d)}_{S(n,d)}L_{N,d}(\la)\simeq L_{n,d}(\la), 
$ 
$
\Res^{S(N,d)}_{S(n,d)}\De_{N,d}(\la)\simeq \De_{n,d}(\la) 
$ 
and \,
$
\Res^{S(N,d)}_{S(n,d)}\nabla_{N,d}(\la)\simeq \nabla_{n,d}(\la).
$
\end{Lemma}

\begin{Lemma} \label{1.5d} {\rm \cite[Theorem 2.7]{BK}} 
Let $k\in\N_+$, $\un = (n_1,\dots,n_k) \in\Comp(k,n)$ and $\ud = (d_1,\dots,d_k) \in\Comp(k,d)$, with $d_t\leq n_t$ for all $t$. Then for the modules $W_1\in\mod{S(n_1,d_1)},\dots,W_k\in\mod{S(n_k,d_k)}$, we have a functorial isomorphism
$$
\Ind_{S(\un,\ud)}^{S(n,d)}(W_1 \boxtimes\dots\boxtimes \,W_k)\simeq 
(\Ind_{S(n_1,d_1)}^{S(n,d_1)} W_1)\otimes\dots\otimes (\Ind_{S(n_k,d_k)}^{S(n,d_k)} W_k).
$$
\end{Lemma}

\subsection{Tilting modules and Ringel duality}
\label{SSTilt}
Throughout the subsection we again assume that $\k=\F$. 
We will use the theory of tilting modules and Ringel duality for quasi-hereditary algebras, see for example \cite[\S\,A4]{Donkin} for details. For a quasi-hereditary (graded super)algebra $A$ with weight poset $(X,\leq)$, a finite dimensional graded $A$-supermodule  is called {\em tilting}\, if it has both a standard and a costandard filtrations. For each $\la\in X$, there exists a unique indecomposable tilting module $T(\la)$ such that $[T(\la) : \De(\la)] = 1$ and $[T(\la) : \De(\mu)] = 0$ unless $\mu \leq\la$. Furthermore, every tilting module is isomorphic to a direct sum of  indecomposable tilting modules $T(\la)$. A {\em full tilting}\, module is a tilting module that contains every $T (\la),\ \la\in X,$ as a summand. 

Given a full tilting module $T$, the {\em Ringel dual}\, of $S$ relative to $T$ is the graded superalgebra $S':= \End_S(T)^\sop$. 
Then $S'$ is also a quasi-hereditary graded superalgebra with weight poset $X$, but ordered with the opposite order. The right tilting modules are defined similarly to the left tilting modules. 

\begin{Lemma} \label{4.5a} {\rm \cite[Lemma 4.5a]{BDK}}
Regarded as a right graded $S'$-supermodule, $T$ is a full tilting module for $S'$. Moreover, $\End_{S'}(T)\cong S$.
\end{Lemma}

Applying the above theory to the Schur algebra $S(n,d)$, we obtain the indecomposable tilting modules $\{T_{n,d}(\la)\mid \la \in \Comp_+(n,d)\}$ over $S(n,d)$. 

\begin{Lemma} \label{LDonkinLa} 
Let $d\leq n$ and $\la \in \Comp_+(n,d)$. Then $\La_{n,d}^{\la'}$ is a tilting module, the indecomposable module 
$T_{n,d}(\la)$ occurs exactly once as a summand of $\La_{n,d}^{\la'}$, and if $T_{n,d}(\mu)$ is a summand of $\La_{n,d}^{\la'}$  then $\mu\unlhd\la$.
\end{Lemma}
\begin{proof}
This is well-known, see for example \cite[Section 3.3(1)]{Donkin}. 
\end{proof}

\begin{Lemma} \label{LTiltSoc}
Let $d\leq n$. Then every composition factor of both the socle and the head of $T_{n,d}(\la)$ belongs to the head of the projective $S(n,d)$-module $S(n,d)\xi_{\om_d}$. 
\end{Lemma}
\begin{proof}
This is well-known. For example, one can use that $S(n,d)\xi_{\om_d} \cong V_{n,1}^{\otimes d}$ is contravariantly self-dual, and the known fact that $T_{n,d}(\la)$ is both a submodule and a quotient of  $S(n,d)\xi_{\om_d}$, which can be seen using Lemmas~\ref{LDonkinLa} and \ref{1.3b}(ii). 
\end{proof}

\section{Affine Brauer tree algebras}

Recall the notation (\ref{EIJ}). 

\subsection{  Brauer tree superalgebras $A_\ell$ and $\Zig_\ell$}
We will need two Morita eqivalent versions of a graded Brauer tree superalgebra. Let $A_\ell$ be the path algebra of the quiver 
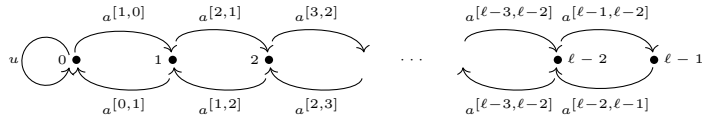
\begin{align*}
\begin{braid}\tikzset{baseline=3mm}
\coordinate (1) at (0,0);
\coordinate (2) at (4,0);
\coordinate (3) at (8,0);
\coordinate (4) at (12,0);
\coordinate (6) at (16,0);
\coordinate (L1) at (20,0);
\coordinate (L) at (24,0);
\draw [thin, black, ->] (-0.3,0.2) arc (15:345:1cm);
\draw [thin, black,->,shorten <= 0.1cm, shorten >= 0.1cm]   (1) to[distance=1.5cm,out=100, in=100] (2);
\draw [thin,black,->,shorten <= 0.25cm, shorten >= 0.1cm]   (2) to[distance=1.5cm,out=-100, in=-80] (1);
\draw [thin,black,->,shorten <= 0.25cm, shorten >= 0.1cm]   (2) to[distance=1.5cm,out=80, in=100] (3);
\draw [thin,black,->,shorten <= 0.25cm, shorten >= 0.1cm]   (3) to[distance=1.5cm,out=-100, in=-80] (2);
\draw [thin,black,->,shorten <= 0.25cm, shorten >= 0.1cm]   (3) to[distance=1.5cm,out=80, in=100] (4);
\draw [thin,black,->,shorten <= 0.25cm, shorten >= 0.1cm]   (4) to[distance=1.5cm,out=-100, in=-80] (3);
\draw [thin,black,->,shorten <= 0.25cm, shorten >= 0.1cm]   (6) to[distance=1.5cm,out=80, in=100] (L1);
\draw [thin,black,->,shorten <= 0.25cm, shorten >= 0.1cm]   (L1) to[distance=1.5cm,out=-100, in=-80] (6);
\draw [thin,black,->,shorten <= 0.25cm, shorten >= 0.1cm]   (L1) to[distance=1.5cm,out=80, in=100] (L);
\draw [thin,black,->,shorten <= 0.1cm, shorten >= 0.1cm]   (L) to[distance=1.5cm,out=-100, in=-100] (L1);
\blackdot(0,0);
\blackdot(4,0);
\blackdot(8,0);
\blackdot(20,0);
\blackdot(24,0);
\draw(0,0) node[left]{$0$};
\draw(4,0) node[left]{$1$};
\draw(8,0) node[left]{$2$};
\draw(14,0) node {$\cdots$};
\draw(20,0) node[right]{$\ell-2$};
\draw(24,0) node[right]{$\ell-1$};
 \draw(-2.6,0) node{$u$};
\draw(2,1.2) node[above]{$a^{[1,0]}$};
\draw(6,1.2) node[above]{$a^{[2,1]}$};
\draw(10,1.2) node[above]{$a^{[3,2]}$};
\draw(18,1.2) node[above]{$a^{[\ell-3,\ell-2]}$};
\draw(22,1.2) node[above]{$a^{[\ell-1,\ell-2]}$};
\draw(2,-1.2) node[below]{$a^{[0,1]}$};
\draw(6,-1.2) node[below]{$a^{[1,2]}$};
\draw(10,-1.2) node[below]{$a^{[2,3]}$};
\draw(18,-1.2) node[below]{$a^{[\ell-3,\ell-2]}$};
\draw(22,-1.2) node[below]{$a^{[\ell-2,\ell-1]}$};
\end{braid}
\end{align*}

\vspace{1mm}
\noindent
generated by the length $0$ paths $\{e^{[j]}\mid j\in J\}$, and the length $1$ paths 
$$\big\{a^{[k,k+1]},a^{[k+1,k]}\mid k\in \{0,1,\dots,\ell-2\}\big\}
\quad\text{and}\quad 
u,
$$
 modulo the following relations:
\begin{align}
&\text{all paths of length three or greater are zero};
\label{EZig1}
\\
&\text{all paths of length two that are not cycles are zero};
\label{EZig2}
\\
&\text{the length-two cycles based at the vertex $i\in\{1,\dots,\ell-1\}$ are equal};
\label{EZig3}
\\
&u^2=a^{[0,1]}a^{[1,0]}. 
\label{EZig4}
\end{align}
Let 
$$
c^{[0]}:=u^2\qquad \text{and}\qquad c^{[i]}:=a^{[i,i-1]}a^{[i-1,i]}\quad \text{for $i=1,\dots,\ell-1$}. 
$$
Then 
\begin{align*}
B_\ell:=\{e^{[j]},c^{[j]}\mid j\in J\}
\cup\big\{a^{[k,k+1]},a^{[k+1,k]}\mid k\in \{0,1,\dots,\ell-2\}\big\}
\cup\{u\}
\end{align*}
is a basis of $A_\ell$. 
We consider $A_\ell$ as a graded superalgebra by setting
\begin{align*}
&\bideg(e^{[j]}):=(0,\0), &\bideg(u):=(2,\1),
\\ 
&\bideg(a^{[k+1,k]}):=(4,\0), &\bideg(a^{[k,k+1]}):=(0,\0).
\end{align*}

We will also need a {\em regrading} $\Zig_\ell$ of $A_\ell$, see Example~\ref{SSRegr}. By definition, 
\begin{equation}\label{EAReGrade}
\Zig_\ell:=\bigoplus_{i,j\in J}(\Uppi \funQ^{2})^{j-i}e^{[i]}A_\ell e^{[j]},
\end{equation}
i.e. $\Zig_\ell$ is obtained from $A_\ell$ by regrading, using the grading supershift parameters 
\begin{equation}\label{EtjEpsj}
t_j:=2j+1-2\ell\quad\text{and}\quad\eps_j:=\bar j:=j\pmod{2}\qquad(j\in J),
\end{equation}
see (\ref{EShiftParameters}). In particular, $A_\ell$ and $\Zig_\ell$ are graded Morita superequivalent, see Example~\ref{SSRegr}. 

The elements of $\Zig_\ell$ corresponding to $e^{[j]},c^{[j]},a^{[i,j]},u\in A_\ell$ will be denoted $\ze^{[j]},\zc^{[j]},\za^{[i,j]},\zu$, respectively. Then  
\begin{align*}
\bideg(\ze^{[j]}):=(0,\0),\ \bideg(\zu)=\bideg(\za^{[k+1,k]})=\bideg(\za^{[k,k+1]}):=(2,\1).
\end{align*}
We have a basis of $\Zig_\ell$ corresponding to the basis $B_\ell$ of $A_\ell$ above: 
\begin{align*}
\zB_\ell:=\{\ze^{[j]},\zc^{[j]}\mid j\in J\}
\cup\big\{\za^{[k,k+1]},\za^{[k+1,k]}\mid k\in \{0,1,\dots,\ell-2\}\big\}
\cup\{\zu\}
\end{align*}
Note that all $c^{[j]}\in A_\ell$ and $\zc^{[j]}\in\Zig_\ell$ have bidegree $(4,\0)$.

Let $j\in J$. We let $\zL_j:=\k\cdot \zv_j$ be the rank $1$ free graded $\k$-supermodule with 
$
\bideg(\zv_j)=(0,\0).
$
 We make $\zL_j$ into a graded $\Zig_\ell$-supermodule via $\ze^{[j]} \zv_j=\zv_j$ and $\zb \zv_j=0$ for all $\zb\in \zB_\ell\setminus\{\ze^{[j]}\}$. 
Taking into account (\ref{EzVMor}), under the graded Morita superequivalence between $\Zig_\ell$ and $A_\ell$, the graded supermodule $\zL_j$ corresponds to the graded $A_\ell$-supermodule $\LL_j=\k\cdot v_j$ with 
\begin{equation}\label{EDegV}
\bideg(v_j)=(2j+1-2\ell,j\pmod{2})
\end{equation}
and the action $e^{[j]} v_j=v_j$ and $b v_j=0$ for all $b\in B_\ell\setminus\{e^{[j]}\}$. Recalling the notation (\ref{EA>0}), note that 
\begin{equation}\label{EzL>0}
\zL_j\simeq \Zig_\ell \ze^{[j]}/(\Zig_\ell \ze^{[j]})^{>0},
\end{equation}  
but it is not true that $
\LL_j\simeq A_\ell e^{[j]}/(A_\ell e^{[j]})^{>0}
$ 
(which is the first indication that the degree shifts we chose to switch from $A_\ell$ to $\Zig_\ell$ are convenient). 

\subsection{  Affine Brauer tree superalgebras $H_d(A_\ell)$ and $H_d(\Zig_\ell)$}
\label{SSAffZig}

Now, fix $d\in\N_+$. We now review the (graded) affine (super)algebra  $H_d(A_\ell)$ and its counterpart $H_d(\Zig_\ell)$, referring the reader to \cite[\S2.2f]{KlLi} for more details. 
For $\bi=i_1\cdots i_d\in  J^d$, denote
\begin{equation}\label{EZEBI}
\begin{split}
e^\bi:=e^{[i_1]}\otimes\dots\otimes e^{[i_d]}\in A_\ell^{\otimes d} 
\quad (\text{resp.}\ \ze^\bi:=\ze^{[i_1]}\otimes\dots\otimes \ze^{[i_d]}\in \Zig_\ell^{\otimes d}).
\end{split}
\end{equation}

Let $z,\zz$ be variables of bidegree $(4,\0)$ and consider the polynomial algebras $\k[z]$ and $\k[\zz]$. We denote by $A_\ell[z]$ (resp. $\Zig_\ell[\zz]$) the free product $\k[z]\star A_\ell$ (resp. $\k[\zz]\star \Zig_\ell$) subject to the relations 
\begin{equation}\label{ETwistedRelations}
uz=-zu\quad \text{and}\quad bz=zb\ \ \text{for all even $b\in B_\ell$}
\end{equation}
\begin{equation}
\label{ETwistedRelations2}
(\text{resp.}\ \ \zb\zz=-\zz\zb\ \ \text{for all odd $\zb\in \zB_\ell$}\quad \text{and}\quad \zb\zz=\zz\zb\ \ \text{for all even $\zb\in \zB_\ell$}).
\end{equation} 
Then  
$A_\ell[z]$ (resp. 
$\Zig_\ell[\zz]$) has basis 
$
\{z^nb\mid n\in\N,\,b\in B_\ell\}
$ 
(resp. 
$
\{\zz^n\zb\mid n\in\N,\,\zb\in\zB_\ell\}
$).

Note that $A_\ell$ (resp. $\Zig_\ell$) can be considered both as a subalgebra and as a quotient algebra of $A_\ell[z]$ (resp. $\Zig_\ell[\zz]$). Considering $A_\ell$ (resp. $\Zig_\ell$) as a quotient of $A_\ell[z]$ (resp. $\Zig_\ell[\zz]$) by the ideal generated by $z$ (resp. $\zz$), we can inflate the graded 
$A_\ell$-supermodule $L_j$ (resp. 
$\Zig_\ell$-supermodule $\zL_j$) to a graded $A_\ell[z]$-supermodule
(resp. $\Zig_\ell[\zz]$-supermodule), which we denote again by $\LL_j$ (resp. $\zL_j$).  Similarly to (\ref{EzL>0}), we have 
\begin{equation}\label{EDeDe+}
\zL_j\simeq \Zig_\ell[\zz] \ze^{[j]}/(\Zig_\ell[\zz] \ze^{[j]})^{>0}.
\end{equation}

Consider the graded superalgebra  
$A_\ell[z]^{\otimes d}$ (resp. $\Zig_\ell[\zz]^{\otimes d}$). Using the notation (\ref{EInsertion}), this algebra has basis
\begin{align*}
\{z_1^{n_1}\cdots z_d^{n_d}b^{(1)}_1\cdots b^{(d)}_d\mid 
n_1,\dots,n_d\in\N,\ b^{(1)},\dots,b^{(d)}\in B_\ell\}
\\
(\text{resp.}\ \{\zz_1^{n_1}\cdots\zz_d^{n_d}\zb^{(1)}_1\cdots \zb^{(d)}_d\mid 
n_1,\dots,n_d\in\N,\ \zb^{(1)},\dots,\zb^{(d)}\in\zB_\ell\}).
\end{align*}

Recall that we always consider the group algebra $\k \Si_d$ as a graded superalgebra concentrated in bidegree $(0,\0)$. We consider the free product $A_\ell[z]^{\otimes d}\star\k \Si_d$\, (resp.  $\Zig_\ell[\zz]^{\otimes d}\star\k \Si_d$). For any $\bi\in J^d$, we have that  $e^\bi\in A_\ell^{\otimes d}\subseteq A_\ell[z]^{\otimes d}\subseteq A_\ell[z]^{\otimes d}\star\k \Si_d$ and similarly we consider $\ze^\bi$ as an element of  $\Zig_\ell[\zz]^{\otimes d}\star\k \Si_d$. 

Recalling the notation (\ref{EWSignAction}), (\ref{EInsertion}), we define the {\em affine Brauer tree (graded super)algebra $H_d(A_\ell)$} to be the free product $A_\ell[z]^{\otimes d}\star\k \Si_d$ subject to the following relations:
\begin{equation}
w\,(a^{(1)}\otimes\dots\otimes a^{(d)})={}^w(a^{(1)}\otimes\dots\otimes a^{(d)})\,w
\label{ERAff1}
\end{equation}
for all $w\in \Si_d,\ a^{(1)},\dots,a^{(d)}\in A_\ell$, and 
\begin{equation}
(s_r z_t - z_{s_r(t)} s_r)e^\bi
=
\begin{cases}
\big((\delta_{r,t}- \delta_{r+1,t})(c_r^{[i_r]} + c_{r+1}^{[i_{r+1}]})
&
\\
\qquad\qquad\qquad+\de_{i_r,0}u_ru_{r+1}\big)e^\bi
&
\text{if}\ i_r = i_{r+1},\\
(\delta_{r,t}- \delta_{r+1,t})a_r^{[i_{r+1},i_r]} a_{r+1}^{[i_r,i_{r+1}]} e^\bi
&
\text{if}\ |i_r-i_{r+1}|=1,\\
0
& 
\textup{otherwise},
\end{cases}
\label{ERAff2}
\end{equation}
for all $1\leq r<d$, $1\leq t\leq d$ and $\bi\in J^d$.

Similarly, the {\em affine Brauer tree (graded super)algebra $H_d(\Zig_\ell)$} is defined to be the free product $A_\ell[z]^{\otimes d}\star\k \Si_d$ subject to the following relations:
\begin{equation}
w\,(\za^{(1)}\otimes\dots\otimes \za^{(d)})={}^w(\za^{(1)}\otimes\dots\otimes \za^{(d)})\,w
\label{ERAff1Zig}
\end{equation}
for all $w\in \Si_d,\ \za^{(1)},\dots,\za^{(d)}\in \Zig_\ell$, and 
\begin{equation}
(s_r \zz_t - \zz_{s_r(t)} s_r)\ze^\bi
=
\begin{cases}
\big((\delta_{r,t}- \delta_{r+1,t})(\zc_r^{[i_r]} + \zc_{r+1}^{[i_{r+1}]})
&
\\
\qquad\qquad\qquad+\de_{i_r,0}\zu_r\zu_{r+1}\big)\ze^\bi
&
\text{if}\ i_r = i_{r+1},\\
(\delta_{r,t}+\delta_{r+1,t})\za_r^{[i_{r+1},i_r]} \za_{r+1}^{[i_r,i_{r+1}]} \ze^\bi
&
\text{if}\ |i_r-i_{r+1}|=1,\\
0
& 
\textup{otherwise}
\end{cases}
\label{ERAff2Zig}
\end{equation}
for all $1\leq r<d$, $1\leq t\leq d$ and $\bi\in J^d$. (A sign change in (\ref{ERAff2Zig}) compared to  (\ref{ERAff2}) is related to the fact that the $\za^{[i,j]}$ are odd, while the $a^{[i,j]}$ are even).

There are natural graded superalgebra homomorphisms (of bidegree  $(0,\0)$): 
\begin{align}\label{EIOTAS}
\iota_1:A_\ell[z]^{\otimes d} \to H_d(A_\ell)
\quad\text{and}
\quad \iota_2:\k \Si_d\to H_d(A_\ell),
\end{align}
and similarly for $H_d(\Zig_\ell)$. 
By the following theorem, these maps are in fact embeddings. 
We will use the same symbols for elements of the domain of these maps as for their images in $H_d(A_\ell)$, and similarly for $H_d(\Zig_\ell)$. 

\begin{Theorem}\label{TAffBasis} \cite[Theorem 3.9]{KlLi}
The map 
$$A_\ell[z]^{\otimes d} \otimes \k \Si_d\to H_d(A_\ell),\ x \otimes y \mapsto \iota_1(x)\iota_2(y)
$$
is an isomorphism of graded superspaces. In particular, 
\begin{align*}
&\{w\,z_1^{ n_1}\cdots z_d^{ n_d} b^{(1)}_1\cdots  b^{(d)}_d\mid 
 n_1,\dots, n_d\in\N,\ b^{(1)},\dots,b^{(d)}\in B_\ell,\,w\in \Si_d\},
 \\
&\{z_1^{ n_1}\cdots z_d^{ n_d} b^{(1)}_1\cdots b^{(d)}_d\,w\mid 
 n_1,\dots, n_d\in\N,\ b^{(1)},\dots,b^{(d)}\in B_\ell,\,w\in \Si_d\}.
\end{align*}
are bases of $H_d(A_\ell)$. The similar statements hold for $H_d(\Zig_\ell)$.
\end{Theorem}

By the theorem and relations, $A_\ell^{\otimes d}\otimes\k \Si_d$ is a graded subsuperalgebra of $H_d(A_\ell)$ isomorphic to the wreath superproduct $A_\ell\swr  \Si_d$. The quotient of $H_d(A_\ell)$ by the graded superideal generated by $z_1,\dots,z_d$ is also isomorphic to $A_\ell\swr  \Si_d$. The similar statements hold for $H_d(\Zig_\ell)$. We identify $A_\ell[z]^{\otimes d}$ (resp. $\Zig_\ell[\zz]^{\otimes d}$) with a graded subsuperalgebra of $H_d(A_\ell)$ (resp. $H_d(\Zig_\ell)$). 
Since $\deg(w)=0$ for all $w\in\Si_d$, we deduce from the theorem:

\begin{Corollary} \label{CHZig>0} We have:
\begin{enumerate}
\item[{\rm (i)}] The algebras $H_d(A_\ell)$ and $H_d(\Zig_\ell)$ are non-negatively graded.
\item[{\rm (ii)}] $H_d(A_\ell)^{>0}=H_d(A_\ell)(A_\ell[z]^{\otimes d})^{>0}$ and $ H_d(\Zig_\ell)^{>0}=H_d(\Zig_\ell)(\Zig_\ell[\zz]^{\otimes d})^{>0}.$
\item[{\rm (iii)}] The degree zero component $H_d(\Zig_\ell)^0$ of $H_d(\Zig_\ell)$ is generated by all $\ze^\bj$ with $\bj\in J^d$ and $w\in\Si_d$.
\end{enumerate}
\end{Corollary}

\subsection{  Comparing $H_d(A_\ell)$ and $H_d(\Zig_\ell)$}
Recall that the graded superalgebra $\Zig_\ell$ is obtained from the graded superalgebra $A_\ell$ by regrading. Already for the graded superalgebras $\Zig_\ell\swr \Si_d$ and $A_\ell\swr \Si_d$
the similar statement is far from clear (in fact, it is even far from obvious that they are isomorphic as algebras). Even less obvious is the same claim on the affine level, namely, that the graded superalgebra $H_d(\Zig_\ell)$ is obtained from the graded superalgebra $H_d(A_\ell)$ by regrading.
Proposition~\ref{PIsoH} below, inspired by \cite[Lemma 5.27]{EK1}, shows that this is the case. 

We note that $\sum_{\bi\in J^n}e^\bi$ is an orthogonal idempotent decomposition of the  identity in $H_d(A_\ell)$, and similarly for $\sum_{\bi\in J^n}\ze^\bj$ in $H_d(\Zig_\ell)$. Recalling (\ref{EtjEpsj}), for $\bi=i_1\cdots i_d\in J^d$, we denote
\begin{equation}\label{ENormBj}\begin{split}
t_\bi&:=t_{i_1}+\dots+t_{i_d}=2(i_1+\dots+i_d)+d(1-2\ell),\\ \eps_{\bi}&:=\eps_{i_1}+\dots+\eps_{i_d} =(i_1+\dots+i_d)\pmod{2}.
\end{split}
\end{equation}
We will also need the following elements: 
\begin{align*}
&\ze':=\sum_{j\in J,\,\, j\,\text{odd}}\ze^{[j]}\in\Zig_\ell, 
&
\ze'':=\sum_{j\in J,\,\, j\,\text{even}}\ze^{[j]}\in\Zig_\ell,
\\ 
&\zh:=\ze'-\ze''\in\Zig_\ell, 
&\zg:=
\ze'\otimes 1+\ze''\otimes\zh
\in \Zig_\ell\otimes\Zig_\ell.
\end{align*}
Recalling the notation (\ref{EInsertion}), we have the elements
$
\zg_{r,r+1}\in \Zig_\ell^{\otimes d}\subseteq H_d(\Zig_\ell) 
$
for $1\leq r<d$. The following properties are easy to check:

\begin{Lemma} 
{\ }

\begin{enumerate}
\item[{\rm (i)}] In $\Zig_\ell$, we have:
\begin{eqnarray}
\label{EZFSquared}
\zh^2&=&1,
\\
\label{EAEE}
\za^{[i,j]}\zh&=&-\zh\za^{[i,j]}\qquad(\text{for all admissible $i,j$}),
\\
\label{EUEE}
\zu\zh&=&-\zu\ \,=\ \,\zh\zu.
\end{eqnarray}
\item[{\rm (ii)}] In $\Zig_\ell[\zz]$, we have 
\begin{equation}
\label{EZEE}
\zz\zh=\zh\zz,
\end{equation}
\item[{\rm (iii)}] In $\Zig_\ell\otimes\Zig_\ell$, we have:
\begin{eqnarray}
\label{EZGFU}
\zg(\zh\otimes\zu)&=&(\zh\otimes\zu)\zg,
\\
\label{EZG1A}
\zg(1\otimes \za^{[i,j]})&=&(\zh\otimes \za^{[i,j]})\zg
\qquad(\text{for all admissible $i,j$}).
\end{eqnarray}
\end{enumerate}
\end{Lemma}

We now prove the main result of this subsection:

\begin{Proposition}\label{PIsoH}
There is an isomorphism of graded superalgebras 
$$\bigoplus_{\bi,\bj\in J^n}\funQ^{t_\bi-t_\bj}\Uppi^{\eps_\bi-\eps_\bj}e^\bj H_d(A_\ell)e^\bi \to H_d(\Zig_\ell)$$ which maps (for all admissible indices $r,i,j$):
\begin{align*}
e^{[i]}_r&\mapsto \ze^{[i]}_r
\\
u_r&\mapsto \zh^{\otimes (r-1)}\otimes \zu\otimes \zh^{\otimes (d-r)}
\\
a^{[i,j]}_r&\mapsto\zh^{\otimes(r-1)}\otimes \za^{[i,j]}\otimes 1^{\otimes (d-r)}
\\
s_r&\mapsto s_r\zg_{r,r+1}
\\
z_r&\mapsto (\zz\zh)_r.
\end{align*}
\end{Proposition}
\begin{proof}
We denote by $X:=\{e^{[i]},u,a^{[i,j]}\}$ the set of the standard generators of $A_\ell$, by $Y:=X\sqcup\{z\}$ the set of the standard generators of $A_\ell[z]$ and by $Z:=\{s_r,x_t\mid 1\leq r<d,\,1\leq t\leq d,\,x\in Y\}$ the set of the standard generators of $H_d(A_\ell)$.  
For every $x\in Z$, we denote by $\phi(x)$ the element of  $H_d(\Zig_\ell)$ which corresponds to it under the assignments in the statement of the proposition, so the proposition claims that there exists an isomorphism $\phi:\bigoplus_{\bi,\bj\in J^n}\funQ^{t_\bi-t_\bj}\Uppi^{\eps_\bi-\eps_\bj}e^\bj H_d(A_\ell)e^\bi \to H_d(\Zig_\ell)$ of graded superalgebras with the prescribed $\phi(x)$ on  the generators $x$. In fact, it suffices to prove that there is such a homomorphism since then it is an isomorphism by bases, see Theorem~\ref{TAffBasis}. 
Noting that $\bideg(\phi(x))=\bideg(x)$ for al $x\in Z$, 
 it suffices to check that $\phi(x)$'s satisfy the defining relations of $H_d(A_\ell)$.

\vspace{2mm}
\noindent
{\sf Claim 1.} For each $1\leq r\leq d$, there is a graded superalgebra homomorphism $\phi_r: A_\ell[z]\to H_d(\Zig_\ell)$ with  
$\phi_r(x)=\phi(x_r)$ for all $x\in Y$. 

\vspace{2mm}
\noindent 
We check that the elements $\phi(x_r)$ satisfy the defining relations of $A_\ell[z]$. The 
relations (\ref{EZig1}),\,(\ref{EZig2}),\,(\ref{EZig3}) are clear. To check the relation (\ref{EZig4}), we need to verify that  
$$\phi(a^{[0,1]}_r)\phi_r(a^{[1,0]})=\phi(u_r)^2$$ which is easy to see using (\ref{EZFSquared}). 
Finally, to check (\ref{ETwistedRelations}), it is immediate that  
$$\phi(e^{[i]}_r)\phi_r(z)=\phi(z_r)\phi(e^{[i]}_r).$$ 
Moreover, $$\phi(a^{[i,j]}_r)\phi_r(z)=\phi(z_r)\phi(a^{[i,j]}_r)$$  thanks to (\ref{ETwistedRelations2}) and (\ref{EAEE}), and 
$$\phi(u_r)\phi(z_r)=-\phi(z_r)\phi(u_r)$$ thanks to  (\ref{ETwistedRelations2}) and (\ref{EUEE}).

\vspace{2mm}
\noindent
{\sf Claim 2.} Let $1\leq r<d$ and $x,y\in Y$. Then $\phi_r(x)\phi_{r+1}(y)=\phi_{r+1}(y)\phi_r(x)$ unless $x=y=u$, in which case $\phi_r(u)\phi_{r+1}(u)=-\phi_{r+1}(u)\phi_r(u)$.

\vspace{2mm}
\noindent
Note, using the notation (\ref{EInsertion}), the equalities  (\ref{EZFSquared}),(\ref{EUEE}) and the fact that $\zu$ is odd, 
\begin{align*}
\phi_r(u)\phi_{r+1}(u)=(\zu\zh\otimes \zh\zu)_{r,r+1}
=(\zh\zu\otimes \zu\zh)_{r,r+1}
=-\phi_{r+1}(u)\phi_r(u).
\end{align*}
Moreover, using (\ref{EAEE}) and the fact that $\za^{[i,j]}$ and $\za^{[k,l]}$ are odd, we get 
\begin{align*}
\phi_r(a^{[i,j]})\phi_{r+1}(a^{[k,l]})&=(\za^{[i,j]}\zh\otimes \za^{[k,l]})_{r,r+1}
=-(\zh\za^{[i,j]}\otimes \za^{[k,l]})_{r,r+1}
\\
&=\phi_{r+1}(a^{[k,l]})\phi_r(a^{[i,j]}).
\end{align*}
Similarly we get
\begin{align*}
\phi_r(a^{[i,j]})\phi_{r+1}(u)&=1^{\otimes(r-1)}\otimes \za^{[i,j]}\zh\otimes\zu\otimes\zh^{\otimes(d-r-1)}
\\
&=-1^{\otimes(r-1)}\otimes \zh \za^{[i,j]}\otimes\zu\otimes\zh^{\otimes(d-r-1)}
\\
&=\phi_{r+1}(u)\phi_r(a^{[i,j]})
\end{align*}
and 
$
\phi_r(u)\phi_{r+1}(a^{[i,j]})=\phi_{r+1}(a^{[i,j]})\phi_r(u)$. 
The remaining cases are easy to check.

\vspace{2mm}
\noindent
{\sf Claim 3.} There is a graded superalgebra homomorphism 
$$\phi: A_\ell[z]^{\otimes d}\to H_d(\Zig_\ell),\ x^{(1)}\otimes\dots\otimes x^{(d)}\mapsto \phi_1(x^{(1)})\cdots \phi_d(x^{(d)}).$$ 

\vspace{2mm}
\noindent
This follows from Claim 2 since in the generating set $Y$ 
the element $u$ is odd and all other elements are even. 

\vspace{2mm}
\noindent
{\sf Claim 4.} There is a graded superalgebra homomorphism 
$\phi: \k\Si_d\to H_d(\Zig_\ell),\ s_r\mapsto s_r\zg_{r,r+1}.$

\vspace{2mm}
\noindent
Indeed, the elements $s_r\zg_{r,r+1}$ are checked to satisfy the Coxeter relations, for example, 
$$
(s_r\zg_{r,r+1})^2=s_r\zg_{r,r+1}s_r\zg_{r,r+1}=s_rs_r\zg_{r,r+1}\zg_{r,r+1}=1
$$
and 
\begin{align*}
s_r\zg_{r,r+1}s_{r+1}\zg_{r+1,r+2}s_r\zg_{r,r+1}&=
s_rs_{r+1}s_r\zg_{r+1,r+2}\zg_{r,r+2}\zg_{r,r+1}
\\&=
s_{r+1}s_rs_{r+1}
\zg_{r,r+1}
\zg_{r,r+2}
\zg_{r+1,r+2}
\\
&=
s_{r+1}\zg_{r+1,r+2}s_r\zg_{r,r+1}s_{r+1}\zg_{r+1,r+2}.
\end{align*}

\vspace{2mm}
\noindent
{\sf Claim 5.} For any $1\leq r<d$ and $x\in X$, we have $\phi(s_r)\phi(x_r)=\phi(x_{r+1})\phi(s_r)$.

\vspace{2mm}
\noindent
Indeed, for $x=u$, using $s_r\zg_{r,r+1}=\zg_{r,r+1}s_r$ and (\ref{EZGFU}), we have 
\begin{align*}
\phi(s_r)\phi(u_r)&=(s_r\zg_{r,r+1})(\zh^{\otimes(r-1)}\otimes\zu\otimes \zh^{\otimes(d-r)})
\\
&=s_r(\zh^{\otimes(r-1)}\otimes(\zg(\zu\otimes\zh))\otimes \zh^{\otimes(d-r-1)})
\\
&=(\zh^{\otimes(r-1)}\otimes(\zg(\zh\otimes\zu))\otimes \zh^{\otimes(d-r-1)})s_r
\\
&=(\zh^{\otimes(r-1)}\otimes((\zh\otimes\zu)\zg)\otimes \zh^{\otimes(d-r-1)})s_r
\\
&=(\zh^{\otimes(r-1)}\otimes(\zh\otimes\zu)\otimes \zh^{\otimes(d-r-1)})\zg_{r,r+1}s_r
\\
&=(\zh^{\otimes r}\otimes\zu\otimes \zh^{\otimes(d-r-1)})s_r\zg_{r,r+1}
\\
&=\phi(u_{r+1})\phi(s_r).
\end{align*}
For $x=a^{[i,j]}$, 
we have similarly, using (\ref{EZG1A}), 
\begin{align*}
\phi(s_r)\phi(a^{[i,j]}_r)&=(s_r\zg_{r,r+1})(\zh^{\otimes(r-1)}\otimes\za^{[i,j]}\otimes 1^{\otimes(d-r)})
\\
&=(\zh^{\otimes(r-1)}\otimes(\zg(1\otimes \za^{[i,j]}))\otimes 1^{\otimes(d-r-1)})s_r
\\
&=(\zh^{\otimes(r-1)}\otimes((\zh\otimes \za^{[i,j]})\zg)\otimes 1^{\otimes(d-r-1)})s_r
\\
&=(\zh^{\otimes r}\otimes \za^{[i,j]}\otimes 1^{\otimes(d-r-1)})s_r\zg_{r,r+1}
\\
&=\phi(a^{[i,j]}_{r+1})\phi(s_r).
\end{align*}
For $x=e^{[i]}$ the argument is similar (but easier). 

\vspace{2mm}
From Claims 3,4,5 we have all the required relations except (\ref{ERAff1}). This final relation for the case $t\not \in\{r,r+1\}$ is easy to check, and in the presence of other relations, the case $t=r+1$ follows from the case $t=r$ (by multiplying with $s_r$ on both sides). So we may assume that $t=r$. In this case, we have using (\ref{ERAff2}), 
\begin{align*}
&(\phi(s_r)\phi(z_r)-\phi(z_{r+1})\phi(s_r))\phi(e^\bi)
\\
=\,
&(s_r\zg_{r,r+1}(\zz\zh)_r-(\zz\zh)_{r+1}s_r\zg_{r,r+1})\ze^\bi
\\
=\,&
(s_r\zg_{r,r+1}\zz_r\zh_r-\zz_{r+1}\zh_{r+1}s_r\zg_{r,r+1})\ze^\bi
\\
=\,&(s_r\zz_r-\zz_{r+1}s_r)\ze^\bi \zg_{r,r+1}\zh_r
\\
=\,&\begin{cases}
(\zc_r^{[i_r]} + \zc_{r+1}^{[i_{r+1}]}
+\de_{i_r,0}\zu_r\zu_{r+1})\ze^\bi\zg_{r,r+1}\zh_r
&
\text{if}\ i_r = i_{r+1},\\
\za_r^{[i_{r+1},i_r]} \za_{r+1}^{[i_r,i_{r+1}]} \ze^\bi\zg_{r,r+1}\zh_r
&
\text{if}\ |i_r-i_{r+1}|=1,\\
0
& 
\textup{otherwise}.
\end{cases}
\end{align*}
If $i_r=i_{r+1}=i$ then $\ze^\bi\zg_{r,r+1}\zh_r=\ze^{\bi}$; if, in addition $i=0$, then 
$$\zu_r\zu_{r+1}\ze^\bi
=\zu_r\zu_{r+1}\zh_r\zh_{r+1}\ze^\bi
=\zu_r\zh_r\zh_{r+1}\zu_{r+1}\ze^\bi
=\phi(u_r)\phi(u_{r+1})\ze^\bi.
$$ 
So in this case we are getting
$$
(\zc_r^{[i]} + \zc_{r+1}^{[i]}
+\de_{i,0}\zu_r\zu_{r+1})\ze^\bi
=(\phi(c_r^{[i]}) + \phi(c_{r+1}^{[i]})
+\de_{i,0}\phi(u_r)\phi(u_{r+1}))\phi(e^\bi),
$$
as required. If $|i_r-i_{r+1}|=1$ then, considering two cases when $i_r$ is even or odd, we get  
\begin{align*}
\za_r^{[i_{r+1},i_r]} \za_{r+1}^{[i_r,i_{r+1}]} \ze^\bi\zg_{r,r+1}\zh_r
&=\za_r^{[i_{r+1},i_r]} \za_{r+1}^{[i_r,i_{r+1}]} (\ze'\otimes 1-\ze''\otimes\zh)_{r,r+1}\ze^\bi
\\
&=\za_r^{[i_{r+1},i_r]}\zh_r \za_{r+1}^{[i_r,i_{r+1}]} \ze^\bi
\\
&=\phi(a_r^{[i_{r+1},i_r]})\phi(a_{r+1}^{[i_r,i_{r+1}]})\ze^\bi,
\end{align*}
as required. The remaining case $|i_r-i_{r+1}|>1$ is clear. 
\end{proof}



\section{Quiver Hecke superalgebras}
\label{SQHSA}

We will work with the quiver Hecke superalgebra of type $A_{2\ell}^{(2)}$, as defined in \cite{KKT}.

\subsection{Definition of the quiver Hecke superalgebra}
\label{SSDefQHSA}

For $i,j,k\in I$, we define polynomials $Q_{i,j}(u,v),B_{i,j,k}(u,v)\in\k[u,v]$ as follows:
\begin{eqnarray*}
Q_{i,j}(u,v)&:=&
\left\{
\begin{array}{ll}
0 &\hbox{if $i=j$,}\\
1 &\hbox{if $|i-j|>1$,}\\
(j-i)(u-v) &\hbox{if $|i-j|=1$ and $1\leq i,j<\ell$,}
\\

(j-i)(u^2-v) &\hbox{if $\{i,j\}=\{0,1\}$ or $\{\ell-1,\ell\}$, and $\ell>1$,}\\

(j-i)(u^4-v) &\hbox{if $\{i,j\}=\{0,1\}$ and $\ell=1$,}
\end{array}
\right.
\\
B_{i,j,k}(u,v)&:=&\left\{
\begin{array}{ll}
-1 &\hbox{if $i=k=j+1$,}\\
1 &\hbox{if $i=k= j-1\not\in \{0,\ell-1\}$,}\\
(u+v) &\hbox{if $i=k= j-1=\ell-1>0$,}\\
(v-u) &\hbox{if $i=k=j-1=0$ and $\ell>1$,}\\
(u^2+v^2)(v-u) &\hbox{if $i=k= j-1=0$ and $\ell=1$,}\\
0 &\hbox{otherwise.}
\end{array}
\right.
\end{eqnarray*}

Let $\theta\in Q_+$ be of height $n$. The {\em quiver Hecke superalgebra} $R_\theta$ is the unital graded $\k$-superalgebra generated by the elements 
$$
\{1_\bi\mid\bi\in I^\theta\}\cup \{y_1,\dots,y_n\}\cup \{\psi_1,\dots,\psi_{n-1}\}
$$
and the following defining relations (for all admissible $r,s,\bi$, etc.)
\begin{equation}
\label{R1}
1_\bi 1_\bj=\de_{\bi,\bj}1_\bi,
\end{equation}
\begin{equation}
\sum_{\bi\in I^\theta}1_\bi=1,
\label{R2}
\end{equation}
\begin{equation}
\label{R2.5}
y_r1_\bi=1_\bi y_r,
\end{equation}
\begin{equation}
\label{R2.75}
\psi_r1_\bi =1_{s_r\cdot \bi}\psi_r,
\end{equation}
\begin{equation}
y_ry_s1_\bi =
\left\{
\begin{array}{ll}
-y_sy_r1_\bi  &\hbox{if $r\neq s$ and $\| i_r\| =\| i_s\| =\1$,}\\
 y_sy_r1_\bi &\hbox{otherwise,}
\end{array}
\right.
\label{R3}
\end{equation}
\begin{equation}
\psi_r y_s1_\bi =(-1)^{\| i_r\|\, \| i_{r+1}\|\, \| i_s\| }y_s \psi_r 1_\bi \qquad(s\neq r,r+1),
\label{R4}
\end{equation}
\begin{equation}
\begin{split}
(\psi_ry_{r+1}-(-1)^{\| i_r\|\, \| i_{r+1}\| }y_r \psi_r) 1_\bi &=
(y_{r+1} \psi_r-(-1)^{\| i_r\|\, \| i_{r+1}\| } \psi_ry_r) 1_\bi 
\\&=
\left\{
\begin{array}{ll}
1_\bi  &\hbox{if $i_r=i_{r+1}$,}\\
0 &\hbox{otherwise,}
\end{array}
\right.
\end{split}
\label{R5}
\end{equation}
\begin{equation}\label{R6}
\psi_r^21_\bi =Q_{i_r,r_{r+1}}(y_r,y_{r+1}),
\end{equation}
\begin{equation}\label{R65}
\psi_r\psi_s 1_\bi =(-1)^{\| i_r\|\, \| i_{r+1}\|\, \| i_s\|\, \| i_{s+1}\| }\psi_s\psi_r 1_\bi \qquad(|r-s|>1),
\end{equation}
\begin{equation}
(\psi_{r+1}\psi_r\psi_{r+1}-\psi_{r}\psi_{r+1} \psi_r) 1_\bi =B_{i_r,i_{r+1},i_{r+2}}(y_r,y_{r+2})1_\bi 
\label{R7}
\end{equation}
The structure of a graded superalgebra on $R_\theta$ is defined by setting 
\begin{align*}
\bideg(\bi)&:=(0,\0),
\\ \bideg(y_s1_\bi )&:=((\al_{i_s}|\al_{i_s}),\| i_s\| ),
\\ \bideg(\psi_r 1_\bi )&:=-((\al_{i_r}|\al_{i_{r+1}}),\| i_r\|\, \| i_{r+1}\| ).
\end{align*}
Sometimes we denote the identity in $R_\theta$ by $1_\theta$.

We will use Khovanov-Lauda diagrams \cite{KL1} to represent the elements of $R_\theta$. In particular, if $\bi=i_1\cdots i_n\in I^\theta$, $1\leq r<n$ and $1\leq s\leq n$, we denote 
\vspace{2mm}
\begin{align*}
1_\bi =
\begin{braid}\tikzset{baseline=3mm}
  \draw (0,0)node[below]{$i_1$}--(0,3);
  \draw (1,0)node[below]{$i_2$}--(1,3);
  \draw[dots] (1.5,2.9)--(4.7,2.9);
  \draw (5,0)node[below]{$i_{n}$}--(5,3);
  \draw[dots] (1.5,0)--(4.7,0);
 \end{braid},
 \ \ 
y_s=
 \begin{braid}\tikzset{baseline=3mm}
  \draw (0,0)node[below]{$i_1$}--(0,3);
  \draw (3,0)node[below]{$i_{s-1}$}--(3,3);
  \draw[dots] (0.5,2.9)--(2.7,2.9);
  \draw (4.4,0)node[below]{$i_{s}$}--(4.4,3);
  \draw (5.7,0)node[below]{$i_{s+1}$}--(5.7,3);
  \draw[dots] (0.5,0)--(2.7,0);
  \draw[dots] (6.1,2.9)--(8.3,2.9);
  \draw[dots] (6.1,0)--(8.3,0);
  \draw (8.5,0)node[below]{$i_{n}$}--(8.5,3);
\blackdot (4.4,1.5);
 \end{braid},
\ \ 
\psi_r 1_\bi =
 \begin{braid}\tikzset{baseline=3mm}
  \draw (0,0)node[below]{$i_1$}--(0,3);
  \draw (3,0)node[below]{$i_{r-1}$}--(3,3);
  \draw[dots] (0.5,2.9)--(2.7,2.9);
  \draw (4.6,0)node[below]{$i_{r}$}--(5.8,3);
  \draw (5.8,0)node[below]{$i_{r+1}$}--(4.6,3);
  \draw (7.4,0)node[below]{$i_{r+2}$}--(7.4,3);
  \draw[dots] (0.5,0)--(2.7,0);
  \draw[dots] (7.8,2.9)--(10,2.9);
  \draw[dots] (7.8,0)--(10,0);
  \draw (10.2,0)node[below]{$i_{n}$}--(10.2,3);
\end{braid}.
\end{align*}

\vspace{1mm}

For every $w\in \Si_n$, we choose a reduced decomposition $w=s_{r_1}\dots s_{r_l}$ and define $\psi_{w}:=\psi_{r_1}\cdots\psi_{r_l}$. In general $\psi_w$ depends on the choice of a reduced decomposition.

\begin{Theorem}\label{TBasis}{\cite[Corollary 3.15]{KKT}} 
Let $\theta\in Q_+$ and $n=\height(\theta)$. Then  
\begin{align*}
&\{\psi_w y_1^{k_1}\dots y_n^{k_n}1_\bi \mid w\in \Si_n,\ k_1,\dots,k_n\in\N, \ \bi\in I^\theta\},
\\
&\{ y_1^{k_1}\dots y_n^{k_n}\psi_w1_\bi \mid w\in \Si_n,\ k_1,\dots,k_n\in\N, \ \bi\in I^\theta\}
\end{align*}
are bases of  $R_\theta$. 
\end{Theorem}

\begin{Lemma} \label{LBKW} {\rm \cite[Proposition 2.5]{BKW}}
Let $\theta\in Q_+$ with $\height(\theta)=n$ and $\bi\in I^\theta$.
\begin{enumerate}
\item[{\rm (i)}] If $w$ is an element of $\Si_n$ with reduced decompositions   $w=s_{t_1}\dots s_{t_m}=s_{r_1}\dots s_{r_m}$. Then in $R_\theta$ we have 
$$\psi_{t_1}\dots\psi_{t_m}1_\bi =\psi_{r_1}\dots\psi_{r_m}1_\bi +(*),$$
where $(*)$ is a linear combination of elements of the form $\psi_uf(y)1_\bi $ such that $u< w$ and $f(y)$ is a polynomial in $y_1,\dots,y_n$.

\item[{\rm (ii)}] If the decomposition $w=s_{t_1}\dots s_{t_m}$ is not reduced, then the element $\psi_{t_1}\dots\psi_{t_m}1_\bi$ can be written as a linear combination of elements of the form
$\psi_{t_{a_1}}\dots\psi_{t_{a_b}}f(y)1_\bi$ 
such that $1\leq a_1<\dots<a_b\leq m$, $b < m$, $s_{t_{a_1}}\dots s_{t_{a_b}}$ is a reduced word, and $f(y)$ is a polynomial in $y_1,\dots,y_n$.
\end{enumerate}

\end{Lemma}

We have a graded superalgebra antiautomorphism which is identity on the generators:
\begin{equation}\label{ETauAntiAuto}
\tau:R_\theta\to R_\theta
\end{equation}
By definition (of superalgebra antiautomorphism), we have $\tau(xy)=\tau(y)\tau(x)$ for all $x,y\in R_\theta$. 

If $V=\bigoplus_{n,\eps} V^n_\eps$ is a 
graded $R_\theta$-supermodule, then the {\em graded dual}
$V^\circledast$ is the graded $R_\theta$-supermodule 
which is the graded dual superspace $V^*$ with the $R_\theta$-action given by $(xf)(v)=f(\tau(x)v)$, for all $f\in V^\circledast, v\in V, x\in
R_\theta$. 
Given an {\em even} homomorphism $\phi:V\to W$ of graded $R_\theta$-supermodules, we have a homomorphism 
\begin{equation}\label{EPhiCircledast}
\phi^\circledast:W^\circledast\to V^{\circledast}
\end{equation}
 of graded $R_\theta$-supermodules so that $\phi^{\circledast}(f)=f\circ \phi$ for all $f\in W^{\circledast}$.

Note that the graded supercomponents $(R_\theta)^n$ are free of finite rank over $\k$ and $(R_\theta)^{n}=0$ for $n\ll0$. 
So, if $\k=\F$ then the same is true for a finitely generated graded $R_\theta$-supermodule $V$ and its {\em word spaces}\, $1_\bi V$ for all $\bi\in I^\theta$. 
At any rate, suppose that $V$ is a graded $R_\theta$-supermodule such that the graded supercomponents $V^n_\eps$ are free of finite rank over $\k$ and $V_{n}=0$ for $n\ll0$. Then we define the {\em formal character of $V$}
$$
\ch_{q,\pi}V:=\sum_{\bi\in I^\theta}(\dim_{q,\pi} 1_\bi  V)\,\bi \in \Z^\pi((q))\cdot I^\theta,
$$
where $\Z^\pi((q))\cdot I^\theta$ denotes the set of all formal $\Z^\pi((q))$-linear combinations of the words in $I^\theta$. 
We say that a graded $R_\theta$-supermodule $V$ {\em has a well-defined formal character}\, if all $V^n_\eps$ are free of finite rank over $\k$ and $V^n=0$ for $n\ll0$. In particular, if $V$ is free of finite rank as a $\k$-module, then it has a well-defined formal character $\ch_{q,\pi} V\in \Z^\pi[q,q^{-1}]\cdot I^\theta$
 and, using (\ref{EDimDual}), we have   
\begin{equation}\label{EChDual}
\ch_{q,\pi} (V^\circledast)=\overline{\ch_{q,\pi} V},
\end{equation}
where we have extended the bar-involution (\ref{EBarInv}) to $\Z^\pi[q,q^{-1}]\cdot I^\theta$ in the obvious way.

In fact, 
$$
\Z^\pi((q))\lan I\ran:=\bigoplus_{\theta\in Q_+}\Z^\pi((q)) \cdot I^\theta
$$ 
is a $\Z^\pi((q))$-algebra with respect to the {\em quantum shuffle product} defined on words as follows:
\begin{equation}\label{EShP}
\bi^1\circ\bi^2:=\sum f_w(\pi,q)i_{w(1)}\cdots i_{w(k+l)},
\end{equation}
where $\bi^1=i_1\cdots i_k$, $\bi^2=i_{k+1}\cdots i_{k+l}$, the sum is over all $w\in \Si_{k+l}$ such that $w^{-1}(1)<\cdots<w^{-1}(k)$, $w^{-1}(k+1)<\cdots<w^{-1}(k+l)$, and 
$$
 f_w(\pi,q):=\sum_{\substack{1\leq r\leq k<s\leq k+l\\ w^{-1}(r)>w^{-1}(s)}}\pi^{\| i_{w(r)}\|\, \| i_{w(s)}\| }q^{-(\al_{i_{w(r)}}|\al_{i_{w(s)}})}
$$
The shuffle product on formal characters corresponds to the parabolic induction on modules, see Corollary~\ref{CShuffle}. 

\subsection{Irreducible modules of quiver Hecke superalgebras}
By \cite[\S5.1, Proposition 6.15]{HW}, \cite[Theorems 8.6,\,8.7(ii)]{KKOII} (and the argument of \cite[Corollary 3.19]{KL1} for absolute irreducibility), we have:

\begin{Lemma} \label{LTypeM} 
Let $\k=\F$. 
Every irreducible graded $R_\theta$-supermodule $L$ is finite dimensional and absolutely irreducible as an $R_\theta$-module. Moreover, there exists $m\in \Z$ such that $(\funQ^mL)^\circledast \simeq\Uppi^\eps \funQ^mL$ for some $\eps\in\Z/2$. 
\end{Lemma}


Recall the set $\Par(\theta)$ of the root partitions of $\theta$ from  \S\ref{SSLTN}. 

\begin{Lemma} \label{LAmount} {\rm \cite[Lemma 3.1.18]{KlLi}} 
Let $\k=\F$. Then $|\Irr(R_\theta)|=|\Par(\theta)|$.  
\end{Lemma}

\subsection{Induction and restriction for quiver Hecke superalgebras}
\label{SSIndRes}
Let $\theta_1,\dots,\theta_n\in Q_+$ and $\theta:=\theta_1+\dots+\theta_n$. Denote 
$\underline{\theta}:=(\theta_1,\dots,\theta_n)\in Q_+^n$, 
$$R_{\underline{\theta}}
=R_{\theta_1,\dots,\theta_n}:=R_{\theta_1}\otimes\dots\otimes R_{\theta_n}$$ and 
\begin{equation}\label{ETheta}
1_{\underline{\theta}}
=1_{\theta_1,\dots,\theta_n}:=\sum_{\bi^1\in I^{\theta_1},\dots,\bi^n\in I^{\theta_n}}1_{\bi^1\cdots\bi^n}\in R_{\theta}.
\end{equation}
By \cite[\S4.1]{KKO}, there is a natural graded superalgebra embedding 
\begin{equation}\label{EIota}
\iota_{\underline{\theta}}:R_{\underline{\theta}}\to 1_{\underline{\theta}}R_{\theta}1_{\underline{\theta}},
\end{equation}
and we identify $R_{\underline{\theta}}$ with a subalgebra of $1_{\underline{\theta}}R_{\theta}1_{\underline{\theta}}$ via this embedding. We refer to this subalgebra as a {\em parabolic subalgebra}. Note that $\iota_{\underline{\theta}}(1_{\theta_1}\otimes\dots\otimes 1_{\theta_n})=1_{\underline{\theta}}$.

We have the exact functors:  
\begin{align*}
\Res^\theta_{\theta_1,\dots,\theta_n}= 
\Res_{\underline{\theta}}^\theta 
&:= 1_{\underline{\theta}} R_{\theta}
\otimes_{R_{\theta}} ?:\mod{R_{\theta}}\rightarrow \mod{R_{\underline{\theta}}},
\\
\Ind ^\theta_{\theta_1,\dots,\theta_n} =\Ind_{\underline{\theta}}^\theta 
&:= R_{\theta} 1_{\underline{\theta}}
\otimes_{R_{\underline{\theta}}} ?:\mod{R_{\underline{\theta}}} \rightarrow \mod{R_{\theta}},\\
\Coind^\theta_{\theta_1,\dots,\theta_n}=\Coind_{\underline{\theta}}^\theta 
&:= \Hom_{R_{\underline{\theta}}}(1_{\underline{\theta}}R_\theta
, ?):\mod{R_{\underline{\theta}}} \rightarrow \mod{R_{\theta}}.
\end{align*}
The functor $\Ind_{\underline{\theta}}^\theta
$ is left adjoint to $\Res_{\underline{\theta}}^\theta$, and  $\Coind_{\underline{\theta}}^\theta
$ is right adjoint to $\Res_{\underline{\theta}}^\theta$. 

The following follows immediately from Theorem~\ref{TBasis}.

\begin{Lemma} \label{LIndBasis} 
Let $\underline{\theta}:=(\theta_1,\dots,\theta_n)\in Q_+^n$,
$\la_r:=\height(\theta_r)$ for $r=1,\dots,n$, and 
$V\in\mod{R_{\underline{\theta}}}$. 
Let $\la:=(\la_1,\dots,\la_n)$ be a composition of $d:=\la_1+\dots+\la_n$. Then, as graded $\k$-supermodules, 
$$\Ind_{\underline{\theta}}^\theta V=\bigoplus_{w\in \D_d^\la} \psi_w\otimes V,
$$
and for each $w\in \D_d^\la$, the map $V\to \psi_w\otimes V,\ v\mapsto \psi_w\otimes v$ is an isomorphism. 
\end{Lemma}

The functors of (co)induction and restriction have obvious parabolic analogues. Given a family $\underline{\ga}=(\ga_{a,b})_{1\leq a\leq m,\,1\leq b\leq n}$ of elements of $Q_+$ such that  $\sum_{a=1}^m\ga_{a,b}=\theta_b$ for all $1\leq b\leq n$, we have  functors
\begin{align*}
\Ind_{\ga_{1,1},\dots,\ga_{m,1}\,;\,\dots\,;\,\ga_{1,n},\dots,\ga_{m,n}}^{\,\theta_1;\,\dots\,;\,\theta_n}
&:
\mod{R_{\ga_{1,1},\dots,\ga_{m,1}\,,\,\dots\,,\,\ga_{1,n},\dots,\ga_{m,n}}}\to\mod{R_{\,\theta_1,\,\dots\,,\,\theta_n}}
\\
\Res_{\ga_{1,1},\dots,\ga_{m,1}\,;\,\dots\,;\,\ga_{1,n},\dots,\ga_{m,n}}^{\,\theta_1;\,\dots\,;\,\theta_n}
&:\mod{R_{\,\theta_1,\,\dots\,,\,\theta_n}}\to \mod{R_{\ga_{1,1},\dots,\ga_{m,1}\,,\,\dots\,,\,\ga_{1,n},\dots,\ga_{m,n}}}.
\end{align*}

For $ g\in \Si_n$ and $\underline{\theta}=(\theta_1,\dots,\theta_n)\in Q_+^n$, recalling (\ref{EThetaParity}), let 
$$
 g\underline{\theta}:=(\theta_{ g^{-1}(1)},\dots,\theta_{ g^{-1}(n)}), 
$$
\begin{align*}
s( g,\underline{\theta})&:=\sum_{\substack{1\leq m<k\leq n,\\ g(m)> g(k)}}(\theta_m \mid \theta_k)\in\Z
\quad\text{and}\quad
\eps( g,\underline{\theta}):=\sum_{\substack{1\leq m<k\leq n,\\  g(m)> g(k)}}\| \theta_m\| \,\| \theta_k\| \in\Z/2.
\end{align*}
There is an isomorphism of  graded superalgebras
$$
\phi^g:R_{ g\underline{\theta}}\to R_{\underline{\theta}},\ 
x_1\otimes\dots\otimes x_n\mapsto 
(-1)^{\sum_{1\leq a<c\leq n, g(a)>g(c)}|x_a||x_c|}
x_{ g(1)}\otimes\dots\otimes x_{ g(n)}.
$$
Note that this makes sense even when some of $\theta_k's$ are $0$ (with $R_0$ being interpreted as the trivial $\k$-algebra $\k$). 
Composing with the isomorphism $\phi^g$, we get a functor
$$
\mod{R_{\underline{\theta}}}\to \mod{R_{ g\underline{\theta}}},\   V\mapsto  V^{\phi^ g}.
$$
Making in addition degree and parity shifts, we get a functor 
\begin{equation}\label{ETwist}
\mod{R_{\underline{\theta}}}\to \mod{R_{ g\underline{\theta}}},\   V\mapsto {}^ {\{g\}}  V:=\Uppi^{-\eps( g,\underline{\theta})}\funQ^{-s( g,\underline{\theta})} V^{\phi^ g}.
\end{equation}

Let $\underline{\theta}=(\theta_1,\dots,\theta_n)\in Q_+^n$ and $\underline{\eta}=(\eta_1,\dots,\eta_m)\in Q_+^m$ with 
$\theta_1+\dots+\theta_n=\eta_1+\dots+\eta_m=:\al.$ We denote by $\D(\underline{\eta},\underline{\theta})$ the set of all tuples $\underline{\ga}=(\ga_{a,b})_{1\leq a\leq m,\,1\leq b\leq n}$ of elements of $Q_+$ such that  $\sum_{a=1}^m\ga_{a,b}=\theta_b$ for all $1\leq b\leq n$ and $\sum_{b=1}^n\ga_{a,b}=\eta_a$ for all $1\leq a\leq m$. For each $\underline{\ga}\in \D(\underline{\eta},\underline{\theta})$, we define permutations 
$w_{m,n}\in\Si_{mn}
$ (not depending on $\underline{\ga}$) 
and 
$
x(\underline{\ga})\in\Si_{\height(\al)}
$ (depending only on the heights $\height(\ga_{a,b})$) 
as follows. The permutation $w_{m,n}\in\Si_{mn}$ maps the ordered $mn$-tuple of symbols 
$$
(\ga_{1,1},\dots,\ga_{m,1},\ga_{1,2},\dots,\ga_{m,2},\dots,\ga_{1,n},\dots,\ga_{m,n})
$$
onto the ordered $mn$-tuple of symbols
$$
(\ga_{1,1},\dots,\ga_{1,n},\ga_{2,1},\dots,\ga_{2,n},\dots,\ga_{m,1},\dots,\ga_{m,n})
$$
while $x(\underline{\ga})$ is the corresponding permutation of integral segments of sizes $\height(\ga_{a,b})$. Denoting $h_{a,b}:=\height(\ga_{a,b})$ and considering the composition 
$$
\chi=(h_{1,1},\dots,h_{m,1},h_{1,2},\dots,h_{m,2},\dots,h_{1,n},\dots,h_{m,n})\in\Comp(mn,\height(\al)),
$$
we have $x(\underline{\ga})=w_{\chi,w_{m,n}}$, using the notation  introduced in (\ref{EWLaSi}).

\begin{Example} \label{EXMackey1} 
Let $\ell=2$ so we have simple roots $\al_0,\al_1,\al_2$. 
Take 
$\underline{\eta}=(\al_0+\al_1,\al_1+\al_2)\in Q_+^2$ and $\underline{\theta}=(\al_2,\al_0+\al_1,\al_1)\in Q_+^3$. Then $\D(\underline{\eta},\underline{\theta})=\{\underline{\ga},\underline{\beta}\}$, where 
\begin{align*}
\left(
\begin{matrix}
\ga_{1,1} & \ga_{1,2} & \ga_{1,3}  \\
\ga_{2,1} & \ga_{2,2} & \ga_{2,3}
\end{matrix}
\right)
&=
\left(
\begin{matrix}
0 & \al_0+\al_1 & 0  \\
\al_2 & 0 & \al_1
\end{matrix}
\right),
\\
\left(
\begin{matrix}
\be_{1,1} & \be_{1,2} & \be_{1,3}  \\
\be_{2,1} & \be_{2,2} & \be_{2,3}
\end{matrix}
\right)
&=
\left(
\begin{matrix}
0 & \al_0 & \al_1  \\
\al_2 & \al_1 & 0
\end{matrix}
\right).
\end{align*}
Now 
\begin{align*}
w_{2,3}&: 1\mapsto 1,\ 2\mapsto 4,\ 3\mapsto 2,\ 4\mapsto 5,\ 5\mapsto 3,\ 6\mapsto 6,
\\
x(\underline{\ga})&: 1\mapsto 3,\ 2\mapsto 1,\ 3\mapsto 2,\ 4\mapsto 4,
\\
x(\underline{\be})&: 1\mapsto 3,\ 2\mapsto 1,\ 3\mapsto 4,\ 4\mapsto 2.
\end{align*}
\end{Example}

\vspace{3mm}
Let $V\in\mod{R_{\underline{\theta}}}$. Note that 
for any $\underline{\ga}\in \D(\underline{\eta},\underline{\theta})$, we have 
$1_{\underline{\eta}}\psi_{x(\underline{\ga})} 1_{\underline{\theta}}\otimes V
\subseteq \Res_{\underline{\eta}}^\al\Ind_{\underline{\theta}}^\al V.$
Let $\leq$ be a total order refining the Bruhat order on $\Si_{\height(\al)}$. 
For $\underline{\ga}\in\D(\underline{\eta},\underline{\theta})$, we consider the submodules
\begin{align}
\label{EFiltF}
F_{\leq \underline{\ga}}(V)&:=\sum_{\underline{\ka}\in \D(\underline{\eta},\underline{\theta})\,\text{with}\, x(\underline{\ka})\leq x(\underline{\ga})}R_{\underline{\eta}}1_{\underline{\eta}}\psi_{x(\underline{\ka})} 1_{\underline{\theta}}\otimes V\subseteq \Res_{\underline{\eta}}^\al\Ind_{\underline{\theta}}^\al V,
\\
F_{<\underline{\ga}}(V)&:=\sum_{\underline{\ka}\in \D(\underline{\eta},\underline{\theta})\,\text{with}\, x(\underline{\ka})< x(\underline{\ga})}R_{\underline{\eta}}1_{\underline{\eta}}\psi_{x(\underline{\ka})} 1_{\underline{\theta}}\otimes V\subseteq \Res_{\underline{\eta}}^\al\Ind_{\underline{\theta}}^\al V.
\end{align}

The Mackey Theorem below for $m=n=2$ follows from \cite[Proposition 4.5]{KKO}. The general case can be deduced from the case $m=n=2$ by induction. See also \cite[Proposition 3.7]{Evseev} or the proof of \cite[Proposition 2.18]{KL1}.

\begin{Theorem} \label{TMackeyKL} {\rm (Mackey Theorem)}
Let $\underline{\theta}=(\theta_1,\dots,\theta_n)\in Q_+^n$ and $\underline{\eta}=(\eta_1,\dots,\eta_m)\in Q_+^m$ with 
$\al:=\theta_1+\dots+\theta_n=\eta_1+\dots+\eta_m$, and  
 $ V\in\mod{R_{\underline{\theta}}}$. Then we have the filtration 
 $
 (F_{\leq \underline{\ga}}(V))_{\underline{\ga}\in\D(\underline{\eta},\underline{\theta})}
 $
 of 
$\Res_{\underline{\eta}}^\al\,\Ind_{\underline{\theta}}^\al  V$. Moreover, for the sub-quotients $S_\ga(V):=F_{\leq \underline{\ga}}(V)/F_{< \underline{\ga}}(V)$ of the filtration we have 
$$
S_\ga(V)\simeq 
\Ind_{\ga_{1,1},\dots,\ga_{1,n}\,;\,\dots\,;\,\ga_{m,1},\dots,\ga_{m,n}}^{\,\eta_1;\,\dots\,;\,\eta_m}
{}^{\{w_{m,n}\}}\big(\Res_{\ga_{1,1},\dots,\ga_{m,1}\,;\,\dots\,;\,\ga_{1,n},\dots,\ga_{m,n}}^{\,\theta_1;\,\dots\,;\,\theta_n}
\, V \big).
$$
\end{Theorem}

\begin{Example} \label{EXMackey2} Continuing with Example~\ref{EXMackey1}, we can take 
$x(\underline{\ga})<x(\underline{\be})$ and so, according to Theorem~\ref{TMackeyKL}, $\Res_{\al_0+\al_1,\al_1+\al_2}^{\al_0+2\al_1+\al_2}\,\Ind_{\al_2,\al_0+\al_1,\al_1}^{\al_0+2\al_1+\al_2}  V$ has a submodule 
$$
F_{\leq \underline{\ga}}(V):=R_{\al_0+\al_1,\al_1+\al_2}1_{\al_0+\al_1,\al_1+\al_2}\psi_{x(\underline{\ga})} 1_{\al_2,\al_0+\al_1,\al_1}\otimes V
$$
isomorphic to 
$$ \Ind^{\al_0+\al_1;\al_1+\al_2}_{\al_0+\al_1;\al_2,\al_1}\,\,{}^{w_{3,2}}\big(\Res_{\al_2;\al_0+\al_1;\al_1}^{\al_2;\al_0+\al_1;\al_1}
\, V \big)
=\Ind^{\al_0+\al_1;\al_1+\al_2}_{\al_0+\al_1;\al_2,\al_1}\,\,{}^{w_{3,2}}V
$$
and the quotient 
\begin{align*}
&\Res_{\al_0+\al_1,\al_1+\al_2}^{\al_0+2\al_1+\al_2}\,\Ind_{\al_2,\al_0+\al_1,\al_1}^{\al_0+2\al_1+\al_2}  V\,/\, F_{\leq \underline{\ga}}(V)
\\
\simeq\, & \Ind^{\al_0+\al_1;\al_1+\al_2}_{\al_0,\al_1;\al_2,\al_1}\,\,{}^{w_{3,2}}\big(\Res_{\al_2;\al_0,\al_1;\al_1}^{\al_2;\al_0+\al_1;\al_1}
\, V \big).
\end{align*}
\end{Example}

\vspace{3mm}
Recalling the shuffle product in $\Z^\pi((q))$ from (\ref{EShP}), a  special case of the Mackey Theorem where each $\eta_k$ is a simple root yields:

\begin{Corollary} \label{CShuffle}
Let $V_1\in\mod{R_{\theta_1}},\dots,V_n\in\mod{R_{\theta_n}}$ have well-defined formal characters. Then $\Ind_{\underline{\theta}}^\theta(V_1\boxtimes\dots\boxtimes V_n)$ has a well-defined formal character, and 
$$
\ch_{\pi,q}\Ind_{\underline{\theta}}^\theta(V_1\boxtimes\dots\boxtimes V_n)=(\ch_{\pi,q} V_1)\circ\dots\circ(\ch_{\pi,q} V_n).
$$
\end{Corollary}

\begin{Corollary} \label{CLinInd}
Suppose $\k=\F$, and $V\in\mod{R_{\theta}}$,  $W\in\mod{R_{\eta}}$ be finite dimensional. Then for any $L\in\Irr(R_{\theta+\eta})$, the multiplicities of $L$ in 
$
\Ind_{\theta,\eta}^\theta(V\boxtimes W)
$
and in 
$
\Ind_{\eta,\theta}^\theta(W\boxtimes V)
$
are the same.
\end{Corollary}
\begin{proof}
It follows from Corollary~\ref{CShuffle} that at $q=1$ and $\pi=1$, the characters of $
\Ind_{\theta,\eta}^\theta(V\boxtimes W)
$
and 
$
\Ind_{\eta,\theta}^\theta(W\boxtimes V)
$
are the same. Now the result follows from the linear independence of the usual formal characters of the irreducible $R_{\theta+\eta}$ modules, see for example \cite[Corollary 8.16]{KKOII}. 
\end{proof}

For $\underline{\theta}:=(\theta_1,\dots,\theta_n)$ as above, recalling (\ref{EThetaParity}), we denote 
\begin{align*}
s(\underline{\theta}):=-\sum_{1\leq m<k\leq n}(\theta_m \mid \theta_k)\in\Z
\quad \text{and}\quad
\eps(\underline{\theta}):=\sum_{1\leq m<k\leq n}\| \theta_m\| \,\| \theta_k\| \in\Z/2.
\end{align*}
Note that $s(\underline{\theta})=s(w_n,\underline{\theta})$ and 
$\eps(\underline{\theta})=\eps(w_n,\underline{\theta})$. 
The following result is essentially proved in \cite[Theorem 2.2]{LV} for quiver Hecke algebras (under some unnecessary additional assumptions on finite dimensionality):

\begin{Lemma} \label{LLV} 
Let $\underline{\theta}:=(\theta_1,\dots,\theta_n)\in Q_+^n$, and 
$V_k\in\mod{R_{\theta_k}}$ for $k=1,\dots,n$.
Then there is a natural isomorphism 
$$
\Ind_{\theta_1,\dots,\theta_n}^\theta(V_1\boxtimes\dots\boxtimes V_n)\simeq 
\funQ^{s(\underline{\theta})}
\Uppi^{\eps(\underline{\theta})}
\,
\Coind_{\theta_n,\dots,\theta_1}^\theta(V_n\boxtimes\dots\boxtimes V_1).
$$
\end{Lemma}
\begin{proof}
It suffices to consider the case $n=2$, i.e. prove that 
$$
\Ind_{\theta,\eta}^{\theta+\eta}(V\boxtimes W)\simeq 
\funQ^{(\theta\mid\eta)}
\Uppi^{\|\theta\|\,\|\eta\|}
\,
\Coind_{\eta,\theta}^{\theta+\eta}(W\boxtimes V).
$$
To prove this, we check the assumptions of Lemma~\ref{LIndCoindIsoGen} (and Example~\ref{ENatural}). 
Let $m=\height(\theta)$ and $n=\height(\eta)$. 
For the assumption (1), the right $R_{\theta,\eta}$-module $R_{\theta+\eta}1_{\theta,\eta}$ has basis $\{\psi_x\mid x\in\D_{m+n}^{(m,n)}\}$ and 
the left $R_{\eta,\theta}$-module $1_{\eta,\theta}R_{\theta+\eta}$ has basis $\{\psi_{y}\mid y\in{}^{(n,m)}\D_{m+n}\}$. For the assumption (2), see Example~\ref{ENatural}. For the assumption (3), the Mackey Theorem yields a filtration of $1_{\eta,\theta}R_{\theta+\eta}1_{\theta,\eta}=\Res_{\eta,\theta}\Ind_{\theta,\eta}R_{\theta,\eta}$ by bimodules with the top quotient 
$1_{\eta,\theta}R_{\theta+\eta}1_{\theta,\eta}/Y$ such that  $R_{\theta,\eta}\to 1_{\eta,\theta}R_{\theta+\eta}1_{\theta,\eta}/Y,\ r\mapsto r\psi_{x_0}$ is an isomorphism of left modules, where $x_0$ is the longest element of ${}^{(n,m)}\D_{m+n}^{(m,n)}$ given by 
$$
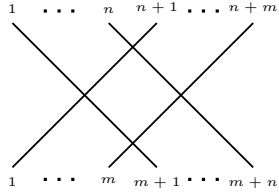
\begin{braid}\tikzset{baseline=2mm}
  \draw (0,0)node[below]{$1$}--(6,6)node[above]{$n+1$};
  \draw[dots] (1.3,-0.5)--(2.8,-0.5);
  \draw[dots] (7.3,-0.5)--(8.8,-0.5);
   \draw[dots] (1.3,6.5)--(2.8,6.5);
  \draw[dots] (7.3,6.5)--(8.8,6.5);
  \draw (4,0)node[below]{$m$}--(10,6)node[above]{$n+m$};
  \draw (6,0)node[below]{$m+1$}--(0,6)node[above]{$1$};
  \draw (10,0)node[below]{$m+n$}--(4,6)node[above]{$n$};
 \end{braid}
 $$
 For the remaining part of the condition (3), taking $\xi:=1_{\eta,\theta}\psi_{x_0}1_{\theta,\eta}$, one needs to check that for $r\otimes s\in R_\theta\otimes R_\eta=R_{\theta,\eta}$ we have 
$$
\xi(r\otimes s)\equiv
(-1)^{|s||r|+|s||\xi|+|r||\xi|}(s\otimes r)\xi 
\pmod{Y},
$$ 
which is checked using Lemma~\ref{LBKW}, cf. the proof of \cite[Theorem 2.2]{LV}. Finally to check the condition (4) of Lemma~\ref{LIndCoindIsoGen}, we need to observe that 
$
{}^{(n,m)}\D_{m+n}=\{x_0x^{-1}\mid \D_{m+n}^{(m,n)}\},
$
and use Lemma~\ref{LBKW} to see that 
$$
1_{\eta,\theta}\psi_{y}\psi_x1_{\theta,\eta}\equiv \de_{x_0x^{-1},y}1_{\eta,\theta}\psi_{x_0}1_{\theta,\eta}\pmod{Y},
$$
completing the proof.
\end{proof}

The anti-automorphism $\tau$ from (\ref{ETauAntiAuto})   restricts to an anti-automorphism of the parabolic subalgebra $R_{\underline{\theta}}$. So, just like for $R_\theta$, for a graded $R_{\underline{\theta}}$-supermodule, we can define its {\em graded dual}
$V^\circledast$. By definition, for the graded supermodules $V_1\in\mod{R_{\theta_1}},\dots,V_n\in\mod{R_{\theta_n}}$, we have
\begin{equation}\label{EBoxCircledAst}
(V_1\boxtimes\dots\boxtimes V_n)^\circledast\simeq 
V_1^\circledast\boxtimes\dots\boxtimes V_n^\circledast.
\end{equation}

\begin{Lemma} \label{LDualInd}
Let $\underline{\theta}:=(\theta_1,\dots,\theta_n)\in Q_+^n$.
\begin{enumerate}
\item[{\rm (i)}] The functor $\Res^\theta_{\underline{\theta}}$ commutes with duality $\circledast$. 
\item[{\rm (ii)}] If $\k=\F$, and $V_k\in\mod{R_{\theta_k}}$ for $k=1,\dots,n$ are finite dimensional, then  
$$
\big(\Ind_{\theta_1,\dots,\theta_n}^\theta(V_1\boxtimes\dots\boxtimes V_n)\big)^\circledast\simeq 
\funQ^{s(\underline{\theta})}
\Uppi^{\eps(\underline{\theta})}
\Ind_{\theta_n,\dots,\theta_1}(V_n^\circledast\boxtimes\dots\boxtimes V_1^\circledast).
$$
\end{enumerate}
\end{Lemma}
\begin{proof}
(i) is clear. 

(ii) follows from (i) and Lemma~\ref{LLV} by uniqueness of adjoint functors as in \cite[Theorem 3.7.5]{Kbook}
\end{proof}

\subsection{Divided power idempotents}\label{SSDPI}
Let $i\in I$ and $m\in \N_+$. Recall that we denote by $w_m$ the longest element of $\Si_m$. If $i\neq 0$, the algebra $R_{m\al_i}$ is known to be the (affine) {\em nil-Hecke algebra}   
and has an idempotent 
\begin{equation}\label{EDivPowId}
1_{i^{(m)}}:=\psi_{w_m}\prod_{s=1}^m y_s^{s-1},
\end{equation}
cf. \cite{KL1}. If $i=0$, the algebra $R_{m\al_0}$ is known to be the (affine) {\em odd nil-Hecke algebra} and has an idempotent 
$1_{0^{(m)}}$ of the form $\psi_{w_m}\prod_{s=1}^m y_s^{s-1}$, cf. \cite{EKL} (the choice of a reduced decomposition for $w_m$ is important here to get the right sign). 
Diagrammatically we will denote 
$$1_{i^{(m)}}:=
\begin{braid}\tikzset{baseline=0.25em}
	\braidbox{0}{1.3}{0}{1}{$i^m$}
\end{braid}
=
\begin{braid}\tikzset{baseline=0.25em}
	\braidbox{0}{2.2}{0}{1}{$i\cdots i$}
\end{braid}.
$$
For example, for $m=2$, we have for any $i\in I$:
\begin{equation}\label{EDivDiag}
1_{i^{(2)}}:=
\begin{braid}\tikzset{baseline=0.25em}
	\braidbox{0}{1.3}{0}{1}{$i^2$}
\end{braid}
=
\begin{braid}\tikzset{baseline=0.25em}
	\braidbox{0}{1.6}{0}{1}{$i\,\,i$}
\end{braid}
=
\begin{braid}\tikzset{baseline=3mm}
  \draw (0,0)node[below]{$i$}--(1,2);
  \draw (1,0)node[below]{$i$}--(0,2);
\blackdot (0.8,0.45);
 \end{braid}.
\end{equation}
The fact that $1_{i^{(m)}}$ is an idempotent is equivalent to 
\begin{equation}\label{EStrIdemp}
1_{i^m}\psi_{w_m}=\psi_{w_m}.
\end{equation}

\begin{Lemma} \label{LW_0Front}
For any $i\in I$ and $m\in\N_+$, in $R_{m\al_i}$ we have 
$$1_{i^{(m)}}R_{m\al_i}=\psi_{w_m}R_{m\al_i}=\psi_{w_m}\k[y_1,\dots,y_m].$$ 
\end{Lemma}
\begin{proof}
The first equality follows from (\ref{EDivPowId}) and (\ref{EStrIdemp}). The second equality follows from Theorem~\ref{TBasis} since in $R_{m\al_i}$ we have $\psi_{w_m}\psi_y=0$ for any non-trivial $y\in \Si_m$.
\end{proof}

Recalling the notation (\ref{EFancyDirectSum}) and (\ref{EBarInv}), we have the following well-known fact (see for example \cite[p. 330]{KL1} for $i\neq 0$ and \cite[Lemma 5.8]{BKodd} for $i=0$):

\begin{Lemma} \label{LNilHeckeProj}
Let $i\in I$. Then, as graded left supermodules,  
$$
R_{m\al_i}\simeq(R_{m\al_i}1_{i^{(m)}})^{\bigoplus q_i^{-m(m-1)/2}[m]^!_i}\simeq 
(1_{i^{(m)}}R_{m\al_i})^{\bigoplus q_i^{m(m-1)/2}\overline{[m]^!_i}}.
$$
\end{Lemma}

Let $\theta\in Q_+$. Recalling the notation $I^{\theta}_{\di}$ from \S\ref{SSLTN}, for a divided power word $\bi = (i_1^{(m_1)},\ldots, i_r^{(m_r)}) \in I^{\theta}_{\di}$, we have the corresponding {\em divided power idempotent} 
\begin{equation}\label{EDPId}
1_\bi :=\iota_{m_1\al_{i_1},\dots,m_r\al_{i_r}}(1_{i_1^{(m_1)}}\otimes\dots\otimes 1_{i_r^{(m_r)}}) \in R_{\theta}.
\end{equation}

\begin{Lemma} \label{LIndStdProj}
Let $\bi^{(1)}\in I^{\theta_1}_\di,\dots,\bi^{(k)}\in I^{\theta_k}_\di$ and $\theta=\theta_1+\dots+\theta_k$. Then there is an isomorphism of graded $R_\theta$-supermodules
\begin{align*}
R_{\theta}&1_{\bi^{(1)}\cdots\bi^{(k)}}\iso
\Ind_{\underline{\theta}}^\theta(R_{\theta_1}1_{\bi^{(1)}}\boxtimes\cdots\boxtimes R_{\theta_k}1_{\bi^{(k)}}), \\ 
&1_{\bi^{(1)}\cdots\bi^{(k)}}\mapsto 1_{\underline{\theta}}\otimes 1_{\bi^{(1)}}\otimes\cdots\otimes 1_{\bi^{(k)}}.
\end{align*}
\end{Lemma}
\begin{proof}
The homomorphism of graded $R_\theta$-supermodules as in the statement exists because in $\Ind_{\underline{\theta}}^\theta(R_{\theta_1}1_{\bi^{(1)}}\boxtimes\cdots\boxtimes R_{\theta_k}1_{\bi^{(k)}})$ we have
$$1_{\bi^{(1)}\cdots\bi^{(k)}}\cdot (1_{\underline{\theta}}\otimes 1_{\bi^{(1)}}\otimes\cdots\otimes 1_{\bi^{(k)}})=1_{\underline{\theta}}\otimes 1_{\bi^{(1)}}\otimes\cdots\otimes 1_{\bi^{(k)}}.
$$
The inverse homomorphism is constructed using the adjointness of $\Ind$ and $\Res$.  
\end{proof}

Let $\bi=i_1^{(m_1)}\cdots i_r^{(m_r)}\in I^\theta_\di$. 
Recalling the notation (\ref{EHatI}) and (\ref{EQInt}), note that $1_\bi  1_{\hat\bi} = 1_\bi  =1_{\hat \bi}1_\bi $, 
and define 
\begin{equation}\label{EHatIPar}
\bi! := [m_1]_{i_1}^! \cdots [m_r]_{i_r}^!
\qquad\text{and}\qquad
\langle \bi \rangle := \sum_{k=1}^r (\al_{i_k}\mid\al_{i_k})m_k (m_k-1)/4.
\end{equation}


\begin{Lemma}\label{LFact}   
 Let $\bi=i_1^{(m_1)}\cdots i_r^{(m_r)}\in I^{\theta}_{\di}$, and\, $U$ (resp.~$V$) be a left (resp. right) graded $R_{\theta}$-supermodule.  
 Then 
 $$
 1_{\hat\bi} U\simeq  (1_\bi  U)^{\oplus q^{\langle \bi\rangle}\bi! } 
  \quad \text{and} \quad V 1_{\hat\bi}\simeq  (V 1_\bi )^{\oplus q^{-\langle \bi\rangle}\overline{\bi!} }
 $$
 In particular, if $1_{\hat\bi}  U$ is a free $\k$-module of rank $(m_1)! \cdots (m_r)!$, then $1_{\bi} U$ is the lowest degree component of $1_{\hat \bi}  U$ and $1_{\bi} U$ is a free $\k$-module of rank $1$. 
\end{Lemma}
\begin{proof}
By Lemma~\ref{LIndStdProj} and adjointness of induction and restriction, 
\begin{align*}
1_{\bi} U
&\simeq\, \Hom_{R_\theta}(R_\theta 1_{\bi}, U)
\\
&\simeq  \Hom_{R_{m_1\al_{i_1},\dots,m_r\al_{i_r}}}(R_{m_1\al_{i_1}}1_{i_1^{(m_1)}}\boxtimes\dots\boxtimes R_{m_r\al_{i_r}}1_{i_r^{(m_r)}}, \Res^\theta_{m_1\al_{i_1},\dots,m_r\al_{i_r}}U)
\end{align*}
and 
\begin{align*}
1_{\hat\bi} U
&\simeq\, \Hom_{R_\theta}(R_\theta 1_{\hat\bi}, U)
\\
&\simeq  \Hom_{R_{m_1\al_{i_1},\dots,m_r\al_{i_r}}}(R_{m_1\al_{i_1}}\boxtimes\dots\boxtimes R_{m_r\al_{i_r}}, \Res^\theta_{m_1\al_{i_1},\dots,m_r\al_{i_r}}U)
\end{align*}
So the isomorphism $
 1_{\hat\bi} U\simeq  (1_\bi  U)^{\oplus q^{\langle \bi\rangle}\bi! }$ follows from Lemma~\ref{LNilHeckeProj}.
The second isomorphism is proved similarly. The final statement follows easily from the first isomorphism.  
\end{proof}

\begin{Lemma} \label{LDivIdCommutation} 
Let $n\in\N_+$, $\eta\in Q_+$ with $\height(\eta)=m$, $i\in I$, $\bj\in I^\eta$, and $\theta=\eta+n\al_i$. 
\begin{enumerate}
\item[{\rm (i)}] Let $x\in\Si_{m+n}$ be the permutation mapping $k\mapsto k+m$ for $k=1,\dots,n$ and $n+l\mapsto l$ for $l=1,\dots,m$. Then in $R_\theta$ we have $\psi_x 1_{i^{(n)}\bj}=1_{\bj i^{(n)}}\psi_x 1_{i^{(n)}\bj}$ (for any choice of a reduced decomposition for $x$ used to define $\psi_x$). 
\item[{\rm (ii)}] Let $y\in\Si_{m+n}$ be the permutation mapping $k\mapsto k+n$ for $k=1,\dots,m$ and $m+l\mapsto l$ for $l=1,\dots,n$. Then in $R_\theta$ we have $\psi_y1_{\bj i^{(n)}}=1_{i^{(n)}\bj}\psi_y1_{\bj i^{(n)}}$ (for any choice of a reduced decomposition for $y$ used to define $\psi_y$). 
\end{enumerate}
\end{Lemma}
\begin{proof}
We prove (i), the proof of (ii) being similar. Recall that $1_{i^{(n)}}=\psi_{w_n}f(y_1,\dots,y_n)$, where $f(y_1,\dots,y_n)=
\prod_{s=1}^n y_s^{s-1}$. We consider $w_n$ as an element of $S_{m+n}$ via the natural embedding $S_n\to S_n\times S_m\leq S_{m+n}$. When we use the embedding  $S_n\to S_m\times S_n\leq S_{m+n}$ instead, we denote the image of $w_n$ by $w_n'$ (it permutes the last $n$ letters). 

Now, we have, using the relations in $R_\theta$, 
\begin{align*}
\psi_x1_{i^{(n)}\bj}&=1_{\bj i^{n}}\psi_x1_{i^{(n)}\bj}
\\&
=1_{\bj i^{n}}\psi_x\psi_{w_n}f(y_1,\dots,y_n)1_{i^n\bj}
\\&=1_{\bj i^{n}}\psi_{w_n'}\psi_xf(y_1,\dots,y_n)1_{i^n\bj}
\\&=1_{\bj i^{(n)}}\psi_{w_n'}\psi_xf(y_1,\dots,y_n)1_{i^n\bj}
\\&=1_{\bj i^{(n)}}\psi_x\psi_{w_n}f(y_1,\dots,y_n)1_{i^n\bj}
\\&=1_{\bj i^{(n)}}\psi_x1_{i^{(n)}\bj},
\end{align*}
where we have used (\ref{EStrIdemp}) for the fourth equality. 
\end{proof}

\chapter{Cuspidal algebras}

\section{Cuspidal modules and cuspidal algebras}

\subsection{Cuspidal systems}
\label{SSCuspSys}

We describe a classification of the graded irreducible $R_\theta$-supermodules in terms of cuspidal systems. Recall the terminology of \S\ref{SSLTN}. 
For now, fix an {\em arbitrary} convex preorder $\preceq$ on $\Phi_+$.

Let $\be\in\Psi$, $m\in\N_+$ and $ V\in\mod{R_{m\be}}$. Then $ V$ is called {\em cuspidal} if $\Res_{\theta,\eta} V\neq 0$ for $\theta,\eta\in Q_+$ implies that $\theta$ is a sum of positive roots $\preceq \be$ and $\eta$ is a sum of positive roots $\succeq\be$ (this and other definitions below depend on the choice of $\preceq$). 

A word $\bi\in I^{m\be}$ is called {\em cuspidal} if $\bi=\bj\bk$ for words $\bj,\bk$ implies that $\wt(\bj)$ is a sum of positive roots $\preceq \be$ and $\wt(\bk)$ is a sum of positive roots $\succeq\be$. 
We denote by 
$I^{m\be}_\cus$ the set of all cuspidal words in $I^{m\be}$. 

Let $(1_\bi \mid \bi\in I^{m\be}\setminus I^{m\be}_\cus)$ be the two-sided ideal of $R_{m\be}$ generated by all $1_\bi$ with $\bi\in I^{m\be}\setminus I^{m\be}_\cus$. 
Then $ V\in\mod{R_{m\be}}$ is cuspidal if and only if every word of $ V$ is cuspidal if and only if $1_\bi  V=0$ for all $\bi\in I^{m\be}\setminus I^{m\be}_\cus$ if and only if $ V$ 
 factors through a module over the {\em cuspidal algebra} 
\begin{equation}\label{ECuspidalAlgebra}
\bar R_{m\be}:=R_{m\be}/(1_\bi \mid \bi\in I^{m\be}\setminus I^{m\be}_\cus).
\end{equation}
Thus $\mod{\bar R_{m\be}}$ can be identified with the full subcategory of $\mod{R_{m\be}}$ which consists of  the (finitely generated) cuspidal (graded super)modules. 

For an element $u\in R_{m\be}$ we denote its image under the natural projection $R_{m\be}\onto \bar R_{m\be}$ again by $u$. So we have elements $1_{m\be},\, y_1\dots,y_{mh},\,\psi_1\dots,\psi_{mh-1}$, etc. in $\bar R_{m\be}$ where $h=\height(\beta)$.

A {\em cuspidal system} (for the fixed convex preorder $\preceq$) is the following data:
\begin{enumerate}
\item[{\rm (Cus1)}] An irreducible cuspidal $L_\be\in\mod{R_\be}$ for every $\be\in \Phi_+^\re$;
\item[{\rm (Cus2)}] An irreducible cuspidal $L_\bmu\in\mod{R_{n\de}}$ for every $n\in\N_+$ and every multipartition $\bmu\in\Par^J(n)$, such that $L_\bmu\not\cong L_\bnu$ if $\bmu\neq\bnu$. 
\end{enumerate}
We call the (graded super)modules $L_\be$ from (Cus1) {\em real irreducible cuspidal modules}, and 
the modules $L(\bmu)$ from (Cus2) {\em  imaginary irreducible cuspidal modules}.


Let $(\um,\bmu)\in \Par(\theta)$ for some $\theta\in Q+$. As  almost all $m_\be$ are zero, we can choose a finite subset
$
\be_1\succ\dots\succ\be_s\succ\de\succ\be_{-t}\succ\dots\succ\be_{-1}
$
of $\Psi$ such that $m_\be=0$ for $\be$'s outside of this subset. Then, denoting $m_u:=m_{\be_u}$, we can write $(\um,\bmu)$ in the form
\begin{equation}\label{EStandForm}
(\um,\bmu)=(\be_1^{m_1},\dots,\be_s^{m_s},\bmu,\be_{-t}^{m_{-t}},\dots,\be_{-1}^{m_{-1}}).
\end{equation}
Denote also 
$$
\lan\um\ran:=(m_1\be_1,\dots,m_s\be_s,m_\de\de,m_{-t}\be_{-t},\dots,m_{-1}\be_{-1})\in Q_+^{s+t+1},
$$
so we have a parabolic subalgebra $R_{\lan\um\ran}\subseteq 1_{\lan\um\ran}R_\theta 1_{\lan\um\ran}$, cf. \S\ref{SSIndRes}.

We define a partial order on $\Par(\theta)$. 
First, there are left and write lexicographic orders $\leq_l$ and $\leq_r$ on the finitary tuples $\um=(m_{\be})_{\be\in\Psi}$, and for two such tuples $\um$ and $\un$ we set 
$\um\leq\un$ if $\um\leq_l \un$ and $\um\geq_r\un$. 
As in \cite{McN1,McN2}, for $(\um,\bmu),(\un,\bnu)\in\Par(\theta)$, we now set 
\begin{equation}\label{EBilex}
(\um,\bmu)\leq (\un,\bnu)\ \iff \ \um< \un,\ \text{or $\um=\un$ and}\ \bmu=\bnu.
\end{equation}

In this subsection, it will be convenient to use the following notation: for graded supermodules $V_1\in\mod{R_{\theta_1}},\,\dots,\,V_n\in\mod{R_{\theta_n}}$, we denote 
\begin{equation*}\label{ECircProd}
V_1\circ\dots\circ V_n:=\Ind_{\underline{\theta}}^\theta(V_1\boxtimes\dots\boxtimes V_n). 
\end{equation*}

Now suppose that $\{L_\be,L_\bmu\mid\be\in\Phi_+^\re,\bmu\in\Par^J\}$ is a cuspidal system for our fixed convex preorder. For a root partition 
$(\um,\bmu)\in\Uppi(\theta)$ 
written in the form (\ref{EStandForm}), we define  
the  {\em (proper) standard (graded super)module} over $R_\theta$:
\begin{equation}\label{EStand}
\Stand(\um,\bmu):=
L_{\be_1}^{\circ m_1} \circ \dots\circ L_{\be_s}^{\circ m_s}\circ L(\bmu)\circ L_{\be_{-t}}^{\circ m_{-t}}\circ\dots\circ  L_{\be_{-1}}^{\circ m_{-1}}.
\end{equation}

\begin{Theorem} \label{THeadIrr} {\rm \cite[Theorem 3.3.13]{KlLi}} 
Let\, $\k=\F$. 
For a given convex preorder there exists a cuspidal system $\{L_\be,L_\bmu\mid \be\in \Phi_+^\re,\,\bmu\in\Par^J\}$. Moreover: 
\begin{enumerate}
\item[{\rm (i)}] For every root partition $(\um,\bmu)$, the standard module  
$
\Stand(\um,\bmu)
$ has an irreducible head; denote this irreducible module $L(\um,\bmu)$.

\item[{\rm (ii)}] $\{L(\um,\bmu)\mid (\um,\bmu)\in \Par(\theta)\}=\Irr(R_\theta)$. 

\item[{\rm (iii)}] $[\Stand(\um,\bmu):L(\um,\bmu)]=1$, and $[\Stand(\um,\bmu):L(\un,\bnu)]\neq 0$ implies $(\un,\bnu)\leq (\um,\bmu)$. 

\item[{\rm (iv)}] $\Res_{\lan\un\ran}L(\um,\bmu)\neq 0$ implies $\un\leq \um$ and 
$$\Res_{\lan\um\ran}L(\um,\bmu)\cong L_{\be_1}^{\circ m_1} \boxtimes \dots\boxtimes L_{\be_s}^{\circ m_s}\boxtimes L(\bmu)\boxtimes L_{\be_{-t}}^{\circ m_{-t}}\boxtimes\dots\boxtimes  L_{\be_{-1}}^{\circ m_{-1}}.$$ 

\item[{\rm (v)}] Let $\be\in \Phi_+^\re$ and $n\in\N_+$. Then $\Stand(\be^n)=L_\be^{\circ n}$ is irreducible, i.e. $L(\be^n)=L_\be^{\circ n}$. 

\item[{\rm (vi)}] Let $n\in\N_+$. Then $\{L(\bmu)\mid\bmu\in \Par^J(n)\}=\Irr(\bar R_{n\de})$ and $\{L(\be^n)\}=\Irr(\bar R_{n\be})$ for any $\be\in\Phi_+^\re$. 
\end{enumerate}
\end{Theorem}

Let $\k=\F$. It is clear from Theorem~\ref{THeadIrr} that for a real root $\beta$, the irreducible real cuspidal (graded super)module $L_\be$ is 
the only irreducible $R_\beta$-module that is cuspidal, and so the choice of $L_\be$ is canonical. On the other hand, the irreducible imaginary cuspidal $R_{n\de}$-modules $\{L(\bmu)\mid \bmu\in\Par^J(n)\}$ are determined up to permutation of the set $ \Par^J(n)$ of their labels, and so in this sense the definition of $L(\bmu)$ (and of $\Stand(\um,\bmu)$) is {\em not canonical}. For now $L(\bmu)$ is just one of the irreducible imaginary cuspidal modules over $R_{n\de}$. One of the goals of this paper is to label the irreducible imaginary cuspidal modules over each $R_{n\de}$ canonically by the multipartitions $\bmu\in\Par^J(n)$.

\subsection{Induction and restriction for cuspidal modules}
We continue to work with a fixed arbitrary convex preorder $\preceq$. 
Let $\be\in\Psi$, $d,n\in\N_+$ and $\la=(\la_1,\dots,\la_n)\in\Comp(n,d)$. Denote 
$$
\la\be:=(\la_1\be,\dots,\la_n\be)\in Q_+^n.
$$ 
Recalling (\ref{ETheta}), we have the idempotent $1_{\la\be}\in R_{d\be}$ and the parabolic subalgebra 
$R_{\la\be}\subseteq 1_{\la\be}R_{d\be}1_{\la\be}$ identified with $R_{\la_1\be}\otimes\dots\otimes R_{\la_n\be}$ via the embedding $\iota_{\la\be}$ of (\ref{EIota}). So we can consider $\bar R_{\la_1\be}\otimes\dots\otimes \bar R_{\la_n\be}$ as a quotient of $R_{\la\be}$. 
Define the {\em cuspidal parabolic subalgebra} $\bar R_{\la\be}\subseteq 1_{\la\be} \bar R_{d\be}1_{\la\be}$ to be the image of $R_{\la\be}$ under the natural projection $1_{\la\be}R_{d\be}1_{\la\be}\onto 1_{\la\be}\bar R_{d\be}1_{\la\be}$. 

\begin{Lemma}\label{L030216} 
{\rm \cite[Lemma 3.3.21]{KlLi}} 
Let $h=\height(\be)$ and $\la\in\Comp(n,d)$.
\begin{enumerate}
\item[{\rm (i)}] The natural map
$$
R_{\la_1\be}\otimes\dots\otimes R_{\la_n\be}\stackrel{\iota_{\la\be}}{\longrightarrow} 1_{\la\be}R_{d\be}1_{\la\be}\onto 1_{\la\be}\bar R_{d\be}1_{\la\be}
$$
factors through $\bar R_{\la_1\be}\otimes\dots\otimes \bar R_{\la_n\be}$ and induces an isomorphism
$$\bar R_{\la_1\be}\otimes\dots\otimes \bar R_{\la_n\be}\iso \bar R_{\la\be}.$$ 
\item[{\rm (ii)}] $\bar R_{d\be} 1_{\la\be}$ is a free right $\bar R_{\la\be}$-module with basis
$\{ \psi_w1_{\la\be} \mid w\in \D_{hd}^{h\la}\}$.
\item[{\rm (iii)}] $1_{\la\be}\bar R_{d\be}$ is a free left $\bar R_{\la\be}$-module with basis
$\{ 1_{\la\be}\psi_w \mid w\in {}^{h\la}\D_{hd}\}$.
\end{enumerate}
\end{Lemma}

In view of the lemma we identify 
\begin{equation}\label{ECuspPar}
\bar R_{\la_1\be}\otimes\dots\otimes \bar R_{\la_n\be}=\bar R_{\la\be}. 
\end{equation}
We call a $R_{\la\be}$-module cuspidal if it factors through $\bar R_{\la\be}$. (This agrees with the definition of a cuspidal $R_{d\de}$-modules for the case $\la=(d)$.) The following two lemmas show, among other things, that cuspidality is preserved under parabolic induction and restriction. 

\begin{Lemma} \label{LCuspRes} 
{\rm \cite[Lemma 3.3.18]{KlLi}}
Let $\la\in\Comp(d)$. 
If $ V\in\mod{R_{d\be}}$ is cuspidal then so is\, 
$\Res^{d\be}_{\la\be} V\in\mod{R_{\la\be}}$.
\end{Lemma}

\begin{Lemma} \label{LTensImagIsImag} 
{\rm \cite[Lemmas 3.3.18,\,3.3.20, Corollary 3.3.22]{KlLi}}
Let $\la\in\Comp(n,d)$. 
\begin{enumerate}
\item[{\rm (i)}] 
If $V\in \mod{R_{\la\be}}$ is cuspidal then so is the $R_{d\be}$-module $\Ind_{\la\be}^{d\be} V$ and there is a natural bidegree $(0,\0)$ isomorphism of graded $R_{d\be}$-supermodules 
\begin{align*}
\Ind_{\la\be}^{d\be} V\iso \bar R_{d\be}1_{\la\be}\otimes_{\bar R_{\la\be}} V, \  1_{\la\be}\otimes v\mapsto
1_{\la\be}\otimes v.
\end{align*}

\item[{\rm (ii)}] If $\bi^{(r)}\in I^{\la_r\be}_\cus$ for $r=1,\dots,n$, then there is a bidegree $(0,\0)$ isomorphism of graded $R_{d\be}$-supermodules 
\begin{align*}
\bar R_{d\be}1_{{\bi^{(1)}\cdots\bi^{(n)}}}&\iso \Ind_{\la\be}^{d\be}\big(\bar R_{\la_1\be}1_{{\bi^{(1)}}}\boxtimes\dots\boxtimes \bar R_{\la_n\be}1_{{\bi^{(n)}}}\big),
\\
1_{{\bi^{(1)}\cdots\bi^{(n)}}}
&\mapsto
1_{\la\be}\otimes 1_{{\bi^{(1)}}}\otimes \dots\otimes 1_{{\bi^{(n)}}}.
\end{align*}
\end{enumerate}
\end{Lemma}

From now on, for cuspidal $V\in \mod{R_{\la\be}}$, we identify the induced (graded super)modules 
$\Ind_{\la\be}^{d\be} V=R_{d\be}1_{\la\be}\otimes_{R_{\la\be}} V$ and $\bar R_{d\be}1_{\la\be}\otimes_{\bar R_{\la\be}} V$ 
as in Lemma~\ref{LTensImagIsImag}(i):
\begin{equation*}\label{EIdInd}
\Ind_{\la\be}^{d\be} V=R_{d\be}1_{\la\be}\otimes_{R_{\la\be}} V=\bar R_{d\be}1_{\la\be}\otimes_{\bar R_{\la\be}} V.
\end{equation*}

\subsection{Imaginary cuspidal algebra}
Fix $d\in\N_+$. 
From now on, for the case $\be=\de$, we redenote the cuspidal algebra $\bar R_{d\be}$ of (\ref{ECuspidalAlgebra}) by $\hat C_d$:
\begin{equation}\label{EHatC}
\hat C_d:=\bar R_{d\de}.
\end{equation}
and refer to it as the {\em imaginary cuspidal (graded super)algebra}. 
Similarly, for $\la\in\Comp(n,d)$, we denote
\begin{equation}\label{ECuspParHat}
\hat C_\la:=\bar R_{\la\de},
\end{equation}
so in view of (\ref{ECuspPar}), we have $\hat C_\la\cong \hat C_{\la_1}\otimes \dots\otimes \hat C_{\la_n}$. According to (\ref{ECuspPar}), we always identify:
\begin{equation}\label{EHatParabolic}
\hat C_\la= \hat C_{\la_1}\otimes \dots\otimes \hat C_{\la_n}
\end{equation}

\begin{Lemma} \label{LImCuspIrrAmount} 
Let $\k=\F$. We have $|\Irr(\hat C_d)|=|\Par^J(d)|$. 
\end{Lemma}
\begin{proof}
This follows from Theorem~\ref{THeadIrr}(vi). 
\end{proof}

\subsection{Imaginary Mackey Theorem}
\label{SSImagMackey}

Let 
\begin{equation}\label{EIotaSd}
\rho_d:\Si_d\to \Si_{dp},\ s_r\mapsto t_r
\end{equation}
be the group embedding 
where $t_r$ is defined as the product of transpositions 
$$
t_r:=\prod_{k=rp-p+1}^{rp}(k,k+p)
\qquad(1\leq r<n).
$$
In other words, for $w\in \Si_d$, $\rho_d(w)$ permutes $d$ blocks of size $p$ according to $w$. 

Recall that the element $\psi_{\rho_d(w)}$ in general depends on the choice of a reduced decomposition for $\rho_d(w)$. We will always use a reduced decomposition consistent with a reduced decomposition for $w$, i.e. for a reduced decomposition $w=s_{r_1}\dots s_{r_k}$, we take $\psi_{\rho_d(w)}=\psi_{t_{r_1}}\dots \psi_{t_{r_k}}$ (now $\psi_{\rho_d(w)}$ only depends on the choice of a reduced decomposition of $w$).

Recall the notation from Section~\ref{SSSG}. In particular, given two compositions $\la,\mu\in\Comp(d)$ and $ x\in {}^\la\D_d^\mu$ we have compositions $\la\cap  x\mu$ and $ x^{-1}\la\cap\mu$ in $\Comp(d)$. Moreover, the corresponding parabolic subalgebras $\hat C_{\la\cap  x\mu}$ and $\hat C_{x^{-1}\la\cap\mu}$ are naturally isomorphic  via an isomorphism 
$
\hat C_{\la\cap  x\mu}\iso \hat C_{x^{-1}\la\cap\mu}
$ 
which permutes the components. Composing with this isomorphism we get a functor 
\begin{equation}\label{ETwistCusp}
\mod{\hat C_{x^{-1}\la\cap\mu}}\to \mod{\hat C_{\la\cap  x\mu}},\ M\mapsto {}^ xM.
\end{equation}
Note that this functor is a special case of the functor (\ref{ETwist}) 
i.e. ${}^ xM={}^ {\{h(x)\}}M$ for a permutation $h(x)$ of the parts of $x^{-1}\la\cap\mu$ which corresponds to the conjugation by $x$. 
Moreover, we do not need any grading and parity shifts here since $(\de,\de)=0$. 

Let $V\in\mod{\hat C_{\mu}}$. Note that 
for any $x\in {}^\la\D_d^\mu$, we have 
$1_{\la\de}\psi_{\rho_d(x)} 1_{\mu\de}\otimes V
\subseteq \Res_{\la\de}^{d\de} \Ind_{\mu\de}^{d\de} V.$
Let $\leq$ be a total order refining the Bruhat order on $\Si_d$. 
For $x\in {}^\la\D_d^\mu$, we consider the submodules
\begin{align*}
\bar F_{\leq x}(V)&:=\sum_{y\in {}^\la\D_d^\mu\,\text{with}\, y\leq x}R_{\la\de}1_{\la\de}\psi_{\rho_d(y)} 1_{\mu\de}\otimes V 
\subseteq 
\Res_{\la\de}^{d\de} \Ind_{\mu\de}^{d\de} V,
\\
\bar F_{< x}(V)&:=\sum_{y\in {}^\la\D_d^\mu\,\text{with}\, y< x}R_{\la\de}1_{\la\de}\psi_{\rho_d(y)} 1_{\mu\de}\otimes V 
\subseteq 
\Res_{\la\de}^{d\de} \Ind_{\mu\de}^{d\de} V.
\end{align*}


\begin{Theorem} \label{TImMackey} 
{\bf (Imaginary Mackey Theorem)}
Let $\la,\mu\in\Comp(d)$, and $V\in\mod{\hat C_{\mu}}$. Then we have the filtration 
 $
 (\bar F_{\leq x}(V))_{x\in {}^\la\D_d^\mu}
 $
 of  
$\Res_{\la\de}^{d\de} \Ind_{\mu\de}^{d\de} V$.
Moreover, the sub-quotients of the filtration are
$$
\bar F_{\leq x}(V)/\bar F_{< x}(V)\simeq 
\Ind_{(\la\cap  x\mu)\de}^{\la\de}\,{}^{x}(\Res^{\mu\de}_{(x^{-1}\la\cap\mu)\de} V). 
$$
\end{Theorem}
\begin{proof}
This follows from the Mackey Theorem~\ref{TMackeyKL} and the definition of cuspidal modules, cf. \cite[Theorem 4.5.4]{KMImag}.
Indeed, by Theorem~\ref{TMackeyKL}, we have the explicit filtration 
 $
 (F_{\leq \underline{\ga}}(V))_{\underline{\ga}\in\D(\la\de,\mu\de)}
 $
 of 
$\Res_{\la\de}^{d\de} \Ind_{\mu\de}^{d\de} V$ with $F_{\leq \underline{\ga}}$ given by 
(\ref{EFiltF}) and the sub-quotients 
$$
S_\ga(V)\simeq 
\Ind_{\ga_{1,1},\dots,\ga_{1,n}\,;\,\dots\,;\,\ga_{m,1},\dots,\ga_{m,n}}^{\,\la_1\de;\,\dots\,;\,\la_m\de}
{}^{\{w_{m,n}\}}\big(\Res_{\ga_{1,1},\dots,\ga_{m,1}\,;\,\dots\,;\,\ga_{1,n},\dots,\ga_{m,n}}^{\,\mu_1\de;\,\dots\,;\,\mu_n\de}
\, V \big),
$$
where $\la=(\la_1,\dots,\la_m)$ and $\mu=(\mu_1,\dots,\mu_n)$. 

Suppose $S_\ga(V)\neq 0$. 
By the definition of cuspidal modules, we have for $b=1,\dots,n$ that $\ga_{1,b}$ is a sum of positive roots $\preceq \de$. Since $\la_1\de=\sum_{b=1}^n\ga_{1,b}$, it follows from the definition of a convex preorder that all $\ga_{1,b}$ must be multiples of $\de$. Now, by Lemma~\ref{LCuspRes}, the module 
$$\Res_{\ga_{1,1},\ga_{2,1}+\dots+\ga_{m,1}\,;\,\dots\,;\,\ga_{1,n},\ga_{2,n}+\dots+\ga_{m,n}}^{\,\mu_1\de;\,\dots\,;\,\mu_n\de}
\, V$$ 
is cuspidal, so for $b=1,\dots,n$ we have that $\ga_{2,b}$ is a sum of positive roots $\preceq \de$. Since $\la_2\de=\sum_{b=1}^n\ga_{2,b}$, it follows from the definition of a convex preorder that all $\ga_{2,b}$ must be multiples of $\de$. 
Continuing this way we deduce that all $\ga_{a,b}$ are multiples of $\de$. Thus, $S_\ga(V)\neq 0$ only if 
$$
\underline{\ga}\in\D(\la\de,\mu\de)_\cus:=\{\underline{\ga}\in \D(\la\de,\mu\de)\mid \text{all $\ga_{a,b}$ are multiples of $\de$}\}.
$$

We now get a bijection 
$\zeta:\D(\la\de,\mu\de)_\cus\to{}^\la\D_d^\mu 
$
with $\rho_d(\zeta(\underline{\ga}))=x(\underline{\ga})$, so 
$F_{\leq \underline{\ga}}(V)=\bar F_{\leq \zeta(\underline{\ga})}(V)$, and the result follows. 
\end{proof}

\begin{Example} \label{EXImMackey}
Let $d=3$, $\la=(1,2)$ and $\mu=(2,1)$. 
Then 
$
{}^{\la}\D^{\mu}_3=\{1,s_1s_2\},
$ 
and $\la\cap\mu=(1,1,1)$, $\la\cap s_1s_2\mu=\la$, $s_2s_1\la\cap \mu=\mu$.
Now, according to Theorem~\ref{TImMackey}, $\Res_{\de,2\de}^{3\de} \Ind_{2\de,\de}^{3\de} V$ has a submodule
$$\bar F_{\leq 1}=R_{\de,2\de}1_{\de,2\de}1_{2\de,\de}\otimes V
\simeq \Ind_{\de;\de,\de}^{\de;2\de}\Res_{\de,\de;\de}^{2\de;\de} V
$$ 
and
$$
\Res_{\de,2\de}^{3\de} \Ind_{2\de,\de}^{3\de} V/F_{\leq 1}\simeq
\Ind_{\de,2\de}^{\de,2\de}\,{}^{s_1s_2}(\Res^{2\de,\de}_{2\de,\de} V)={}^{s_1s_2}V.
$$
\end{Example}

\section{Gelfand-Graev truncation of imaginary cuspidal algebra}
\label{SGG}
{\em From now on, we fix a convex preorder $\preceq$ on $\Phi_+$ as in Example~\ref{ExConPr}}, in particular (\ref{ESharp}) holds. When we speak of cuspidality from now on, we mean cuspidality with respect to this specially chosen $\preceq$. 

Recall from (\ref{ECuspidalAlgebra}) and (\ref{EHatC}) that for $d\in \N_+$, the imaginary cuspidal algebra $\hat C_{d}$ is an explicit  quotient of $R_{d\de}$. The algebra $R_{d\de}$ is generated by the elements $\psi_r,y_s,1_\bi $, and we use the same notation for the corresponding elements of $\hat C_d$.

\subsection{Gelfand-Graev words}\label{SSGGW}
For $j\in J$, in \cite[\S4.2a]{KlLi}, we have defined (divided power) words
\begin{equation*}
\begin{split}
\ggw^{j}&:=\ell\, (\ell-1)^{(2)}\,\cdots\, ( j+1)^{(2)}\,  j\, \cdots\, 1\, 0^{(2)}\, 1\, \cdots\,  j\,\in\, I^{\de}_\di. 
\end{split}
\end{equation*}
For $j\in J$ and $m\in \N$, we will consider more general divided power words in $I^{m\de}_\di$:
\begin{equation}\label{EGGW}
\begin{split}
\ggw^{m,j}&:=\ell^{(m)}(\ell-1)^{(2m)}\,\cdots\, ( j+1)^{(2m)} j^{(m)}\cdots\, 1^{(m)} 0^{(2m)}1^{(m)}\cdots\,  j^{(m)}. 
\end{split}
\end{equation}
Using the notation from (\ref{EHatI}), we also have the words 
$$
\hat\ggw^{m,j}:=\widehat{\ggw^{m,j}}=\ell^m\,(\ell-1)^{2m}\,\cdots\, ( j+1)^{2m}\, j^m\,\cdots \,1^m\, 0^{2m}\,1^m\,\cdots\,  j^m\,\in \,I^{m\de}.
$$
Let $d\in \N$, $n\in\N_+$ and  
recall the 
colored compositions $\Comp^\col(n,d)$ from \S\ref{SSPar}. 

\begin{Lemma} \label{LIndecGGW}
Let $\bi^{(1)}\in I^{d_1\de},\dots,\bi^{(k)}\in I^{d_k\de}$ for $d_1,\dots,d_k\in\N_+$. If\, $\hat\ggw^{m,j}=\bi^{(1)}\cdots\bi^{(k)}$ then $k=1$. 
\end{Lemma}

We refer to the words $\hat\ggw^{m,j}$ as {\em indecomposable Gelfand-Graev words} and to the divided power words $\ggw^{m,j}$ as {\em indecomposable divided power Gelfand-Graev words}. From the definition:

Given $(\mu,\bj)\in\Comp^\col(n,d)$, we define more general {\em Gelfand-Graev words}
\begin{equation}\label{EGHatG}
\ggw^{\mu,\bj}:=\ggw^{\mu_1,j_1}\cdots \ggw^{\mu_n,j_n}\in I^{d\de}_\di
\quad \text{and}\quad \hat\ggw^{\mu,\bj}:=\widehat{\ggw^{\mu,\bj}}=\hat\ggw^{\mu_1,j_1}\cdots \hat\ggw^{\mu_n,j_n}\in I^{d\de}. 
\end{equation}
Recalling the set $\EC^\col(d)$ of colored {\em essential}  compositions  from \S\ref{SSPar}, note that the words $\{\hat \ggw^{\mu,\bj}\mid (\mu,\bj)\in\EC^\col(d)\}$ are distinct and exhaust all non-divided power Gelfand-Graev words in $I^{d\de}$. Similalrly, the divided power words $\{\ggw^{\mu,\bj}\mid (\mu,\bj)\in\EC^\col(d)\}$ are distinct and exhaust all divided power Gelfand-Graev words in $I^{d\de}$. We introduce the sets of {\em rank $d$ Gelfand-Graev words}:
$$
\GGW^{(d)}:=\{\hat \ggw^{\mu,\bj}\mid (\mu,\bj)\in\EC^\col(d)\}\quad\text{and}\quad \GGW^{(d)}_\di:=\{\ggw^{\mu,\bj}\mid (\mu,\bj)\in\EC^\col(d)\}
$$
It is clear from the definitions that $\GGW^{(d)}\subseteq I^{d\de}_\cus$. Moreover, from the definitions:

\begin{Lemma} \label{LGGSplit} 
Let $\mu\in\Comp(n,d)$ and $\bi^{(k)}\in I^{\mu_k\de}$ for all $k=1,\dots,n$. Then 
$\bi^{(k)}\in \GGW^{(\mu_k)}$ for all $k=1,\dots,n$ if and only if $\bi^{(1)}\cdots\bi^{(n)}\in\GGW^{(d)}$. 
\end{Lemma}

Recall the notation (\ref{EIotaSd}) and (\ref{EWLaSi}).

\begin{Lemma} \label{LTrickyGGW}
Let $(\nu,\bj)\in\EC^\col(n,d)$, $\la\in\EC(m,d)$ and $w\in{}^\la\D_d^\nu$. If $\bi^{(r)}\in I^{\la_r\de}$ for $r=1,\dots,m$. Suppose that  
$\rho_d(w)\cdot\hat \ggw^{\nu,\bj}=\bi^{(1)}\cdots\bi^{(m)}$. Then $m=n$, $w=w_{\nu,\si}$ and $\la=\si\mu$ for some $\si\in\Si_n$,  and 
$$\rho_d(w)\cdot\hat \ggw^{\nu,\bj}=\hat\ggw^{\nu_{\si^{-1}(1)},j_{\si^{-1}(1)}}\cdots \hat\ggw^{\nu_{\si^{-1}(n)},j_{\si^{-1}(n)}},
$$
in particular, $\rho_d(w)\cdot\hat \ggw^{\nu,\bj}\in\GGW^{(d)}$. 
\end{Lemma}
\begin{proof}
We may assume that all parts of $\nu$ and all parts of $\la$ are non-zero. 
We use the notation of \S\ref{SSSG}. In particular, $w=w_A$ for some $A\in M(\la,\nu)$ and $w_A:\ttB^A_{r,s}\nearrow\ttT^A_{r,s}$ for all $1\leq r\leq m$ and $1\leq s\leq n$. Moreover, 
$[1,dp]$ is an increasing disjoint union of the segments 
$
\hat\ttB^{A}_{1,1},\dots,\hat\ttB^{A}_{m,1},\dots,\hat\ttB^{A}_{1,n},\dots,\hat\ttB^{A}_{m,n},
$
of sizes $|\hat\ttB^{A}_{r,s}|=a_{r,s}p$ as well as  
an increasing disjoint union of the segments 
$
\hat\ttT^{A}_{1,1},\dots,\hat\ttT^{A}_{1,n},\dots,\hat\ttT^{A}_{m,1},\dots,\hat\ttT^{A}_{m,n},
$
of sizes $\hat |\ttT^{A}_{r,s}|=a_{r,s}p$,  
and $\rho_d(w_A):\hat\ttB^A_{r,s}\nearrow\hat\ttT^A_{r,s}$ for all $1\leq r\leq m$ and $1\leq s\leq n$. 

For $s=1,\dots, n$, we can write 
$\hat\ggw^{\nu_s,j_s}=\bk^{(1,s)}\cdots\bk^{(m,s)}$ for $\bk^{(r,s)}\in I^{a_{r,s}p}$. Then $\bi^{(r)}=\bk^{(r,1)}\cdots\bk^{(r,n)}$. 
Since $\hat\ggw^{\nu_s,j_s}$ is cuspidal, we have  $\bk^{(1,s)}\in I^{\theta_{1,s}}$ for $\theta_s$ a sum of positive roots $\preceq\de$. 
Since $\bi^{(1)}=\bk^{(1,1)}\cdots\bk^{(1,n)}\in I^{\la_1\de}$, we deduce using by Lemma~\ref{LCon3} that each $\theta_{1,s}$ is a multiple of $\de$. Now, we deduce similarly that each $\theta_{2,s}$ is a multiple of $\de$, \dots, each $\theta_{m,s}$ is a multiple of $\de$.

Using Lemma~\ref{LIndecGGW}, we now deduce that $m=n$ and there is $\si\in\Si_n$ such that 
$\la_r=a_{r,\si^{-1}(r)}=\mu_{\si^{-1}(r)}$ for all $r=1,\dots,n$. So $w=w_{\mu,\si}$, and the rest follows. 
\end{proof}

In some frequently occurring special cases it will be convenient to use a special notation. 
In the special case where  $\bj=j^n$, we denote
\begin{equation}\label{EGGIOneColorNot}
(\mu,j):=(\mu,j^n).
\end{equation}
In particular, we have 
\begin{equation}\label{EGGIOneColorIdNotNew}
\ggw^{\mu,j}:=\ggw^{\mu,j^n}\quad\text{and}\quad   \hat\ggw^{\mu,j}:=\hat\ggw^{\mu,j^n}.
\end{equation}
In the other special case where $\mu=\om_d$ and $\bj=j_1\cdots j_d$, we denote
\begin{equation}\label{EGGWJ^n}
 \ggw^{\bj}:=\ggw^{\om_d,\bj}=\ggw^{j_1}\cdots \ggw^{j_d}.
\end{equation}

\subsection{Gelfand-Graev idempotents}
\label{SSGGIdempotents}
For $(\mu,\bj)\in\Comp^\col(n,d)$, the corresponding {\em Gelfand-Graev idempotent} is 
\begin{equation}\label{EGGIdempotent}
\ggi^{\mu,\bj}:=1_{\ggw^{\mu,\bj}}\in \hat C_d.
\end{equation}
This is a divided power idempotent, cf. (\ref{EDPId}). 
Occasionally, we will also need the non-divided power version:
\begin{equation}\label{EGGIdempotentNonDiv}
\hat\ggi^{\mu,\bj}:=1_{\hat\ggw^{\mu,\bj}}\in \hat C_d.
\end{equation}
Recalling the special cases (\ref{EGGW}), (\ref{EGGIOneColorIdNotNew}) and (\ref{EGGWJ^n}), we also denote 
\begin{equation}\label{EGGIOneColorIdNot'}
\ggi^{d,j}:=\ggi^{(d),j}=1_{\ggw^{d,j}},
\quad
\ggi^{\mu,j}:=1_{\ggw^{\mu,j}}
\quad\text{and}\quad \ggi^{\bj}:=1_{\ggw^\bj}.
\end{equation}


\subsection{  Gelfand-Graev idempotent truncation $C_d$ of $\hat C_d$}
We have observed in \S\ref{SSGGW} 
that the words $\{\hat \ggw^{\mu,\bj}\mid (\mu,\bj)\in\EC^\col(d)\}$ are distinct, so the 
idempotents $\{\hat \ggi^{\mu,\bj}\mid (\mu,\bj)\in\EC^\col(d)\}$ in $\hat C_d$ are orthogonal. Hence the 
idempotents $\{\ggi^{\mu,\bj}\mid (\mu,\bj)\in\EC^\col(d)\}$ in $\hat C_d$ are orthogonal. Define the idempotent 
\begin{equation}\label{EGad}
\ggi_d:=\sum_{(\mu,\bj)\in\EC^\col(d)}\ggi^{\mu,\bj}\in \hat C_d,
\end{equation}
and the {\em Gelfand-Graev truncated cuspidal algebra}
$$
C_d:=\ggi_d\hat C_d\ggi_d.
$$

Let $\la=(\la_1,\dots,\la_n)\in\Comp(n,d)$. 
Recalling (\ref{EHatParabolic}), 
we have the cuspidal parabolic subalgebra 
$
\hat C_{\la_1}\otimes\cdots\otimes \hat C_{\la_n}=\hat C_{\la}\subseteq 1_{\la\de}\hat C_d1_{\la\de}.
$ 
Define the idempotent 
$$
\ggi_\la=\ggi_{\la_1,\dots,\la_n}:=\ggi_{\la_1}\otimes \dots\otimes \ggi_{\la_n}\in\hat C_{\la}.
$$
Truncating with this idempotent we get the parabolic subalgebra 
\begin{eqnarray}\label{ECPar}
C_{\la}\,\,:=\,\,\ggi_{\la}\hat C_{\la}\ggi_{\la}\,\,\subseteq\,\, \ggi_{\la}C_d\ggi_\la.
\end{eqnarray}
The isomorphism $\hat C_{\la}\cong \hat C_{\la_1}\otimes\cdots\otimes \hat C_{\la_n}$ yields an isomorphism 
 $C_\la\cong C_{\la_1}\otimes \cdots\otimes C_{\la_n}$, which we use to identify  
\begin{equation}\label{CParIdentify}
C_{\la}=C_{\la_1}\otimes \cdots\otimes C_{\la_n}.
\end{equation}


\subsection{   Functors relating $\mod{\hat C_d}$ and $\mod{C_d}$} 
As a special case of (\ref{EResGeneral}) and (\ref{EIndGeneral}), we have functors 
\begin{equation}\label{EFGFunctors}
\begin{split}
 \funF_d:=\ggi_d\hat C_d\otimes_{\hat C_d}-:\mod{\hat C_d}\to\mod{C_d}, 
\\
\funG_d:=\hat C_d\ggi_d\otimes_{C_d}-:\mod{C_d}\to\mod{\hat C_d}.
\end{split}
\end{equation}
More generally, for $\la\in\Comp(n,d)$, we have functors 
\begin{equation}\label{EFGFunctorsLa}
\begin{split}
\funG_\la:=\hat C_{\la}\ggi_\la\otimes_{C_\la}-:\mod{C_\la}\to\mod{\hat C_{\la}}, 
\\
 \funF_\la:=\ggi_\la\hat C_{\la}\otimes_{\hat C_{\la}}-:\mod{\hat C_{\la}}\to\mod{C_\la}
\end{split}
\end{equation}
with
\begin{equation}\label{EFGFunctorsLaOuter}
\begin{split}
 \funG_\la(W_1\boxtimes\dots\boxtimes W_n)
 \simeq \funG_{\la_1}(W_1)\boxtimes\dots\boxtimes\funG_{\la_n}(W_n),
\\
\funF_\la(V_1\boxtimes\dots\boxtimes V_n)\simeq \funF_{\la_1}(V_1)\boxtimes\dots\boxtimes\funF_{\la_n}(V_n)
\end{split}
\end{equation}
for $V_k\in\mod{\hat C_{\la_k}}$ and $W_k\in\mod{C_{\la_k}}$, $k=1,\dots,n$. Note that 
\begin{equation}\label{EFAfterG}
\funF_\la\circ \funG_\la\simeq \id_{\mod{C_\la}}.
\end{equation}

We will eventually prove that the functors $\funF_d$ and $\funG_d$ are quasi-inverse equivalences, in particular, $C_d$ is graded Morita superequivalent to $\hat C_d$. If $\k=\F$ and $\cha \F>d$ or $\cha \F=0$, recalling the notation (\ref{EOmd}), this is known to be true even for the smaller idempotent truncation $\ggi_{\om_d}C_d\ggi_{\om_d}$ (denoted $B_d$ in \cite{KlLi}): 

\begin{Theorem} \label{TMord=1} 
{\rm \cite[Theorems 4.2.55, 4.5.9]{KlLi}}
Let $\k=\F$. 
If $\cha \F>d$ or $\cha \F=0$ then the graded superalgebras  $\ggi_{\om_d}C_d\ggi_{\om_d}$ and $C_d$ are graded Morita superequivalent to $\hat C_d$. In particular, the left module $C_d\ggi_{\om_d}$ is a projective generator for $C_d$.
\end{Theorem}


\begin{Corollary} \label{CMord=1} 
The functors $\funF_1$ and $\funG_1$ are quasi-inverse equivalences.  
\end{Corollary}
\begin{proof}
Note that $f_1=f_{\om_1}$.  By Theorem~\ref{TMord=1}, for $\k=\F$, the functors $\funF_1$ and $\funG_1$ are quasi-inverse equivalences. Now, the result for $\k=\O$ follows using Lemma~\ref{LMorExtScal}. 
\end{proof}

In Corollary~\ref{CEquivFG}(ii), we will prove that the functors $\funF_d$ and $\funG_d$ are quasi-inverse equivalences by induction on $d$ using Corollary~\ref{CMord=1} as the induction base. Until then, while working with a fixed $d\in\N_+$, we make the inductive assumption:

\begin{Inductive Assumption}\label{IA}
For all $c<d$, the functors $\funF_c$ and $\funG_c$ are quasi-inverse equivalences.  
\end{Inductive Assumption}

\begin{Lemma} \label{LGLaFLa}
Let $\la\in\Comp(d)$ have more than one non-zero part. In the presence of Inductive Assumption~\ref{IA}, the functors $\funF_\la$ and $\funG_\la$ are quasi-inverse equivalences.   
\end{Lemma}
\begin{proof}
Let $\la=(\la_1,\dots,\la_n)$. By the assumption on $\la$, we have $\la_k<d$ and $\funF_{\la_k}$ and $\funG_{\la_k}$ are quasi-inverse equivalences for all $k=1,\dots,n$ by the Inductive Assumption. In particular, we have $\hat C_{\la_k}=\hat C_{\la_k}\ggi_{\la_k}\hat C_{\la_k}$ for all $k$. 

Now, we always have $\funF_\la\funG_\la W\simeq W$ for $W\in\mod{C_\la}$, while $\funG_\la\funF_\la V\simeq V$ for $V\in\mod{\hat C_{\la}}$ is equivalent to $\hat C_{\la}=\hat C_{\la}\ggi_\la\hat C_{\la}$. But 
$$\hat C_{\la}\ggi_\la\hat C_{\la}=\hat C_{\la_1}\ggi_{\la_1}\hat C_{\la_1}\otimes\dots\otimes \hat C_{\la_n}\ggi_{\la_n}\hat C_{\la_n}
=\hat C_{\la_1}\otimes\dots\otimes \hat C_{\la_n}=\hat C_{\la}
$$ 
by the previous paragraph.
\end{proof}

\subsection{Gelfand-Graev induction and restriction}
Let $\la\in\Comp(d)$. 
As a special case of (\ref{EResGeneral}) and (\ref{EIndGeneral}), 
we have the functors of {\em Gelfand-Graev restriction and induction}:
\begin{equation}\label{EGGIR}
\begin{split}
\GGR_\la^d:=\ggi_\la C_d\otimes_{C_d}- :\mod{C_d}\to\mod{C_\la},
\\
\GGI_\la^d:=C_d\ggi_\la\otimes_{C_\la}- :\mod{C_\la}\to\mod{C_d}.
\end{split}
\end{equation}
If $\mu\in\Comp(d)$ and $\la$ is a refinement of $\mu$, we have more  generally 
\begin{align*}
\GGR_\la^\mu&=\ggi_\la C_\mu\otimes_{C_\mu}- :\mod{C_\mu}\to\mod{C_\la},
\\
\GGI_\la^\mu&=C_\mu\ggi_\la\otimes_{C_\la}- :\mod{C_\la}\to\mod{C_\mu}.
\end{align*}
We have transitivity of Gelfand-Graev induction:

\begin{Lemma} 
If $\la,\mu,\nu\in\Comp(d)$ such that $\la$ is a refinement of $\mu$ and $\mu$ is a refinement of $\nu$ then  $\GGI^\nu_\mu\circ\GGI^\mu_\la \simeq \GGI^\nu_\la$. 
\end{Lemma}
\begin{proof}
Using $\ggi_\mu\ggi_\la=\ggi_\la$, we have 
$$
C_\nu\ggi_\mu\otimes_{C_\mu}C_\mu\ggi_\la\otimes_{C_\la}V\simeq 
C_\nu\ggi_\mu\ggi_\la\otimes_{C_\la}V= C_\nu\ggi_\la\otimes_{C_\la}V
$$
for $V\in \mod{C_\la}$. 
\end{proof}

Recalling the functors (\ref{EFGFunctors}) and (\ref{EFGFunctorsLa}), note that 
\begin{equation}\label{EGGIInd}
\GGI_\la^d\simeq \funF_d\circ \Ind_{\la\de}^{d\de}\circ \funG_\la
\quad\text{and}\quad
\GGR_\la^d\simeq \funF_\la\circ \Res_{\la\de}^{d\de}\circ \funG_d.
\end{equation}
Indeed, for the first isomorphism, let $V\in\mod{C_\la}$. Using Lemma~\ref{LTensImagIsImag}(i) and the equality $1_{\la\de}\ggi_\la=\ggi_\la$, we have  
\begin{align*}
\funF_d\big( \Ind_{\la\de}^{d\de}\, \funG_\la (V)\big)
&=\ggi_d\hat C_d\otimes_{\hat C_d}\hat C_d1_{\la\de} \otimes_{\hat C_{\la}}\otimes \hat C_{\la}\ggi_\la\otimes_{C_\la}V
\\
&\simeq \ggi_d\hat C_d\ggi_\la\otimes_{C_\la}V
\\
&= C_d\ggi_\la\otimes_{C_\la}V
\\&= \GGI_\la^dV.
\end{align*}
The second isomorphism is proved similarly, but using Lemma~\ref{LCuspRes} instead of Lemma~\ref{LTensImagIsImag}. 

Even more generally, 
for $\la,\mu\in\Comp(d)$ such that $\la$ is a refinement of $\mu$ we have by a similar argument:
\begin{equation}\label{EGGIIndGen}
\GGI_\la^\mu\simeq \funF_\mu\circ \Ind_{\la\de}^{\mu\de}\circ \funG_\la
\quad\text{and}\quad
\GGR_\la^\mu\simeq \funF_\la\circ \Res_{\la\de}^{\mu\de}\circ \funG_\mu.
\end{equation}

\begin{Lemma} \label{LFFExact}
Let $\la\in\Comp(d)$. 
The functors $\GGI_\la^d$ and $\GGR_\la^d$ are exact. 
\end{Lemma}
\begin{proof}
This is clear for $\GGR_\la^d$. For $\GGI_\la^d$, if $\la$ has only one non-zero part, there is nothing to prove. Otherwise, we have 
$\GGI_\la^d\simeq \funF_d\circ \Ind_{\la\de}^{d\de}\circ \funG_\la$ by (\ref{EGGIInd}), and the exactness of $\GGI_\la^d$ follows from that of $\funF_d$, 
$\Ind_{\la\de}^{d\de}$ and  $\funG_\la$, the last one coming from Lemma~\ref{LGLaFLa}. 
\end{proof}

\begin{Lemma} \label{LfunGInd}
For $\la\in\Comp(d)$ and $V\in\mod{C_{\la}}$, we have a functorial isomorphism  $\funG_d\,\GGI^d_\la (V)\simeq \Ind_{\la\de}^{d\de}(\funG_\la\, V)$ of graded $\hat C_d$-supermodules. 
\end{Lemma}
\begin{proof}
We have 
\begin{align*}
\funG_d\,\GGI^d_\la (V)&\simeq\hat C_d\ggi_d\otimes_{C_d} C_d\ggi_\la\otimes_{C_\la} V
\\
&\simeq\hat C_d\ggi_d\ggi_\la\otimes_{C_\la} V
\\&=\hat C_d\ggi_\la\otimes_{C_\la} V
\\
&=\hat C_d 1_{\la\de}\ggi_\la\otimes_{C_\la} V
\\&\simeq \hat C_d 1_{\la\de}\otimes_{\hat C_{\la}}\hat C_\la\ggi_\la\otimes_{C_\la} V
\\&\simeq \Ind_{\la\de}^{d\de}(\funG_\la V),
\end{align*}
as required.
\end{proof}

We have that $\GGI_\la^d$ is left adjoint to $\GGR_\la^d$. 
The following lemma partially describes the right adjoint:

\begin{Lemma} \label{LAnotherAdj} 
Let $\la=(\la_1,\dots,\la_n)\in\Comp(n,d)$ and define $\la^\op:=(\la_n,\dots,\la_1)\in\Comp(n,d)$. Then for $W\in\mod{C_d}$ and $V_k\in\mod{C_{\la_k}}$ for $k=1,\dots,n$, there is a functorial isomorphism 
$$
\Hom_{C_\la}(\GGR^d_\la W\,,\, V_1\boxtimes\dots\boxtimes V_n)\simeq \Hom_{C_d}(W\,,\,\GGI_{\la^\op}^d (V_n\boxtimes\dots\boxtimes V_1)).
$$
\end{Lemma}
\begin{proof}
If $\la$ has only one non-zero part, there is nothing to prove. Otherwise, $\funF_\la$ is an equivalence with quasi-inverse $\funG_\la$ by Lemma~\ref{LGLaFLa}. Moreover, $\funG_d$ is always left adjoint to $\funF_d$. So we have 
\begin{align*}
\Hom_{C_\la}(\GGR^d_\la W, V_1\boxtimes\dots\boxtimes V_n)
&\simeq 
\Hom_{C_\la}((\funF_\la\circ \Res^{d\de}_{\la\de}\circ \funG_d) (W), V_1\boxtimes\dots\boxtimes V_n)
\\
&\simeq 
\Hom_{\hat C_{\la}}((\Res^{d\de}_{\la\de}\circ \funG_d) (W), \funG_\la(V_1\boxtimes\dots\boxtimes V_n))
\\
&\simeq 
\Hom_{\hat C_{\la}}((\Res^{d\de}_{\la\de}\circ \funG_d) (W), \funG_{\la_1}V_1\boxtimes\dots\boxtimes \funG_{\la_n}V_n)
\\
&\simeq 
\Hom_{\hat C_d}(\funG_d W, \Coind^{d\de}_{\la\de}(\funG_{\la_1}V_1\boxtimes\dots\boxtimes \funG_{\la_n}V_n))
\\
&\simeq 
\Hom_{\hat C_d}(\funG_d W, \Ind^{d\de}_{\la^\op\de}(\funG_{\la_n}V_n\boxtimes\dots\boxtimes \funG_{\la_1}V_1))
\\
&\simeq 
\Hom_{C_{d}}(W, \funF_d \Ind^{d\de}_{\la^\op\de}(\funG_{\la_n}V_n\boxtimes\dots\boxtimes \funG_{\la_1}V_1))
\\
&\simeq 
\Hom_{C_{d}}(W, \funF_d \Ind^{d\de}_{\la^\op\de}\funG_{\la^\op}(V_n\boxtimes\dots\boxtimes V_1))
\\
&\simeq 
\Hom_{C_d}(W,\GGI_{\la^\op}^d (V_n\boxtimes\dots\boxtimes V_1)),
\end{align*}
where we have used Lemma~\ref{LLV} for the fifth isomorphism.  
\end{proof}

\begin{Lemma} \label{LGGIndProjId}
Let $\la\in\Comp(n,d)$ and 
$\bi^{(1)}\in I^{\la_1\de}_\di,\dots,\bi^{(n)}\in I^{\la_n\de}_\di$ be Gelfand-Graev words. Then there is an isomorphism of graded $C_d$-supermodules
\begin{align*}
C_d1_{\bi^{(1)}\cdots\bi^{(n)}}&\iso\GGI_{\la}^d\big( C_{\la_1}1_{\bi^{(1)}}\boxtimes\,\cdots\,\boxtimes\, C_{\la_n}1_{\bi^{(n)}}\big),
\\ 
1_{\bi^{(1)}\cdots\bi^{(n)}}&\mapsto \ggi_\la\otimes 1_{\bi^{(1)}}\otimes \cdots\,\otimes 1_{\bi^{(n)}}.
\end{align*}
\end{Lemma}
\begin{proof}
This follows easily from (\ref{EGGIInd}), Lemma~\ref{LTensImagIsImag}(ii), and the fact that the concatenation $\bi^{(1)}\cdots\bi^{(n)}$ is a Gelfand-Graev word.
\end{proof}

\subsection{Gelfand-Graev Mackey}
Here we establish a version of the Mackey Theorem for Gelfand-Graev  induction and restriction. 
First, we need the following lemma. Recall the embedding $\rho_d$ from (\ref{EIotaSd}).

\begin{Lemma} \label{LGGOneSide}
Let $\la,\mu\in \Comp(d)$ and $x\in {}^\la\D_d^\mu$. Then in $\hat C_d$ we have $$1_{\la\de}\psi_{\rho_d(x)}\ggi_\mu=\ggi_{\la}\psi_{\rho_d(x)}\ggi_\mu.$$ 
\end{Lemma}
\begin{proof}
Since $\ggi_\mu$ is a sum of Gelfand-Graev idempotents of the form $\ggi^{\nu,\bj}$ for a composition $\nu$ being a refinement of a composition $\mu$, it suffices to show that for any refinement $\nu\in\Comp(n,d)$ of $\mu$ we have $1_{\la\de}\psi_{\rho_d(x)}\ggi^{\nu,\bj}=\ggi_{\la}\psi_{\rho_d(x)}\ggi^{\nu,\bj}.$ 
Since $\nu$ is a refinement of $\mu$, the assumption $x\in {}^\la\D_d^\mu$ implies $x\in {}^\la\D_d^\nu$. Now Lemma~\ref{LTrickyGGW} implies that for some $\si\in\Si_n$, we have $x=w_{\nu,\si}$ and  
$$1_{\la\de}\psi_{\rho_d(x)}\ggi^{\nu,\bj}=
1_{\la\de}\hat\ggi^{\si\nu,\si\bj}
\psi_{\rho_d(x)}\ggi^{\nu,\bj}.
$$
By Lemma~\ref{LGGSplit}, if $\si\nu$ is not a refinement of $\la$ then $1_{\la\de}\hat\ggi^{\si\nu,\si\bj}=0$ and so both $1_{\la\de}\psi_{\rho_d(x)}\ggi^{\nu,\bj}$ and $\ggi_{\la}\psi_{\rho_d(x)}\ggi^{\nu,\bj}=\ggi_{\la}1_{\la\de}\psi_{\rho_d(x)}\ggi^{\nu,\bj}$ are zero. 

So we may assume that $\si\nu$ is a refinement of $\la$. Then $1_{\la\de}\hat\ggi^{\si\nu,\si\bj}
=
\hat\ggi^{\si\nu,\si\bj}$, so 
$$1_{\la\de}\psi_{\rho_d(x)}\ggi^{\nu,\bj}=
\hat\ggi^{\si\nu,\si\bj}
\psi_{\rho_d(x)}\ggi^{\nu,\bj}.
$$
Now, repeatedly using Lemma~\ref{LDivIdCommutation}, we deduce that 
$$\hat\ggi^{\si\nu,\si\bj}
\psi_{\rho_d(x)}\ggi^{\nu,\bj}
=\ggi^{\si\nu,\si\bj}
\psi_{\rho_d(x)}\ggi^{\nu,\bj}
=\ggi_\la\ggi^{\si\nu,\si\bj}
\psi_{\rho_d(x)}\ggi^{\nu,\bj}=\ggi_\la\psi_{\rho_d(x)}\ggi^{\nu,\bj},
$$ 
as required.
\end{proof}

Recall the notation of \S\ref{SSSG}. In particular, given $\la,\mu\in\Comp(d)$ and $x\in {}^\la\D_d^\mu$ we have the compositions $\la\cap x\mu$ and $x^{-1}\la\cap\mu$ 
in $\Comp(d)$. Moreover, the corresponding parabolic algebras $C_{\la\cap x\mu,\de}$ and $C_{x^{-1}\la\cap\mu,\de}$ are naturally isomorphic  via an isomorphism 
$C_{\la\cap x\mu,\de}\iso C_{x^{-1}\la\cap\mu,\de}$ 
which permutes the components. Composing with this isomorphism we get a functor 
$$
\mod{C_{x^{-1}\la\cap\mu}}\to \mod{C_{\la\cap x\mu}},\ V\mapsto {}^xV.
$$
Recalling from (\ref{ETwist}) the similar functor $\mod{\hat C_{x^{-1}\la\cap\mu}}\to \mod{\hat C_{\la\cap x\mu}},\ W\mapsto {}^xW$, note that for $V\in\mod{C_{\mu}}$, we have 
\begin{equation}\label{ETwistRC}
\begin{split}
\funF_{\la\cap x\mu}{}^x\big(\Res^{\mu\de}_{(x^{-1}\la\cap\mu)\de} \funG_\mu V\big)
&\simeq 
{}^{x}(\funF_{x^{-1}\la\cap\mu}
\Res^{\mu\de}_{(x^{-1}\la\cap\mu)\de}\funG_\mu V)
\\
&\simeq
{}^{x}(\GGR^\mu_{x^{-1}\la\cap\mu} V),
\end{split}
\end{equation}
where we have used (\ref{EGGIIndGen}) for the second isomorphism. 

For $w\in \Si_d$, we set 
\begin{equation}\label{EBW}
b_w:=\ggi_d\psi_{\rho_d(w)}\ggi_d\in C_d.
\end{equation}

Let $V\in\mod{C_\mu}$. Note that 
for any $x\in {}^\la\D_d^\mu$, we have 
$\ggi_{\la}b_x \ggi_\mu\otimes V
\subseteq \GGR_{\la}^{d} \GGI_{\mu}^{d} V.$
Let $\leq$ be a total order refining the Bruhat order on $\Si_d$. 
For $x\in {}^\la\D_d^\mu$, we consider the submodules
\begin{align*}
\Phi_{\leq x}(V)&:=\sum_{y\in {}^\la\D_d^\mu\,\text{with}\, y\leq x}C_{\la}\ggi_{\la}b_y \ggi_{\mu}\otimes V 
\subseteq 
\GGR_{\la}^{d} \GGI_{\mu}^{d} V,
\\
\Phi_{< x}(V)&:=\sum_{y\in {}^\la\D_d^\mu\,\text{with}\, y< x}C_{\la}\ggi_{\la}b_y \ggi_{\mu}\otimes V 
\subseteq 
\GGR_{\la}^{d} \GGI_{\mu}^{d} V.
\end{align*}

\begin{Theorem} \label{TGGMackey} {\bf (Gelfand-Graev Mackey Theorem)} 
Let $\la,\mu\in\Comp(d)$, and $V\in\mod{C_{\mu}}$. 
Then we have the filtration 
 $
 (\Phi_{\leq x}(V))_{x\in {}^\la\D_d^\mu}
 $
 of  
$\GGR_{\la}^{d} \GGI_{\mu}^{d} V$.
Moreover, for the sub-quotients of the filtration we have 
$$ 
\Phi_{\leq x}(V)/\Phi_{< x}(V)\simeq\GGI_{\la\cap x\mu}^{\la}{}^{x}(\GGR^\mu_{x^{-1}\la\cap\mu} V).
$$
\end{Theorem}
\begin{proof}
We identify
\begin{align*}
\GGR_{\la}^{d} \GGI_{\mu}^{d} V
&=\ggi_\la C_d\ggi_\mu\otimes_{C_\mu} V
=
\ggi_\la \hat C_d1_{\mu\de}\ggi_\mu V
\\
&=
\ggi_\la 1_{\la\de}\hat C_d 1_{\mu\de} \otimes_{\hat C_{\mu}} \hat C_{\mu}\ggi_\mu\otimes_{C_\mu} V
\\
&=
\ggi_\la \Res^{d\de}_{\la\de}\Ind^{d\de}_{\mu\de} \, \funG_\mu V.
\end{align*}
By Theorem~\ref{TImMackey}, this has a filtration 
$
(\ggi_\la\bar F_{\leq x}(\funG_\mu V))_{x\in {}^\la\D_d^\mu}
$ 
where 
\begin{align*}
\ggi_\la\bar F_{\leq x}(\funG_\mu V)
&=
\sum_{y\in {}^\la\D_d^\mu\,\text{with}\, y\leq x}
\ggi_\la
\hat C_{\la}1_{\la\de}b_y1_{\mu\de}\otimes_{\hat C_{\mu}}
(\funG_\mu V)
\\
&=
\sum_{y\in {}^\la\D_d^\mu\,\text{with}\, y\leq x}
\ggi_\la
\hat C_{\la}1_{\la\de}b_y 1_{\mu\de}\otimes_{\hat C_{\mu}}
\hat C_{\mu}\ggi_\mu\otimes_{C_\mu} V
\\
&=
\sum_{y\in {}^\la\D_d^\mu\,\text{with}\, y\leq x}
\ggi_\la
\hat C_{\la}1_{\la\de}b_y 1_{\mu\de}\ggi_\mu\otimes_{C_\mu} V
\\
&=
\sum_{y\in {}^\la\D_d^\mu\,\text{with}\, y\leq x}
\ggi_\la
\hat C_{\la}1_{\la\de}b_y \ggi_\mu\otimes_{C_\mu} V
 \\
&=
\sum_{y\in {}^\la\D_d^\mu\,\text{with}\, y\leq x}
\ggi_\la
\hat C_{\la}\ggi_{\la}b_y \ggi_\mu\otimes_{C_\mu} V
\\
&=
\sum_{y\in {}^\la\D_d^\mu\,\text{with}\, y\leq x}
C_{\la}\ggi_{\la}b_y \ggi_\mu\otimes_{C_\mu} V
\\
&=\Phi_{\leq x}(V),
\end{align*}
where we have used Lemma~\ref{LGGOneSide} for the fifth equality. 
Now, using Theorem~\ref{TImMackey}, we also have 
\begin{align*}
\Phi_{\leq x}(V)/\Phi_{< x}(V)
&=
\ggi_\la\bar F_{\leq x}(\funG_\mu V)\,/\,\ggi_\la\bar F_{< x}(\funG_\mu V)
\\
&\simeq\ggi_\la\big(\bar F_{\leq x}(\funG_\mu V)/\bar F_{< x}(\funG_\mu V)\big)
\\
&\simeq
\ggi_\la\,\Ind_{(\la\cap  x\mu)\de}^{\la\de}\,{}^{x}(\Res^{\mu\de}_{(x^{-1}\la\cap\mu)\de}\, \funG_\mu V)
\end{align*}
We may assume that $\la$ has more than one non-zero part since otherwise there is nothing to prove. It follows that  the composition $\la\cap  x\mu$ also has more than one non-zero part. 
So 
\begin{align*}
&\ggi_\la\,\Ind_{(\la\cap  x\mu)\de}^{\la\de}\,{}^{x}(\Res^{\mu\de}_{(x^{-1}\la\cap\mu)\de} \funG_\mu V)
\\
\simeq\,&
\funF_\la\,\Ind_{(\la\cap  x\mu)\de}^{\la\de}\,\funG_{\la\cap  x\mu}\funF_{\la\cap  x\mu}{}^{x}(\Res^{\mu\de}_{(x^{-1}\la\cap\mu)\de} \funG_\mu V)
\\
\simeq\,& 
\GGI_{\la\cap  x\mu}^\la\,{}^{x}(\GGR^\mu_{x^{-1}\la\cap\mu} V),
\end{align*}
where we have used Lemma~\ref{LGLaFLa} for the first isomorphism and (\ref{ETwistRC}),\,(\ref{EGGIIndGen}) for the second isomorphism. 
\end{proof}

\begin{Corollary} \label{CLa1^d}
Let $\la\in\Comp(d)$. Then $f_\la C_df_{\om_d}=C_\la f_{\om_d}C_df_{\om_d}$.
\end{Corollary}
\begin{proof}
We have $C_\la f_{\om_d}C_df_{\om_d}=f_\la C_\la f_{\om_d}C_df_{\om_d}\subseteq f_\la C_df_{\om_d}$. On the other hand, by the Gelfand-Graev Mackey, we have 
$$
f_\la C_df_{\om_d}\simeq \GGR^d_\la\GGI_{\om_d}^d C_{\om_d}=
\sum_{x\in {}^\la\D_d}C_{\la}\ggi_{\la}b_x \ggi_{\om_d}\otimes C_{\om_d}
$$
So 
$$
f_\la C_df_{\om_d}
=\sum_{x\in {}^\la\D_d}C_{\la}\ggi_{\la}b_x\ggi_{\om_d}
=\sum_{x\in {}^\la\D_d}C_{\la}b_x \ggi_{\om_d}
$$
and it remains to note, using Lemma~\ref{LGGOneSide} with $\la=\mu=\om_d$, that $b_x \ggi_{\om_d}\in f_{\om_d}C_df_{\om_d}$ for any $x\in \Si_d$. 
\end{proof}

\subsection{  The isomorphism $\ggi_{\om_d}C_d\ggi_{\om_d}\cong H_d(A_\ell)$}
\label{SSIso}
Following \cite[\S4.2c]{KlLi}, we have the following elements of $C_1$:
\begin{align}
\label{EUFormula}
\dot u&:=\ggi^0y_p\ggi^0,
\\
\label{EZFormula}
\dot z&:=\textstyle\sum_{j\in J}\ggi^jy_{p-j-1}y_{p-j}\ggi^j, 
\\
\label{EAII-1Formula}
\dot a^{[j,j-1]}&:=\ggi^j\psi_{p-1}\cdots\psi_{p-2j}\ggi^{j-1}\qquad (j=1,\dots,\ell-1),
\\
\label{EAI-1IFormula}
\dot a^{[j-1,j]}&:=(-1)^j\ggi^{j-1}\psi_{p-2j-1}\cdots\psi_{p-1}\ggi^{j}\qquad (j=1,\dots,\ell-1).
\end{align}

We have the parabolic subalgebra 
$
C_1^{\otimes d}=C_{\om_d} \subseteq \ggi_{\om_d}C_d\ggi_{\om_d}.
$ 
Given an element $x\in C_1$ and $1\leq r\leq d$ and recalling the notation (\ref{EInsertion}), we have the element
$$
x_r:=\ggi_1^{\otimes (r-1)}\otimes x\otimes \ggi_1^{\otimes(d-r)}\in C_1^{\otimes d}\subseteq \ggi_{\om_d}C_d\ggi_{\om_d}.
$$
So we now also have the elements 
\begin{equation}\label{EElementsCd}
\dot u_r,\,\dot z_r,\,
\dot a^{[j,j-1]}_r,\, \dot a^{[j-1,j]}_r\in \ggi_{\om_d}C_d\ggi_{\om_d}\qquad(1\leq r\leq d).
\end{equation}

Further, following \cite[(4.3.28)]{KlLi}, we denote $\psi(r):=\psi_r\psi_{r+1}\cdots\psi_{r+p-1}$ for $r=1,\dots,p$, and 
set  
\begin{equation}\label{EUpSigma}
\upsigma:=\ggi_{\om_2}\psi(p)\psi(p-1)\cdots\psi(1)\ggi_{\om_2}\in \ggi_{\om_2}C_2\ggi_{\om_2}
\end{equation}
Now, following \cite[(4.4.11)]{KlLi}, we define $\dot s\in \ggi_{\om_2}C_2\ggi_{\om_2}$ via: 
\begin{equation}\label{ETau}
\dot s\ggi^{ij}=(\upsigma+(-1)^{i}\de_{i,j})\ggi^{ij}\qquad(i,j\in J).
\end{equation}
For $1\leq r<d$, we define 
\begin{eqnarray*}\label{EIotaTauR}
\dot s_{r}&:=&\ggi_1^{\otimes (r-1)}\otimes \dot s\otimes \ggi_1^{\otimes (d-r-1)}\in C_1^{\otimes (r-1)}\otimes \ggi_{\om_2}C_2\ggi_{\om_2}\otimes C_1^{\otimes (d-r-1)}\subseteq  \ggi_{\om_d}C_d\ggi_{\om_d},
\\
\upsigma_{r}&:=&\ggi_1^{\otimes (r-1)}\otimes \upsigma\otimes \ggi_1^{\otimes (d-r-1)}\in C_1^{\otimes (r-1)}\otimes \ggi_{\om_2}C_2\ggi_{\om_2}\otimes C_1^{\otimes (d-r-1)}\subseteq  \ggi_{\om_d}C_d\ggi_{\om_d}.
\end{eqnarray*}
Recalling (\ref{EIotaSd}) and (\ref{EBW}), we observe that
\begin{equation}\label{ESiRho}
\upsigma_r=\psi_{\rho_d(s_r)}\ggi_{\om_d}=b_{s_r}\ggi_{\om_d}.
\end{equation}

For $w\in \Si_d$ with an arbitrarily chosen reduced decomposition $w=s_{t_1}\cdots s_{t_l}$, we define
\begin{equation}\label{ETauSigma}
\dot w=\dot s_{t_1}\cdots\dot s_{t_l}\quad\text{and}\quad
\upsigma_w=\upsigma_{t_1}\cdots\upsigma_{t_l}. 
\end{equation}
We point out that in general $\upsigma_w$ depends on the choice of a reduced decomposition of $w$, while $\dot w$ does not, as  Theorem~\ref{TMorIso} below shows. We note that the element $\dot w$ was denoted $\uptau_w$ in \cite{KlLi}. 
In view of (\ref{ESiRho}) and (\ref{EBW}), we may assume that
\begin{equation}\label{ESiRhoW}
\upsigma_w=\psi_{\rho_d(w)}\ggi_{\om_d}=b_w\ggi_{\om_d}\qquad(w\in \Si_d).
\end{equation}

Now \cite[Lemma 4.4.29]{KlLi} gives:

\begin{Lemma} \label{LSiTau} 
Let $w\in \Si_d$. Then $\dot w=\upsigma_w+\sum_{u<w}c_u\upsigma_u$ for some coefficients $c_u\in\k$. 
\end{Lemma}

Recalling now the notation (\ref{EZEBI}), we have

\begin{Theorem} \label{TMorIso} There exists an isomorphism of graded superalgebras 
\begin{align*}
F_d:H_d(A_\ell)&\iso  \ggi_{\om_d}C_d\ggi_{\om_d},\\ 
e^\bj&\mapsto \ggi^\bj&(\bj\in J^d),\\ 
u_r&\mapsto \dot u_r,\ z_r\mapsto \dot z_r,\ 
a^{[j,j-1]}_r\mapsto \dot a^{[j,j-1]}_r,\ 
a^{[j-1,j]}_r\mapsto \dot a^{[j-1,j]}_r
&(1\leq r\leq d),\\ 
w&\mapsto\dot w&(w\in\Si_d).
\end{align*}
\end{Theorem}
\begin{proof}
In {\rm \cite[Theorems 4.5.8]{KlLi}} this is proved over any field. The result for $\k=\O$ follows.
\end{proof}

From now on, we identify the graded $\k$-superalgebras $ \ggi_{\om_d}C_d\ggi_{\om_d}$ and $H_d(A_\ell)$ via the isomorphism of Theorem~\ref{TMorIso}. 

\begin{Corollary} \label{CParCuspBasis}
As a left/right graded $C_{\om_d}$-supermodule, $\ggi_{\om_d}C_d\ggi_{\om_d}$ is free with basis $\{\dot w\mid w\in \Si_d\}$. 
\end{Corollary}
\begin{proof}
In view of Theorem~\ref{TAffBasis}, as a left/right 
graded $A_\ell[z]^{\otimes d}$-supermodule, $H_d(A_\ell)$ is free with basis given by the elements $w\in \Si_d$. Now the claim follows from Theorem~\ref{TMorIso}. 
\end{proof}

\subsection{  The graded supermodules $\LL_j$ and $\hat \LL_j$}
\label{SSL}
For $j\in J$, we recall from \S\ref{SSAffZig} the graded supermodule  $\LL_j=\k\cdot v_j$ over $H_1(A_\ell)$. But we have identified $H_1(A_\ell)=\ggi_1C_1\ggi_1=C_1$. So we will also consider $L_j$ as graded $C_1$-supermodule. Note that 
\begin{equation}\label{E100824_3}
\ggi^{i} v_j=\de_{i,j}\hat v_j\qquad(i,j\in J).
\end{equation}

Define the graded $\hat C_1$-supermodules 
\begin{equation}\label{EHatL}
\hat \LL_j:=\funG_1(\LL_j)=\hat C_1\ggi_1\otimes_{C_1} \LL_j\qquad(j\in J). 
\end{equation}
Also set
\begin{equation}\label{EHatv}
\hat v_j:=\ggi_1\otimes v_j\in \hat \LL_j.
\end{equation}
Note that 
\begin{equation}\label{E100824_2}
\ggi^{i} \hat v_j=\de_{i,j}\hat v_j\qquad(i,j\in J).
\end{equation}
Recalling (\ref{EDegV}), we have 
\begin{equation}\label{EBidegvjhatvj}
\bideg(v_{j})=\bideg(\hat v_{j})=(1+2j-2\ell,\,j\pmod{2}).
\end{equation}

Using the Morita equivalence of Corollary~\ref{CMord=1}, if $\k=\F$, then $\hat \LL_j$ is irreducible. Moreover, suppose the ground field $\F$ is an $\O$-module. 
Writing $\hat \LL_{j,\k}$ when we want to specify the ground ring $\k$, and applying (\ref{EFunExtendScal}), we get
\begin{equation}\label{ELJExtScal}
\F\otimes_\O\hat \LL_{j,\O}\simeq \hat \LL_{j,\F}.
\end{equation}

\begin{Remark} \label{RHatLFree} 
It is easy to see that as a $\k$-module, $\hat \LL_j$ is free of finite rank. 
The claim follows from (\ref{ELJExtScal}) if we can show that $\dim \hat \LL_{j,\F}$ does not depend on the characteristic of $\F$. This can be seen using for example the argument of \cite[Remark 5.5]{Kcusp}. 
\end{Remark}


\chapter{Regrading}
Let us identify the idempotent truncation $\ggi_{\om_d}C_d\ggi_{\om_d}$ with the affine zigzag algebra $H_d(A_\ell)$ as graded superalgebras via the isomorphism $F_d$ of Theorem~\ref{THCIsoZ}.
Recall from Proposition~\ref{PIsoH} that $H_d(A_\ell)$ has an important regrading $H_d(\Zig_\ell)$, so we can also regrade $\ggi_{\om_d}C_d\ggi_{\om_d}$ accordingly. In this chapter we will upgrade this regrading to a regrading $\zC_d$ of the whole algebra $C_d$. This regrading is important in particular because $\zC_d$ is non-negatively graded, see Theorem~\ref{TNonNeg} below. We will also describe a basis 
$\{\lgathz_{\la,\bi}\}$ the degree zero component $(\ggis^{\la,\bi}\zC_{\la}\ggis^{i_1^{\la_1}\cdots i_n^{\la_n}})^0$, see Lemma~\ref{LGBasis}. The special elements $\lgathz_{\la,\bi}\in\zC_d$ and their counterparts $\lgath_{\la,\bi}\in C_d$ introduced in \S\ref{SSUpsilon} below, will play an important role later on. 

\section{Regrading the truncated cuspidal algebra and modules over it.} 
\subsection{  Regrading $\zC_d$ of $C_d$.}
\label{SSRegradingCd}
Let $(\mu,\bj)\in\Comp^\col(m,d)$. 
Recalling (\ref{EIJ}), we set
\begin{equation}\label{EShiftsCuspidal}
t_{\mu,\bj}:=dp+\sum_{s=1}^m\mu_s^2(2j_s-4\ell)
\quad\text{and}\quad
\eps_{\mu,\bj}:=\sum_{s=1}^m\mu_sj_s\pmod{2}.
\end{equation}
As usual, in some special cases we use a simplified notation:
if $\mu=(d)$ and $\bj=j$, we write $t_{d,j}:=t_{(d),j}$ and $\eps_{d,j}:=\eps_{(d),j}$; 
if $j_s=j$ for all $s=1,\dots,m$, we write $t_{\mu,\bj}:=t_{\mu,j}$ and $\eps_{\mu,\bj}:=\eps_{\mu,j}$; 
if $\mu=\om_d$, we write $t_{\mu,\bj}:=t_{\bj}$ and $\eps_{\mu,\bj}:=\eps_{\bj}$ which agrees with (\ref{ENormBj}).

Recalling the general setting of (\ref{EShiftParameters}), the parameters (\ref{EShiftsCuspidal}) will be taken as grading supershift parameters which correspond to the orthogonal decomposition (\ref{EGad}) in $C_d$. This yields the new graded superalgebra 
\begin{equation}\label{EReGradingC}
\zC_d=\bigoplus_{(\la,\bi),(\mu,\bj)\in\EC^\col(d)}
\funQ^{t_{\la,\bi}-t_{\mu,\bj}}\Uppi^{\eps_{\la,\bi}-\eps_{\mu,\bj}} \ggi^{\mu,\bj}C_d\ggi^{\la,\bi}.
\end{equation}
We denote by $\zc\in\zC_d$ the element corresponding to an  element $c\in C_d$; for example, we have the idempotents $\ggis^{\la,\bi}$ and $\ggis_\nu$ in $\zC_d$ corresponding to the idempotents $\ggi^{\la,\bi}$ and $\ggi_\nu$ in $C_d$, respectively. 

For $j\in J$, recalling (\ref{EQInt}), we define
$$
m_{d,j}:=[d]^!_\ell[2d]^!_0\bigg(\prod_{i=j+1}^{\ell-1}[2d]^!_i\bigg)\bigg(
\prod_{i=1}^j([d]^!_i)^2\bigg)\in\Z^\pi[q,q^{-1}],
$$
and then for $(\la,\bi)\in\Comp^\col(n,d)$, we define 
$$
m_{\la,\bi}:=m_{\la_1,i_1}\cdots m_{\la_n,i_n}\in\Z^\pi[q,q^{-1}].
$$

\begin{Lemma} \label{LMBar}
For any $(\la,\bi)\in\Comp^\col(n,d)$, we have $\overline{m_{\la,\bi}}=\pi^d m_{\la,\bi}$. 
\end{Lemma}
\begin{proof}
Using (\ref{EBarGood}) and (\ref{E2n!Bar}), we get 
$$
\overline{m_{\la,\bi}}=\overline{m_{\la_1,i_1}}\cdots \overline{m_{\la_n,i_n}}
=(\pi^{\la_1}m_{\la_1,i_1})\cdots (\pi^{\la_n}m_{\la_n,i_n})=\pi^d m_{\la,\bi},
$$
as required.
\end{proof}

Recall the notation (\ref{EHatIPar}). 

\begin{Lemma} \label{LTComp} 
Let $(\la,\bi)\in\Comp^\col(n,d)$. 
Then 
$\ggw^{\la,\bi}!=m_{\la,\bi}$ and $\langle \ggw^{\la,\bi}\rangle=-t_{\la,\bi}$. 
\end{Lemma}
\begin{proof}
We may assume that $n=1$, i.e. $\ggw^{\la,\bi}=\ggw^{d,j}$. That $\ggw^{d,j}!=m_{d,j}$ follows immediately from the definitions. Moreover, taking into account that $(\al_\ell|\al_\ell)=8$, $(\al_0|\al_0)=2$, and $(\al_i|\al_i)=4$ for all other $i$, we get 
\begin{align*}
\langle \ggw^{d,j}\rangle
&=8\frac{d(d-1)}{4}+\sum_{i=j+1}^{\ell-1}4\frac{2d(2d-1)}{4}+2\sum_{i=1}^j4\frac{d(d-1)}{4}+2\frac{2d(2d-1)}{4}
\\
&=-dp-d^2(2j-4\ell),
\end{align*}
which is $-t_{d,j}$ as required. 
\end{proof}

Recall the notation (\ref{EFancyDirectSum}).

\begin{Lemma} \label{L140923}
Let $(\la,\bi)\in\Comp^\col(d)$, and $V$ (resp. $W$) be a left (resp. right) graded $\hat C_d$-supermodule. 
Then 
 $$
 \hat \ggi^{\la,\bi} V\simeq  (\ggi^{\la,\bi}  V)^{\oplus q^{-t_{\la,\bi}}m_{\la,\bi} } 
  \quad \text{and} \quad W \hat \ggi^{\la,\bi}\simeq  (W  \ggi^{\la,\bi})^{\oplus q^{t_{\la,\bi}}\overline{m_{\la,\bi}} }.
 $$
\end{Lemma}
\begin{proof}
Let 
 $\bk=\bl^{\la,\bi}$. Then by definition we have $\ggi^{\la,\bi}=1_{\bk}$ and $\hat\ggi^{\la,\bi}=1_{\hat\bk}$.  Now, by Lemmas~\ref{LFact} and~\ref{LTComp}, we have  
 $
 \hat \ggi^{\la,\bi} V=1_{\hat\bk} V\simeq  (1_\bk  V)^{\oplus q^{\langle \bk\rangle}\bk! }
= (\ggi^{\la,\bi}  V)^{\oplus q^{-t_{\la,\bi}}m_{\la,\bi} }.
 $
The second isomorphism is proved similarly. 
\end{proof}

\begin{Lemma} \label{LSymmetry}
Let $(\la,\bi),(\mu,\bj)\in\Comp^\col(d)$. We have isomorphisms of graded superspaces:
\begin{enumerate}
\item[{\rm (i)}] 
$
\hat\ggi^{\mu,\bj}\hat C_d\hat\ggi^{\la,\bi}\simeq(
\ggis^{\mu,\bj}\zC_d\ggis^{\la,\bi})^{\oplus \pi^{d+\eps_{\mu,\bj}-\eps_{\la,\bi}} m_{\la,\bi}m_{\mu,\bj}};
$ 
\item[{\rm (ii)}] $\ggis^{\mu,\bj}\zC_d\ggis^{\la,\bi}\simeq \ggis^{\la,\bi}\zC_d\ggis^{\mu,\bj}$.
\end{enumerate}
\end{Lemma}
\begin{proof}
(i) By Lemmas~\ref{L140923},\,\ref{LMBar} and (\ref{EFG}), we have
\begin{align*}
\hat\ggi^{\mu,\bj}\hat C_d\hat\ggi^{\la,\bi}
\simeq
(\ggi^{\mu,\bj} C_d\ggi^{\la,\bi})^{\oplus q^{t_{\la,\bi}-t_{\mu,\bj}}\overline{m_{\la,\bi}}m_{\mu,\bj}}
\simeq(
\ggis^{\mu,\bj} \zC_d\ggis^{\la,\bi})^{\oplus \pi^{d+\eps_{\mu,\bj}-\eps_{\la,\bi}}m_{\la,\bi}m_{\mu,\bj}}.
\end{align*}

(ii) Note applying $\tau$ from (\ref{ETauAntiAuto}) that $\hat\ggi^{\mu,\bj}\hat C_d\hat\ggi^{\la,\bi}\simeq \hat\ggi^{\la,\bi}\hat C_d\hat\ggi^{\mu,\bj}$ as superspaces. So (ii) 
follows form (i) and the uniqueness theorem for finitely generated modules over a PID applied to each graded component.  
\end{proof}

\subsection{Functors for regraded algebras}
\label{SSFunctorsRegr}
As in (\ref{EzVMor}), we have mutually quasi-inverse equivalence functors 
\begin{align}
\label{EFunH}
\funh_d&:\mod{C_d}\to\mod{\zC_d},\ V\mapsto \bigoplus_{(\la,\bi)\in\EC^\col(d)} 
\funQ^{-t_{\la,\bi}}\Uppi^{-\eps_{\la,\bi}}\ggi^{\la,\bi}V,
\\
\fung_d&:\mod{\zC_d}\to\mod{C_d},\ \zV\mapsto \bigoplus_{(\la,\bi)\in\EC^\col(d)} 
\funQ^{t_{\la,\bi}}\Uppi^{\eps_{\la,\bi}}
\ggis^{\la,\bi} \zV
\end{align}

Since $\sum_{(\la,\bi)\in\EC^\col(d)} 
\ggi^{\la,\bi}$ is the identity in $C_d$, we can identify $V$ and $\zV:=\funh_d V$ as $\k$-modules; in other words $\funh_d V$ is $V$ with a different grading. So for $v\in V$ we can speak of {\em the element of $\funh_d V$ corresponding to $v$}; we usually denote this element $\zv$. Then, recalling the notation of \S\ref{SSRegradingCd}, we have for any $c\in C_d$ and $v\in V$:
\begin{equation}\label{ECVCorresponds}
\text{$\zc\zv$ is the element of $\funh_d V$ corresponding to $cv$.}
\end{equation}

Let $\nu\in\Comp(n,d)$ and recall from 
(\ref{ECPar}),\,(\ref{CParIdentify})
the parabolic subalgebra 
$$
C_{\nu_1}\otimes \cdots\otimes C_{\nu_n}= C_{\nu}=\ggi_{\nu}\hat C_{\nu}\ggi_{\nu}\subseteq\ggi_{\nu}C_d\ggi_\nu.
$$ 
We denote 
$$
\EC^\col(\nu):=\{(\la,\bi)\in\EC^\col(d)\mid \text{$\la$ is a refinement of $\nu$}\}. 
$$
Then 
$
\ggi_{\nu}=\sum_{(\la,\bi)\in\EC^\col(\nu)}\ggi^{\la,\bi},
$
and we have a parabolic subalgebra 
\begin{equation}\label{ERegrParabolic}
\zC_\nu:=\bigoplus_{(\la,\bi),(\mu,\bj)\in\EC^\col(\nu)}
\funQ^{t_{\la,\bi}-t_{\mu,\bj}}\Uppi^{\eps_{\la,\bi}-\eps_{\mu,\bj}} \ggi^{\mu,\bj}C_\nu\ggi^{\la,\bi}\subseteq \ggis_{\nu}\zC_d\ggis_\nu
\end{equation}
which can again be identified with $\zC_{\nu_1}\otimes \cdots\otimes \zC_{\nu_n}$ via the same isomorphism as for $C_\nu$ (one needs to make the obvious observation that the grading shifts match):
\begin{equation}\label{EzCParIdentify}
\zC_\nu=\zC_{\nu_1}\otimes \cdots\otimes \zC_{\nu_n}.
\end{equation}
 Again, we have mutually inverse equivalence functors 
\begin{align*}
\funh_\nu&:\mod{C_\nu}\to\mod{\zC_\nu},\ V\mapsto \bigoplus_{(\la,\bi)\in\EC^\col(\nu)} 
\funQ^{-t_{\la,\bi}}\Uppi^{-\eps_{\la,\bi}}\ggi^{\la,\bi}V,
\\
\fung_\nu&:\mod{\zC_\nu}\to\mod{C_\nu},\ \zV\mapsto \bigoplus_{(\la,\bi)\in\EC^\col(\nu)} 
\funQ^{t_{\la,\bi}}\Uppi^{\eps_{\la,\bi}}
\ggis^{\la,\bi} \zV
\end{align*}

\subsection{Regraded Gelfand-Graev induction and restriction}
We now also have the {\em shifted Gelfand-Graev induction and restriction} functors
\begin{equation}\label{EGGIRS}
\begin{split}
\GGIS_\nu^d:=\zC_d\ggis_\nu\otimes_{\zC_\nu}- :\mod{\zC_\nu}\to\mod{\zC_d},
\\
\GGRS_\nu^d:=\ggis_\nu \zC_d\otimes_{\zC_d}- :\mod{\zC_d}\to\mod{\zC_\nu}.
\end{split}
\end{equation}

\begin{Lemma} \label{LGGISEq}
Let $\nu\in\Comp(d)$. For\, $\zV\in \mod{\zC_\nu}$ and\, $\zW\in\mod{\zC_d}$, we have functorial isomorphisms
\begin{align*}
\GGIS_\nu^d\,\zV\simeq \funh_d\,\GGI_\nu^d\,\fung_\nu\,\zV
\quad\text{and}\quad
\GGRS_\nu^d\,\zV\simeq \funh_\nu\,\GGI_\nu^d\,\fung_d\,\zW.
\end{align*}
\end{Lemma}
\begin{proof}
For the first isomorphism we have 
\begin{align*}
&\,\funh_d\,\GGI_\nu^d\,\fung_\nu\,\zV
\\=\,&
\bigoplus_{(\mu,\bj)\in\EC^\col(d)} 
\funQ^{-t_{\mu,\bj}}\Uppi^{-\eps_{\mu,\bj}}\ggi^{\mu,\bj} C_d\ggi_\nu\otimes_{C_\nu}\Big(\bigoplus_{(\la,\bi)\in\EC^\col(\nu)} 
\funQ^{t_{\la,\bi}}\Uppi^{\eps_{\la,\bi}}
\ggis^{\la,\bi} \zV\Big)
\\
=\,&
\bigoplus_{(\mu,\bj)\in\EC^\col(d)} \bigoplus_{(\la,\bi)\in\EC^\col(\nu)} 
\funQ^{t_{\la,\bi}-t_{\mu,\bj}}\Uppi^{\eps_{\la,\bi}-\eps_{\mu,\bj}}\ggi^{\mu,\bj} C_d\ggi_\nu\ggi^{\la,\bi}\otimes_{\zC_\nu}\,\zV
\\
=\,&
\bigoplus_{(\mu,\bj)\in\EC^\col(d)} \bigoplus_{(\la,\bi)\in\EC^\col(d)} 
\funQ^{t_{\la,\bi}-t_{\mu,\bj}}\Uppi^{\eps_{\la,\bi}-\eps_{\mu,\bj}}\ggi^{\mu,\bj} C_d\ggi^{\la,\bi}\ggi_\nu\otimes_{\zC_\nu}\,\zV
\\
=\,&
\zC_d\ggis_\nu\otimes_{\zC_\nu}\,\zV
\\
=
\,&
\GGIS_\nu^d\,\zV.
\end{align*}
The second isomorphism is proved similarly.
\end{proof}

\begin{Corollary} \label{CGGsIngId}
Let $\la\in\Comp(n,d)$ and 
$\bi^{(1)}\in I^{\la_1\de}_\di,\dots,\bi^{(n)}\in I^{\la_n\de}_\di$ be Gelfand-Graev words. Then 
$
\zC_d1_{\bi^{(1)}\cdots\bi^{(n)}}\simeq\GGIS_{\la}^d\big( \zC_{\la_1}1_{\bi^{(1)}}\boxtimes\,\cdots\,\boxtimes\, \zC_{\la_n}1_{\bi^{(n)}}\big)
$
\end{Corollary}
\begin{proof}
This follows from Lemma~\ref{LGGIndProjId}
using Lemma~\ref{LGGISEq}.
\end{proof}

Let $\zV\in\mod{\zC_\mu}$. Recall (\ref{EBW}). Note that 
for any $x\in {}^\la\D_d^\mu$, we have 
$\ggis_{\la}\zb_x \ggis_\mu\otimes \zV
\subseteq \GGRS_{\la}^{d} \GGIS_{\mu}^{d} \zV.$
Let $\leq$ be a total order refining the Bruhat order on $\Si_d$. 
For $x\in {}^\la\D_d^\mu$, we consider the submodules
\begin{align*}
\Phi_{\leq x}(\zV)&:=\sum_{y\in {}^\la\D_d^\mu\,\text{with}\, y\leq x}\zC_{\la}\ggis_{\la}\zb_y \ggis_{\mu}\otimes \zV 
\subseteq 
\GGRS_{\la}^{d} \GGIS_{\mu}^{d} \zV,
\\
\Phi_{< x}(\zV)&:=\sum_{y\in {}^\la\D_d^\mu\,\text{with}\, y< x}\zC_{\la}\ggis_{\la}\zb_y \ggis_{\mu}\otimes \zV 
\subseteq 
\GGRS_{\la}^{d} \GGIS_{\mu}^{d} V.
\end{align*}

\begin{Theorem} \label{CGGMackeyS} 
 {\bf (Shifted Gelfand-Graev Mackey Theorem)} 
Let $\la,\mu\in\Comp(d)$, and $\zV\in\mod{\zC_{\mu}}$. 
Then we have a filtration 
 $
 (\Phi_{\leq x}(\zV))_{x\in {}^\la\D_d^\mu}
 $
 of  
$\GGRS_{\la}^{d}\, \GGIS_{\mu}^{d}\, \zV$ 
with sub-quotients 
$$ 
\Phi_{\leq x}(\zV)/\Phi_{< x}(\zV)\simeq\GGIS_{\la\cap x\mu}^{\la}{}^{x}(\GGRS^\mu_{x^{-1}\la\cap\mu} \zV).
$$
\end{Theorem}
\begin{proof}
By Lemma~\ref{LGGISEq}, we have 
$\GGRS_{\la}^{d}\, \GGIS_{\mu}^{d}\, \zV\simeq \funh_\la\,\GGR_{\la}^{d}\, \GGI_{\mu}^{d}\, \fung_\mu\zV$. The result follows by applying Theorem~\ref{TGGMackey} to $\GGR_{\la}^{d}\, \GGI_{\mu}^{d} (\, \fung_\mu\zV)$ and Lemma~\ref{LGGISEq} again. 
\end{proof}

\subsection{  The idempotent truncation $\ggis_{\om_d}\zC_d\ggis_{\om_d}$}

We recall from (\ref{EElementsCd}) and (\ref{ETauSigma}) the explicit elements $\dot u_r,\dot z_r,
\dot a^{[j,j-1]}_r, \dot a^{[j-1,j]}_r, \dot w\in \ggi_{\om_d}C_d\ggi_{\om_d}$
According to our general convention the corresponding elements of the regraded algebra $\ggis_{\om_d}\zC_d\ggis_{\om_d}$ are denoted $
\dot \zu_r,\dot \zz_r,
\dot \za^{[j,j-1]}_r,
\dot \za^{[j-1,j]}_r,
\dot \zw$. 
By Theorem~\ref{TMorIso}, there is an explicit isomorphism $F_d:H_d(A_\ell)\iso \ggi_{\om_d} C_d\ggi_{\om_d}$. Combining with  Proposition~\ref{PIsoH}, we obtain:

\begin{Theorem} \label{THCIsoZ}
There is an isomorphism of graded superalgebras
\begin{align*}
\zF_d: H_d(\Zig_\ell)&\iso  \ggis_{\om_d}\zC_d\ggis_{\om_d},\\ 
\ze^\bj&\mapsto \ggis^\bj,\\ 
\zu_r&\mapsto \dot \zu_r\sum_{\bj\in J^d}(-1)^{d+j_1+\dots+j_d}\ggis^\bj,\\ 
 \zz_r&\mapsto \dot \zz_r\sum_{\bj\in J^d}(-1)^{j_r+1}\ggis^\bj,\\ 
\za^{[i,j]}_r&\mapsto \dot \za^{[i,j]}_r\sum_{\bj\in J^d}(-1)^{j_1+\dots+j_{r-1}+r-1}\ggis^\bj,\\ 
s_r&\mapsto \dot \zs_r\sum_{\bj\in J^d}(-1)^{(j_r+1)(j_{r+1}+1)}\ggis^\bj.
\end{align*}
\end{Theorem}

From Corollary~\ref{CParCuspBasis}, we get: 

\begin{Corollary} \label{CParCuspBasisZ}
As a left/right $\zC_{\om_d}$-module, $\ggis_{\om_d}\zC_d\ggis_{\om_d}$ is free with basis $\{\dot \zw\mid w\in \Si_d\}$. 
\end{Corollary}

Note that under the isomorphism of Theorem~\ref{THCIsoZ}, 
the subalgebra $\Zig_\ell[\zz]^{\otimes d}\subseteq H_d(\Zig_\ell)$ 
corresponds to the parabolic subalgebra
$
\zC_{\om_d}\subseteq \ggis_{\om_d}\zC_d\ggis_{\om_d}.
$
In particular, $\zC_{\om_d}$ is generated by the  
$\ggis^\bj,\dot \zu_r,\dot \zz_r,\dot \za^{[i,j]}_r$. 


\begin{Lemma} \label{L7.1} 
The algebra $\ggis_{\om_d}\zC_d\ggis_{\om_d}$ is non-negatively graded with the degree zero component $(\ggis_{\om_d}\zC_d\ggis_{\om_d})^0=\ggis_{\om_d}\zC_d^0\ggis_{\om_d}$ generated by all\, $\ggis^\bj$ and\, $\dot \zw$. Moreover, 
$$
(\ggis_{\om_d}\zC_d\ggis_{\om_d})^{>0}=\ggis_{\om_d}\zC_d^{>0}\ggis_{\om_d}=\ggis_{\om_d}\zC_d\zC_{\om_d}^{>0}=\ggis_{\om_d}\zC_d\ggis_{\om_d}\zC_{\om_d}^{>0}.
$$
\end{Lemma}
\begin{proof}
The equalities $(\ggis_{\om_d}\zC_d\ggis_{\om_d})^0=\ggis_{\om_d}\zC_d^0\ggis_{\om_d}$ and $(\ggis_{\om_d}\zC_d\ggis_{\om_d})^{>0}=\ggis_{\om_d}\zC_d^{>0}\ggis_{\om_d}$ are obvious, and 
the first claim follows from Theorem~\ref{THCIsoZ} and Corollary~\ref{CHZig>0}(i). The equality $\ggis_{\om_d}\zC_d\zC_{\om_d}^{>0}=\ggis_{\om_d}\zC_d\ggis_{\om_d}\zC_{\om_d}^{>0}$ follows from $\zC_{\om_d}^{>0}=\ggis_{\om_d}\zC_{\om_d}^{>0}$. 

Recalling the isomorphism $\zF_d$ of Theorem~\ref{THCIsoZ}, we have $(\ggis_{\om_d}\zC_d\ggis_{\om_d})^{>0}=\zF_d(H_d(\Zig_\ell)^{>0})$. By Corollary~\ref{CHZig>0}(ii), we have $H_d(\Zig_\ell)^{>0}=H_d(\Zig_\ell)(\Zig_\ell[z]^{\otimes d})^{>0}$, and $\zF_d(H_d(\Zig_\ell)(\Zig_\ell[z]^{\otimes d})^{>0})=\ggis_{\om_d}\zC_d\ggis_{\om_d}\zC_{\om_d}^{>0}$. 
So $(\ggis_{\om_d}\zC_d\ggis_{\om_d})^{>0}=\ggis_{\om_d}\zC_d\ggis_{\om_d}\zC_{\om_d}^{>0}$.
\end{proof}

\begin{Lemma} \label{LPositiveTau} 
In $\zC_d$, for any $w\in \Si_d$, we have $\zC_{\om_d}^{>0}\dot\zw=\dot\zw\,\zC_{\om_d}^{>0}$. 
\end{Lemma}
\begin{proof}
We have $\zC_{\om_d}=\zF_d(\Zig_\ell[\zz]^{\otimes d})$.  
It is easy to see using the relations in $H_d(\Zig_\ell)$ and the fact that $s_r$ has degree $0$ that $(\Zig_\ell[\zz]^{\otimes d})^{>0}s_r=s_r(\Zig_\ell[\zz]^{\otimes d})^{>0}$ for all $1\leq r<d$. So  
$(\Zig_\ell[\zz]^{\otimes d})^{>0}w=w(\Zig_\ell[\zz]^{\otimes d})^{>0}$ for all $w\in\Si_d$. 
The result follows on application of $\zF_d$. 
\end{proof}

\begin{Lemma} \label{L7.7}
For any $\bj\in J^d$ we have $\ggis^\bj\zC_{\om_d}^0\ggis^\bj=\k \cdot \ggis^\bj$. 
\end{Lemma}
\begin{proof}
We have $\ggis^\bj\zC_{\om_d}^0\ggis^\bj=\zF_d(\ze^\bj(\Zig_\ell[\zz]^{\otimes d})^0\ze^\bj)=\zF_d(\k\cdot\ze^\bj)=\k \cdot \ggis^\bj$. 
\end{proof}

\section{Non-negativity of grading and the elements $\lgath_{\la,\bi}$}

\subsection{Special segments}
\label{SSNot}
Let $j\in J$ and $d\in\N_+$. Define the composition
\begin{equation}\label{ENu(i)}
\nu^{d,j}:=(d,(2d)^{\ell-1-j},d^j,2d,d^j)\in\Comp(dp).
\end{equation}
Corresponding to the composition $\nu^{d,j}$ we have the decomposition of the segment $[1,dp]$ as an increasing disjoint union of the segments (this is a special notation replacing the segments  $\ttP^{\nu^{d,j}}_r$ from (\ref{EPartitionOfComposition})):
\begin{equation}\label{ESegDec}
[1,dp]=\ttX_\ell^{d,j}\sqcup \ttX_{\ell-1}^{d,j}\sqcup\dots \sqcup \ttX_{j+1}^{d,j}\sqcup \ttY^{d,j}_j\sqcup\dots\sqcup \ttY^{d,j}_1\sqcup
\ttX^{d,j}_0\sqcup \ttZ^{d,j}_1\sqcup\dots \sqcup\ttZ^{d,j}_j,
\end{equation}
such that 
\begin{align*}
\ttX^{d,j}_i&=\{r\in [1,dp]\mid \hat\ggw^{d,j}_r=i\}\qquad(\text{for}\ j<i\leq \ell\ \text{or}\ i=0),
\\
\ttY^{d,j}_i\sqcup \ttZ^{d,j}_i&=\{r\in [1,dp]\mid \hat\ggw^{d,j}_r=i\}\qquad(\text{for}\ 1\leq i\leq j). 
\end{align*}
For $1\leq i\leq j$, define 
$
\ttX^{d,j}_i:=\ttY^{d,j}_i\sqcup \ttZ^{d,j}_i
$
 so that now {\em for all} $i\in I$ we have 
 \begin{equation}\label{ESegDec1}
 \ttX^{d,j}_i=\{r\in [1,dp]\mid \hat\ggw^{d,j}_r=i\}.
\end{equation}

More generally, let $(\la,\bi)\in\Comp^\col(n,d)$. Recalling (\ref{EPartitionOfComposition}), (\ref{EkLa}), 
we have the decomposition of the segment $[1,dp]$ as an increasing disjoint union of the segments 
$$
[1,dp]=\ttP^{p\la}_1\sqcup\dots\sqcup \ttP^{p\la}_n.
$$
For each $1\leq r\leq n$, we have the decomposition of the segment $\ttP^{p\la}_r$ as an increasing disjoint union of the segments:
\begin{equation}\label{ESegDecU}
\ttP_r^{p\la}=\ttX_{r,\ell}^{\la,\bi}\sqcup \ttX_{r,\ell-1}^{\la,\bi}\sqcup\dots \sqcup \ttX_{r,i_r+1}^{\la,\bi}\sqcup \ttY^{\la,\bi}_{r,i_r}\sqcup\dots\sqcup \ttY^{\la,\bi}_{1,i_r}\sqcup
\ttX^{\la,\bi}_{r,0}\sqcup \ttZ^{\la,\bi}_{r,1}\sqcup\dots \sqcup\ttZ^{\la,\bi}_{r,i_r},
\end{equation}
where 
\begin{align*}
\ttX^{\la,\bi}_{r,i}&=\{r\in \ttP^{p\la}_r\mid \hat\ggw^{\la,\bi}_r=i\}\qquad(\text{for}\ i_r<i\leq \ell\ \text{or}\ i=0),
\\
\ttY^{\la,\bi}_{r,i}\sqcup \ttZ^{\la,\bi}_{r,i}&=\{r\in \ttP^{p\la}_r\mid \hat\ggw^{\la,\bi}_r=i\}\qquad(\text{for}\ 1\leq i\leq i_r). 
\end{align*}
For $1\leq i\leq i_r$ we define 
$
\ttX^{\la,\bi}_{r,i}:=\ttY^{\la,\bi}_{r,i}\sqcup \ttZ^{\la,\bi}_{r,i}
$
 so that now 
$
 \ttX^{\la,\bi}_{r,i}=\{r\in \ttP_r^{p\la}\mid \hat\ggw^{\la,\bi}_r=i\} 
$
for all $i\in I$.

\begin{Example} 
{\rm 
(i) Let $\ell=3$, $j=1$, and $d=2$. Then 
$$\hat\ggw^{d,j}=3\,3\,2\,2\,2\,2\,1\,1\,0\,0\,0\,0\,1\,1$$ and the decomposition (\ref{ESegDec}) is illustrated by the following picture:
$$
\overbrace{3\,3}^{\ttX^{d,j}_3}\,\overbrace{2\,2\,2\,2}^{\ttX^{d,j}_2}\,\overbrace{1\,1}^{\ttY^{d,j}_1}\,\overbrace{0\,0\,0\,0}^{\ttX^{d,j}_0}\,\overbrace{1\,1}^{\ttZ^{d,j}_1}.
$$
In other words, the decomposition (\ref{ESegDec}) is $[1,14]=[1,2]\sqcup[3,6]\sqcup[7,8]\sqcup[9,12]\sqcup[13,14]$. 
}

(ii) Let $\ell=3$, $d=3$, $\la=(2,1)$ and $\bi=11$. Then 
$$\hat\ggw^{\la,\bi}=3\,\,3\,\,2\,\,2\,\,2\,\,2\,\,1\,\,1\,\,0\,\,0\,\,0\,\,0\,\,1\,\,1\,\,
3\,\,2\,\,2\,\,1\,\,0\,\,0\,\,1
$$ and the decomposition (\ref{ESegDecU}) is illustrated by the following picture:
$$
\underbrace{\overbrace{3\,\,3}^{\ttX^{\la,\bi}_{1,3}}\,\,\overbrace{2\,\,2\,\,2\,\,2}^{\ttX^{\la,\bi}_{1,2}}\,\,\overbrace{1\,\,1}^{\ttY^{\la,\bi}_{1,1}}\,\,\overbrace{0\,\,0\,\,0\,\,0}^{\ttX^{\la,\bi}_{1,0}}\,\,\overbrace{1\,\,1}^{\ttZ^{\la,\bi}_{1,1}}}_{\ttP^{p\la}_1}\,\,\underbrace{
\overbrace{3}^{\ttX^{\la,\bi}_{2,3}}\,\,\overbrace{2\,\,2}^{\ttX^{\la,\bi}_{2,2}}\,\,\overbrace{1}^{\ttY^{\la,\bi}_{2,1}}\,\,\overbrace{0\,\,0}^{\ttX^{\la,\bi}_{2,0}}\,\,\overbrace{1}^{\ttZ^{\la,\bi}_{2,1}}}_{\ttP^{p\la}_2}.
$$
\end{Example}

\subsection{   The elements $\lgath_{d,j}$ and $\gath_{d,j}$}
\label{SSUpsilon}

Let $d\in\N_+$ and $j\in J$. 
Note that the set 
$$
\{x\in \D^{p^d}_{dp}\mid x\cdot(\hat\ggw^j)^d=\hat\ggw^{d,j}\}
$$
has a unique longest element which we denote $\ga_{d,j}$ and the unique shortest element which we denote $\chi_{d,j}$. 
We denote 
\begin{equation}\label{EUpsilon}
\lgath_{d,j}:=\psi_{\ga_{d,j}}\ggi^{j^d}\in \hat C_d
\quad\text{and}\quad
\gath_{d,j}:=\psi_{\chi_{d,j}}\ggi^{j^d}\in \hat C_d.
\end{equation}
For $(\la,\bi)\in\Comp^\col(n,d)$, we denote
\begin{equation}\label{EUpsilonLa}
\lgath_{\la,\bi}:=\lgath_{\la_1,i_1}\otimes\dots\otimes \lgath_{\la_n,i_n}\in \hat C_{\la_1}\otimes\dots\otimes \hat C_{\la_n}=\hat C_\la. 
\end{equation}
An easy computation shows that 
\begin{equation}\label{EUDeg}
\begin{split}
\bideg(\lgath_{d,j})&=(d(d-1)(2j-4\ell),\0),\\ \bideg(\lgath_{\la,\bi})&=\big(\sum_{r=1}^n\la_r(\la_r-1)(2i_r-4\ell)\,,\ \0\big). 
\end{split}
\end{equation}

\begin{Example} 
{\rm 
Suppose $\ell=3$, $j=1$ and $d=2$. Then $\lgath_{d,j}$ is represented by the following Khovanov-Lauda diagram:
$$
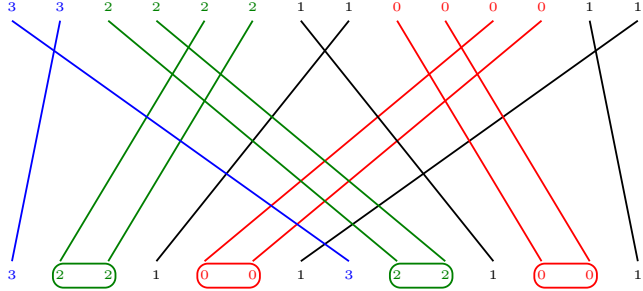
\begin{braid}\tikzset{baseline=-.3em}
\darkgreenbraidbox{1.9}{4.1}{-6.1}{-5.1}{};
\redbraidbox{7.9}{10.1}{-6.1}{-5.1}{};
\redbraidbox{21.9}{24.1}{-6.1}{-5.1}{};
\darkgreenbraidbox{15.9}{18.1}{-6.1}{-5.1}{};
	
\draw[blue](0,-5)node[below]{$3$}--(2,5)node[above]{$3$};
\draw[darkgreen](2,-5)node[below]{$2$}--(8,5)node[above]{$2$};
\draw[darkgreen](4,-5)node[below]{$2$}--(10,5)node[above]{$2$};
\draw(6,-5)node[below]{$1$}--(14,5)node[above]{$1$};
\draw[red](8,-5)node[below]{$0$}--(20,5)node[above]{$0$};
\draw[red](10,-5)node[below]{$0$}--(22,5)node[above]{$0$};
\draw(12,-5)node[below]{$1$}--(26,5)node[above]{$1$};
\draw[blue](14,-5)node[below]{$3$}--(0,5)node[above]{$3$};
\draw[darkgreen](16,-5)node[below]{$2$}--(4,5)node[above]{$2$};
\draw[darkgreen](18,-5)node[below]{$2$}--(6,5)node[above]{$2$};
\draw(20,-5)node[below]{$1$}--(12,5)node[above]{$1$};
\draw[red](22,-5)node[below]{$0$}--(16,5)node[above]{$0$};
\draw[red](24,-5)node[below]{$0$}--(18,5)node[above]{$0$};
\draw(26,-5)node[below]{$1$}--(24,5)node[above]{$1$};
\end{braid}
$$
and $\gath_{d,j}$ is represented by the following Khovanov-Lauda diagram:
$$
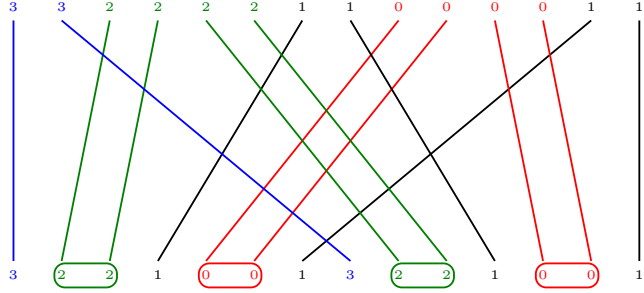
\begin{braid}\tikzset{baseline=-.3em}
\darkgreenbraidbox{1.9}{4.1}{-6.1}{-5.1}{};
\redbraidbox{7.9}{10.1}{-6.1}{-5.1}{};
\redbraidbox{21.9}{24.1}{-6.1}{-5.1}{};
\darkgreenbraidbox{15.9}{18.1}{-6.1}{-5.1}{};
	
\draw[blue](0,-5)node[below]{$3$}--(0,5)node[above]{$3$};
\draw[darkgreen](2,-5)node[below]{$2$}--(4,5)node[above]{$2$};
\draw[darkgreen](4,-5)node[below]{$2$}--(6,5)node[above]{$2$};
\draw(6,-5)node[below]{$1$}--(12,5)node[above]{$1$};
\draw[red](8,-5)node[below]{$0$}--(16,5)node[above]{$0$};
\draw[red](10,-5)node[below]{$0$}--(18,5)node[above]{$0$};
\draw(12,-5)node[below]{$1$}--(24,5)node[above]{$1$};
\draw[blue](14,-5)node[below]{$3$}--(2,5)node[above]{$3$};
\draw[darkgreen](16,-5)node[below]{$2$}--(8,5)node[above]{$2$};
\draw[darkgreen](18,-5)node[below]{$2$}--(10,5)node[above]{$2$};
\draw(20,-5)node[below]{$1$}--(14,5)node[above]{$1$};
\draw[red](22,-5)node[below]{$0$}--(20,5)node[above]{$0$};
\draw[red](24,-5)node[below]{$0$}--(22,5)node[above]{$0$};
\draw(26,-5)node[below]{$1$}--(26,5)node[above]{$1$};
\end{braid}.
$$
}
\end{Example}

It is easy to see that 
\begin{equation}\label{EKappa}
\ga_{d,j}=\kappa_{d,j}\chi_{d,j}
\end{equation} 
for some $\kappa_{d,j}\in\Si_{\nu^{d,j}}$ such that 
 $\ttl(\ga_{d,j})=\ttl(\kappa_{d,j})+\ttl(\chi_{d,j})$.
Letting $\udl_{d,j}:=\psi_{\kappa_{d,j}}$, we may assume that 
\begin{equation}\label{EKappaH}
\lgath_{d,j}=\udl_{d,j}\gath_{d,j}.
\end{equation}

\begin{Lemma} \label{LG1}
Let $(\la,\bi)\in \Comp^\col(n,d)$. Then $\lgath_{\la,\bi}=
\ggi^{\la,\bi}\lgath_{d,j}\ggi^{i_1^{\la_1}\cdots\,i_n^{\la_n}}\in C_\la$. 
In particular, $\lgath_{d,j}=
\ggi^{d,j}\lgath_{d,j}\ggi^{j^d}\in C_d$. 
\end{Lemma}
\begin{proof}
It suffices to establish the special case $\lgath_{d,j}=
\ggi^{d,j}\lgath_{d,j}\ggi^{j^d}\in C_d$. 
As $\ggi^{j^d}$ is an idempotent, we have $\lgath_{d,j}=\lgath_{d,j}\ggi^{j^d}$ by definition. To prove that  $\lgath_{d,j}=\ggi^{d,j}\lgath_{d,j}$, recall the notation (\ref{ENu(i)}). Observe, using the braid relations (\ref{R7}) with no error term, that we can write $\lgath_{d,j}=\psi_{w_0}\psi_x\ggi^{j^d}$, where $w_{0}$ is the longest element in the parabolic subgroup $\Si_{\nu^{d,j}}$. Now we can apply (\ref{EStrIdemp}). 
\end{proof}

\begin{Lemma} \label{LUpsilonPart} 
Let $(\la,\bi)\in \Comp^\col(n,d)$. Then $\lgath_{\la,\bi}\neq 0$. In particular, in $\hat C_d$, we have that $\lgath_{d,j}\neq 0$ and hence $\gath_{d,j}\neq 0$. 
\end{Lemma}
\begin{proof}
It suffices to prove that $\lgath_{d,j}\neq 0$. 
Recalling the $\hat C_1$-module $\hat \LL_j$ from (\ref{EHatL}), 
the module $\Ind_{\de^d}^{d\de}\hat \LL_j^{\boxtimes d}\neq 0$ is cuspidal by Lemma~\ref{LTensImagIsImag}(i), and it suffices to prove that $\lgath_{d,j} \Ind_{\de^d}^{d\de}\hat \LL_j^{\boxtimes d}\neq 0$. In fact we show that 
$$
\lgath_{d,j} \otimes \hat \LL_j^{\boxtimes d}=
\lgath_{d,j} 1_{\de^d}\otimes \hat \LL_j^{\boxtimes d}\subseteq \Ind_{\de^d}^{d\de}\hat \LL_j^{\boxtimes d}
$$ 
is non-zero. 
By definition, we have $\ggi^{j}\LL_j\neq 0$, so $\ggi^{j^d}\hat \LL_j^{\boxtimes d}\neq 0$. Now, since $\ga_{d,j}\in\D_{pd}^{p^d}$, we have by Lemma~\ref{LIndBasis} that $\lgath_{d,j} \otimes (\hat \LL_j)^{\boxtimes d}=\psi_{\ga_{d,j}}1_{\de^d}\otimes \ggi^{j^d}(\hat \LL_j)^{\boxtimes d}\neq 0$. 
\end{proof}

\subsection{  The idempotent truncation $\ggi^{d,j} C_d\ggi_{\om_d}$} In this subsection we establish the equality  $\ggi^{d,j} C_d\ggi_{\om_d}= \lgath_{d,j} C_d\ggi_{\om_d}$. 

\begin{Lemma} \label{LConsecutive} 
Let $\bi^{(1)},\dots,\bi^{(d)}\in I^\de_\cus$, $\bi=\bi^{(1)}\cdots\bi^{(d)}$, and $w\in{}^{\nu^{d,j}}\D^{p^d}_{dp}$ satisy $w\cdot \bi=\hat\ggw^{d,j}$.  Then
$\bi^{(1)}=\dots=\bi^{(d)}=\hat\ggw^j$ and $w=\chi_{d,j}$. 
\end{Lemma}
\begin{proof}
Write $\bi=i_1\cdots i_{dp}$. 
We use the notation of \S\ref{SSNot}, in particular the sets $\ttX^{d,j}_i, \ttY^{d,j}_i, \ttZ^{d,j}_i$. 
We simplify the notation
$\ttP_r:=\ttP^{p\om_d}_r:=[p(r-1)+1,pr]$ for $1\leq r\leq d$, and also 
denote $\ttX_i^{(r)}:=\{t\in\ttP_r\mid i_t=i\}$ for $i\in I$. 
The assumption $w\in{}^{\nu^{d,j}}\D^{p^d}_{dp}$ means that $w$ is increasing on each $\ttP_r$, and $w^{-1}$ is increasing on each $\ttX^{d,j}_i$ for $j<i\leq \ell$ or $i=0$ and on each $\ttY^{d,j}_i$, $\ttZ^{d,j}_i$ for $1\leq i\leq j$. The assumption that $w\cdot \bi=\hat\ggw^{d,j}$ means that $w$ maps $\bigsqcup_{r=1}^d\ttX_i^{(r)}$ onto $\ttX^{d,j}_i$ for all $i\in I$. Thus it suffices to prove that $\bi^{(r)}=\ggw^j$ for all $r=1,\dots,d$, and then it automatically follows that $w=\chi_{d,j}$.

We now claim that for $r=1,\dots,d$ the word $\bi^{(r)}$ begins with $\ell (\ell-1)^2\cdots (j+1)^2$ (i.e. every $\bi^{(r)}$ agrees with $\ggw^j$ in the first $2(\ell-j)-1$ positions). We apply inverse induction on $i=\ell,\ell-1,\dots,j$ to prove that  $\bi^{(r)}$ begins with $\ell (\ell-1)^2\cdots (i+1)^2$ 
for all $r$.
As each $\bi^{(r)}$ is cuspidal, we have that $i^{(r)}$ begins with $\ell$ for all $r$. This gives the induction base. The inductive step follows using the fact that each $\bi^{(r)}$ is cupsidal, contains only two letters equal to $i+1$, and the properties of $w$ listed in the previous paragraph. 

Now, we certainly have that $\bi^{(r)}$ begins with $\ell (\ell-1)^2\cdots (j+1)^2j$ for all $r$ since $\bi^{(r)}\in I^\de$ are cuspidal. On the other hand we now cannot have $\bi^{(r)}$ beginning with $\ell (\ell-1)^2\cdots (j+1)^2j^2$ for any $r$ since $|\ttY^{d,j}_j|=d$. So it follows that each $\bi^{(r)}$ begins with $\ell (\ell-1)^2\cdots (j+1)^2j(j-1)$. Continuing like this, we deduce that each $\bi^{(r)}$ begins with $\ell (\ell-1)^2\cdots (j+1)^2j(j-1)\cdots1$ and then even with $\ell (\ell-1)^2\cdots (j+1)^2j(j-1)\cdots1 0^2$. Now it is easy to see that in fact 
$$
\bi^{(r)}=\ell (\ell-1)^2\cdots (j+1)^2j(j-1)\cdots1 0^21\cdots j=\ggw^j
$$ 
as required.
\end{proof}

\begin{Lemma} \label{LNonZeroU}
Let $\bj\in J^d$, $j\in J$, and $u\in\Si_{dp}$ satisfy $u\cdot \hat\ggw^{\bj}=\hat\ggw^{d,j}$. If $\psi_u\hat\ggi^{\bj}\neq 0$ in $\hat C_d$ then 
$u=x\chi_{d,j}y$ for some $x\in\Si_{\nu^{d,j}}$ and $y\in\Si_{p^d}$ such that $y\cdot  \hat\ggw^{\bj}=\hat\ggw^{j^d}$. 
\end{Lemma}
\begin{proof}
We can write $u=xwy$ for $x\in\Si_{\nu^{d,j}}$, $w\in{}^{\nu^{d,j}}\D^{p^d}_{dp}$ and $y\in\Si_{p^d}$ with $\ttl(u)=\ttl(x)+\ttl(w)+\ttl(y)$. We may assume that $\psi_{u}=\psi_{x}\psi_{w}\psi_{y}$. Now $\psi_u\hat\ggi^{\bj}\neq 0$ implies $\psi_y\hat\ggi^{\bj}\neq 0$, so $y\cdot \hat\ggw^{\bj}=\bi^{(1)}\cdots\bi^{(d)}$ for some 
$\bi^{(1)},\dots,\bi^{(d)}\in I^\de_\cus$. Now, since $x \cdot \hat\ggw^{d,j}=\hat\ggw^{d,j}$, we can apply Lemma~\ref{LConsecutive} to deduce $w=\chi_{d,j}$ and $y\cdot  \hat\ggw^{\bj}=\ggw^{j^d}$. 
\end{proof}

\begin{Lemma} \label{LYY'} 
Let $Y\in\k[y_1,\dots,y_{dp}]$. Then there exists $Y'\in\k[y_1,\dots,y_{dp}]$ such that $Y\psi_{\chi_{d,j}}\hat\ggi^{j^d}=\psi_{\chi_{d,j}}Y'\hat\ggi^{j^d}$ in $\hat C_d$. 
\end{Lemma}
\begin{proof}
We consecutively move the dots corresponding to the elements $y_t$ appearing in $Y$ past the crossings of the Khovanov-Lauda diagram for  $\psi_{\chi_{d,j}}\hat\ggi^{j^d}$ using the relations (\ref{R5}). The error terms appearing on each step are of the form $Y''\psi_uY'''\hat\ggi^{j^d}$ with $\ttl(u)<\ttl(\chi_{d,j})$ and $u\cdot \hat\ggw^{j^d}=\hat\ggw^{d,j}$. By Lemma~\ref{LNonZeroU}, such error terms are zero. 
\end{proof}

\begin{Proposition} \label{LDifficult}
We have $\ggi^{d,j} C_d\ggi_{\om_d}= \lgath_{d,j} C_d\ggi_{\om_d}$. 
\end{Proposition}
\begin{proof}
By Lemma~\ref{LG1} we have  $\ggi^{d,j}\lgath_{d,i}=\lgath_{d,i}$ so $\ggi^{d,j}C_d\ggi_{\om_d}\supseteq  \lgath_{d,i} C_d\ggi_{\om_d}$. 

For the opposite inclusion, let $0\neq c\in \ggi^{d,j} C_d\ggi_{\om_d}$. 
As $\ggi_{\om_d}=\sum_{\bj\in J^d}\ggi^{\bj}$, by Theorem~\ref{TBasis}, we may assume that $c=\ggi^{d,j}\psi_u Y \ggi^{\bj}=\ggi^{d,j}\psi_u \hat\ggi^{\bj}Y \ggi^{\bj}$ for some $\bj\in J^d$, some $Y\in\k[y_1,\dots,y_{dp}]$, and some $u\in\Si_{dp}$ satisfying $u\cdot (\hat\ggw^{\bj})=\hat\ggw^{d,j}$. 
By Lemma~\ref{LNonZeroU}, we may assume that  
$\psi_u=\psi_x\psi_{\chi_{d,j}}\psi_y$ for some $x\in\Si_{\nu^{d,j}}$ and $y\in\Si_{p^d}$ with $y\cdot  \hat\ggw^{\bj}=\hat\ggw^{j^d}$. 

Let $w_{d,j}:=w_{\nu^{d,j}}$ be the longest element of the parabolic subgroup $\Si_{\nu^{d,j}}$. By Lemma~\ref{LW_0Front}, we have 
$
\ggi^{d,j}\psi_x=\psi_{w_{d,j}}Y'
$
for some $Y'\in\k[y_1,\dots,y_{dp}]$. 
So 
\begin{align*}
c&=\ggi^{d,j}\psi_x\psi_{\chi_{d,j}}\psi_y Y \ggi^{\bj}
\\&=\ggi^{d,j}\ggi^{d,j}\psi_x\psi_{\chi_{d,j}}\hat\ggi^{j^d}\psi_y Y \ggi^{\bj}
\\&=\ggi^{d,j}\psi_{w_{d,j}}Y'\psi_{\chi_{d,j}}\hat\ggi^{j^d}\psi_y Y \ggi^{\bj}
\\&=\ggi^{d,j}\psi_{w_{d,j}}\psi_{\chi_{d,j}}Y''\hat\ggi^{j^d}\psi_y Y \ggi^{\bj}
\end{align*}
for some $Y''\in\k[y_1,\dots,y_{dp}]$, where we have used Lemma~\ref{LYY'} for the last equality.  
Recalling the notation (\ref{EKappa}), we note 
$
w_{d,j}
=\kappa_{d,j}z
$
for some element $z\in\Si_d$ such that $\ttl(w_{d,j})=\ttl(\kappa_{d,j})+\ttl(z)$. Moreover, $z\chi_{d,j}=\chi_{d,j}z'$ where $z'=(w_{1,j},\dots,w_{1,j})\in\Si_p\times\dots\times \Si_p=\Si_{p^d}\leq \Si_{dp}$ and $\ttl(\chi_{d,j}z')=\ttl(\chi_{d,j})+\ttl(z')$. Note by Lemma~\ref{LW_0Front} that 
$\psi_{z'}\hat \ggi^{j^d}\hat C_d=\ggi^{j^d}\hat C_d$. So, using (\ref{EKappaH}) and Lemma~\ref{LG1}, we get 
\begin{align*}
c&=\ggi^{d,j}\psi_{w_{d,j}}\psi_{\chi_{d,j}}Y''\hat\ggi^{j^d}\psi_y Y\ggi^{\bj}
\\&
=\ggi^{d,j}\psi_{\kappa_{d,j}}\psi_{\chi_{d,j}}\psi_{z'}\hat\ggi^{j^d}Y''\hat\ggi^{j^d}\psi_y Y\ggi^{\bj}
\\&
\subseteq \ggi^{d,j}k_{d,j}\psi_{\chi_{d,j}}\ggi^{j^d}\hat C_d\ggi^{\bj}
\\&
=\ggi^{d,j}k_{d,j}\gath_{d,j}C_d\ggi^{\bj}
\\&
=\ggi^{d,j}\lgath_{d,j}C_d\ggi^{\bj}
\\&
=\lgath_{d,j}C_d\ggi^{\bj},
\end{align*}
which completes the proof since  $C_d\ggi^{\bj}\subseteq C_d\ggi_{\om_d}$. 
\end{proof}

\begin{Lemma} 
Let $(\la,\bi)\in\Comp^\col(n,d)$. Then
$$
\ggi^{\la,\bi} C_d\ggi_{\om_d}=\lgath_{\la,\bi} C_d\ggi_{\om_d}
=\lgath_{\la,\bi}\ggi^{i_1^{\la_1}\cdots i_n^{\la_n}} C_d\ggi_{\om_d}
=\lgath_{\la,\bi}\ggi_{\om_d} C_d\ggi_{\om_d}.
$$
\end{Lemma}
\begin{proof}
Since $
\lgath_{\la,\bi}=\ggi^{\la,\bi}\lgath_{\la,\bi}\ggi^{i_1^{\la_1}\cdots i_n^{\la_n}}=\ggi^{\la,\bi}\lgath_{\la,\bi}\ggi_{\om_d}$ it suffices to prove the inclusion $\ggi^{\la,\bi} C_d\ggi_{\om_d}\subseteq\lgath_{\la,\bi} C_d\ggi_{\om_d}$. 
By Corollary~\ref{CLa1^d}, we have 
$$\ggi^{\la,\bi} C_d\ggi_{\om_d}=\ggi^{\la,\bi} C_\la\ggi_{\om_d} C_d \ggi_{\om_d}.$$ 
Now, the required  inclusion 
follows from 
\begin{align*}
\ggi^{\la,\bi} C_\la\ggi_{\om_d}&=\ggi^{\la_1,i_1} C_{\la_1}\ggi_{\om_{\la_1}}
\otimes \dots\otimes 
\ggi^{\la_n,i_n} C_{\la_n}\ggi_{\om_{\la_n}}
\\&=\lgath_{\la_1,i_1} C_{\la_1}\ggi_{\om_{\la_1}}
\otimes \dots\otimes 
\lgath_{\la_n,i_n} C_{\la_n}\ggi_{\om_{\la_n}}
\\&=\lgath_{\la,\bi} C_{\la}\ggi_{\om_{d}},
\end{align*}
where we have used Proposition~\ref{LDifficult} for the second equality. 
\end{proof}


\subsection{   The algebra $\zC_d$ is non-negatively graded}
In this subsection we prove that $\zC_d$ is non-negatively graded, i.e. $\zC_d^{<0}=0$, see the notation (\ref{EA>0}).

\begin{Lemma} \label{LNonCusp}
Let $d\in\N_+$. 
The words in $I^{d\de}$ beginning in the following patterns are not cuspidal:
\begin{align*}
&\ell^d(\ell-1)^{2d}\cdots (i+1)^{2d}i^d\cdots (k+1)^dk^{d-1}(k-1)^d\cdots
 &(\text{for $1<k\leq i<\ell$}),
 \\
&\ell^d(\ell-1)^{2d}\cdots (i+1)^{2d}i^d\cdots 2^d1^{d-1}0^{2d}\cdots 
&(\text{for $1\leq i<\ell$}),
\\
&\ell^d(\ell-1)^{2d}\cdots (j+1)^{2d}j^{d}(j-1)^{2d}\cdots 
&(\text{for $2\leq j<\ell$}),
\\
&\ell^d(\ell-1)^{2d}\cdots (j+1)^{2d}j^{a_j}\cdots 1^{a_1}0^{2d}\cdots\cdots 
&(\text{for $1\leq j<\ell$ and $a_j<d$}).
\end{align*}
\end{Lemma}
\begin{proof}
Checked explicitly using the definition of cuspidality and (\ref{ESharp}). 
\end{proof}

\begin{Lemma} \label{Lb=1}
Let $i,j\in J$. We have:
\begin{enumerate}
\item[{\rm (i)}]  
$\ggis^{d,i}\zC_d\ggis^{d,i}=(\ggis^{d,i}\zC_d\ggis^{d,i})^{\geq 0}$ and $(\ggis^{d,i}\zC_d\ggis^{d,i})^0$ is spanned by $\ggis^{d,i}$.

\item[{\rm (ii)}] If $|i-j|=1$ then $\ggis^{d,j}\zC_d\ggis^{d,i}=(\ggis^{d,j}\zC_d\ggis^{d,i})^{>0}$.
 
\item[{\rm (iii)}] If $|i-j|>1$ then $\hat\ggi^{d,j}C_d\ggi^{d,i}=0$; in particular, 
$\ggis^{d,j}\zC_d\ggis^{d,i}=0$. 
\end{enumerate}
\end{Lemma}
\begin{proof}
Throughout the proof we will use the notation (\ref{ESegDec}),\,(\ref{ESegDec1}),\,(\ref{ENEArrow}). 

By Lemma~\ref{LSymmetry}(ii), we may assume that $i\leq j$. 
Recall that $\ggi^{d,i}=1_{\ggw^{d,i}}$ and $\ggi^{d,j}=1_{\ggw^{d,j}}$. 
As $t_{d,i}-t_{d,j}=2d^2(i-j)$,  it suffices to prove the following claims:

(a) $(1_{\ggw^{d,i}}C_d1_{\ggw^{d,i}})^{< 0}=0$ and $(1_{\ggw^{d,i}}C_d1_{\ggw^{d,i}})^0$ is spanned by $1_{\ggw^{d,i}}$, 

(c) $(1_{\ggw^{d,i+1}}C_d1_{\ggw^{d,i}})^{\leq 2d^2}=0$, 

(d) $1_{\hat\ggw^{d,j}}C_d1_{\ggw^{d,i}}=0$ if $j-i>1$. 

By Theorem~\ref{TBasis}, we have that  
$1_{\hat\ggw^{d,j}}C_d1_{\ggw^{d,i}}$ is spanned by the elements of the form $1_{\hat\ggw^{d,j}}F(y_1,\dots,y_{dp})\psi_w1_{\ggw^{d,i}}$, where $F(y_1,\dots,y_{dp})$ is a polynomial and $w\in \Si_{dp}$ satisfies $w\cdot \hat\ggw^{d,i}=\hat\ggw^{d,j}$, which is equivalent to $w(\ttX^{d,i}_k)=\ttX^{d,j}_k$ for all $k\in I$. 
Moreover, since $\psi_x1_{\ggw^{d,i}}=0$ for any $x$ in the parabolic subgroup $\Si_{\nu^{d,i}}$, we may assume that $w\in\D^{\nu^{d,i}}_{dp}$, i.e. $w$ is increasing on all $\ttX^{(i)}_k$ with $i<k\leq \ell$ or $k=0$ and on all $\ttY^{(i)}_k,\ttZ^{(i)}_k$ with $1\leq k\leq i$. Since $i\leq j$, we have 
that $w:\ttX^{(i)}_k\nearrow \ttX^{(j)}_k$ for $j<k\leq \ell$ and $k=0$, and $w:\ttX^{(i)}_k\nearrow \ttY^{(j)}_k\sqcup  \ttZ_k^{(j)}$ for $i< k\leq j$. 
Since all $y_r$'s have positive degrees it now suffices to prove the following claims:

(a$'$) $1_{\ggw^{d,i}}\psi_w1_{\ggw^{d,i}}\in C_d$ is non-zero only if $w=1$, 

(b$'$) a non-zero $1_{\ggw^{d,i+1}}\psi_w1_{\ggw^{d,i}}\in C_d$ has degree $>2d^2$ 

(c$'$) $1_{\hat\ggw^{d,j}}\psi_w1_{\ggw^{d,i}}=0$ in $C_d$ for $|i-j|>1$.

To prove (a$'$), suppose $1_{\ggw^{d,i}}\psi_w1_{\ggw^{d,i}}\in C_d$ is non-zero. To prove that $w=1$, it suffices to prove that $w:\ttY_k^{(i)}\nearrow \ttY_k^{(j)}$ and  $w:\ttZ_k^{(i)}\nearrow \ttZ_k^{(j)}$ for all $k=1,\dots,i$. We apply induction on $k$. Suppose $k\geq 1$ and we already know that $w:\ttY_l^{(i)}\nearrow \ttY_l^{(j)}$ and  $w:\ttZ_l^{(i)}\nearrow \ttZ_l^{(j)}$ for all $1\leq l<k$. Let $t=\max(\ttY_k^{(i)})$. To prove that $w:\ttY_k^{(i)}\nearrow \ttY_k^{(j)}$ and  $w:\ttZ_k^{(i)}\nearrow \ttZ_k^{(j)}$, it suffices to show that $w(t)\in \ttY_k^{(j)}$. If $w(t)\in \ttZ_k^{(j)}$, we can write $w=yx$ with $\ttl(w)=\ttl(y)+\ttl(x)$,   $x(r)=r$ for all $r\in \ttY_k^{(i)}\setminus\{t\}$ and $x(r)=r-1$ for all $r\in \ttY_{k-1}^{(i)}$ (if $k=1$ this should be interpreted as $x(r)=r-1$ for all $r\in \ttX^{(i)}_0$). So we may assume that $\psi_w=\psi_y\psi_x$, and  then $\psi_x1_{\ggw^{d,i}}=0$ since 
$$x\cdot \hat \ggw^{d,i}=
\left\{
\begin{array}{ll}
\ell^d(\ell-1)^{2d}\cdots (i+1)^{2d}i^d\cdots (k+1)^dk^{d-1}(k-1)^d &\hbox{if $k>1$,}\\
\ell^d(\ell-1)^{2d}\cdots (i+1)^{2d}i^d\cdots 2^d1^{d-1}0^{2d} &\hbox{if $k=1$,}
\end{array}
\right.
$$
which is not cuspidal by Lemma~\ref{LNonCusp}. 

To prove (b$'$), as in the case $j=i$ we prove by induction on $k$ that $1_{\ggw^{d,i+1}}\psi_w1_{\ggw^{d,i}}$ is non-zero only if 
$w:\ttY_k^{(i)}\nearrow \ttY_k^{(i+1)}$ and  $w:\ttZ_k^{(i)}\nearrow \ttZ_k^{(i+1)}$ for all $k=1,\dots,i$. Therefore 
$
\deg(1_{\ggw^{d,i+1}}\psi_w1_{\ggw^{d,i}})=4d^2
>2d^2$ as required. 

To prove (c$'$), let $j>i+1$.  We can write $w=yx$ with $\ttl(w)=\ttl(y)+\ttl(x)$ and 
$$
x\cdot \hat \ggw^{d,i}=\ell^d(\ell-1)^{2d}\cdots (j+1)^{2d}j^d(j-1)^{2d}\cdots
$$
which is not cuspidal by Lemma~\ref{LNonCusp}. So $1_{\hat \ggw^{d,j}}C_d1_{\ggw^{d,i}}=0$. 
\end{proof}

Recall the notation of \S\ref{SSNot}. 
For $\mu\in\EC(n,d)$ and $i\in I$, we define the permutation 
$q_{\mu,i}\in\Si_{dp}$ by the following conditions:

(1) $q_{\mu,i}:\ttX^{d,i}_k\nearrow \bigsqcup_{s=1}^n\ttX_{s,k}^{\mu,i}$ for $i<k\leq \ell$ or $k=0$;

(2) $q_{\mu,i}:\ttY^{d,i}_k\nearrow \bigsqcup_{s=1}^n\ttY_{s,k}^{\mu,i}$ 
and $q_{\mu,i}:\ttZ^{d,i}_k\nearrow \bigsqcup_{s=1}^n\ttZ_{s,k}^{\mu,i}$
for $1\leq k\leq i$;

\noindent
Note that $q_{\mu,i}\cdot\hat\ggw^{d,i}=\hat\ggw^{\mu,i}$ and $w_{(d),i}=1$. Define $a_{\mu,i}:=\ggi^{\mu,i}\psi_{q_{\mu,i}}\ggi^{d,i}\in \ggi^{\mu,i}C_d\ggi^{d,i}$. We have the corresponding element $\za_{\mu,i}
\in \ggis^{\mu,i}\zC_d\ggis^{d,i}$.

\begin{Lemma} \label{LMinAB=1} 
Let $i\in J$ and $(\mu,\bj)\in\EC^\col(n,d)$. Then we have 
$\ggis^{\mu,\bj}\zC_d\ggis^{d,i}=(\ggis^{\mu,\bj}\zC_d\ggis^{d,i})^{\geq 0}$. Moreover:
\begin{enumerate}
\item[{\rm (i)}] $(\ggis^{\mu,i}\zC_d\ggis^{d,i})^{0}$ is spanned by $\za_{\mu,i}$;
\item[{\rm (ii)}] $\ggis^{\mu,\bj}\zC_d\ggis^{d,i}=(\ggis^{\mu,\bj}\zC_d\ggis^{d,i})^{>0}$ unless $j_s=i$ for all $s=1,\dots,n$;
\item[{\rm (iii)}] $\ggis^{\mu,\bj}\zC_d\ggis^{d,i}=0$ unless $|j_s-i|\leq 1$ for all $s=1,\dots,n$.
\end{enumerate} 
\end{Lemma}
\begin{proof}
We apply induction on $n$, the base case $n=1$ being Lemma~\ref{Lb=1}. Let $n>1$. It suffices to prove the following claims:

(a) $\ggi^{\mu,\bj}C_d\ggi^{d,i}=(\ggi^{\mu,\bj}C_d\ggi^{d,i})^{\geq t_{d,i}- t_{\mu,\bj}}$;

(b)  $\ggi^{\mu,\bj}C_d\ggi^{d,i}=(\ggi^{\mu,\bj}C_d\ggi^{d,i})^{>t_{d,i}-t_{\mu,\bj}}$ unless $j_s=i$ for all $s=1,\dots,n$, in which case we have that $(\ggi^{\mu,i}C_d\ggi^{d,i})^{t_{\mu,i}-t_{d,i}}$ is spanned by $a_{\mu,i}$;

(c) $\hat \ggi^{\mu,\bj}C_d\ggi^{d,i}=0$ if $|i-j_s|>1$ for some $s\in\{1,\dots,n\}$.

We have increasing disjoint unions of the segments 
\begin{align*}
&[1,dp]=\ttX_\ell^{d,i}\sqcup \ttX_{\ell-1}^{d,i}\sqcup\dots \sqcup \ttX_{i+1}^{d,i}\sqcup \ttY^{d,i}_i\sqcup\dots\sqcup \ttY^{d,i}_1\sqcup
\ttX^{d,i}_0\sqcup \ttZ^{d,i}_1\sqcup\dots \ttZ^{d,i}_i,
\\
&[1,p\mu_1]=\ttX_{1,\ell}^{\mu,j_1}\sqcup \ttX_{1,\ell-1}^{\mu,j_1}\sqcup\dots \sqcup \ttX_{1,{j_1}+1}^{\mu,j_1}\sqcup \ttY^{\mu,j_1}_{1,j_1}\sqcup\dots\sqcup \ttY^{\mu,j_1}_{1,1}\sqcup
\ttX^{\mu,j_1}_{1,0}\sqcup \ttZ^{\mu,j_1}_{1,1}\sqcup\dots \ttZ^{\mu,j_1}_{1,j_1}.
\end{align*}


By Theorem~\ref{TBasis}, we have that 
$\hat\ggi^{\mu,\bj}C_d\ggi^{d,i}$ is spanned by the elements of the form $\hat\ggi^{\mu,\bj}F(y_1,\dots,y_{dp})\psi_w\ggi^{d,i}$, where $F(y_1,\dots,y_{dp})$ is a polynomial and $w\in \Si_{dp}$ satisfies $w\cdot \hat\ggw^{d,i}=\hat\ggw^{\mu,\bj}$.  
Moreover, since $\psi_x\ggi^{d,i}=0$ for any $x$ in the parabolic subgroup $\Si_{\nu^{d,i}}$, we may assume that $w\in\D^{\nu^{d,i}}_{dp}$, i.e. $w$ is increasing on all $\ttX^{d,i}_k$ with $i<k\leq \ell$ or $k=0$ and on all $\ttY^{d,i}_k,\ttZ^{d,i}_k$ with $1\leq k\leq i$. 
Since all $y_r$'s have positive degrees it suffices to prove that for 
a non-zero $\hat \ggi^{\mu,\bj}\psi_w\ggi^{d,i}\in C_d$ we have:

(1) $\deg(\ggi^{\mu,\bj}\psi_w\ggi^{d,i})\geq t_{d,i}-t_{\mu,\bj}$,

(2)  $\deg(\ggi^{\mu,\bj}\psi_w\ggi^{d,i})> t_{d,i}-t_{\mu,\bj}$, unless $w=q_{\mu,i}$ and $j_s=i$ for all $s=1,\dots,n$,

(3) $|j_s-i|\leq 1$ for all $s=1,\dots,n$.

For segments $\ttX$ and $\ttY$ we will use the notation $\ttX<_l\ttY$ if we can write $\ttY$ as an increasing disjoint union of the segments $\ttY=\ttX\sqcup\ttZ$.

Since $w$ is increasing on all $\ttX^{d,i}_k$ with $i<k\leq \ell$ or $k=0$, we have for such $k$ that $\tilde \ttX_{k}:=w^{-1}(\ttX_{1,k}^{\mu,j_1}) <_l \ttX_k^{d,i}$.  Since $w$ is increasing on $\ttY^{d,i}_k,\ttZ^{d,i}_k$ with $1\leq k\leq i$, we have for such $k$ that $\tilde \ttX_{k}:=w^{-1}(\ttX_{1,k}^{\mu,j_1})=
\tilde\ttY_{k}\sqcup \tilde\ttZ_{k}$ with $
\tilde\ttY_{k} <_l \ttY_k^{d,i}$ and $
\tilde\ttZ_{k} <_l \ttZ_k^{d,i}$.
Set
$$
\tilde \ttX:=\tilde\ttX_{\ell}\sqcup \tilde\ttX_{\ell-1}\sqcup\dots \sqcup \tilde\ttX_{i+1}\sqcup \tilde\ttY_{i}\sqcup\dots\sqcup \tilde\ttY_{1}\sqcup
\tilde\ttX_{0}\sqcup \tilde\ttZ_{1}\sqcup\dots \tilde\ttZ_{i}\subseteq [1,dp].
$$
Denote $\ttx_k:=|\tilde\ttX_k|$, $\tty_k:=|\tilde\ttY_{k}|$ and $\ttz_k:=|\tilde\ttZ_{k}|$ for all admissible $k$.

Let $u\in\Si_{dp}$ be the permutation which maps $\tilde X$ increasingly onto $[1,\mu_1p]$ and $[1,dp]\setminus\tilde X$ increasingly onto $[\mu_1p+1,dp]$. Then we can write $w=xyu$ with $\ttl(w)=\ttl(x)+\ttl(y)+\ttl(u)$ and $(x,y)\in\Si_{\mu_1p}\times\Si_{dp-\mu_1p}<\Si_{dp}$. Note that $
u\cdot\hat\ggw^{d,i}=\bi^{(1)}\bi^{(2)}$, where
\begin{align*}
\bi^{(1)}&:=\ell^{\,\ttx_\ell}(\ell-1)^{\,\ttx_{\ell-1}}\cdots (i+1)^{\,\ttx_{i+1}}i^{\,\tty_{i}}\cdots 1^{\,\tty_{1}}0^{\,\ttx_{0}}1^{\,\ttz_{1}}\cdots i^{\,\ttz_{i}},
\\
\bi^{(2)}&:=\ell^{d\ell-\ttx_\ell}(\ell-1)^{2d\ell-\ttx_{\ell-1}}\cdots (i+1)^{2d\ell-\ttx_{i+1}}i^{d\ell-\tty_{i}}\cdots 1^{d\ell-\tty_{1}}0^{2d\ell\ttx_{0}}1^{d\ell-\ttz_{1}}\cdots i^{d\ell-\ttz_{i}}.
\end{align*}
So, by Lemma~\ref{LNonCusp}, $u\cdot\hat\ggw^{d,i}$ is non-cuspidal and $\psi_u\ggi^{d,i}$ is zero in $C_d$, unless $\bi^{(1)}=\hat\ggw^{\mu_1,i}$ and $\bi^{(2)}=\ggw^{d-\mu_1,i}$. So we may assume that $\bi^{(1)}=\hat\ggw^{\mu_1,i}$ and $\bi^{(2)}=\ggw^{d-\mu_1,i}$. Then  we also have $x\cdot\hat\ggw^{\mu_1,i}=\hat\ggw^{\mu_1,j_1}$ and $y\cdot\hat\ggw^{d-\mu_1,i}=\hat\ggw^{\bar\mu,\bar\bj}$, where we have set
$\bar\mu:=(\mu_2,\dots,\mu_n)$ and $\bar\bj:=j_2\cdots j_n$. 

An elementary computation shows that 
$
\deg(\psi_u\ggi^{d,i})=4\mu_1(2\ell-i)(d-\mu_1).
$ 
Now, by the inductive assumption, we have that:

(1$'$) $\deg(\ggi^{\mu_1,j_1}\psi_x\ggi^{\mu_1,i})\geq t_{(\mu_1),j_1}-t_{(\mu_1),i}$ and $\deg(\ggi^{\bar\mu,\bar\bj}\psi_y\ggi^{d-\mu_1,i})\geq t_{\bar\mu,\bar\bj}-t_{(d-\mu_1),i}$;

(2$'$)  $\deg(\ggi^{\mu_1,j_1}\psi_x\ggi^{\mu_1,i})>t_{(\mu_1),j_1}-t_{(\mu_1),i}$ unless $j_1=i$ and $x=1$; moreover, 
 $\deg(\ggi^{\bar\mu,\bar\bj}\psi_y\ggi^{d-\mu_1,i})> t_{\bar\mu,\bar\bj}-t_{(d-\mu_1),i}$ unless $j_s=i$ for all $s=2,\dots,n$  and $y=q_{\bar\mu,i}$. 

(3$'$) $\hat\ggi^{\mu_1,j_1}\psi_x\ggi^{\mu_1,i}=0$ unless $|j_1-i|\leq 1$ and $\hat\ggi^{\bar\mu,\bar\bj}\psi_y\ggi^{d-\mu_1,i}=0$ unless $|j_s-i|\leq 1$ for all $s=2,\dots,n$. 

Another elementary computation shows that 
$$
4\mu_1(2\ell-i)(d-\mu_1)+t_{(\mu_1),j_1}-t_{(\mu_1),i}+t_{\bar\mu,\bar\bj}-t_{(d-\mu_1),i}=t_{\mu,\bj}-t_{d,i}
$$
and $q_{\mu,i}=q_{\bar\mu,i}u$, 
so (1),(2),(3) above follow from (1$'$),(2$'$),(3$'$), respectively. 
\end{proof}

For $(\la,\bi)\in\Comp^\col(d)$ and $j\in J$, we define 
$$
d_j(\la,\bi):=\sum_{r\,\,\text{with}\,\, i_r=j}\la_r,
$$ 
and, recalling the notation (\ref{ELaJ}), set 
\begin{equation}\label{EUdLaBi}
\ud(\la,\bi):=(d_0(\la,\bi),\dots,d_{\ell-1}(\la,\bi))\in \Comp(J,d).
\end{equation}

\begin{Theorem} \label{TNonNeg}
The algebra $\zC_d$ is non-negatively graded. Also, $\ggis^{\mu,\bj}\zC_d^0\ggis^{\la,\bi}=0$ unless $ \ud(\la,\bi)= \ud(\mu,\bj)$. 
\end{Theorem}
\begin{proof}
Let $\la=(\la_1,\dots,\la_m)$ and $\mu=(\mu_1,\dots,\mu_n)$. 
If $\min(m,n)=1$ then by Lemma~\ref{LSymmetry}(ii), we may assume that $m=1$, in which case the theorem follows from Lemma~\ref{LMinAB=1}. So we may assume that $m,n\geq 2$. 

Let $\bar\la=(\la_2,\dots,\la_m)$, $\bar\bi=i_2\cdots i_m$, and $\bar\mu=(\mu_2,\dots,\mu_n)$, $\bar\bj=j_2\cdots j_n$. Note that 
$ \ud(\la,\bi)= \ud((\la_1),i_1)+ \ud(\bar\la,\bar\bi)$ and $ \ud(\mu,\bj)= \ud((\mu_1),j_1)+ \ud(\bar\mu,\bar\bj)$. 

By Corollary~\ref{CGGsIngId}, we have 
$$\zC_d\ggis^{\la,\bi}\simeq \GGIS_{\la_1,d-\la_1}^d(\zC_{\la_1}\ggis^{\la_1,i_1}\boxtimes \zC_{d-\la_1}\ggis^{\bar\la,\bar\bi}).
$$ 
By the inductive assumption, $\zC_{\la_1}$ and $\zC_{d-\la_1}$ are non-negatively graded, which now implies that so is $\zC_d\ggis^{\la,\bi}$. This implies the first claim. 

To prove the second claim, from the Mackey Theorem~\ref{CGGMackeyS}, we have that 
$$
\ggis_{\mu_1,d-\mu_1}\zC_d\ggis^{\la,\bi}\simeq \GGRS^d_{\mu_1,d-\mu_1}\GGIS_{\la_1,d-\la_1}^d(\zC_{\la_1}\ggis^{\la_1,i_1}\boxtimes \zC_{d-\la_1}\ggis^{\bar\la,\bar\bi})
$$
has a filtration with subquotients of the form 
$$
\GGIS^{\mu_1,d-\mu_1}_{r,s;t,u}\,\,{}^x\big((\GGRS^{\la_1}_{r,t}\zC_{\la_1}\ggis^{\la_1,i_1})\boxtimes(\GGRS^{d-\la_1}_{s,u}\zC_{d-\la_1}\ggis^{\bar\la,\bar\bi})\big).
$$
It follows that $\ggis^{\mu,\bj}\zC_d\ggis^{\la,\bi}$ has a filtration with subquotients of the form 
$$
\big((\ggis^{\mu_1,j_1}\zC_{\mu_1}\ggis_{r,s})\boxtimes(\ggis^{\bar\mu,\bar\bj}\zC_{d-\mu_1}\ggis_{t,u})\big)\otimes_{\zC_{r,s;t,u}}{}^x\big((\ggis_{r,t}\zC_{\la_1}\ggis^{\la_1,i_1})\boxtimes(\ggis_{s,u}\zC_{d-\la_1}\ggis^{\bar\la,\bar\bi})\big).
$$

Recalling the concatenation notation (\ref{EConcat}), we have 
\begin{align*}
\ggis^{\mu_1,j_1}\zC_{\mu_1}\ggis_{r,s}&=\bigoplus_{(\rho,\br)\in\EC^\col(r),\, (\si,\bs)\in\EC^\col(s)} \ggis^{\mu_1,j_1}\zC_{\mu_1}\ggis^{\rho\si,\br\bs},
\\
\ggis^{\bar\mu,\bar\bj}\zC_{d-\mu_1}\ggis_{t,u}&=\bigoplus_{(\tau,\bt)\in\EC^\col(t),\, (\upsilon,\bu)\in\EC^\col(u)} \ggis^{\bar\mu,\bar\bj}\zC_{d-\mu_1}\ggis^{\tau\upsilon,\bt\bu},
\\
\ggis_{r,t}\zC_{\la_1}\ggis^{\la_1,i_1}&=\bigoplus_{(\rho',\br')\in\EC^\col(r),\, (\tau',\bt')\in\EC^\col(t)} \ggis^{\rho'\tau',\br'\bt'}\zC_{\la_1}\ggis^{\la_1,i_1},
\\
\ggis_{s,u}\zC_{d-\la_1}\ggis^{\bar\la,\bar\bi}&=\bigoplus_{(\si',\bs')\in\EC^\col(s),\, (\upsilon',\bu')\in\EC^\col(u)} \ggis^{\si'\upsilon',\bs'\bu'}\zC_{d-\la_1}\ggis^{\bar\la,\bar\bi}. 
\end{align*}

By the inductive assumption, we have for all $(\rho,\br)\in\EC^\col(r)$, $(\si,\bs)\in\EC^\col(s)$, $(\tau,\bt)\in\EC^\col(t)$ and $(\upsilon,\bu)\in\EC^\col(u)$: 
\begin{enumerate}
\item[$\bullet$] $(\ggis^{\mu_1,j_1}\zC_{\mu_1}\ggis^{\rho\si,\br\bs})^0=0$ unless $ \ud(\rho,\br)+ \ud(\si,\bs)= \ud((\mu_1),j_1)$;
\item[$\bullet$] $(\ggis^{\bar\mu,\bar\bj}\zC_{d-\mu_1}\ggis^{\tau\upsilon,\bt\bu})^0=0$ unless $ \ud(\tau,\bt)+ \ud(\upsilon,\bu)= \ud(\bar\mu,\bar\bj)$;
\item[$\bullet$] $(\ggis^{\rho'\tau',\br'\bt'}\zC_{\la_1}\ggis^{\la_1,i_1})^0=0$ unless $ \ud(\rho',\br')+ \ud(\tau',\bt')= \ud((\la_1),i_1)$;
\item[$\bullet$] $(\ggis^{\si'\upsilon',\bs'\bu'}\zC_{d-\la_1}\ggis^{\bar\la,\bar\bi})^0=0$ unless $ \ud(\si',\bs')+ \ud(\upsilon',\bu')= \ud(\bar\la,\bar\bi)$.
\end{enumerate}
Finally, note that 
\begin{align*}
&\big((\ggis^{\mu_1,j_1}\zC_{\mu_1}\ggis^{\rho\si,\br\bs})\boxtimes(\ggis^{\bar\mu,\bar\bj}\zC_{d-\mu_1}\ggis^{\tau\upsilon,\bt\bu})\big)
\\
&\otimes_{\zC_{r,s;t,u}}\hspace{-1mm}{}^x\big((\ggis^{\rho'\tau',\br'\bt'}\zC_{\la_1}\ggis^{\la_1,i_1})\boxtimes(\ggis^{\si'\upsilon',\bs'\bu'}\zC_{d-\la_1}\ggis^{\bar\la,\bar\bi})\big)=0
\end{align*}
unless $(\rho,\br)=(\rho',\br')$, $(\si,\bs)=(\si',\bs')$, $(\tau,\bt)=(\tau',\bt')$ and $(\upsilon,\bu)=(\upsilon',\bu')$. 
So $(\ggis^{\mu,\bj}\zC_d\ggis^{\la,\bi})^0=0$ unless 
\begin{align*}
 \ud(\mu,\bj)&= \ud((\mu_1),j_1)+ \ud(\bar\mu,\bar\bj)\\&= \ud(\rho,\br)+ \ud(\si,\bs)+ \ud(\tau,\bt)+ \ud(\upsilon,\bu)
\\&= \ud(\rho',\br')+ \ud(\si',\bs')+ \ud(\tau',\bt')+ \ud(\upsilon',\bu')
\\&= \ud((\la_1),i_1)+ \ud(\bar\la,\bar\bi)
\\&= \ud(\la,\bi)
\end{align*}
as required. 
\end{proof}

For $j\in J$, we have the idempotent  
\begin{equation}\label{EFDJ}
\ggis_{d,j}:=\sum_{\la\in \EC(d)}\ggis^{\la,j}\in\zC_d
\end{equation}
(this is not to be confused with the idempotents $\ggi^{d,j}\in C_d$ or $\ggis^{d,j}\in\zC_d$, cf. (\ref{EGGIOneColorIdNot'})). 

\begin{Corollary} \label{C030924}
Let $\la\in\Comp(d)$ and $j\in J$. Then 
$\zC_d^0\ggis^{\la,j}=\ggis_{d,j}\zC_d^0\ggis^{\la,j}$ and $\ggis^{\la,j}\zC_d^0=\ggis^{\la,j}\zC_d^0\ggis_{d,j}$
\end{Corollary}
\begin{proof}
By (\ref{EGad}), we have 
$$\zC_d\ggis^{\la,j}=\sum_{(\mu,\bj)\in\EC^\col(d)}\ggis^{\mu,\bj}\zC_d\ggis^{\la,j}.
$$
By Theorem~\ref{TNonNeg}, we have $\ggis^{\mu,\bj}\zC_d^0\ggis^{\la,j}=0$ unless $\bj$ is of the form $j^k$, which implies the first equality. The second equality is proved similarly. 
\end{proof}

\subsection{On weight spaces $\ggis^\bla\zC_d^0\ggis^\bmu$}
\label{SSWtSpaces}
Recall from (\ref{EBeta}) that we consider the multicompositions from $\Comp^J(n,d)$ as elements of $\Comp^\col(n\ell,d)$ via the embedding $\ttb:\Comp^J(n,d)\to \Comp^\col(n\ell,d)$. So, for  $\bla\in \Comp^J(n,d)$, we use the notation 
\begin{equation}\label{EBlaNot}
\ggis^\bla:=\ggis^{\ttb(\bla)}=\ggis^{\tta(\bla),0^n\cdots (\ell-1)^n},\ 
\ud(\bla):=\ud(\ttb(\bla)),\ \text{etc.}
\end{equation}
Note that $
\ud(\bla)=(|\la^{(0)}|,\dots, |\la^{(\ell-1)}|).
$
The second statement of Theorem~\ref{TNonNeg} yields:

\begin{Lemma} \label{LBiWeightSpaceTriv}
If $\bla,\bmu\in\Comp^J(n,d)$ and $\ud(\bla)\neq \ud(\bmu)$ then $\ggis^\bla \zC_d^0\ggis^\bmu=0$.
\end{Lemma}

So we concentrate on the weight spaces $\ggis^\bla \zC_d^0\ggis^\bmu=0$ for $\ud(\bla)= \ud(\bmu)$. The next lemma gives an upper bound on the size of these weight spaces when both $\bla$ and $\bmu$ concentrated in one color $j$. We will later prove that this upper bound is sharp, i.e. $\ggis^{\la,j}\zC_d^0\ggis^{\mu,j}$ is free of rank $|{}^\la\D_d^\mu|$.

\begin{Lemma} \label{LDoubleCosets}
Let $\la,\mu\in\Comp(d)$ and $j\in J$. Then the graded $\k$-superpsace $\ggis^{\la,j}\zC_d^0\ggis^{\mu,j}$ is spanned by $|{}^\la\D_d^\mu|$ elements. 
\end{Lemma}
\begin{proof}
By the Mackey Theorem~\ref{CGGMackeyS}, 
$$
\ggis^{\la,j}\zC_d\ggis^{\mu,j}=\ggis^{\la,j}\GGRS^d_\la\GGIS_\mu^d\zC_\mu\ggis^{\mu,j}
$$
has filtration with sub-quotients 
$$
\ggis^{\la,j}\GGIS^\la_{\la\cap x\mu}{}^{x}(\GGRS^\mu_{x^{-1}\la\cap\mu}\zC_\mu\ggis^{\mu,j})
=\ggis^{\la,j}\zC_\la\ggis_{\la\cap x\mu}\otimes_{\zC_{\la\cap x\mu}}
{}^{x}(\ggis_{x^{-1}\la\cap\mu}\zC_\mu\ggis^{\mu,j}).
$$
Since all algebras are non-negatively graded by Theorem~\ref{TNonNeg}, we deduce that the graded $\k$-supermodule $\ggis^{\la,j}\zC_d^0\ggis^{\mu,j}$ has filtration with sub-quotients 
$$
\ggis^{\la,j}\zC^0_\la\ggis_{\la\cap x\mu}\otimes_{\zC^0_{\la\cap x\mu}}
{}^{x}(\ggis_{x^{-1}\la\cap\mu}\zC^0_\mu\ggis^{\mu,j})\qquad(x\in{}^\la\D_d^\mu).
$$
By the second claim of Theorem~\ref{TNonNeg}, for each $x$ we have 
\begin{align*}
\ggis^{\la,j}\zC^0_\la\ggis_{\la\cap x\mu}
=\ggis^{\la,j}\zC^0_\la\ggis^{\la\cap x\mu,j}
\quad \text{and}\quad
\ggis_{x^{-1}\la\cap\mu}\zC_\mu^0\ggis^{\mu,j}=\ggis^{x^{-1}\la\cap\mu,j}\zC^0_\mu\ggis^{\mu,j}.
\end{align*}
By Lemma~\ref{LMinAB=1}, $\ggis^{x^{-1}\la\cap\mu,j}\zC^0_\mu\ggis^{\mu,j}$ is spanned by one element. The same is true for $\ggis^{\la,j}\zC^0_\la\ggis^{\la\cap x\mu,j}$ by Lemmas~\ref{LMinAB=1} and \ref{LSymmetry}(ii). 
The result follows. 
\end{proof}

Our next goal is to prove Proposition~\ref{PDoubleCosetsGen} generalizing Lemma~\ref{LDoubleCosets}. Recall from (\ref{ERegrParabolic}) and (\ref{EzCParIdentify}) that for a composition $\ud=(d_0,\dots,d_{\ell-1})\in\Comp(J,d)$, we have the parabolic subalgebra $\zC_\ud=\zC_{d_0}\otimes\dots\otimes\zC_{d_{\ell-1}}\subseteq \ggis_\ud C_d\ggis_\ud$. 

\begin{Proposition}\label{PDoubleCosetsGen} Let $\bla,\bmu\in\Comp^J(n,d)$ with 
$$
\ud(\bla)= \ud(\bmu)=:\ud=(d_0,\dots,d_{\ell-1}).
$$
Then 
\begin{equation}\label{E030924}
\ggis^\bla \zC_d^0\ggis^\bmu=\ggis^\bla \zC_\ud^0\ggis^\bmu\simeq (\ggis^{\la^{(0)},0} \zC_{d_0}^0\ggis^{\mu^{(0)},0}) \otimes\dots\otimes (\ggis^{\la^{(\ell-1)},\ell-1} \zC_{d_{\ell-1}}^0\ggis^{\mu^{(\ell-1)},\ell-1}).
\end{equation}
In particular, the graded $\k$-superpsace $\ggis^{\bla}\zC_d^0\ggis^{\bmu}$ is spanned by 
$\prod_{j\in J}|{}^{\la^{(j)}}\D_{d_j}^{\,\mu^{(j)}}|=|{}^{\tta(\bla)}\D_\ud^{\,\tta(\bmu)}|
$ 
elements. 
\end{Proposition}
\begin{proof}
The isomorphism in (\ref{E030924}) is clear from the identification $\zC_\ud=\zC_{d_0}\otimes\dots\otimes\zC_{d_{\ell-1}}$ and the fact that all algebras are positively graded by Theorem~\ref{TNonNeg}. So, 
once the equality $\ggis^\bla \zC_d^0\ggis^\bmu=\ggis^\bla \zC_\ud^0\ggis^\bmu$ is proved, everything else follows from Lemma~\ref{LDoubleCosets}. To prove the equality, we apply the Mackey Theorem~\ref{CGGMackeyS} to
$$
\ggis^\bla \zC_d\ggis^\bmu=\ggis^\bla \GGRS^d_{\ud}\GGIS_\ud^d\zC_\ud\ggis^\bmu
$$
Theorem~\ref{CGGMackeyS} yields a filtration with sub-quotients isomorphic to 
$$
\ggis^\bla\GGIS^\ud_{\ud\cap x\ud}
{}^{x}(\GGRS^\ud_{x^{-1}\ud\cap\ud}\zC_\ud\ggis^{\bmu})
\simeq
\ggis^\bla\GGIS^\ud_{\ud\cap x\ud}
{}^{x}(\ggis_{x^{-1}\ud\cap\ud}\zC_\ud\ggis^{\bmu})\qquad(x\in{}^\ud\D_d^\ud).
$$
Since all algebras are non-negatively graded by Theorem~\ref{TNonNeg}, we deduce that the graded $\k$-supermodule $\ggis^{\bla}\zC_d^0\ggis^{\bmu}$ has filtration with sub-quotients 
$$
\ggis^\bla\zC_\ud^0\ggis_{\ud\cap x\ud}\otimes_{\zC_{\ud\cap x\ud}^0}
{}^{x}(\ggis_{x^{-1}\ud\cap\ud}\zC_\ud^0\ggis^{\bmu})\qquad(x\in{}^\ud\D_d^\ud).
$$
By Theorem~\ref{CGGMackeyS}, the case $x=1$ corresponds to the bottom level of the filtration and contributes the $\k$-subspace 
$$
\ggis^\bla\zC_\ud^0\ggis_{\ud}\otimes_{\zC_{\ud}^0}
\ggis_{\ud}\zC_\ud^0\ggis^{\bmu}=
\ggis^\bla\zC_\ud^0\otimes_{\zC_{\ud}^0}
\zC_\ud^0\ggis^{\bmu}\simeq \ggis^\bla \zC_\ud^0\ggis^\bmu.
$$
It remains to prove that for $1\neq x\in{}^\ud\D_d^\ud$, we have 
$$
\ggis^\bla\zC_\ud^0\ggis_{\ud\cap x\ud}\otimes_{\zC_{\ud\cap x\ud}^0}
{}^{x}(\ggis_{x^{-1}\ud\cap\ud}\zC_\ud^0\ggis^{\bmu})=0.
$$
We denote 
$$
\ud\cap x\ud=(d_{0,0},\dots,d_{0,\ell-1};\dots;d_{\ell-1,0},\dots,d_{\ell-1,\ell-1})
$$
so that 
$$
x^{-1}\ud\cap \ud=(d_{0,0},\dots,d_{\ell-1,0};\dots;d_{0,\ell-1},\dots,d_{\ell-1,\ell-1})
$$
and $\sum_{i\in J}d_{i,j}=d_j=\sum_{i\in J}d_{j,i}$. 
Define the idempotents 
\begin{align*}
\ggis(\ud\cap x\ud)&:=\ggis_{d_{0,0},0}\otimes\dots\otimes\ggis_{d_{0,\ell-1},0}\otimes\dots\otimes\ggis_{d_{\ell-1,0},\ell-1}\otimes\dots\otimes\ggis_{d_{\ell-1,\ell-1},\ell-1},
\\
\ggis'(\ud\cap x\ud)&:=\ggis_{d_{0,0},0}\otimes\dots\otimes\ggis_{d_{0,\ell-1},\ell-1}\otimes\dots\otimes\ggis_{d_{\ell-1,0},0}\otimes\dots\otimes\ggis_{d_{\ell-1,\ell-1},\ell-1}
\end{align*}
in the parabolic subalgebra $\zC_{\ud\cap x\ud}$,
and
$$
\ggis(x^{-1}\ud\cap \ud):=\ggis_{d_{0,0},0}\otimes\dots\otimes\ggis_{d_{\ell-1,0},0}\otimes\dots\otimes\ggis_{d_{0,\ell-1},\ell-1}\otimes\dots\otimes\ggis_{d_{\ell-1,\ell-1},\ell-1}
$$
in the parabolic subalgebra $\zC_{x^{-1}\ud\cap \ud}$.

By Corollary~\ref{C030924}, we have 
$
\ggis^\bla\zC_\ud^0\ggis_{\ud\cap x\ud}
=\ggis^\bla\zC_\ud^0\ggis(\ud\cap x\ud)
$ and $
\ggis_{x^{-1}\ud\cap\ud}\zC_\ud^0\ggis^{\bmu}=
\ggis(x^{-1}\ud\cap\ud)\zC_\ud^0\ggis^{\bmu},
$
and so 
$$
{}^{x}(\ggis_{x^{-1}\ud\cap\ud}\zC_\ud^0\ggis^{\bmu})=
{}^{x}(\ggis(x^{-1}\ud\cap\ud)\zC_\ud^0\ggis^{\bmu})
=
\ggis'(\ud\cap x\ud)\,{}^{x}(\ggis(x^{-1}\ud\cap\ud)\zC_\ud^0\ggis^{\bmu})
$$

For $x\neq 1$, we must have that for some $j\in J$ and distinct $i_1,i_2\in J$, we have $d_{j,i_1}\neq 0\neq d_{j,i_2}$, and so $\ggis(\ud\cap x\ud)\ggis'(\ud\cap x\ud)=0$. So in this case, we have  that 
\begin{align*}
&\ggis^\bla\zC_\ud^0\ggis_{\ud\cap x\ud}\otimes_{\zC_{\ud\cap x\ud}^0}
{}^{x}(\ggis_{x^{-1}\ud\cap\ud}\zC_\ud^0\ggis^{\bmu})
\\
=\,\,
&\ggis^\bla\zC_\ud^0\ggis(\ud\cap x\ud)\otimes_{\zC_{\ud\cap x\ud}^0}
\ggis'(\ud\cap x\ud)\,
{}^{x}(\ggis(x^{-1}\ud\cap\ud)\zC_\ud^0\ggis^{\bmu})
\\
=\,\,
&\ggis^\bla\zC_\ud^0\ggis(\ud\cap x\ud)\ggis'(\ud\cap x\ud)\otimes_{\zC_{\ud\cap x\ud}^0}
\ggis'(\ud\cap x\ud)\,
{}^{x}(\ggis(x^{-1}\ud\cap\ud)\zC_\ud^0\ggis^{\bmu})
\end{align*}
equals $0$, as desired.
\end{proof}

\subsection{  More on the elements $\lgath_{d,j}$}
\label{SSUpsilonz}
Let $j\in J$, $d\in\N_+$ and $(\la,\bi)\in\Comp^\col(n,d)$.  Recalling the elements $\lgath_{d,j}\in C_d$ and $\lgath_{\la,\bi}\in C_\la$ from (\ref{EUpsilon}) and (\ref{EUpsilonLa}), we have the corresponding elements  
\begin{equation}\label{EGathZ}
\lgathz_{d,j}\in\zC_d\quad \text{and}\quad  
\lgathz_{\la,\bi}:=\lgathz_{\la_1,i_1}\otimes\dots \lgathz_{\la_n,i_n}\in \zC_\la\subseteq \ggis_\la\zC_d\ggis_\la. 
\end{equation}
By Lemmas~\ref{LG1} and \ref{LUpsilonPart}, we have 
\begin{equation}\label{E290923}
0\neq \lgathz_{d,j}=\ggis^{d,j}\lgathz_{d,j}\ggis^{j^d}\quad\text{and}\quad  
0\neq \lgathz_{\la,\bi}=\ggis^{\la,\bi}\lgathz_{\la,\bi}\ggis^{i_1^{\la_1}\cdots i_n^{\la_n}}.
\end{equation}

\begin{Lemma} \label{LDegreeG}
We have $\bideg(\lgathz_{\la,\bi})=(0,\0)$. 
\end{Lemma}
\begin{proof}
Using (\ref{EReGradingC}), (\ref{EShiftsCuspidal}) and (\ref{EUDeg}), we get 
\begin{align*}
\bideg(\lgathz_{\la,\bi})
=&\,\bideg(\lgath_{\la,\bi})+(t_{\om_d,i_1^{\la_1}\cdots i_n^{\la_n}},\eps_{\om_d,i_1^{\la_1}\cdots i_n^{\la_n}})-(t_{\la,\bi},\eps_{\la,\bi})
\\=&\,\big(\textstyle\sum_{r=1}^n\la_r(\la_r-1)(2i_r-4\ell)\,,\,\0\big)
\\&+
(dp+\textstyle\sum_{r=1}^n\la_r(2i_r-4\ell)\,,\,\sum_{r=1}^n\la_ri_r\pmod{2})
\\&
-\big(dp
+\textstyle\sum_{r=1}^n\la_r^2(2i_r-4\ell)\,,\,\sum_{r=1}^n\la_ri_r\pmod{2}\big)
\\=&\,(0,\0),
\end{align*}
as desired.
\end{proof}

\begin{Lemma} \label{LGBasis}
We have that $\{\lgathz_{\la,\bi}\}$ is a $\k$-basis of  $(\ggis^{\la,\bi}\zC_{\la}\ggis^{i_1^{\la_1}\cdots i_n^{\la_n}})^0$. 
\end{Lemma}
\begin{proof}
By Theorem~\ref{TNonNeg}, each $\zC_{\la_r}$ 
is non-negatively graded, so it suffices to prove the special case that $\{\lgathz_{d,j}\}$ is the basis of  $(\ggis^{d,j}\zC_d\ggis^{j^d})^0$. 
By Lemmas~\ref{LSymmetry}(ii) and \ref{LMinAB=1}(i), the $\k$-module $(\ggis^{d,j}\zC_{d}\ggis^{j^{d}})^0$ is generated by one element. 
Moreover, by Lemmas~\ref{LUpsilonPart} and \ref{LDegreeG}, we have that $\lgathz_{d,j}\in (\ggis^{d,j}\zC_{d}\ggis^{j^{d}})^0$ is non-zero when working over $\O$ and over any field, so $\{\lgathz_{\la,\bi}\}$ must be the basis of this $\k$-module. 
\end{proof}

\begin{Corollary} \label{CLowestDegree} 
We have that $\{\lgath_{\la,\bi}\}$ is a $\k$-basis of the lowest degree component of  $\ggi^{\la,\bi}C_{\la}\ggi^{i_1^{\la_1}\cdots i_n^{\la_n}}$. 
\end{Corollary}
\begin{proof}
By Theorem~\ref{TNonNeg}, $\zC_\la$ is non-negatively graded. So Lemma~\ref{LGBasis} implies that $\{\lgathz_{\la,\bi}\}$ is a $\k$-basis of the lowest degree component of $\ggis^{\la,\bi}\zC_{\la}\ggis^{i_1^{\la_1}\cdots i_n^{\la_n}}$. The result follows since $C_\la$ is obtained from $\zC_\la$ be regrading.
\end{proof}

Recall from (\ref{ETauSigma}) the elements $\dot w\in \ggi_{\om_d}C_d\ggi_{\om_d}$.

\begin{Proposition} \label{PUpsilon} 
Let $j\in J$ and $d\in\N_+$. Then in $C_d$ we have: 
\begin{enumerate}
\item[{\rm (i)}]  $\lgath_{d,j}\dot w=(-1)^{\ttl(w)(j+1)}\lgath_{d,j}$ for all $w\in \Si_d$;
\item[{\rm (ii)}] for any $\la\in\Comp(d)$ we have that $\lgath_{\la,j}\dot w=(-1)^{\ttl(w)(j+1)}\lgath_{\la,j}$ for all $w\in \Si_\la$. 
\end{enumerate}
\end{Proposition}
\begin{proof}
Note that (ii) follows from (i) since 
$
\lgath_{\la,j}:=\lgath_{\la_1,j}\otimes\dots\otimes \lgath_{\la_n,j}\in C_{\la_1}\otimes\dots\otimes C_{\la_n}=C_\la. 
$ 
To prove (i), it suffices to show that $\lgath_{d,j}\dot  s_r=(-1)^{j+1}\lgath_{d,j}$ for all $1\leq r<d$. 
By Corollary~\ref{CLowestDegree}, $\{\lgath_{d,j}\}$ is a basis of  the lowest degree component of $\ggi^{d,j}C_d\ggi_{\om_d}$. By Theorem~\ref{TMorIso}, $\dot  s_r$ is a degree zero element of order $2$. So 
$\lgath_{d,j}\dot s_r=c\lgath_{d,j}$ for $c=\pm 1$. To see that actually $c=(-1)^{j+1}$, 
recall that 
$$
\dot s_{r}=\ggi_1^{\otimes (r-1)}\otimes \dot s\otimes \ggi_1^{\otimes (d-r-1)}\in C_1^{\otimes (r-1)}\otimes \ggi_{\om_2}C_2\ggi_{\om_2}\otimes C_1^{\otimes (d-r-1)}\subseteq  \ggi_{\om_d}C_d\ggi_{\om_d},
$$
and note that 
we can write 
$\lgath_{d,j}=yx$ for some $y\in C_d$ and 
$$
x=\ggi_1^{\otimes (r-1)}\otimes \lgath_{2,j}\otimes \ggi_1^{\otimes (d-r-1)}\in C_1^{\otimes (r-1)}\otimes \ggi^{2,j}C_2\ggi_{\om_2}\otimes C_1^{\otimes (d-r-1)}\subseteq  \ggi_{\om_d}C_d\ggi_{\om_d},
$$
so the computation reduced to the case $d=2$, in which case we need to prove $\lgath_{2,j}\dot s=(-1)^{j+1}\lgath_{2,j}$, knowing that $\lgath_{2,j}\dot s=c\lgath_{2,j}$ for $c=\pm 1$. 

Recalling the element $\upsigma$ from (\ref{EUpSigma}), observe that $\upsigma\ggi^{j,j}=z\lgath_{2,j}$ for 
some $z\in C_2$.  Hence we have $z\lgath_{2,j}\dot s\ggi^{j,j}=cz\lgath_{2,j}\ggi^{j,j}$ or  
\begin{equation}\label{E100824}
\upsigma\dot s\ggi^{j,j}=c\upsigma\ggi^{j,j}. 
\end{equation}
On the other hand, by (\ref{ETau}), we have $\dot s\ggi^{j,j}=(\upsigma+(-1)^j)\ggi^{j,j}$, and by Theorem~\ref{TMorIso}, we have $\dot s^2\ggi^{j,j}=\ggi^{j,j}$, so 
$
(\upsigma+(-1)^j)\dot s\ggi^{j,j}=\ggi^{j,j},
$
whence we have
$$
\upsigma\dot s\ggi^{j,j}=(-1)^{j+1}\dot s\ggi^{j,j}+\ggi^{j,j}
=(-1)^{j+1}(\upsigma+(-1)^j)\ggi^{j,j}+\ggi^{j,j}=(-1)^{j+1}\upsigma\ggi^{j,j}.
$$
Combining with (\ref{E100824}), we deduce that $c=(-1)^{j+1}$ as required. 
\end{proof}

\chapter{Imaginary Schur-Weyl duality}\label{SSTensSpace}
Fix $j\in J$. Recall from \S\ref{SSL} the graded $C_1$-supermodule $\LL_j=\k\cdot v_j$ and the graded $\hat C_1$-supermodule $\hat \LL_j$ generated by the vector $\hat v_j$ such that $\ggi^{i} v_j=\de_{i,j}v_j$ and $\ggi^{i} \hat v_j=\de_{i,j}\hat v_j$ 
for all $i\in J$. Recall from (\ref{EBidegvjhatvj}) that  
$$\bideg(v_{j})=\bideg(\hat v_{j})=(1+2j-2\ell,\,j\pmod{2}).$$ 

\section{Colored imaginary tensor space}

\subsection{  The imaginary tensor space $\M_{d,j}$}
For $d\in \N$, we define two versions of {\em ($d$th) imaginary tensor superspace (of color $j$)} which will correspond to each other under the functors $\funF_d$ and $\funG_d$ from (\ref{EFGFunctors}) (at the moment we do not yet know that $\funF_d$ and $\funG_d$ are equivalences---this will be established in \S\ref{SMorita}):
\begin{eqnarray}
\label{EM}
\M_{d,j}&:=&\GGI_{\om_d}^d \LL_j^{\boxtimes d}=C_d\ggi_{\om_d}\otimes_{C_{\om_d}}\LL_j^{\boxtimes d},
\\
\label{EHatM}
\hat \M_{d,j}&:=&\Ind_{\de^d}^{d\de}\,\hat \LL_j^{\boxtimes d}=\hat C_d1_{\de^d}\otimes_{\hat C_{\om_d}}\hat \LL_j^{\boxtimes d}.
\end{eqnarray}
(see Lemma~\ref{LTensImagIsImag}(i) for the last equality). Let 
$$
v_{d,j}:=\ggi_{\om_d}\otimes v_j^{\otimes d}\in \M_{d,j}\quad\text{and}\quad
\hat v_{d,j}:=1_{\de^d}\otimes \hat v_j^{\otimes d}\in \hat \M_{d,j},
$$
so the $C_d$-module $\M_{d,j}$ is generated by 
$
v_{d,j}
$
and 
the $\hat C_d$-module 
$\hat \M_{d,j}$ is generated by 
$
\hat v_{d,j}.
$
Note that 
\begin{equation}\label{EvdWeight}
\ggi^{j^d}v_{d,j}=v_{d,j}\quad \text{and}\quad \ggi^{j^d}\hat v_{d,j}=\hat v_{d,j}.
\end{equation}

\begin{Lemma}\label{LHatMG}
We have:
\begin{enumerate}
\item[{\rm (i)}] $\M_{d,j}\simeq\funF_d(\hat \M_{d,j})$.
\item[{\rm (ii)}] $\hat \M_{d,j}\simeq\funG_d(\M_{d,j})= \hat C_d\ggi_d\otimes_{C_d}\M_{d,j}$ and the map $\M_{d,j}\to  \hat \M_{d,j},\ m\mapsto \ggi_d\otimes m$ is injectve. 
\end{enumerate}  
\end{Lemma}
\begin{proof}
As $\M_{d,j}\simeq\funF_d(\funG_d(\M_{d,j}))$  
by (\ref{EFAfterG}), it suffices to prove (ii). 
We have 
\begin{align*}
\funG_d(\M_{d,j})&=\hat C_d\ggi_{d}\otimes_{C_d}\M_{d,j}
\\
&=\hat C_d\ggi_{d}\otimes_{C_d}C_d\ggi_{\om_d}\otimes_{C_{\om_d}}\LL_j^{\boxtimes d}
\\
&\simeq \hat C_d\ggi_{\om_d}\otimes_{C_{\om_d}}\LL_j^{\boxtimes d}
\\
&\simeq 
\hat C_d1_{\de^d}\otimes_{\hat C_{\om_d}} \hat C_{\om_d}\ggi_{\om_d}\otimes_{C_{\om_d}} \LL_j^{\boxtimes d}
\\
&\simeq
\hat C_d1_{\de^d}\otimes_{\hat C_{\om_d}}\hat \LL_j^{\boxtimes d}
\\&=\hat \M_{d,j}.
\end{align*}
Finally, the map in (ii) is just the isomorphism between $\M_{d,j}$ and $\ggi_d\hat \M_{d,j}$. 
\end{proof}

More generally, let $\la\in\Comp(n,d)$. Recalling  the parabolic subalgebra $C_\la=C_{\la_1}\otimes\dots\otimes C_{\la_n}\subseteq \ggi_\la C_d\ggi_\la$, we have the graded $C_\la$-supermodule 
$$
\M_{\la,j}:=
\M_{\la_1,j}\boxtimes\dots\boxtimes \M_{\la_n,j}
\simeq \GGI_{\om_d}^\la \LL_j^{\boxtimes d} 
$$
with generator $v_{\la,j}:=v_{\la_1,j}\otimes\dots\otimes v_{\la_n,j}$. 
There are obvious analogous modules $\hat \M_{\la,j}$ over $\hat C_\la$. 
Note using (\ref{EBidegvjhatvj}) that 
\begin{equation}\label{EBiDegVj}
\bideg(v_{\la,j})=\bideg(\hat v_{\la,j})=(d(1+2j-2\ell),dj\pmod{2}). 
\end{equation}

\subsection{  The space $\ggi_{\om_d}M_{d,j}$}

Recall from (\ref{ETauSigma}) the elements $\dot w\in \ggi_{\om_d}C_d\ggi_{\om_d}$. 

\begin{Lemma} \label{LResMdOne} 
We have:
\begin{enumerate}
\item[{\rm (i)}] $\ggi_{\om_d}\M_{d,j}=\ggi^{j^d}\M_{d,j}$ has $\k$-basis $\{\dot wv_{d,j}\mid w\in \Si_d\}$. 
\item[{\rm (ii)}] $\ggi_{\om_d}\hat \M_{d,j}=\ggi^{j^d}\hat \M_{d,j}$ has $\k$-basis $\{\dot w\hat v_{d,j}\mid w\in \Si_d\}$. 
\end{enumerate}
\end{Lemma}
\begin{proof}
Using Lemma~\ref{LHatMG}, we get $\ggi_{\om_d}\hat \M_{d,j}\simeq \ggi_{\om_d}\ggi_d\hat \M_{d,j}\simeq \ggi_{\om_d}\M_{d,j}$, so (ii) follows from (i). To prove (i), 
note that 
$$\ggi_{\om_d}\M_{d,j}=\GGR_{\om_d}^d\M_{d,j}=\GGR_{\om_d}^d\GGI_{\om_d}^d\LL_j^{\boxtimes d}.
$$ 
By Theorem~\ref{TGGMackey}, taking into account (\ref{ESiRhoW}), this has a filtration 
$
(\Phi_{\leq w})_{w\in \Si_d},
$
where 
$$
\Phi_{\leq w}=\sum_{x\in \Si_d\,\text{with}\, x\leq w}C_{\om_d}\ggi_{\om_d}\upsigma_{x} \ggi_{\om_d}\otimes \LL_j^{\boxtimes d} 
$$
and 
$$ 
\Phi_{\leq x}(V)/\Phi_{< x}(V)\simeq \,{}^{x}(\LL_j^{\boxtimes d})\simeq \LL_j^{\boxtimes d}.
$$
Note that under the isomorphism above, we have that the coset $\upsigma_{x} v_{d,j}+\Phi_{< x}(V)$ corresponds to $v_j^{\otimes d}$, and 
$\ggi_{\om_d}\upsigma_{x} v_{d,j}=\ggi^{j^d}\upsigma_{x} v_{d,j}=\upsigma_{x} v_{d,j}$, which implies that $\ggi_{\om_d}\M_{d,j}=\ggi^{j^d}\M_{d,j}$ has $\k$-basis $\{\upsigma_wv_{d,j}\mid w\in \Si_d\}$. Now we can apply Lemma~\ref{LSiTau}. 
\end{proof}

\begin{Lemma} \label{LResMdTwo} 
We have:
\begin{enumerate}
\item[{\rm (i)}] For each $w\in \Si_d$, the vector $\dot wv_{d,j}$ spans a graded $C_{\om_d}$-subsupermodule of $\ggi_{\om_d}\M_{d,j}$ isomorphic to $\LL_j^{\boxtimes d}$. 
\item[{\rm (ii)}] $\GGR^d_{\om_d}\M_{d,j}\simeq (\LL_j^{\boxtimes d})^{\oplus d!}$. 
\end{enumerate} 
\end{Lemma}
\begin{proof}
By Lemma~\ref{LResMdOne}, (ii) follows from (i). To prove (i), note that
$$
\ggi_{\om_d}\M_{d,j}\simeq \GGR^d_{\om_d}\M_{d,j}=\GGR^d_{\om_d}\GGI^d_{\om_d}\LL_j^{\boxtimes d}
=\ggi_{\om_d}C_d\ggi_{\om_d}\otimes_{C_1^{\otimes d}}\LL_j^{\boxtimes d}.
$$ 
Now using Theorem~\ref{TMorIso}, the statement (i) becomes a statement about $H_d(A_\ell)$ which is easily checked using the defining relations for $H_d(A_\ell)$.
\end{proof}

\subsection{  Symmetric group action on $M_{d,j}$}
\label{SSSymAct}
By adjointness of $\GGI^d_{\om_d}$ and $\GGR^d_{\om_d}$, 
we have an explicit isomorphism
$$
\End_{C_d}(\M_{d,j})=\Hom_{C_d}(\GGI^d_{\om_d}\LL_j^{\boxtimes d},\GGI^d_{\om_d}\LL_j^{\boxtimes d})
\iso \Hom_{C_{\om_d}}(\LL_j^{\boxtimes d},\GGR^d_{\om_d}\M_{d,j}).
$$
It now follows from Lemma~\ref{LResMdTwo} that for every $w\in \Si_d$, there exists a unique $C_d$-endomorphism $\phi_w$ of $\M_{d,j}$ which is characterized by the property 
\begin{equation}\label{EPhiW}
\phi_w(v_{d,j})=(-1)^{\ttl(w)(j+1)}\dot w v_{d,j}. 
\end{equation}
The sign $(-1)^{\ttl(w)(j+1)}$ is chosen for technical reasons---it makes some of the notions of \S\ref{SSSymDivExt}, such as imaginary symmetric, exterior and divided powers, more natural. In fact, the convenience of the sign is already clear from 
Proposition~\ref{PUpsilon}. 

Recall that we always consider $\k \Si_d$ as concentrated in bidegree $(0,\0)$. 
Applying Lemma~\ref{LResMdTwo} one more time, we get:

\begin{Theorem} \label{TEndMd} 
There is an isomorphism of graded superalgebras 
$$
\k \Si_d \iso \End_{C_d}(\M_{d,j})^\sop
$$
which maps $w\in \Si_d$ to the unique endomorphism $\phi_w$ such that  $$\phi_w(v_{d,j})=(-1)^{\ttl(w)(j+1)}\dot w v_{d,j}.$$ 
\end{Theorem}

By Theorem~\ref{TEndMd}, we can consider $\M_{d,j}$ as a graded $(C_d,\k \Si_d)$-bisupermodule with $mw=\phi_w(m)$ for any $m\in \M_{d,j}$ and $w\in\Si_d$; in particular, 
\begin{equation}\label{EActions_r}
(cv_{d,j}) w=(-1)^{\ttl(w)(j+1)}c\dot w v_{d,j}\qquad(c\in C_d,\ w\in\Si_d).
\end{equation}


\begin{Corollary} \label{CEndTens}
Let $\la\in\Comp(d)$. There  is an isomorphism of graded superalgebras 
$$
\k \Si_\la \iso \End_{C_\la}(\M_{\la,j})^\sop
$$
which maps $w\in \Si_\la$ to the unique endomorphism $\phi_w$ with  
$$\phi_w(v_{\la,j})=(-1)^{\ttl(w)(j+1)}\dot w v_{\la,j}.$$ 
\end{Corollary}

By Corollary~\ref{CEndTens}, we can consider $\M_{\la,j}$ as a graded $(C_\la,\k \Si_\la)$-bisupermodule with 
\begin{equation}\label{ERAction}
mw=\phi_w(m) \qquad (m\in \M_{\la,j},\, w\in \Si_\la).
\end{equation}
In particular, 
\begin{equation}\label{EVWLa}
(cv_{\la,j}) w=(-1)^{\ttl(w)(j+1)}c\dot w v_{\la,j} \qquad(c\in C_d,\ w\in\Si_\la).
\end{equation}

\subsection{  Gelfand-Graev induction and restriction on imaginary tensor spaces}
Let $\la\in\Comp(d)$. By transitivity of the Gelfand-Graev induction, we have $$\M_{d,j}\simeq \GGI_\la^d \M_{\la,j},$$ and so we have a homomorphism 
of graded $C_\la$-supermodules
$$
i_\la:\M_{\la,j}\to \GGR_\la^d \M_{d,j},\ v_{\la,j}\mapsto \ggi_\la v_{d,j}
$$

\begin{Lemma} \label{LResLaTensSpace}
For each $\la\in \Comp(d)$, the homomorphism $i_\la$ injective. Identifying $\M_{\la,j}$ with $i_\la(\M_{\la,j})$, we have 
$$
\GGR_\la^d \M_{d,j}=\bigoplus_{x\in{}^\la\D_d} \M_{\la,j} x.
$$
\end{Lemma}
\begin{proof}
By (\ref{ESiRhoW}), we have $\psi_{\rho_d(x)} 
v_{d,j}=\upsigma_{x}  v_{d,j}$ for all $x\in\Si_d$, so 
by Theorem~\ref{TGGMackey}, 
$$\GGR^d_\la \M_{d,j}=\GGR^d_\la\GGR^d_{\om_d} \LL_j^{\boxtimes d}$$ has a filtration 
 $
 (\Phi_{\leq x})_{x\in {}^\la\D_d},
 $
where 
$$
\Phi_{\leq x}=\sum_{y\in {}^\la\D_d\,\text{with}\, y\leq x}C_{\la}\ggi_{\la}\upsigma_{y} \ggi_{\om_d}\otimes \LL_j^{\boxtimes d} 
$$
and 
$$ 
\Phi_{\leq x}(V)/\Phi_{< x}(V)\simeq\GGI_{\om_d}^{\la}\,{}^{x}(\LL_j^{\boxtimes d})\simeq \M_{\la,j}.
$$
Note that under the isomorphism above, we have that the coset $\upsigma_{x} v_{d,j}+\Phi_{< x}(V)$ corresponds to $v_{\la,j}$. 

When $x=1$, we deduce that $i_\la$ is injective. Moreover, using Lemma~\ref{LSiTau}, we have the filtration $
 (\Phi'_{\leq x})_{x\in {}^\la\D_d}$ of $\GGR^d_\la \M_{d,j}$, 
where 
\begin{align*}
\Phi'_{\leq x}&=\sum_{y\in {}^\la\D_d\,\text{with}\, y\leq x}C_{\la}\ggi_{\la}\dot y \ggi_{\om_d}\otimes \LL_j^{\boxtimes d} 
\\&=\sum_{y\in {}^\la\D_d\,\text{with}\, y\leq x}C_{\la}\ggi_{\la}v_{d,j}y
\\&=\bigoplus_{y\in {}^\la\D_d\,\text{with}\, y\leq x}\M_{\la,j} y 
\end{align*}
which implies the lemma.
\end{proof}

By functoriality of $\GGI_\la^d$, we have that $\GGI_\la^d \M_{\la,j}$ becomes a graded $(C_d,\k \Si_\la)$-bisupermodule. Similarly, $\GGR_\la^d \M_{d,j}$ becomes a graded $(C_\la,\k \Si_d)$-bisupermodule.

\begin{Lemma} \label{LggIndResSymGroup} 
Let $\la\in\Comp(d)$. Then:
\begin{enumerate}
\item[{\rm (i)}] $\GGI_\la^d \M_{\la,j}\simeq \M_{d,j}$ as graded $(C_d,\k \Si_\la)$-bisupermodules, where the right action of $\Si_\la$ on $\M_{d,j}$ is via restriction from $\Si_d$.
\item[{\rm (ii)}] $\GGR_\la^d \M_{d,j}\simeq \M_{\la,j}\otimes_{\k \Si_\la}\k \Si_d$ as graded $(C_\la,\k \Si_d)$-bisupermodules. 
\end{enumerate}
\end{Lemma}
\begin{proof}
(i) By transitivity of Gelfand-Graev induction, $\GGI_\la^d \M_{\la,j}\simeq \M_{d,j}$ as graded $C_d$-supermodules. By (\ref{ERAction}), the isomorphism is compatible with the $\Si_\la$-action. 

(ii) follows easily from Lemma~\ref{LResLaTensSpace}.
\end{proof}

\begin{Corollary} \label{CFunPerm}
Let $\la\in\Comp(d)$. 
The following pairs of functors are isomorphic:
\begin{enumerate}
\item[{\rm (i)}] $\GGI_\la^d\circ (\M_{\la,j}\otimes_{\k \Si_\la}\, -)$ and $(\M_{d,j}\otimes_{\k \Si_d}\, -)\circ \Ind_{\Si_\la}^{\Si_d}\,:\, \mod{\k \Si_\la}\to \mod{C_{d}}$. 
\item[{\rm (ii)}] $\GGR_\la^d\circ (\M_{d,j}\otimes_{\k \Si_d}\, -)$ and $(\M_{\la,j}\otimes_{\k \Si_\la}\, -)\circ \Res_{\Si_\la}^{\Si_d}\,:\, \mod{\k  \Si_d}\to \mod{C_{\la}}$. 
\end{enumerate}
\end{Corollary}
\begin{proof}
(i) Take $V\in \mod{\k \Si_\la}$. Using Lemma~\ref{LggIndResSymGroup}(i), we have natural isomorphisms
\begin{align*}
\M_{d,j}\otimes_{\k \Si_d} \Ind_{\Si_\la}^{\Si_d}V&=\M_{d,j}\otimes_{\k \Si_d} \k \Si_d\otimes_{\k \Si_\la}V
\\&\simeq \M_{d,j}\otimes_{\k \Si_\la}V
\\&\simeq (\GGI^n_\la \M_{\la,j})\otimes_{\k \Si_\la}V
\\
&
\simeq \GGI^n_\la (\M_{\la,j}\otimes_{\k \Si_\la}V),
\end{align*}
as required.

(ii) Using Lemma~\ref{LggIndResSymGroup}(ii), for an $V\in\mod{\k \Si_d}$, we have natural isomorphisms
\begin{align*}
\GGR_\la^n (\M_{d,j}\otimes_{\k \Si_d} V)&\simeq(\GGR_\la^n \M_{d,j})\otimes_{\k \Si_d} V
\\&\simeq (\M_{\la,j}\otimes_{\k \Si_\la}\k \Si_d)\otimes_{\k \Si_d} V
\\
&\simeq \M_{\la,j}\otimes_{\k \Si_\la}  \Res_{\Si_\la}^{\Si_d} V,
\end{align*}
as required.
\end{proof}

\subsection{  Weight basis of  $\M_{d,j}$}
Recall the elements $\lgath_{d,j}\in \ggi^{d,j}C_d\ggi^{j^d}$ and 
$\lgath_{\la,j}=\lgath_{\la_1,j}\otimes \dots\otimes \lgath_{\la_n,j}\in \ggi^{\la,j}C_\la\ggi^{j^d}$ for $\la\in\Comp(n,d)$ from \S\ref{SSUpsilon}.

\begin{Lemma} \label{LTrivU}
For any $w\in\Si_d$, we have $\lgath_{d,j}v_{d,j}w=\lgath_{d,j}v_{d,j}$. 
More generally, for $\la\in\Comp(d)$ and $w\in\Si_\la$, we have $\lgath_{\la,j}v_{d,j}w=\lgath_{\la,j}v_{d,j}$. 
\end{Lemma}
\begin{proof}
For $w\in\Si_\la$, we have 
$$
\lgath_{\la,j}v_{d,j}w=(-1)^{\ttl(w)(j+1)}\lgath_{\la,j}\dot wv_{d,j}
=\lgath_{\la,j}v_{d,j},
$$
thanks to Proposition~\ref{PUpsilon}(ii). 
\end{proof}

\begin{Lemma} \label{LMWtSpacesNew} 
Let $\la\in\Comp(d)$. Then $\ggi^{\la,j}\M_{d,j}$ has $\k$-basis $\{\lgath_{\la,j}  v_{d,j}x\mid x\in{}^\la\D_d\}$. 
\end{Lemma}
\begin{proof}
Using Lemma~\ref{LggIndResSymGroup}(ii), we get
$$\ggi^{\la,j}\M_{d,j}=\ggi^{\la,j}f_\la\M_{d,j}=\ggi^{\la,j}\GGR^d_\la\M_{d,j}
\simeq\ggi^{\la,j}\M_{\la,j}\otimes_{\k\Si_\la}\k\Si_d,
$$ 
with $\ggi^{\la,j}v_{d,j}\in \ggi^{\la,j}\M_{d,j}$ corresponding to 
$\ggi^{\la,j}v_{\la,j}\otimes1$. 
So, it suffices to note that $\ggi^{\la,j}\M_{\la,j}$ has basis $\lgath_{\la,j}  v_{\la,j}$. If $\la=(\la_1,\dots,\la_n)$ then $\ggi^{\la,j}\M_{\la,j}\simeq \ggi^{\la_1,j}\M_{\la_1,j}\otimes\dots\otimes \ggi^{\la_n,j}\M_{\la_n,j}$, so the required claim reduces to the case where $\la$ has only one part, i.e. to the claim that 
$\ggi^{d,j}\M_{d,j}$ has basis $\lgath_{d,j}  v_{d,j}$. 

By Proposition~\ref{LDifficult}, we have 
\begin{align*}
\ggi^{d,j}\M_{d,j}&=\ggi^{d,j}C_d\ggi_{\om_d} \otimes_{C_{\om_d}}L_j^{\boxtimes d}
\\&=\lgath_{d,j} C_d\ggi_{\om_d}\otimes_{C_{\om_d}}L_j^{\boxtimes d}
\\&=\lgath_{d,j}\ggi_{\om_d} C_d\ggi_{\om_d}\otimes_{C_{\om_d}}L_j^{\boxtimes d}
\\&=\lgath_{d,j}\ggi_{\om_d}M_{d,j},
\end{align*}
where we have used $\lgath_{d,j}=\lgath_{d,j}\ggi_{\om_d} $ for the second equality. By Lemma~\ref{LResMdOne}(i), $\ggi_{\om_d}M_{d,j}$ 
has 
$\k$-basis $\{\dot wv_{d,j}\mid w\in\Si_d\}=\{\pm v_{j, d}w\mid w\in\Si_d\}$. Since $\lgath_{d,j}v_{j, d}w= \lgath_{d,j}v_{j, d}$ by Lemma~\ref{LTrivU}, we deduce that $\lgath_{d,j}v_{j, d}$ spans $\ggi^{d,j}\M_{d,j}$. So it suffices to prove that 
$\ggi^{d,j}\M_{d,j}=1_{\ggw^{d,j}}\M_{d,j}$ is non-zero and free as a $\k$-module. 
By Lemma~\ref{LHatMG}, we have 
$$
1_{\ggw^{d,j}}\M_{d,j}\simeq 1_{\ggw^{d,j}}\ggi_d\Ind_{\de^d}^{d\de}\, \hat \LL_j^{\boxtimes d}
= 1_{\ggw^{d,j}}\,\Ind_{\de^d}^{d\de}\,\hat \LL_j^{\boxtimes d}.
$$
By Lemma~\ref{LIndBasis}, the $\k$-module $1_{\hat \ggw^{d,j}}\,\Ind_{\de^d}^{d\de}\,\hat \LL_j^{\boxtimes d}$ is free and non-zero, so $1_{\ggw^{d,j}}\M_{d,j}$ is also free and non-zero by Lemma~\ref{LFact}. 
\end{proof}

\begin{Lemma} \label{LMWtSpaces} 
We have 
$\M_{d,j}=\bigoplus_{\la\in\EC(d)}\ggi^{\la,j} \M_{d,j}$. 
\end{Lemma}
\begin{proof}
We just need to prove that the orthogonal idempotents appearing in the sum (\ref{EGad}) act on $\M_{d,j}$ as zero unless they are of the form $\ggi^{\la,j}$. Using Lemma~\ref{LResLaTensSpace} and induction on $d$, this reduces to the case where $\la$ has just one non-zero part, i.e. we need to prove that $\ggi^{d,i}\M_{d,j}=0$ for $i\neq j$. 

By definition, $\lgath_{d,i}=\lgath_{d,i}\ggi^{i^d}=\lgath_{d,i}\ggi^{i^d}f_{\om_d}$. So, using Proposition~\ref{LDifficult}, we get 
\begin{align*}
\ggi^{d,i}\M_{d,j}&=\ggi^{d,i}C_d\ggi_{\om_d} \otimes_{C_{\om_d}}L_j^{\boxtimes d}
\\&=\lgath_{d,i} C_d\ggi_{\om_d}\otimes_{C_{\om_d}}L_j^{\boxtimes d}
\\&=\lgath_{d,i}\ggi^{i^d}f_{\om_d} C_d\ggi_{\om_d}\otimes_{C_{\om_d}}L_j^{\boxtimes d}.
\end{align*}
But for $i\neq j$, we have 
$$\ggi^{i^d}f_{\om_d} C_d\ggi_{\om_d}\otimes_{C_{\om_d}}L_j^{\boxtimes d}=
\ggi^{i^d}f_{\om_d}M_{d,j}=
0,$$ 
thanks to Lemma~\ref{LResMdTwo}(ii). 
\end{proof}


\begin{Corollary} \label{CBasisIndGG}
We have that $\{\lgath_{\la,j} v_{d,j}x\mid \la\in\EC(d),\,x\in{}^\la\D_d\}$ is a $\k$-basis of $\M_{d,j}$.
\end{Corollary}
\begin{proof}
This follows from Lemmas~\ref{LMWtSpaces} and \ref{LMWtSpacesNew}. 
\end{proof}

Let $\la\in\Comp(d)$. By (\ref{EUDeg}) and (\ref{EBiDegVj}), we can write
\begin{equation}\label{EA}
\begin{split}
\bideg(\lgath_{d,j} v_{d,j})&=(a_{d,j},\,dj\pmod{2}), 
\\
\bideg(\lgath_{\la,j} v_{\la,j})&=(a_{\la,j},\,dj\pmod{2})
\end{split}
\end{equation}
for 
\begin{equation}\label{EAEquals}
\begin{split}
a_{d,j}
&:=d(1+2\ell)+d^2(2j-4\ell),
\\ 
a_{\la,j}
&\textstyle:=d(1+2\ell)+(2j-4\ell)\sum_{r=1}^n\la_r^2.
\end{split}
\end{equation}

\begin{Lemma} \label{LBasisIndGGNew} 
For each $\la\in \Comp(d)$, 
we have a bidegree $(0,\0)$ isomorphism of right $\k\Si_d$-modules
$$
\funQ^{-a_{\la,j}}\Uppi ^{dj}\ggi^{\la,j} \M_{d,j}\iso \k_{\Si_\la}\otimes_{\k\Si_\la}\k\Si_d
$$
which maps $\lgath_{\la,j} v_{d,j}$ to  $1\otimes 1_{\Si_d}$. 
\end{Lemma}
\begin{proof}
We have 
$$\bideg(\lgath_{\la,j} v_{d,j})=\bideg(\lgath_{\la,j} v_{\la,j})=(a_{\la,j},dj\pmod{2}).$$ Now the result follows from 
Lemmas~\ref{LMWtSpacesNew} and \ref{LTrivU}. 
\end{proof}

\section{Morita equivalence of $\hat C_d$ and $C_d$}
\label{SMorita}

We now pause to prove that the functors $\funF_d$ and $\funG_d$ from (\ref{EFGFunctors}) are quasi-inverse equivalences, thus resolving Inductive Assumption~\ref{IA}. 

\subsection{Counting the irreducibles}
Let $\k=\F$. 
By Lemma~\ref{LImCuspIrrAmount}, we have $|\Irr(\hat C_d)|=|\Par^J(d)|$. 
Since $C_d=\ggi_d\hat C_d\ggi_d$ is an idempotent truncation of $\hat C_d$, we have 
\begin{equation}\label{IttCCount}
\Irr(C_d)\leq |\Par^J(d)|.
\end{equation}
Moreover, 
in view of Lemmas~\ref{cor:idmpt_Mor} and \ref{LMorExtScal}, 
to prove that the functors $\funF_d$ and $\funG_d$ are quasi-inverse equivalences, 
it suffices to prove that 
$
|\Irr(C_d)|= |\Par^J(d)|.
$

Recall from (\ref{EGad}) that $\ggi_d$, which is the identity in $C_d$, decomposes as the sum of orthogonal idempotents $\ggi_d=\sum_{(\mu,\bj)\in\EC^\col(d)}\ggi^{\mu,\bj}\in \hat C_d$. 
In order to bound $|\Irr(C_d)|$ below, we introduce the {\em Gelfand-Graev formal character}\, of a finite dimensional graded $C_d$-supermodule $V$: 
$$
\operatorname{GGCh}(V):=\sum_{(\mu,\bj)\in\EC^\col(d)}(\dim \ggi^{\mu,\bj}V)\cdot (\mu,\bj) \in \Z\cdot\EC^\col(d).
$$
where $\Z\cdot\EC^\col(d)$ is the free $\Z$-module  on the  
basis $\EC^\col(d)$.

Recalling the notation (\ref{Ettx}) and (\ref{EGGIOneColorNot}) and the notation $\la'$ for transposed partition, a  key observation is:

\begin{Lemma} \label{LGGChMy}
Let\, $\k=\F$, $j\in J$ and\, $\la\in\Par(d)$. 
Then 
$$
\operatorname{GGCh}(\M_{d,j}\tty_\la)=(\la',j)+\sum_{\mu\lhd\la'}c_{\mu,j}\cdot (\mu,j)
$$
for some $c_{\mu,j}\in\N$. 
\end{Lemma}
\begin{proof}
Let $\mu\in \EC(d)$. 
By Lemma~\ref{LBasisIndGGNew}, we have an isomorphism of vector spaces 
$$\ggi^{\mu,j} \M_{d,j}\tty_\la\cong \k_{\Si_\mu}\otimes_{\k\Si_\mu}\k\Si_d\tty_\la.$$ 
But by \cite[Lemma 4.6]{JamesBook}, we have $\k_{\Si_\mu}\otimes_{\k\Si_\mu}\k\Si_d\tty_\la=0$ unless $\mu\unlhd\la'$, and $\k_{\Si_{\la'}}\otimes_{\k\Si_{\la'}}\k\Si_d\tty_{\la}$ is $1$-dimensional.
\end{proof}

\begin{Corollary} \label{CCFMjAmountNew}
Let $\k=\F$ and $j\in J$. For every $\la\in\Par(d)$ there exists a unique up to isomorphism irreducible graded $C_d$-supermodule $L_j(\la)$ such that $L_j(\la)$ is a composition factor of $\M_{d,j}\tty_{\la'}$ and 
\begin{equation}\label{EFChL}
\operatorname{GGCh}(L_{j}(\la))=(\la,j)+\sum_{\mu\lhd\la}b_{\mu,j}\cdot (\mu,j)
\end{equation}
for some $b_{\mu,j}\in\N$. 
In particular, the number of non-isomorphic composition factors of $M_{d,j}$ is at least $|\Par(d)|$. 
\end{Corollary}
\begin{proof}
This follows from Lemma~\ref{LGGChMy} using the obvious fact that 
the Gelfand-Graev formal character of a (finite dimensional) graded supermodule is the sum of the Gelfand-Graev formal characters of its composition factors. 
\end{proof}

\begin{Corollary} \label{CCompFactorsMy} 
Let $\k=\F$, $j\in J$, $d\in\N$ and $\la\in\Par(d)$. Then $L_j(\la')$ appears as a composition factor of $M_{d,j}\tty_\la$ with multiplicity $1$, and $L_j(\mu)$ is a composition factor of $M_{d,j}\tty_\la$ only if $\mu\unlhd \la'$. 
\end{Corollary}

For $\bla=(\la^{(0)},\dots,\la^{(\ell-1)})\in\Par^J(d)$ with $d_j:=|\la^{(j)}|$ for all $j\in J$, we define
\begin{equation}\label{ELBla}
L(\bla):=\GGI^{d}_{d_0,\dots,d_{\ell-1}}\big(L_0(\la^{(0)})\boxtimes\dots\boxtimes L_{\ell-1}(\la^{(\ell-1)})\big).
\end{equation}

We can now canonically parametrize all irreducible graded $C_d$-supermodules as well as the composition factors of each $M_{d,j}$. 

\begin{Theorem} \label{LAmountAmount} 
Let $\k=\F$. 
 Then: 
\begin{enumerate}
\item[{\rm (i)}] 
$\Irr(C_d)=\{L(\bla)\mid\bla\in\Par^J(d)\}$; 

\item[{\rm (ii)}] 
$\{L_{j}(\la)\mid \la\in\Par(d)\}$ is a complete and non-redundant set of composition factors of $\M_{d,j}$ up to isomorphism.
\end{enumerate}
\end{Theorem}
\begin{proof}
We prove the theorem by induction on $d$. 
Let $d\in\N$ and suppose the theorem is true for all $d'<d$. 

\vspace{2mm}
\noindent
{\sf Claim 1.} Let $\bla,\bmu\in\Par^J(d)$ and set $d_j:=|\la^{(j)}|$ for all $j\in J$. If the irreducible graded $C_{d_0,\dots,d_{\ell-1}}$-supermodule $L_0(\la^{(0)})\boxtimes \dots\boxtimes L_{\ell-1}(\la^{(\ell-1)})$ appears as a composition factor of $\GGR_{d_0,\dots,d_{\ell-1}}^dL(\bmu)$ then $\bla=\bmu$ and the multiplicity of this composition factor is~$1$. 

\vspace{1 mm}
To prove Claim 1, we set $m_j:=|\mu^{(j)}|$ for all $j\in J$. 
We may assume that at least two of the $m_j$ are non-zero or at least two of the $d_j$ are non-zero, since otherwise the claim boils down to the proof that $L_j(\la)\cong L_i(\mu)$ only if $j=i$ and $\la=\mu$, which is clear by considering the Gelfand-Graev formal characters of $L_j(\la)$ and $L_i(\mu)$, see (\ref{EFChL}).  

By Theorem~\ref{TGGMackey}, 
$$
\GGR_{d_0,\dots,d_{\ell-1}}^dL(\bmu)=\GGR_{d_0,\dots,d_{\ell-1}}^d\GGI_{m_0,\dots,m_{\ell-1}}^d\big(L_0(\mu^{(0)})\boxtimes\dots\boxtimes L_{\ell-1}(\mu^{(\ell-1)})\big)
$$
has filtration with subfactors of the form
\begin{equation}\label{E5.4}
\GGI^{d_0\,;\,\dots\,;\,d_{\ell-1}}_{n_{0,0},\dots,n_{\ell-1,0}\,;\,\dots\,;\,n_{0,\ell-1},\dots,n_{\ell-1,\ell-1}}\,V
\end{equation}
where: 

(a) $\sum_{i\in J}n_{i,j}=d_j$ for all $j\in J$, 

(b) $\sum_{j\in J}n_{i,j}=m_i$ for all $i\in J$, and 

(c) $V$ is obtained by an appropriate twisting of the module 
$$
\big(\GGR^{m_0}_{n_{0,0},\dots,n_{0,\ell-1}}L_0(\mu^{(0)})\big)\boxtimes\dots\boxtimes
\big(\GGR^{m_{\ell-1}}_{n_{\ell-1,0},\dots,n_{\ell-1,\ell-1}}L_{\ell-1}(\mu^{(\ell-1)})\big).
$$
Note by (a),(b) and the previous paragraph that all $n_{i,j}<d$. 

Assume that the module in (\ref{E5.4}) is non-zero. By Lemma~\ref{LggIndResSymGroup}(ii), if $L_1\boxtimes\dots\boxtimes L_{\ell-1}$ is a composition factor of $\GGR^{m_i}_{n_{i,0},\dots,n_{i,\ell-1}}L_i(\mu^{(i)})$ then each $L_j$ is a composition factors of $M_{n_{i,j},i}$, so by the inductive assumption $L_j$ is of the form $L_i(\mu^{(i,j)})$ for some $\mu^{(i,j)}\in\Par(n_{i,j})$. Hence $
\GGR_{d_0,\dots,d_{\ell-1}}^dL(\bmu)$ has a filtration with factors of the form 
$
X_0\boxtimes\dots\boxtimes X_{\ell-1}
$
where for each $j\in J$ we have set 
$$
X_j:=\GGI^{d_j}_{n_{0,j},\dots,n_{\ell-1,j}}\big(
L_0(\mu^{(0,j)})\boxtimes\dots\boxtimes L_{\ell-1}(\mu^{(\ell-1,j)})
\big).
$$
By the inductive hypothesis, each $X_j$ is irreducible. By considering Gelfand-Graev formal characters we deduce that $X_j\cong L_j(\la^{(j)})$ if and only if $\mu^{(j,j)}=\la^{(j)}$ and $\mu^{(i,i)}=\varnothing$ for all $i\neq j$. Thus, $n_{j,j}=d_j$ and $n_{i,j}=0$ for all $i\neq j$. We conclude that $m_j=d_j$ and $\mu^{(j)}=\la^{(j)}$ for all $j$. This completes the proof of Claim 1.

\vspace{2mm}
\noindent
{\sf Claim 2.} Let $\bla\in\Par^J(d)$. The $C_d$-module $L(\bla)$ has simple head; denote it by $L^\bla$. Then the multiplicity of $L^\bla$ in $L(\bla)$ is $1$. 

\vspace{1 mm}
To prove Claim 2, suppose that an irreducible $C_d$-module $L$ appears in the head of $L(\bla)$. Then, setting $d_j:=|\la^{(j)}|$ for all $j\in J$, 
 by adjointness of Gelfand-Graev induction and restriction, we have that the space
\begin{align*}
\Hom_{C_d}(L(\bla),L)
=\Hom_{C_{d_0,\dots,d_{\ell-1}}}(L_0(\la^{(0)})\boxtimes\dots\boxtimes L_{\ell-1}(\la^{(\ell-1)}),\GGR^{d}_{d_0,\dots,d_{\ell-1}}\,L)
\end{align*}
is non-zero. 
So each irreducible $L$ in the head of $L(\bla)$ contributes a composition factor $L_0(\la^{(0)})\boxtimes\dots\boxtimes L_{\ell-1}(\la^{(\ell-1)})$ into $\GGR^{d}_{d_0,\dots,d_{\ell-1}} L(\bla)$. But by Claim 1, the multiplicity of 
$L_0(\la^{(0)})\boxtimes\dots\boxtimes L_{\ell-1}(\la^{(\ell-1)})$ in $\GGR^{d}_{d_0,\dots,d_{\ell-1}} L(\bla)$ is $1$, so Claim 2 follows. 

\vspace{2mm}
\noindent
{\sf Claim 3.} If $\bla,\bmu\in\Par^J(d)$ are distinct then  $L^\bla\not\cong L^\bmu$.  

\vspace{1 mm}
Indeed, if $L^\bla\cong L^\bmu$ then $L^\bmu$ is a quotient of $L(\bla)$. By the adjointness of Gelfand-Graev induction and restriction, we have that $L_0(\la^{(0)})\boxtimes\dots\boxtimes L_{\ell-1}(\la^{(\ell-1)})$ is a submodule of $\GGR^{d}_{d_0,\dots,d_{\ell-1}} L^\bmu$, hence $L_0(\la^{(0)})\boxtimes\dots\boxtimes L_{\ell-1}(\la^{(\ell-1)})$ is a composition factor of $\GGR^{d}_{d_0,\dots,d_{\ell-1}} L(\bmu)$. Now $\bla=\bmu$ by Claim 1. 

\vspace{2mm}
Now we can finish the proof of the theorem. By (\ref{IttCCount}) and Claim 3, we have $\Irr(C_d)=\{L^\bla\mid\bla\in\Par^J(d)\}$. So to finish the proof of (i), it suffices to show that $L(\bmu)=L^\bmu$ for all $\bmu\in\Par^J(d)$. If some $L(\bmu)$ is not irreducible, then by Claim 2, it has a composition factor $L^\bla$ in its socle with $\bla\neq\bmu$. This yields a non-zero homomorphism $L(\bla)\to L(\bmu)$, so by adjointness of Gelfand-Graev induction and restriction, we have that $L_0(\la^{(0)})\boxtimes\dots\boxtimes L_{\ell-1}(\la^{(\ell-1)})$ is a submodule of $\GGR^{d}_{d_0,\dots,d_{\ell-1}} L(\bmu)$. By Claim 1, we have $\bla=\bmu$ which is a contradiction. We have proved (i)

(ii) follows from (i) since $L(\bla)$ has  
a Gelfand-Graev formal character compatible with being a composition factor of $M_{d,j}$ only if $\la^{(i)}=\varnothing$ for all $i\neq j$. 
\end{proof}

\begin{Corollary} \label{CEquivFG}
We have: 
\begin{enumerate}
\item[{\rm (i)}] If $\k=\F$ then  $|\Irr(C_d)|= |\Par^J(d)|$.
\item[{\rm (ii)}] The functors $\funF_d$ and $\funG_d$ from (\ref{EFGFunctors}) are quasi-inverse equivalences. Hence the functors $\funF_\la$ and $\funG_\la$ from (\ref{EFGFunctors}) are quasi-inverse equivalences for any $\la\in\Comp(d)$. 
\item[{\rm (iii)}] If $\k=\F$ then the functors $\funF_d$ and $\funG_d$ from (\ref{EFGFunctors}) restrict to quasi-inverse equivalences of the categories of {\em finite dimensional} graded supermodules over $\hat C_d$ and~$C_d$. 
\end{enumerate} 
\end{Corollary}
\begin{proof}
(i) follows from Theorem~\ref{LAmountAmount}(i). 

(ii) follows from (i) in view of Lemmas~\ref{cor:idmpt_Mor} and \ref{LMorExtScal}. 

(iii) We just need to make sure that the functors preserve the finite dimensionality. This follows from the fact that the irreducible modules over $\hat C_d$ are finite dimensional, see Lemma~\ref{LTypeM}. 
\end{proof}

\begin{Corollary} \label{CFGProj}
Let $\la\in\Comp(d)$. Then 
the left graded $C_\la$-supermodule
$\ggi_\la C_d$ is finitely generated projective and 
the right graded $C_\la$-supermodule
$C_d\ggi_\la$ is finitely generated projective. 
\end{Corollary}
\begin{proof}
Note that $\ggi_\la C_d=\ggi_\la\hat C_d\ggi_d$ is a direct summand of the left graded $C_\la$-supermodule $\ggi_\la\hat C_d$, so it suffices to prove that $\ggi_\la\hat C_d$ is finitely generated projective. But $\ggi_\la\hat C_d=\ggi_\la1_{\la\de}\hat C_d\simeq \funF_\la(1_{\la\de}\hat C_d)$, and since $\funF_\la$ is an equivalence by Corollary~\ref{CEquivFG}(ii), it suffices to prove  that the left graded $\hat C_{\la\de}$-supermodule $1_{\la\de}\hat C_d$ is finitely generated projective. But this supermodule is in fact finitely generated free by Lemma~\ref{L030216}(iii). This completes the proof for left modules. For right modules the argument is similar. 
\end{proof}

\begin{Corollary} \label{CGGRProjtoProj} 
Suppose that $P$ is a finitely generated projective graded $C_d$-supermodule. Then for any $\la\in\Par(d)$, we have that $\GGR_\la^d P$ is a finitely generated projective graded $C_\la$-supermodule. 
\end{Corollary}
\begin{proof}
By Corollary~\ref{CFGProj}, the functor $\GGR^d_\la$ sends the regular graded $C_d$-supermodule $C_d$ to a finitely generated projective graded $C_\la$-supermodule $f_\la C_d$. This implies the corollary since  $\GGR^d_\la$ commutes with taking direct sums and direct summands, as well as with grading and parity shifts. 
\end{proof}

\begin{Corollary} \label{CTensorcd}
Let $\k=\F$, $j\in J$ and $c,d\in\N$. Then $L_j((c+d))$ is a composition factor of $\GGI_{c,d}^{c+d}\  L_j((c))\boxtimes L_j((d))$. 
\end{Corollary}
\begin{proof}
Since $(c+d)$ is the most dominant partition of $c+d$, it suffices to prove that $(c+d,j)$ appears in the Gelfand-Graev formal character of $\GGI_{c,d}^{c+d}\  L_j((c))\boxtimes L_j((d))$. 
But by (\ref{EFGFunctorsLaOuter}) and (\ref{EGGIInd}), we have 
$$\GGI_{c,d}^{c+d}(L_j(c)\boxtimes L_j(d))\simeq \funF_{c+d}\,\Ind_{c\de,d\de}^{(c+d)\de}\,\funG_c\big(L_j((c))\big)\boxtimes \funG_d\big(L_j((d))\big).
$$
Since $\funF_c\funG_c\big(L_j((c))\big)\simeq L_j((c))$, we have that the inequality $\ggi^{c,j}L_j((c))\neq 0$ implies that 
$\ggi^{c,j}\funG_c\big(L_j((c))\big)\neq 0$. So $1_{\hat\ggw^{c,j}}\funG_c\big(L_j((c))\big)\neq 0$. Similarly $1_{\hat\ggw^{d,j}}\funG_d\big(L_j((d))\big)\neq 0$. Since the word $\hat g^{c+d,j}$ is a shuffle of the words $\hat\ggw^{c,j}$ and $\hat\ggw^{d,j}$, it follows from Corollary~\ref{CShuffle} that $$1_{\hat g^{c+d,j}}\Ind_{c\de,d\de}^{(c+d)\de}\,\funG_c\big(L_j((c))\big)\boxtimes \funG_d\big(L_j((d))\big)\neq 0.$$
Byg Lemma~\ref{LFact} this implies 
$$\ggi^{c+d,j}\Ind_{c\de,d\de}^{(c+d)\de}\,\funG_c\big(L_j((c))\big)\boxtimes \funG_d\big(L_j((d))\big)\neq 0.$$
Therefore
$$\ggi^{c+d,j}\funF_{c+d}\Ind_{c\de,d\de}^{(c+d)\de}\,\funG_c\big(L_j((c))\big)\boxtimes \funG_d\big(L_j((d))\big)\neq 0,$$
as required.
\end{proof}

\subsection{Gelfand-Graev duality}
Throughout the subsection we assume that $\k=\F$ and work with finite dimensional (graded super)modules. 
Recall the duality $\circledast$  from \S\S\ref{SSDefQHSA},\ref{SSIndRes}. We will use $\circledast$ to define a duality on finite dimensional graded $C_d$-supermodules which we denote by $\sharp$. Namely, given $V\in\mod{C_d}$, we define
\begin{equation}\label{ECircledAstGG}
V^\sharp :=\funF_d\big((\funG_d V)^\circledast\big).
\end{equation}
More generally, 
for $\la\in\Comp(d)$ and $V\in\mod{C_\la}$, we define
\begin{equation}\label{ECircledAstGGPar}
V^\sharp :=\funF_\la\big((\funG_\la V)^\circledast\big).
\end{equation}
Note, using (\ref{EFGFunctorsLaOuter}) and (\ref{EBoxCircledAst}), that for $\la\in\Comp(n,d)$ 
and $V_1\in\mod{C_{\la_1}}, \dots , V_n\in\mod{C_{\la_n}}$ 
we have 
\begin{equation}\label{ESharpBoxTimes}
(V_1\boxtimes\dots\boxtimes V_n)^\sharp\simeq 
V_1^\sharp\boxtimes\dots\boxtimes V_n^\sharp.
\end{equation}

\begin{Lemma} \label{LSharpHom} 
We have that $\sharp$ defines a contravariant functor on the 
finite dimensional graded $C_d$-supermodules. Moreover, $\sharp\circ \sharp$ is equivalent to the identity functor and we have a functorial isomorphism 
$$\Hom_{C_d}(V,W)\cong \Hom_{C_d}(W^\sharp,V^\sharp)
$$
\end{Lemma}
\begin{proof}
This follows from the corresponding properties of $\circledast$ and the fact that the functors $\funF_d$ and $\funG_d$ are quasi-inverse equivalences, see Corollary~\ref{CEquivFG}(iii). 
\end{proof}

\begin{Lemma} \label{LSharpResComm}
For $\la\in\Comp(d)$ and $V\in\mod{C_d}$, we have a functorial isomorphism $\GGR^d_\la(V^\sharp)\simeq (\GGR^d_\la\,V)^\sharp$. 
\end{Lemma}
\begin{proof}
We have 
\begin{align*}
\GGR^d_\la(V^\sharp)\stackrel{(\ref{ECircledAstGG})}{=}
\ \ \ \ \ \ 
&\GGR^d_\la\,\funF_d((\funG_dV)^\circledast)
\\
\stackrel{(\ref{EGGIInd})}{\simeq} 
\ \ \ \ \ 
&\funF_\la\,\Res^{d\de}_{\la\de}\,\funG_d\funF_d((\funG_dV)^\circledast)
\\
\stackrel{\text{Corollary~\ref{CEquivFG}(ii)}}{\simeq} 
&\funF_\la\,\Res^{d\de}_{\la\de}\,(\funG_dV)^\circledast
\\
\stackrel{\text{Lemma~\ref{LDualInd}(i)}}{\simeq} 
\ 
&\funF_\la\Big(\big(\Res^{d\de}_{\la\de}\,(\funG_dV)\big)^\circledast\Big)
\\
\stackrel{\text{Corollary~\ref{CEquivFG}(ii)}}{\simeq} &\funF_\la\Big(\big(\funG_\la\funF_\la\,\Res^{d\de}_{\la\de}\,(\funG_dV)\big)^\circledast\Big)
\\
\stackrel{(\ref{EGGIInd})}{\simeq} 
\ \ \ \ \ \,
& \funF_\la(\funG_\la \,\GGR^{d}_{\la}\,V)^\circledast
\\
\stackrel{(\ref{ECircledAstGGPar})}{=} 
\ \ \ \ \ \ 
&(\GGR^{d}_{\la}\,V)^\sharp
\end{align*}
as required.
\end{proof}

\begin{Lemma} \label{LSharpIndComm}
For $\la\in\Comp(n,d)$, $V_1\in\mod{C_{\la_1}},\dots,V_n\in\mod{C_{\la_n}}$, we have functorial isomorphisms  
$$\GGI^d_\la\big(V_1^\sharp\boxtimes\dots\boxtimes V_n^\sharp\big)\simeq \GGI^d_\la\big((V_1\boxtimes\dots\boxtimes V_n)^\sharp\big)\simeq \big(\GGI^d_{\la^\op}(V_n\boxtimes\dots\boxtimes V_1)\big)^\sharp.
$$ 
\end{Lemma}
\begin{proof}
The first isomorphism comes from (\ref{ESharpBoxTimes}). 
Now, 
\begin{align*}
\GGI^d_\la\big((V_1\boxtimes\dots\boxtimes V_n)^\sharp\big)\stackrel{(\ref{EGGIInd}), (\ref{ECircledAstGGPar})}{=}\ \ \  &\funF_d\,\Ind^{d\de}_{\la\de}\,\funG_\la\funF_\la\Big(\big(\funG_\la(V_1\boxtimes\dots\boxtimes V_n)\big)^\circledast\Big)
\\
\stackrel{\text{Corollary~\ref{CEquivFG}(ii)}}{\simeq}
\ 
&\funF_d\,\Ind^{d\de}_{\la\de}\,\Big(\big(\funG_\la(V_1\boxtimes\dots\boxtimes V_n)\big)^\circledast\Big)
\\
\stackrel{(\ref{EFGFunctorsLaOuter})}{\simeq} \ \ \ \,\,\ \ \,
&\funF_d\,\Ind^{d\de}_{\la\de}\,\Big(\big(\funG_{\la_1}(V_1)\boxtimes\dots\boxtimes \funG_{\la_n}(V_n)\big)^\circledast\Big)
\\
\stackrel{(\ref{EBoxCircledAst})}{\simeq}  \ \ \ \ \ \,\, \,& \funF_d\,\Ind^{d\de}_{\la\de}\,\big(\funG_{\la_1}(V_1)^\circledast\boxtimes\dots\boxtimes \funG_{\la_n}(V_n)^\circledast\big)
\\
\stackrel{\text{Lemma~\ref{LDualInd}(ii)}}{\simeq}\ \, &\funF_d\,\Big(\big(\Ind^{d\de}_{\la^\op\de}\,\funG_{\la_n}(V_n)\boxtimes\dots\boxtimes \funG_{\la_1}(V_1)\big)^\circledast\Big)
\\
{\simeq}\ \ \ \ \ \  \ \ 
&\funF_d\,\Big(\big(\Ind^{d\de}_{\la^\op\de}\,\funG_{\la^\op}(V_n\boxtimes\dots\boxtimes V_1)\big)^\circledast\Big)
\\
\stackrel{\text{Corollary~\ref{CEquivFG}(ii)}}{\simeq}\  &\funF_d\,\Big(\big(\funG_d\funF_d\Ind^{d\de}_{\la^\op\de}\,\funG_{\la^\op}(V_n\boxtimes\dots\boxtimes V_1)\big)^\circledast\Big)
\\
\stackrel{(\ref{EGGIInd})}{\simeq}\  \ \ \ \ \ \,&\funF_d\,\Big(\big(\funG_d(\GGI^{d}_{\la^\op}\,V_n\boxtimes\dots\boxtimes V_1)\big)^\circledast\Big)
\\
\stackrel{(\ref{ECircledAstGG})}{=} \ \ \ \ \ \,\,& (\GGI^{d}_{\la^\op}\,V_n\boxtimes\dots\boxtimes V_1)^\sharp,
\end{align*}
as required. 
\end{proof}

\begin{Lemma} \label{LGGIrrSelfD} 
Each isomorphism class of irreducible graded $C_d$-supermodules contains an irreducible graded supermodule $L$ such that $L^\sharp\simeq \Uppi^\eps L$ for some $\eps\in\Z/2$ (depending on $L$). In particular, for every irreducible graded $C_d$-supermodule $V$, we have $V^\sharp\cong V$. 
\end{Lemma}
\begin{proof}
Let $L$ be an irreducible graded $C_d$-supermodule. 
By Corollary~\ref{CEquivFG}, $\funG_d L$ is an irreducible graded $\hat C_d$-supermodule. By Lemma~\ref{LTypeM}, we may assume that $(\funG_d L)^\circledast\simeq \Uppi^\eps\funG_d L$ for some $\eps\in\Z/2$. So
$$
L^\sharp=\funF_d\big((\funG_d L)^\circledast\big)\simeq \funF_d(\Uppi^\eps\funG_d(L))
\simeq \Uppi^\eps \funF_d\funG_d(L)
\simeq \Uppi^\eps L,
$$
as required.
\end{proof}

\begin{Lemma} \label{LLJSelfD}
We have that 
$(\hat \LL_j)^\circledast\simeq  \Uppi \hat \LL_j$ and $(\LL_j)^\sharp \simeq  \Uppi\LL_j$. 
\end{Lemma}
\begin{proof}
We have $\LL_j=\funF_1(\hat \LL_j)=\ggi_1 \hat \LL_j$. In particular, recalling (\ref{EDegV}), we have that $\dim_{q,\pi}\ggi^j\hat \LL_j=q^{2j+1-2\ell}\pi^j$. Since $\ggi^j=1_{\ggw^j}$, it now follows from Lemma~\ref{LFact} that 
$$\dim_{q,\pi}1_{\hat\ggw^j}\hat \LL_j=\pi^{j}(q^2+q^{-2})^{\ell-1-j}(\pi q+q^{-1}).$$ We now deduce from Lemma~\ref{LTypeM} that $(\hat \LL_j)^\circledast\simeq \Uppi\hat \LL_j$, which also implies that $(\LL_j)^\sharp \simeq \Uppi\LL_j$. 
\end{proof}

\begin{Lemma} \label{LSocleHead}
We have that $\hat \M_{d,j}^\circledast\simeq \Uppi^d\hat \M_{d,j}$ and $\M_{d,j}^\sharp\simeq \Uppi\M_{d,j}$. In particular, the socle and the head of $\hat \M_{d,j}$ (resp. $\M_{d,j}$) are (not necessarily evenly) isomorphic to each other.\end{Lemma}
\begin{proof}
By Lemma~\ref{LLJSelfD}, we have $(\hat \LL_j)^\circledast\simeq \Uppi\hat \LL_j$ . So, using Lemma~\ref{LDualInd}(ii), we get  
$$
\hat \M_{d,j}^\circledast=
(\Ind_{\de^d}^{d\de}\hat \LL_j^{\boxtimes d})^\circledast \simeq  \Ind_{\de^d}^{d\de}(\Uppi\hat \LL_j)^{\boxtimes d}
\simeq  \Ind_{\de^d}^{d\de}\Uppi^d\hat \LL_j^{\boxtimes d}
\simeq  \Uppi^d\Ind_{\de^d}^{d\de}\hat \LL_j^{\boxtimes d}
=\Uppi^d\hat \M_{d,j}.$$ 
Now, by (\ref{ECircledAstGG}) and Lemma~\ref{LHatMG}, we have 
$$
\M_{d,j}^\sharp= \funF_d\big((\funG_d \, \M_{d,j})^\circledast\big)
\simeq \funF_d\big((\hat \M_{d,j})^\circledast\big)
\simeq \funF_d(\Uppi^d\hat \M_{d,j})\simeq 
\Uppi^d\funF_d(\hat \M_{d,j})=
\Uppi^d\M_{d,j}.
$$
The final statement now follows using Lemma~\ref{LGGIrrSelfD}. 
\end{proof}

\begin{Corollary} \label{CSocleHead}
Let $\la\in\Comp(d)$. Then $\hat \M_{\la,j}^\circledast\cong \hat \M_{\la,j}$ and $\M_{\la,j}^\sharp\cong \M_{\la,j}$. In particular, the socle and the head of $\hat \M_{\la,j}$ (resp. $ \M_{\la,j}$) are (not necessarily evenly) isomorphic to each other.
\end{Corollary}

\chapter{Imaginary Howe duality}
As in Section~\ref{SSTensSpace}, we again fix $j\in J$. We can consider the imaginary tensor space $M_{d,j}$ as a (faithful) graded supermodule over {\em imaginary graded Schur superalgebra} $\IS_{d,j}:=C_d/\Ann_{C_d}(\M_{d,j})$. In this chapter we will prove that this graded supermodule is projective and its endomorphism algebra is $\k\Si_d$, see Theorem~\ref{TMdProj}. But this `imaginary Schur-Weyl duality' is not sufficient to describe $\IS_{d,j}$ up to Morita equivalence since in general $\M_{d,j}$ is not a projective generator for $\IS_{d,j}$. We will eventually construct the desired projective generator as a direct sum  $\Ga_j(n,d):=\bigoplus_{\la\in \Comp(n,d)} \Di_j^\la$ of `imaginary divided powers' modules (for $n\geq d$), and the endomorphism algebra of $\Ga_j(n,d)$ will turn out to be the classical Schur algebra $S(n,d)$. This will lead to a Morita equivalence between the imaginary and the classical Schur algebras.

\section{Imaginary Schur superalgebra} 
\subsection{  Imaginary Schur superalgebra $\IS_{d,j}$ and its parabolic analogues $\IS_{\la,j}$}
\label{SSIS}
We define the {\em imaginary graded Schur superalgebra}
\begin{equation}\label{DIS}
\IS_{d,j}:=C_d/\Ann_{C_d}(\M_{d,j}).
\end{equation}
In particular, $\M_{d,j}$ can be considered as a faithful graded supermodule over $\IS_{d,j}$. 
(Occasionally, we will also use the graded superalgebra 
$
\hat C_d/\Ann_{\hat C_d}(\hat \M_{d,j})
$ 
to which we do not assign any special name.) 

More generally, for $\la\in\Comp(n,d)$, we set
$$
\IS_{\la,j}:=C_\la/\Ann_{C_\la}(\M_{\la,j})\cong \IS_{\la_1,j}\otimes\dots\otimes \IS_{\la_n,j}.
$$
In particular, $\M_{\la,j}$ can be considered as a faithful graded supermodule over $\IS_{\la,j}$. Identifying 
$
\IS_{\la,j}=\IS_{\la_1,j}\otimes\dots\otimes \IS_{\la_n,j}
$, we have 
$\M_{\la,j}\cong \M_{\la_1,j}\boxtimes\dots\boxtimes \M_{\la_n,j}$ as graded $\IS_{\la,j}$-supermodules. 

\subsection{  Imaginary tensor space $\M_{d,j}$ as a module over $\IS_{d,j}$}
As in \S\ref{SSSymAct}, we can now consider 
$\M_{d,j}$ as a graded $(\IS_{d,j},\k \Si_d)$-bisupermodule and,  more generally, 
$\M_{\la,j}$ as a graded $(\IS_{\la,j},\k \Si_\la)$-bisupermodule. 
By functoriality of $\GGI_\la^d$, we have that $\GGI_\la^d \M_{\la,j}$ becomes a graded $(\IS_{d,j},\k \Si_\la)$-bisupermodule. Similarly, taking into account Lemma~\ref{LggIndResSymGroup}(ii), $\GGR_\la^d \M_{d,j}$ becomes a graded $(\IS_{\la,j},\k \Si_d)$-bisupermodule.

\begin{Lemma} \label{LggIndResSymGroupImSch} 
Let $\la\in\Comp(d)$. Then:
\begin{enumerate}
\item[{\rm (i)}] $\GGI_\la^d \M_{\la,j}\simeq \M_{d,j}$ as graded $(\IS_{d,j},\k \Si_\la)$-bisupermodules, where the right action of $\Si_\la$ on $\M_{d,j}$ is via restriction from $\Si_d$.
\item[{\rm (ii)}] $\GGR_\la^d \M_{d,j}\simeq \M_{\la,j}\otimes_{\k \Si_\la}\k \Si_d$ as graded $(\IS_{\la,j},\k \Si_d)$-bisupermodules. 
\end{enumerate}
\end{Lemma}
\begin{proof}
Follows immediately from Lemma~\ref{LggIndResSymGroup}.
\end{proof}

\begin{Remark}
Lemma~\ref{LggIndResSymGroupImSch} shows that the $C_\la$-action on $\GGR_\la^d \M_{d,j}$ factors through the quotient $\IS_{\la,j}$ and the $C_d$-action on $\GGI_\la^d \M_{\la,j}$ factors through the quotient $\IS_{d,j}$. 
In Corollary~\ref{CHCIC} we will prove a stronger result that 
the functor $\GGR_\la^d$ sends $\IS_{d,j}$-modules to $\IS_{\la,j}$-modules and the functor 
$\GGI_\la^d$ sends $\IS_{\la,j}$-modules to $\IS_{d,j}$-modules.
\end{Remark}

Recall the graded superalgebra embedding $C_{\la}\to \ggi_\la C_d\ggi_\la$ 
from (\ref{ECPar}).

\begin{Lemma} \label{LBarParabolic} 
Let $\la\in\Comp(d)$. Denote $\bar\ggi_\la:=\ggi_\la+\Ann_{C_d}(\M_{d,j})\in\IS_{d,j}$. The embedding $C_{\la}\to \ggi_\la C_d\ggi_\la$ induces a graded superalgebra embedding  $\IS_{\la,j}\to \bar\ggi_\la \IS_{d,j}\bar\ggi_\la$. 
\end{Lemma}
\begin{proof}
The surjection $C_d\to \IS_{d,j}$ with kernel $\Ann_{C_d}(\M_{d,j})$ induces a surjection $\ggi_\la C_d\ggi_\la\stackrel{\pi}{\to} \bar\ggi_\la\IS_{d,j}\bar\ggi_\la$ with kernel $\ggi_\la \Ann_{C_d}(\M_{d,j})\ggi_\la$, 
so the composition $C_{\la}\to \ggi_\la C_d\ggi_\la\stackrel{\pi}{\to} \bar\ggi_\la\IS_{d,j}\bar\ggi_\la$ has kernel $C_{\la}\cap \ggi_\la \Ann_{C_d}(\M_{d,j})\ggi_\la$. But by Lemma~\ref{LggIndResSymGroup}(ii), the graded $C_{\la}$-supermodule $\ggi_\la\M_{d,j}$ is a direct sum of copies of $\M_{\la,j}$, so  $C_{\la}\cap \ggi_\la \Ann_{C_d}(\M_{d,j})\ggi_\la=\Ann_{C_\la}(\M_{\la,j})$. 
\end{proof}

We make use of the following {\em Schubert's Criterion} (see \cite[\S3.2]{BDK} for attributions):

\begin{Lemma} \label{LSchub}  
Let $A$ be a graded $\k$-superalgebra and $P\in\mod{A}$ be projective. Let $Z\subseteq P$ be a graded subsupermodule such that every homomorphism in $\Hom_A(P,P/Z)$ annihilates $Z$. Then $P/Z$ is a  projective graded  $(A/\Ann_A(P/Z))$-supermodule.
\end{Lemma}
\begin{proof}
The proof in \cite[Lemma 3.2a]{BDK} goes through in the  setting of finitely generated graded supermodules. 
\end{proof}

\begin{Theorem} \label{TMdProj}
We have that $\M_{d,j}$ is a projective graded $\IS_{d,j}$-supermodule and there is an isomorphism (of graded superalgebras)\,  $\End_{\IS_{d,j}}(\M_{d,j})\cong\k \Si_d$.
\end{Theorem}
\begin{proof}
The second statement comes from Theorem~\ref{TEndMd}. 

For the first statement we will apply Lemma~\ref{LSchub}. 
Let $P:=C_d\ggi^{j^d}$. By (\ref{EvdWeight}), we have $\ggi^{j^d}v_{d,j}=v_{d,j}$, so, as $v_{d,j}$ generates $M_{d,j}$, there is a 
surjective homomorphism 
$
\pi: P\onto \M_{d,j}, \ \ggi^{j^d}\mapsto v_{d,j}.
$
Let $Z=\Ker \pi$. To verify the assumptions of Lemma~\ref{LSchub}, it suffices to prove that any $F\in\Hom_{C_d}(P,\M_{d,j})$ can be written as $F=\phi\circ \pi$ for some $\phi\in\End_{C_d}(\M_{d,j})$. Since $P$ is generated by $\ggi^{j^d}$, it suffices to show that $F(\ggi^{j^d})=\phi(\pi(\ggi^{j^d}))=\phi(v_{d,j})$ for some $\phi\in\End_{C_d}(\M_{d,j})$. We have that $F(\ggi^{j^d})\in \ggi^{j^d}\M_{d,j}$, so by Lemma~\ref{LResMdOne}(i), we can write 
$$
F(\ggi^{j^d})=\textstyle\sum_{w\in \Si_d}c_w\dot w v_{d,j} \qquad(c_w\in\k). 
$$
In view of 
(\ref{EPhiW}), 
we can now take $\phi:=\sum_{w\in \Si_d}(-1)^{\ttl(w)(j+1)}c_w\phi_w$. 
\end{proof}

\begin{Corollary} \label{CMLaProj} 
Let $\la\in\Comp(d)$. Then $\M_{\la,j}$ is a projective graded $\IS_{\la,j}$-supermodule and $\End_{\IS_{\la,j}}(\M_{d,j})\cong\k \Si_\la$
\end{Corollary}

Recall that by  Lemma~\ref{LHatMG}, we have $\hat \M_{d,j}\simeq \funG_d(M_{d,j})=\hat C_d\ggi_d\otimes_{C_d}\M_{d,j}$, where $\M_{d,j}$ is a graded $(C_d,\k\Si_d)$-bisupermodule. So we have a right $\k\Si_d$-module structure on $\hat \M_{d,j}$ or a homomorphism of algebras 
\begin{equation}\label{EHatMSd}
\k\Si_d\to \End_{\hat C_d}(\hat \M_{d,j})^\sop=\End_{\hat C_d/\Ann_{\hat C_d}(\hat \M_{d,j})}(\hat \M_{d,j})^\sop
\end{equation}
which maps $w\in \Si_\la$ to the unique endomorphism $\al_w$ with  
$$
\al_w(\ggi_{d}\otimes v_{d,j})=(-1)^{\ttl(w)(j+1)}\ggi_{d}\otimes \dot w v_{d,j}.
$$

\begin{Corollary} \label{CEndHatM} 
The homomorphism (\ref{EHatMSd}) is an isomorphism. 
\end{Corollary}
\begin{proof}
This follows from 
Theorem~\ref{TEndMd} and 
the fact that $\funG_d$ is an equivalence by  Corollary~\ref{CEquivFG}(ii). 
\end{proof}

\begin{Theorem} \label{TMdProjHat}
We have that $\hat \M_{d,j}$ is a projective graded $\hat C_d/\Ann_{\hat C_d}(\hat \M_{d,j})$-supermodule and $\End_{\hat C_d/\Ann_{\hat C_d}(\hat \M_{d,j})}(\hat \M_{d,j})\cong\k \Si_d$.
\end{Theorem}
\begin{proof}
The proof is similar to that of Theorem~\ref{TMdProj}, using Corollary~\ref{CEndHatM} instead of Theorem~\ref{TEndMd}, and 
Lemma~\ref{LResMdOne}(ii) instead of Lemma~\ref{LResMdOne}(i). 
\end{proof}

\section{Imaginary symmetric and divided powers}\label{SSSymDivExt}

\subsection{  The modules $\Sy_{d,j}$ and $\Di_{d,j}$}
Recall from (\ref{Ettx}) the elements $\ttx_d$ and $\tty_d$ of $\k \Si_d$. Define the {\em imaginary symmetric} and {\em imaginary divided powers} as the following graded $C_d$-supermodules:
\begin{align}
\Sy_{d,j}&:=\M_{d,j}/\spa_\k\{mg-m\mid g\in\Si_d,\, m\in \M_{d,j}\},
\\
\label{EGaDJ}
\Di_{d,j}&:=\{m\in \M_{d,j}\mid mg=m \ \text{for all}\ g\in\Si_d\}.
\end{align}

An important role will also be played by the graded $C_d$-supermodules 
$$\M_{d,j}\tty_d\quad \text{and}\quad \M_{d,j}\ttx_d,$$ 
the first of which referred to as the {\em imaginary exterior power}. Note that $\M_{d,j}\tty_d\neq 0\neq \M_{d,j}\ttx_d$ for example by Theorem~\ref{TEndMd}, and $\M_{d,j}\ttx_d$ is a graded  subsupermodule of $\Di_{d,j}$. 

The graded $C_d$-supermodules $\Sy_{d,j},\,\Di_{d,j},\,\M_{d,j}\tty_d,\, \M_{d,j}\ttx_d$ factor through the quotient $\IS_{d,j}$ of $C_d$ to yield graded $\IS_{d,j}$-supermodules. 

To prove Lemma~\ref{LDualSharp} below, we will also need the following $\hat C_d$-versions of $\Sy_{d,j},\Di_{d,j}$ 
\begin{align*}
\hat\Sy_{d,j}&:=\hat \M_{d,j}/\spa_\k\{mg-m\mid g\in\Si_d,\, m\in \hat \M_{d,j}\},
\\
\hat\Di_{d,j}&:=\{m\in \hat \M_{d,j}\mid mg=m \ \text{for all}\ g\in\Si_d\},
\end{align*}
as well as the $\hat C_d$-modules $\hat \M_{d,j}\tty_d$ and $\hat \M_{d,j}\ttx_d$.

\begin{Lemma} \label{L5.2.2} 
We have $\Sy_{d,j}\simeq \M_{d,j}\otimes_{\k\Si_d}\k_{\Si_d}$
as graded left $C_d$-supermodules and
$\hat\Sy_{d,j}\simeq \hat\M_{d,j}\otimes_{\k\Si_d}\k_{\Si_d}$
as graded left $\hat C_d$-supermodules.  
\end{Lemma}
\begin{proof}
By definition, 
$$\M_{d,j}\otimes_{\k\Si_d}\k_{\Si_d}\simeq (\M_{d,j}\otimes_{\k}\k_{\Si_d})/\spa_\k\{mg\otimes 1-m\otimes 1\mid g\in\Si_d,\, m\in \M_{d,j}\}.$$ 
Identifying $\M_{d,j}\otimes_{\k}\k_{\Si_d}=\M_{d,j}$, this immediately gives the first claim. The second claim is proved similarly. 
\end{proof}

\begin{Lemma} 
We have that $\{v_{d,j}\tty_d\}$ is a basis of the $\k$-module  $\ggi^{j^d}\M_{d,j}\tty_d$, and $\{v_{d,j}\ttx_d\}$ is a basis of the $\k$-module $\ggi^{j^d}\M_{d,j}\ttx_d$. 
\end{Lemma}
\begin{proof}
By (\ref{EActions_r}) and Lemma~\ref{LResMdOne}(i), we have that $\ggi^{j^d}\M_{d,j}$ is a free right $\k\Si_d$-module with basis $v_{d,j}$. This implies the lemma.   
\end{proof}

\begin{Lemma} \label{Lttx} 
Let $\k=\F$. 
\begin{enumerate}
\item[{\rm (i)}] $\M_{d,j}\tty_d$ and $\M_{d,j}\ttx_d$ are irreducible $C_d$-modules, not isomorphic to each other if $\cha \F\neq 2$. 
\item[{\rm (ii)}] $\hat\M_{d,j}\tty_d$ and $\hat \M_{d,j}\ttx_d$ are irreducible $\hat C_d$-modules, not isomorphic to each other if $\cha \F\neq 2$. 
\end{enumerate} 
\end{Lemma}
\begin{proof}
(i) By Theorem~\ref{TMdProj}, 
$\M_{d,j}$ is a projective graded $\IS_{d,j}$-supermodule with $\k \Si_d=\End_{\IS_{d,j}}(\M_{d,j})^\sop$. Moreover, by Lemma~\ref{LSocleHead}, every composition factor of the socle of $\M_{d,j}$ appears in its head.  Also, left ideals $\F \Si_d\ttx_d$ and  $\F \Si_d\tty_d$ are irreducible $\F \Si_d$-modules, not isomorphic to each other if $\cha \F\neq 2$. 
Now the claim follows from Lemmas~\ref{L3.1e},\,\ref{L3.1f}. 

(ii) The proof is similar to (i) using Theorem~\ref{TMdProjHat} instead of Theorem~\ref{TMdProj}. 
\end{proof}

\begin{Lemma} \label{LSymHead}
Let $\k=\F$. 
\begin{enumerate}
\item[{\rm (i)}] $\Sy_{d,j}$ has simple head isomorphic to $\M_{d,j}\ttx_d$, and no other composition factor of $\Sy_{d,j}$ appears in the head of $\M_{d,j}$.
\item[{\rm (ii)}] $\hat \Sy_{d,j}$ has simple head isomorphic to $\hat \M_{d,j}\ttx_d$, and no other composition factor of $\hat \Sy_{d,j}$ appears in the head of $\hat \M_{d,j}$.
\end{enumerate} 
\end{Lemma}
\begin{proof}
(i) Let $\funf$, $\fung$ and  $\funq$ be the functors defined in \S\ref{SSSF}, taking the projective module $P$ to be the $\IS_{d,j}$-module $M_{d,j}$, so that $H=\F\Si_d$, see Theorem~\ref{TMdProj}. By Lemma~\ref{L5.2.2}, we have $\Sy_{d,j}\simeq\fung(\k_{\Si_d})$. On the other hand, by Lemma~\ref{LSocleHead},  every composition factor of the socle of $\M_{d,j}$ appears in its head, so by Lemma~\ref{L3.1f}, we have $$\M_{d,j}\ttx_d\simeq \funq\circ\fung(\k\Si_d\ttx_d)\simeq \funq\circ\fung(\k_{\Si_d})\simeq \funq(\Sy_{d,j}).
$$
By Lemma~\ref{Lttx}(i), the graded supermodule $\M_{d,j}\ttx_d$ is irreducible. We deduce that $\M_{d,j}\ttx_d$ appears in the head of $\Sy_{d,j}$ and no other composition factor of $\Sy_{d,j}$ appear in the head of $\M_{d,j}$. Since $\Sy_{d,j}$ is a quotient of $\M_{d,j}$, this means that $\Sy_{d,j}$ has simple head. 

(ii) is proved similarly using Theorem~\ref{TMdProjHat} instead of Theorem~\ref{TMdProj}, and Lemma~\ref{Lttx}(ii) instead of Lemma~\ref{Lttx}(i). 
\end{proof}

\begin{Lemma} \label{LHatSHatZ}
We have $\funG_d\Sy_{d,j}\simeq \hat\Sy_{d,j}$ and $\funF_d\hat\Di_{d,j}\simeq \Di_{d,j}$.
\end{Lemma}
\begin{proof}
By Lemma~\ref{LHatMG}, we identify $\funG_d \M_{d,j}=\hat \M_{d,j}= \hat C_d\ggi_d\otimes_{C_d}\M_{d,j}$. 

For the first isomorphism, since $\Sy_{d,j}$ is a quotient of $\M_{d,j}$ there is a 
short exact sequence $0\to K\to \M_{d,j}\to \Sy_{d,j}\to 0.$ 
Applying the exact functor $\funG_d$ we get the exact sequence 
$$
0\to \funG_d K\stackrel{\phi}{\to} \funG_d \M_{d,j}\to \funG_d\Sy_{d,j}\to 0.
$$
We need to prove that $\Im \phi=\hat K:=\spa_\k\{mg-m\mid g\in\Si_d,\, m\in \hat \M_{d,j}\}$. 
Note that $\Im \phi$ consists of all elements of the form $\sum_{k} r_{k}\ggi_d\otimes(m_{k}-m_{k}g_{k})$ with $r_{k}\in\hat C_d$, $m_{k}\in \M_{d,j}$ and $g_{k}\in \Si_d$. But 
\begin{align*}
\textstyle\sum_{k} r_{k}\ggi_d\otimes(m_{k}-m_{k}g_{k})
=\sum_{k} \big(r_{k}\ggi_d\otimes m_{k} 
- (r_{k}\ggi_d\otimes m_{k})g_{k}\big)
\subseteq \hat K.
\end{align*}
On the other hand, $\hat K$ is a span of the elements of the form $$\textstyle\sum_{k} r_{k}\ggi_d\otimes m_{k}
-\Big(\sum_{k} r_{k}\ggi_d\otimes m_{k}\Big)g=\sum_{k} r_{k}\ggi_d\otimes (m_{k}-m_kg)\subseteq \Im \phi.
$$

For the second isomorphism, applying $\funF_d$ to the embedding $\hat \Di_d\to \hat \M_{d,j}$ we get the embedding $\funF_d\hat \Ga_d\to \funF_d\hat \M_{d,j}=\M_{d,j}$, whose image consists of all elements of the form $\sum_k\ggi_dr_k\ggi_dm_k$ such that $\sum_kr_k\ggi_d\otimes m_kg=\sum_kr_k\ggi_d\otimes m_k$ for all $g\in\Si_d$, which is easily seen to be $\Di_{d,j}$.
\end{proof}

\subsection{  Duality between $\Di_{d,j}$ and $\Sy_{d,j}$}
Our next goal is to prove that $\Di_{d,j}\cong \Sy_{d,j}^\sharp$ when $\k=\F$. We first establish the $\hat C_d$-version of this:

\begin{Lemma} \label{LDualHatZ}
If $\k=\F$ then $\hat \Di_{d,j}\cong \hat \Sy_{d,j}^\circledast$. 
\end{Lemma}
\begin{proof}
By Lemma~\ref{LSocleHead}, there exists a 
$\hat C_d$-isomorphism 
$\phi:\hat \M_{d,j}\iso \Uppi^d \hat \M_{d,j}^\circledast.$  
For $\theta\in\End_{\hat C_d}(\hat \M_{d,j})$, we define $\ka(\theta):=\phi\circ\theta\circ\phi^{-1} \in \End_{\hat C_d}(\hat \M_{d,j}^\circledast)
$. 
This defines an algebra isomorphism
$$
\kappa:\End_{\hat C_d}(\hat \M_{d,j})^\sop\to \End_{\hat C_d}(\hat \M_{d,j}^\circledast)^\sop.
$$
Note by Corollary~\ref{CEndHatM} that $\End_{\hat C_d}(\hat \M_{d,j})^\sop\cong \F \Si_d$ is concentrated in parity $\0$, so, recalling (\ref{EPhiCircledast}), we also have an algebra isomorphism
$$
\circledast:\End_{\hat C_d}(\hat \M_{d,j})\iso \End_{\hat C_d}(\hat \M_{d,j}^\circledast)^\sop,\ \theta\mapsto \theta^\circledast. 
$$
Using the standard right $\End_{\hat C_d}(\hat \M_{d,j})^\sop$-module structure on $\hat \M_{d,j}$ and the standard right $\End_{\hat C_d}(\hat \M_{d,j}^\circledast)^\sop$-module structure on $\hat \M_{d,j}^\circledast $, we can write:
$$
v\cdot\ka^{-1}(\theta^\circledast)=\phi^{-1}(\phi(v)\cdot \theta^\circledast)\qquad(v\in \hat \M_{d,j},\,\theta\in \End_{\hat C_d}(\hat \M_{d,j})^\sop).
$$
We have an algebra isomorphism 
$$
\si:=\kappa^{-1}\circ \circledast: \End_{\hat C_d}(\hat \M_{d,j})\iso \End_{\hat C_d}(\hat \M_{d,j})^\sop.
$$

We consider a non-degenerate bilinear form $(\cdot,\cdot)$ on $\hat \M_{d,j}$ defined from
$(v,w):=\phi(v)(w)$. For any $x\in\F\Si_d$ and $v,w\in\hat \M_{d,j}$, we have
\begin{align*}
(v\cdot\si(x),w)&=(v\cdot\ka^{-1}(x^\circledast),w)
\\
&=(\phi^{-1}(\phi(v)\cdot x^\circledast),w)
\\
&=(\phi(v)\cdot x^\circledast)(w)
\\
&=x^\circledast(\phi(v))(w)
\\
&=\phi(v)(x(w))
\\
&=\phi(v)(w\cdot x)
\\
&=(v,w\cdot x).
\end{align*}
Let 
$$\textstyle\operatorname{Aug}:\F\Si_d\to\F,\ \sum_{g\in\Si_d}c_gg\mapsto \sum_{g\in\Si_d}c_g$$
be the augmentation homomorphism. 
By definition, $\hat\Sy_{d,j}=\hat\M_{d,j}/K$ where $K:= \spa_\k\{w(x-\operatorname{Aug}(x))\mid w\in \hat\M_{d,j},\,x\in\F\Si_d\}
$. Now, 
\begin{align*}
\hat\Sy_{d,j}^\circledast
&\cong 
\{v\in \hat \M_{d,j}\mid (v,k)=0\ \text{for all $k\in K$}\}
\\
&\cong 
\{v\in \hat \M_{d,j}\mid \big(v,w(x-\operatorname{Aug}(x))\big)=0\ \text{for all $w\in \hat \M_{d,j}$ and $x\in\F\Si_d$}\}
\\
&=
\{v\in \hat \M_{d,j}\mid (v(\si(x)-\operatorname{Aug}(x)),w)=0\ \text{for all $w\in \hat \M_{d,j}$ and $x\in\F\Si_d$}\}
\\
&=
\{v\in \hat \M_{d,j}\mid v(\si(x)-\operatorname{Aug}(x))=0\ \text{for all $x\in\F\Si_d$}\}
\\
&=
\{v\in \hat \M_{d,j}\mid vx=\operatorname{Aug}(\si^{-1}(x))v\ \text{for all $x\in\F\Si_d$}\}.
\end{align*}
We have an algebra homomorphism $\operatorname{Aug}\circ \si^{-1}:\F\Si_d^\op\to \F$, so either $\operatorname{Aug}\circ \si^{-1}$ is trivial or a sign homomorphism. We need to make sure it is actually the trivial homomorphism. In particular, we may assume that $p\neq 2$. 

If $\operatorname{Aug}\circ \si^{-1}$ is a sign homomorphism then we have proved that $\hat\Sy_{d,j}^\circledast
\cong \{v\in \hat \M_{d,j}\mid wg=\sgn(g)w\supseteq \hat \M_{d,j}\tty_d$. Then by Lemmas~\ref{Lttx} and \ref{LTypeM}, the irreducible module $(\hat\M_{d,j}\tty_d)^\circledast\cong \hat\M_{d,j}\tty_d$ appears in the head of $\hat \Sy_{d,j}$. But by Lemma~\ref{LSymHead}(ii), the head of $\hat \Sy_{d,j}$ is isomorphic to $\hat \M_{d,j}\ttx_d$. It remains to apply   Lemma~\ref{Lttx}(ii).
\end{proof}

\begin{Lemma} \label{LDualSharp} 
If $\k=\F$ then $\Di_{d,j}\cong \Sy_{d,j}^\sharp$. 
\end{Lemma}
\begin{proof}
We have
$$
\Sy_{d,j}^\sharp=\funF_d((\funG_d\Sy_{d,j})^\circledast)\simeq
\funF_d(\hat \Sy_{d,j}^\circledast)
\cong
\funF_d(\hat \Di_{d,j})
\simeq
\Di_{d,j},
$$
where we have used (\ref{ECircledAstGG}) and Lemmas~\ref{LHatSHatZ}, \ref{LDualHatZ}. 
\end{proof}

\begin{Corollary} \label{CZSocle}
If $\k=\F$ then $\M_{d,j}\ttx_d\subseteq \Di_{d,j}$ is the simple socle of $\Di_{d,j}$ and no composition factor of $\Di_{d,j}/\M_{d,j}\ttx_d$ is isomorphic to a  submodule of $\M_{d,j}$.
\end{Corollary}
\begin{proof}
This follows from Lemmas~\ref{LSymHead}, \ref{LDualSharp} and \ref{LGGIrrSelfD}. 
\end{proof}

\subsection{  Parabolic versions $\Sy_{\la,j}$ and $\Di_{\la,j}$}
Let $\la\in\Comp(n,d)$
and recall the notation (\ref{Ettx}). 
We consider the following graded $C_\la$-supermodules:
\begin{align*}
\Sy_{\la,j}&:=\M_{\la,j}/\spa_\k\{mg-m\mid g\in\Si_\la,\, m\in \M_{\la,j}\}\simeq \Sy_{\la_1,j}\boxtimes\dots\boxtimes\Sy_{\la_n,j},
\\
\Di_{\la,j}&:=\{m\in \M_{\la,j}\mid mg=m \ \text{for all}\ g\in\Si_\la\}\simeq \Di_{\la_1,j}\boxtimes\dots\boxtimes\Di_{\la_n,j}.
\end{align*}
We will also use the graded $C_\la$-supermodules 
$$\M_{\la,j}\tty_\la\simeq \M_{\la_1,j}\tty_{\la_1}\boxtimes\dots\boxtimes\M_{\la_n,j}\tty_{\la_n}\quad \text{and}\quad 
\M_{\la,j}\ttx_\la\simeq \M_{\la_1,j}\ttx_{\la_1}\boxtimes\dots\boxtimes\M_{\la_n,j}\ttx_{\la_n}.
$$
Note that $\M_{\la,j}\ttx_\la\subseteq \Di_{\la,j}$. 

The graded $C_d$-supermodules $\Sy_{\la,j},\,\Di_{\la,j},\,\M_{\la,j}\tty_\la,\, \M_{\la,j}\ttx_\la$ factor through the quotient $\IS_{\la,j}$ of $C_\la$ to yield graded $\IS_{\la,j}$-supermodules.

Since $\M_{\la,j}\simeq \M_{\la_1,j}\boxtimes\cdots\boxtimes \M_{\la_n,j}$, $\Sy_{\la,j}\simeq \Sy_{\la_1,j}\boxtimes\cdots\boxtimes\Sy_{\la_n,j}$ and 
$\Di_{\la,j}\simeq \Di_{\la_1,j}\boxtimes\cdots\boxtimes\Di_{\la_n,j}$, the following three lemma follow from Lemmas~\ref{CZSocle},~\ref{L5.2.2} and \ref{LDualSharp}, respectively. 

\begin{Lemma} \label{LZSocleLa}
If $\k=\F$ and $\la\in\Comp(d)$. Then $\M_{\la,j}\ttx_\la\subseteq \Di_{\la,j}$ is the simple socle of $\Di_{\la,j}$ and no composition factor of $\Di_{\la,j}/\M_{\la,j}\ttx_\la$ is isomorphic to a  submodule of $\M_{\la,j}$.
\end{Lemma}

\begin{Lemma} \label{L5.2.2La} 
Let $\la\in\Comp(d)$. Then $\Sy_{\la,j}\simeq \M_{\la,j}\otimes_{\k\Si_\la}\k_{\Si_\la}$
as graded left $C_\la$-supermodules.  
\end{Lemma}

\begin{Lemma} \label{LDualSharpPar} 
If $\k=\F$ then $\Di_{\la,j}\simeq \Sy_{\la,j}^\sharp$ for any $\la\in\Comp(d)$. 
\end{Lemma}

\begin{Lemma} \label{L5.3.1} 
Let $\la\in\Comp(d)$. We have:
\begin{enumerate}
\item[{\rm (i)}] $\GGR^d_\la\Sy_{d,j}\simeq \Sy_{\la,j}$.
\item[{\rm (ii)}] If $\k=\F$ then $\GGR^d_\la\Di_{d,j}\simeq \Di_{\la,j}$. 
\end{enumerate} 
\end{Lemma}
\begin{proof}
(i) By Lemmas~\ref{L5.2.2} and \ref{L5.2.2La}, we have $\Sy_{d,j}\simeq \M_{d,j}\otimes_{\k\Si_d}\k_{\Si_d}$ and $\Sy_{\la,j}\simeq \M_{\la,j}\otimes_{\k\Si_\la}\k_{\Si_\la}$. Using Corollary~\ref{CFunPerm}(ii), we now get
\begin{align*}
\GGR^d_\la\Sy_{d,j}&\simeq 
\GGR^d_\la(\M_{d,j}\otimes_{\k\Si_d}\k_{\Si_d})
\\
&\simeq \M_{\la,j}\otimes_{\k\Si_\la}\Res^{\k\Si_d}_{\k\Si_\la}\k_{\Si_d}
\\
&\simeq \M_{\la,j}\otimes_{\k\Si_\la}\k_{\Si_\la}
\\
&\simeq \Sy_{\la,j}.
\end{align*}

(ii) follows from (i) on dualizing using Lemmas~\ref{LDualSharp}, \ref{LDualSharpPar} and \ref{LSharpResComm}. 
\end{proof}

\subsection{  Parabolic versions $\Sy_j^\la$ and $\Di_j^\la$}
We consider the following graded $C_d$-supermodules:
\begin{align*}
\Sy_j^\la&:=\M_{d,j}/\spa_\k\{mg-m\mid g\in\Si_\la,\, m\in \M_{d,j}\},
\\
\Di_j^\la&:=\{m\in \M_{d,j}\mid mg=m \ \text{for all}\ g\in\Si_\la\}.
\end{align*}
We will also use the graded $C_d$-supermodules 
$$
\M_{d,j}\tty_\la\quad \text{and}\quad \M_{d,j}\ttx_\la.
$$ 
Note that 
\begin{equation}\label{EMGa}
\M_{d,j}\ttx_\la\subseteq \Di^\la_j. 
\end{equation}

The graded $C_d$-supermodules $\Sy_j^\la,\,\Di_j^\la,\,\M_{d,j}\tty_\la,\, \M_{d,j}\ttx_\la$ factor through the quotient $\IS_{d,j}$ of $C_d$ to yield graded $\IS_{d,j}$-supermodules. 

\begin{Lemma} \label{LUpperInd}
For $\la\in\Comp(d)$, we have: 
\begin{enumerate}
\item[{\rm (i)}] $\Sy_j^\la\simeq \GGI_\la^d\, \Sy_{\la,j}$;
\item[{\rm (ii)}] $\Di_j^\la\simeq \GGI_\la^d\, \Di_{\la,j}$;
\item[{\rm (iii)}] $\M_{d,j}\tty_\la\simeq \GGI_\la^d (\M_{\la,j}\tty_\la)$;
\item[{\rm (iv)}] $\M_{d,j}\ttx_\la\simeq \GGI_\la^d (\M_{\la,j}\ttx_\la)$.
\end{enumerate}
\end{Lemma}
\begin{proof}
(i),\,(iii),\,(iv) We identify $\M_{d,j}$ and $\GGI_\la^d\, \M_{\la,j}$ as graded $(C_d,\k\Si_\la)$-bisupermodules 
via the isomorphism of Lemma~\ref{LggIndResSymGroup}(i). Then the quotient $\Sy_j^\la$ of $\M_{d,j}$ is identified with the quotient $\GGI_\la^d\, \Sy_{\la,j}$ of $\GGI_\la^d\, \M_{\la,j}$ proving (i), the submodule $\M_{d,j}\tty_\la\subseteq M_{d,j}$ is identified with the submodule $\GGI_\la^d(\M_{\la,j}\tty_\la)\subseteq \GGI_\la^d\, \M_{\la,j}$ proving (iii), and the submodule $\M_{d,j}\ttx_\la\subseteq \M_{d,j}$ is identified with the submodule $\GGI_\la^d (\M_{\la,j}\ttx_\la)\subseteq\GGI_\la^d\, \M_{\la,j}$, proving (iv).

(ii) 
Let 
\begin{align*}
\hat \Di_{\la,j}&:=\{m\in \hat \M_{\la,j}\mid mg=m \ \text{for all}\ g\in\Si_\la\}
\\
\hat \Di_j^\la&:=\{m\in \hat \M_{d,j}\mid mg=m \ \text{for all}\ g\in\Si_\la\}.
\end{align*}

Note  that 
\begin{equation}\label{E1}
\funF_d\hat\Di_j^\la\simeq \Di_j^\la.
\end{equation}
Indeed, use Lemma~\ref{LHatMG} to identify $\hat \M_{d,j}=\hat C_df_d\otimes_{C_d}\M_{d,j}$ and $\funF_d\hat \M_{d,j}=\M_{d,j}$. Then applying $\funF_d$ to the embedding $\hat \Di^\la\to \hat \M_{d,j}$, we get the embedding $\al:\funF_d\hat \Di^\la\to \funF_d\hat \M_{d,j}=\M_{d,j}$, whose image consists of all elements of the form $\sum_k\ggi_dr_k\ggi_dm_k$, with $r_{k}\in\hat C_d$, $m_{k}\in \M_{d,j}$, such that in $\hat\M_{d,j}$ we have 
$\sum_kr_k\ggi_d\otimes m_kg=\sum_kr_k\ggi_d\otimes m_k$ for all $g\in\Si_\la$. So $\Im\al=\Di_j^\la$.

Note next that 
\begin{equation}\label{E2}
\hat\Di_j^\la\simeq  \Ind_{\la\de}^{d\de} \hat \Di_{\la,j}.
\end{equation}
Indeed, use Lemma~\ref{LTensImagIsImag} to identify $\hat \M_{d,j}=\Ind_{\la\de}^{d\de} \hat \M_{\la,j}$. 
By Lemma~\ref{L030216}(ii), we then have $\hat \M_{d,j}=\bigoplus_{w\in \D_{pd}^{p\la}}\psi_w1_{\la\de}\otimes \hat \M_{\la,j}$ with all the natural maps $\M_{\la,j}\to \psi_w1_{\la\de}\otimes \hat \M_{\la,j},\ m\mapsto \psi_w1_{\la\de}\otimes m$ being isomorphisms. Then the submodule $\hat \Di_j^\la$ of $\hat \M_{d,j}$ is identified with the submodule 
$$
\bigoplus_{w\in \D_{pd}^{p\la}}\psi_w1_{\la\de}\otimes\hat \Di_{\la,j}\simeq 
\Ind_{\la\de}^{d\de} \hat \Di_{\la,j}\subseteq \Ind_{\la\de}^{d\de} \hat \M_{\la,j}.
$$ 

Now, 
$$
\Di_j^\la\simeq\funF_d\hat\Di_j^\la\simeq \funF_d\, \Ind_{\la\de}^{d\de}\, \hat \Di_{\la,j}
\simeq \funF_d\, \Ind_{\la\de}^{d\de}\,\funG_\la\funF_\la \hat \Di_{\la,j}
\simeq \funF_d\, \Ind_{\la\de}^{d\de}\,\funG_\la \Di_{\la,j}
\simeq \GGI_{\la}^{d}\,\Di_{\la,j},
$$
where we have used (\ref{E1}) for the first isomorphism, (\ref{E2}) for the second isomorphism, Corollary~\ref{CEquivFG}(ii) for the third isomorphism, Lemma~\ref{LHatSHatZ} for the fourth isomorphism, and (\ref{EGGIInd}) for the last isomorphism. 
\end{proof}

\begin{Lemma} \label{LBetaInd}
For $\la\in\Comp(d)$, we have 
$\Sy_j^\la\simeq \M_{d,j}\otimes_{\k\Si_d}\otimes \Ind_{\Si_\la}^{\Si_d}\k_{\Si_\la}$.
\end{Lemma}
\begin{proof}
We have 
$$
\M_{d,j}\otimes_{\k\Si_d}\otimes \Ind_{\Si_\la}^{\Si_d}\k_{\Si_\la}
\simeq \GGI_\la^d(\M_{\la,j}\otimes_{\k\Si_\la}\k_{\Si_\la})
\simeq \GGI_\la^d\Sy_{\la,j}\simeq \Sy_j^\la.
$$
where we have used Corollary~\ref{CFunPerm}(i) for the first isomorphism, Lemma~\ref{L5.2.2La} for the second isomorphism, and Lemma~\ref{LUpperInd} for the third isomorphism. 
\end{proof}

\begin{Lemma} \label{LS^LaDual}
Let $\k=\F$ and $\la\in\Comp(d)$. Then $(\Sy_j^\la)^\sharp\simeq \Di_j^{\la^\op}$. 
\end{Lemma}
\begin{proof}
We have 
\begin{align*}
(\Sy_j^\la)^\sharp\simeq (\GGI_\la^d \Sy_{\la,j})^\sharp
\simeq (\GGI_{\la^\op}^d \Sy_{\la^\op}^\sharp)
\simeq (\GGI_{\la^\op}^d \Di_{\la^\op,j})
\simeq \Di_j^{\la^\op},
\end{align*}
where we have used Lemma~\ref{LUpperInd}(i), Lemma~\ref{LSharpIndComm} for the second isomorphism,  Lemma~\ref{LDualSharpPar} for the third isomorphism and Lemma~\ref{LUpperInd}(ii) for the fourth isomorphism. 
\end{proof}

\begin{Lemma} \label{LZModMx}
Let $\k=\F$ and $\la\in\Comp(d)$. Then no composition factor of\, $\Di_j^\la/\M_{d,j}\ttx_\la$ is isomorphic to a submodule of $\M_{d,j}$.
\end{Lemma}
\begin{proof}
Let $\funf,\fung$ and $\funq$ be the functors defined in \S\ref{SSSF} corresponding to $P$ being the projective $\IS_{\la^\op,j}$-module $\M_{\la^\op,j}$, see Corollary~\ref{CMLaProj}. Note that the $\F\Si_d$-module $\Ind_{\F\Si_{\la^\op}}^{\F\Si_d}\k_{\Si_{\la^\op}}$ is isomorphic to the left ideal $\F\Si_d\ttx_{\la^\op}$ of $\F\Si_d$. By Corollary~\ref{CSocleHead} and Lemma~\ref{L3.1f}, we have $\funq\circ \fung(\Ind_{\F\Si_{\la^\op}}^{\F\Si_d}\k_{\Si_{\la^\op}})\simeq \M_{d,j}\ttx_{\la^\op}$. But by Lemma~\ref{LBetaInd}, we have 
$$
\fung(\Ind_{\F\Si_{\la^\op}}^{\F\Si_d}\k_{\Si_{\la^\op}})=\M_{d,j}\otimes_{\F\Si_d}\Ind_{\F\Si_{\la^\op}}^{\F\Si_d}\k_{\Si_{\la^\op}}\simeq \Sy_j^{\la^\op}.
$$
So there is a short exact sequence of the form $0\to V\to \Sy_j^{\la^\op}\to W\to 0$ where $W\simeq \M_{d,j}\ttx_{\la^\op}$ and $V$ has no composition actors in common with the head of $\M_{d,j}$. Dualizing and using Lemma~\ref{LS^LaDual}, we have a short exact sequence $0\to W^\sharp\to \Di_j^{\la}\to V^\sharp\to 0$. Moreover, by Lemma~\ref{LGGIrrSelfD}, $V^\sharp$ has the same composition factors as $V$, so $V^\sharp$ has no composition factors in common with the head of $\M_{d,j}$. Finally, 
\begin{align*}
W^\sharp\ \simeq\ \ \ \ \ \ \ \ \ \ \  &(\M_{d,j}\ttx_{\la^\op})^\sharp
\\ 
\stackrel{\text{Lemma}~\ref{LUpperInd}\text{(iv)}}{\simeq}\   \ \ &\big(\GGI_{\la^\op}^d (\M_{\la^\op,j}\ttx_{\la^\op})\big)^\sharp
\\
\simeq\ \ \ \ \ \  \ \ \ \ \ \,  &\big(\GGI_{\la^\op}^d (\M_{\la_n,j}\ttx_{\la_n}\boxtimes\dots\boxtimes \M_{\la_1,j}\ttx_{\la_1})\big)^\sharp
\\
\stackrel{\text{Lemma}~\ref{LSharpIndComm}}{\simeq}\    \ \  \ \ \,&\GGI_{\la}^d \big((\M_{\la_1,j}\ttx_{\la_1})^\sharp\boxtimes\dots\boxtimes (\M_{\la_n,j}\ttx_{\la_n})^\sharp\big)
\\
\stackrel{\text{Lemmas~\ref{Lttx}(i),\ref{LGGIrrSelfD}}}{\cong}\   &\GGI_{\la}^d \big(\M_{\la_1,j}\ttx_{\la_1}\boxtimes\dots\boxtimes \M_{\la_n,j}\ttx_{\la_n}\big)
\\
\stackrel{\text{Lemma}~\ref{LUpperInd}\text{(iv)}}{\simeq}\  \ \ \,&\M_{d,j}\ttx_\la,
\end{align*}
which implies the result using the Jordan-H\"older theorem. 
\end{proof}

\section{Classical Schur algebras as endomorphism algebras}
Recall the classical Schur algebra $S(n,d)=\End_{\k\Si_d}\big(\textstyle\bigoplus_{\la\in\Comp(n,d)}M^\la\big)$ from \S\ref{SSSchur} 
with the basis $\{\xi^g_{\la,\mu}\mid \la,\mu\in\Comp(n,d),\, g\in {}^\la\D_d^\mu\}$ from Lemma~\ref{LSchurBasis}. The basis elements $\xi^g_{\la,\mu}$ are closely related to the elements $X_{\mu,\la}^g\in\k\Si_d$  and $Y_{\mu,\la}^g\in\k\Si_d$ defined in (\ref{EX}). 

\subsection{  The endomorphism algebras of the modules $\bigoplus_{\la\in\Comp(n,d)}\M_{d,j}\ttx_\la$ and $\bigoplus_{\la\in\Comp(n,d)}\M_{d,j}\tty_\la$}

It follows from (\ref{ECleverX}) that for any $g\in\Si_d$ and $\la,\mu\in\Comp(d)$, we have
$$
\M_{d,j}\ttx_\mu X_{\mu,\la}^g
=\M_{d,j} \Big(\sum_{w\in \Si_\mu g\Si_\la\cap\D^\la}w\Big)\,\ttx_\la
\subseteq \M_{d,j}\ttx_\la,
$$
so the right multiplication by $X_{\mu,\la}^g$ defines a homomorphism from $\M_{d,j}\ttx_\mu$ to $\M_{d,j}\ttx_\la$. 
Denote by $\upxi^g_{\la,\mu}$ the endomorphism of $\bigoplus_{\la\in\Comp(n,d)}\M_{d,j}\ttx_\la$ which is zero on the summands $\M_{d,j}\ttx_\nu$ with $\nu\neq \mu$ and sends $\M_{d,j}\ttx_\mu$ to $\M_{d,j}\ttx_\la$ via the homomorphism given  by the right multiplication by $X^g_{\mu,\la}$.

\begin{Theorem} \label{TSchurEnd1}
Let $\k=\F$. Then there is a graded superalgebra isomorphism
$$\textstyle S(n,d)\iso \End_{\IS_{d,j}}\big(\bigoplus_{\la\in\Comp(n,d)}\M_{d,j}\ttx_\la\big), \ \xi^g_{\la,\mu}\mapsto \upxi^g_{\la,\mu}.
$$
\end{Theorem}
\begin{proof}
Let $\funf$, $\fung$ and  $\funq$ be the functors defined in \S\ref{SSSF}, taking the projective module $P$ to be the $\IS_{d,j}$-module $M_{d,j}$, so that $H=\F\Si_d$, see Theorem~\ref{TMdProj}. We have an equivalence $\funq\circ\fung$ of graded supercategories coming from Theorem~\ref{T3.1d}.
By Lemma~\ref{LSocleHead},  every composition factor of the socle of $\M_{d,j}$ appears in its head, so by Lemma~\ref{L3.1f}, for every $\la\in\Comp(n,d)$, we have 
$$
\M_{d,j}\ttx_\la\simeq \funq\circ\fung(\k\Si_d\ttx_\la)\simeq \funq\circ\fung(M^\la). 
$$ 
So  
$$
\textstyle \bigoplus_{\la\in\Comp(n,d)}\M_{d,j}\ttx_\la\simeq \funq\circ\fung\big(\bigoplus_{\la\in\Comp(n,d)}M^\la\big),
$$
and the endomorphism algebras of $\bigoplus_{\la\in\Comp(n,d)}\M_{d,j}\ttx_\la$ and $\bigoplus_{\la\in\Comp(n,d)}M^\la$ are isomorphic by Theorem~\ref{T3.1d}. The latter is $S(n,d)$ by definition. 
It remains to use Lemma~\ref{L3.1f} one more time to 
check that the image of $\xi^g_{\la,\mu}$ under the functor $\funq\circ\fung$ is precisely the endomorphism $\upxi^g_{\la,\mu}$. 
\end{proof}

It follows from (\ref{ECleverY}) that for any $g\in\Si_d$, we have
$$
\M_{d,j}\tty_\mu Y_{\mu,\la}^g
=\M_{d,j} \Big(\sum_{w\in \Si_\mu g\Si_\la\cap\D^\la}\sgn(w)w\Big)\,\tty_\la
\subseteq \M_{d,j}\tty_\la,
$$
so the right multiplication by $Y_{\mu,\la}^g$ defines a homomorphism from $\M_{d,j}\tty_\mu$ to $\M_{d,j}\tty_\la$. 
Denote by $\upeta^g_{\la,\mu}$ the endomorphism of $\bigoplus_{\la\in\Comp(n,d)}\M_{d,j}\tty_\la$ which is zero on the summands $\M_{d,j}\tty_\nu$ with $\nu\neq \mu$ and sends $\M_{d,j}\tty_\mu$ to $\M_{d,j}\tty_\la$ via the homomorphism induced by the right multiplication by $Y^g_{\mu,\la}$. The same argument as in the proof of Theorem~\ref{TSchurEnd1} yields:

\begin{Theorem} \label{TSchurEnd2}
Let $\k=\F$. Then there is a graded superalgebra isomorphism
$$S(n,d)\iso \End_{\IS_{d,j}}\big(\textstyle\bigoplus_{\la\in\Comp(n,d)}\M_{d,j}\tty_\la\big), \ \xi^g_{\la,\mu} \mapsto \upeta^g_{\la,\mu}.
$$
\end{Theorem}

\subsection{  The endomorphism algebra of $\bigoplus_{\la\in\Comp(n,d)}\Di_j^\la$} Our real object of interest is the endomorphism algebra of $\bigoplus_{\la\in\Comp(n,d)}\Di_j^\la$. In this subsection we will use Theorem~\ref{TSchurEnd1} to show that this is also the classical Schur algebra.

\begin{Lemma} 
Let $\k=\F$, $\la,\mu\in\Comp(d)$ and $g\in\Si_d$. Then 
$\Di_j^\mu X^g_{\mu,\la}\subseteq \Di_j^\la$, so the right multiplication by $X^g_{\mu,\la}$ induces a homomorphism $\Di_j^\mu \to \Di_j^\la$ of graded $\IS_{d,j}$-supermodules.
\end{Lemma}
\begin{proof}
It suffices to check that $\Di_j^\mu X^g_{\mu,\la}(u-1)=0$ for all $u\in\Si_\la$. 
But for any such $u$, we have using (\ref{ECleverX}) that 
$$\M_{d,j}\ttx_\mu X^g_{\mu,\la}(u-1)
=\M_{d,j}\Big(\sum_{w\in \Si_\mu g\Si_\la\cap\D^\la}w\Big)\ttx_\la (u-1)
=
0.$$
As the right multiplication by $X^g_{\mu,\la}(u-1)$ yields a graded $\IS_{d,j}$-supermodule homomorphism $\Di_j^\mu\to \M_{d,j}$, we deduce from Lemma~\ref{LZModMx} that $\Di_j^\mu X^g_{\mu,\la}(u-1)=0$.  
\end{proof}

In view of the lemma, we denote by $\upzeta^g_{\la,\mu}$ the endomorphism of  $\bigoplus_{\la\in\Comp(n,d)}\Di_j^\la$ which is zero on the summands $\Di_j^\nu$ with $\nu\neq \mu$ and sends $\Di_j^\mu$ to $\Di_j^\la$ via the homomorphism given by the right multiplication by $X^g_{\mu,\la}$. 

\begin{Theorem} \label{TSchurEnd3}
Let $\k=\F$. Then there is a graded superalgebra isomorphism
$$\textstyle S(n,d)\iso \End_{\IS_{d,j}}\big(\bigoplus_{\la\in\Comp(n,d)}\Di_j^\la\big), 
\ \xi^g_{\la,\mu} \mapsto 
\upzeta^g_{\la,\mu}.
$$ 
\end{Theorem}
\begin{proof}
We have 
$$A:=\bigoplus_{\la\in\Comp(n,d)}\M_{d,j}\ttx_\la\subseteq B:=\bigoplus_{\la\in\Comp(n,d)}\Di_j^\la.
$$ 
Let $S$ be the graded subsuperalgebra of $\End_{\IS_{d,j}}(B)$ consisting of the endomorphisms which preserve $A$. Restriction yields a graded superalgebra homomorphism $S\to\End_{\IS_{d,j}}(A)$. This homomorphism is injective by Lemma~\ref{LZModMx}. Moreover, this homomorphism is surjective since by Theorem~\ref{TSchurEnd1}, the elements $\upxi^g_{\la,\mu}$ span $S$, and by definition $\upxi^g_{\la,\mu}$ is the restriction of $\upzeta^g_{\la,\mu}$. In particular, $\{\upzeta^g_{\la,\mu}\mid \la,\mu\in\Comp(n,d),\, g\in {}^\la\D_d^\mu\}$ is a basis of $S$. It remains to prove using dimensions that $S=\End_{\IS_{d,j}}(B)$. On expanding the direct sums this will follow by dimensions if we can show that
$$
\dim\Hom_{C_d}(\Di_j^\mu,\Di_j^\la)=\dim\Hom_{\F\Si_d}(M^\la,M^\mu)
$$
for all $\la,\mu\in\Comp(n,d)$. But
\begin{align*}
\dim\Hom_{C_d}(\Di_j^\mu,\Di_j^\la)
&=\dim\Hom_{C_d}((\Di_j^\la)^\sharp,(\Di_j^\mu)^\sharp)
\\
&=\dim\Hom_{C_d}(\Sy_j^{\la^\op},\Sy_j^{\mu^\op})
\\
&=\dim\Hom_{C_d}(\fung M^{\la^\op},\fung M^{\mu^\op})
\\
&=\dim\Hom_{C_d}( M^{\la^\op},\funf\fung M^{\mu^\op})
\\
&=\dim\Hom_{C_d}( M^{\la^\op},M^{\mu^\op})
\\
&=\dim\Hom_{C_d}( M^{\la},M^{\mu})
\end{align*}
using consecutively 
Lemma~\ref{LSharpHom}, Lemma~\ref{LS^LaDual}, Lemma~\ref{LBetaInd}, the fact that the functor $\fung$ is left adjoint to the functor $\funf$ (see \S\ref{SSSF}), Lemma~\ref{L3.1a}, and the isomorphism $M^{\nu^\op}\cong M^\nu$ coming from the fact that the subgroups $\Si_{\nu^\op}$ and $\Si_\nu$ of $\Si_d$ are conjugate. 
\end{proof}

\chapter{Morita equivalence of imaginary and classical Schur algebras}
We continue to work with a fixed $j\in J$. 
In Theorem~\ref{TSchurEnd3}, we have established the isomorphism $\textstyle S(n,d)\cong \End_{\IS_{d,j}}\big(\bigoplus_{\la\in\Comp(n,d)}\Di_j^\la\big).
$ The main result of this chapter is that, under the assumption $n\geq d$, the graded supermodule $\bigoplus_{\la\in\Comp(n,d)}\Di_j^\la$ is a projective generator for the imaginary graded Schur superalgebra $\IS_{d,j}$, thus $\IS_{d,j}$ is graded Morita superequivalent to the classical Schur algebra $S(n,d)$. 

To establish projectivity of the graded $\IS_{d,j}$-supermodules $\Di_j^\la$, we will compare them to certain projective graded $C_d$-supermodules which we call (imaginary) Gelfand-Grave modules.

\section{Imaginary Gelfand-Graev modules}
\label{SSGGM}

\subsection{  The modules $\G_{d,j}$ and $\G_{\la,j}$}

Recall from \S\S\ref{SSGGW},\,\ref{SSGGIdempotents} the (divided power) Gelfand-Graev word 
$\ggw^{d,j}$ and the corresponding idempotent $\ggi^{d,j}=1_{\ggw^{d,j}}$. 
We define the {\em Gelfand-Graev module} to be  
the projective graded $C_{d}$-supermodule
\begin{align*}
\G_{d,j}:=
C_{d}\ggi^{d,j}.
\end{align*}

We identify $\k\Si_d=\End_{C_{d}}(\M_{d,j})^\sop$ 
via the isomorphism of 
Theorem~\ref{TEndMd}. Then $
\Hom_{C_{d}}(\G_{d,j},\M_{d,j})
$ is naturally a right $\k\Si_d$-module.

\begin{Lemma} \label{LHomGM} 
We have an isomorphism of right $\k\Si_d
$-modules
$$
\Hom_{C_{d}}(\G_{d,j},\M_{d,j})\cong \k_{\Si_d}.
$$
\end{Lemma}
\begin{proof}
By Lemma~\ref{LMWtSpacesNew}, 
$\ggi^{d,j}\M_{d,j}$ has $\k$-basis $\{\lgath_{d,j}v_{d,j}\}$. 
So, as $\k$-modules, 
$$
\Hom_{C_{d}}(\G_{d,j},\M_{d,j})\cong 
\ggi^{d,j}\M_{d,j}\cong \k.
$$   
Now, as $\k\Si_d$-modules, $\Hom_{C_{d}}(\G_{d,j},\M_{d,j})\cong \k_{\Si_d}
$ thanks to Lemma~\ref{LTrivU}.
\end{proof}

Recall from Theorem~\ref{LAmountAmount}(ii) the composition factors $\{L_j(\la)\mid\la\in\Par(d)\}$ of $M_{d,j}$. We denote by $P_j(\la)$ the projective cover of $L_j(\la)$ as a graded $C_d$-supermodule. 

\begin{Lemma} \label{LProjCover} 
Let $\k=\F$. Then $\G_{d,j}\cong P_j\big((d)\big)$. 
\end{Lemma}
\begin{proof}
For any $V\in\mod{C_d}$, we have 
$$
\Hom_{C_d}(\G_{d,j},V)=\Hom_{C_d}(
C_d\ggi^{d,j},V)\cong 
\ggi^{d,j}V.
$$
So the lemma follows from Theorem~\ref{LAmountAmount}(ii) and (\ref{EFChL}).\end{proof}

More generally, for $\la\in\Comp(n,d)$, we consider 
the projective graded supermodule over the parabolic $C_{\la}=C_{\la_1}\otimes \dots\otimes C_{\la_n}$ 
$$
\G_{\la,j}:=
\G_{\la_1,j}\boxtimes\dots\boxtimes \G_{\la_n,j}.
$$ 
and the projective graded $C_d$-supermodule 
\begin{equation}\label{EGUpperLa}
\G_j^\la:=\GGI_\la^d\G_{\la,j}\simeq 
C_n\ggi^{\la,j},
\end{equation}
see Lemma~\ref{LGGIndProjId} for the last isomorphism.

We identify $\k\Si_\la=\End_{C_{d}}(\M_{\la,j})^\sop$ 
via the isomorphism of 
Corollary~\ref{CEndTens}, so $
\Hom_{C_{\la}}(\G_{\la,j},\M_{\la,j})
$ is naturally a right $\k\Si_\la$-module. 
From Lemma~\ref{LHomGM}, we get:

\begin{Lemma} \label{LHomGMLa}
We have an isomorphism of right $\k\Si_\la$-modules
$$
\Hom_{C_{\la}}(\G_{\la,j},\M_{\la,j})\cong \k_{\Si_\la}.
$$
\end{Lemma}


From Lemma~\ref{LProjCover}, we get:

\begin{Lemma} \label{LProjCoverLa} 
Let $\k=\F$ and $\la\in\Comp(n,d)$. Then $$\G_{\la,j}\cong P_j\big((\la_1)\big)\boxtimes\dots\boxtimes P_j\big((\la_n)\big).$$ 
\end{Lemma}

\begin{Lemma} 
Let $\k=\F$ and $\la\in\Par(d)$. Then $\dim \Hom_{C_d}(\G_j^\la,L_j(\la))=1$. 
\end{Lemma}
\begin{proof}
We have 
$$
\dim\Hom_{C_d}(\G_j^\la,L_j(\la))=\dim\Hom_{C_d}(C_n\ggi^{\la,j},L_j(\la))=\dim \ggi^{\la,j}L_j(\la)=1
$$
by (\ref{EGUpperLa}) 
and (\ref{EFChL}). 
\end{proof}

\begin{Proposition}\label{PResGGModule}
Let $\k=\F$ and $\la\in\Comp(n,d)$. Then $\GGR^d_\la \G_{d,j}\cong \G_{\la,j}$. 
\end{Proposition}
\begin{proof}
By Corollary~\ref{CGGRProjtoProj}, $\GGR^d_\la \G_{d,j}$ is a projective graded supermodule over $C_\la$. Moreover, by Lemmas~\ref{LAnotherAdj} and \ref{LHomGM}, we have 
\begin{align*}
\dim\Hom_{C_\la}(\GGR^d_\la \G_{d,j},M_{\la,j})
&=
\dim\Hom_{C_d}(\G_{d,j},\GGI^d_{\la^\op}M_{\la^\op,j})
\\&=\dim\Hom_{C_d}(\G_{d,j},M_{d,j})=1,
\end{align*}
hence exactly one composition factor of $M_{\la,j}$ appears in the head of $\GGR^d_\la \G_{d,j}$. 
By Lemma~\ref{LProjCoverLa}, $\G_{\la,j}$ is the projective cover of $L_j\big((\la_1)\big)\boxtimes\dots\boxtimes L_j\big((\la_n)\big)$. So to complete the proof of the proposition it remains to show that for an irreducible graded $C_\la$-supermodule $L$, we have 
$
\Hom_{C_\la}(\GGR^d_\la \G_{d,j},L)=0
$
unless $L\cong L_j((\la_1))\boxtimes\dots L_j((\la_n))$. 

Recalling the classification of irreducible modules from Theorem~\ref{LAmountAmount}, an arbitrary irreducible graded supermodule over $C_\la$ has form $L(\bmu_{(1)})\boxtimes\dots\boxtimes L(\bmu_{(n)})$ where $\bmu_{(r)}=(\mu_{(r)}^{(0)},\dots,\mu_{(r)}^{(\ell-1)})\in\Par^J(\la_r)$ for $r=1,\dots,n$. 
Now, by Lemma~\ref{LAnotherAdj}, we have 
\begin{align*}
&\Hom_{C_\la}\big(\GGR^d_\la \G_{d,j}, L(\bmu_{(1)})\boxtimes\dots\boxtimes L(\bmu_{(n)})\big)
\\
\simeq\, &\Hom_{C_d}\big(\G_{d,j}, \GGI_{\la^\op}^d \big(L(\bmu_{(n)})\boxtimes\dots\boxtimes L(\bmu_{(1)})\big)\big).
\end{align*}

For $r=1,\dots,n$, 
by definition of $L(\bmu_{(r)})$, we have 
$$L(\bmu_{(r)})=\GGI^{\la_r}_{|\mu_{(r)}^{(0)}|,\dots,|\mu_{(r)}^{(\ell-1)}|}\big(L(\mu_{(r)}^{(0)})\boxtimes\dots\boxtimes L(\mu_{(r)}^{(\ell-1)})\big).$$ 
Note that for any $i\in J$, we have that all composition factors of 
$$\GGI_{|\mu_{(1)}^{(i)}|,\dots,|\mu_{(n)}^{(i)}|}^{|\mu_{(1)}^{(i)}|+\dots+|\mu_{(n)}^{(i)}|}\big((L(\mu_{(1)}^{(i)})\boxtimes\dots\boxtimes L(\mu_{(n)}^{(i)})\big)
$$
are composition factors of $M_{|\mu_{(1)}^{(i)}|+\dots+|\mu_{(n)}^{(i)}|,i}$, i.e. have form $L_i(\nu^{(i)})$ for a partition $\nu^{(i)}\in \Par(|\mu_{(1)}^{(i)}|+\dots+|\mu_{(n)}^{(i)}|)$. 
So, using Corollary~\ref{CLinInd}, we deduce that all composition factors of $\GGI_{\la^\op}^d \big(L(\bmu_{(n)})\boxtimes\dots\boxtimes L(\bmu_{(1)})\big)$ are of the form $L(\bnu)$ for $\bnu=(\nu^{(0)},\dots,\nu^{(\ell-1)})$ with each $\nu^{(i)}\in \Par(|\mu_{(1)}^{(i)}|+\dots+|\mu_{(n)}^{(i)}|)$. 
By Lemma~\ref{LProjCover}, we deduce that for an irreducible $C_d$-module $L$, we have 
$
\Hom_{C_\la}(\GGR^d_\la \G_{d,j},L)=0
$
unless $L\cong L_j(\mu^{(1)})\boxtimes\cdots\boxtimes L_j(\mu^{(n)})$, where $\mu^{(r)}\in\Par(\la_r)$ for $r=1,\dots, n$. 

Moreover, again 
\begin{align*}
&\Hom_{C_\la}(\GGR^d_\la \G_{d,j},L_j(\mu^{(1)})\boxtimes\cdots \boxtimes L_j(\mu^{(n)}))
\\
\simeq\, &\Hom_{C_d}(\G_{d,j},\GGI^d_{\la^\op} \big(L_j(\mu^{(n)})\boxtimes\cdots\boxtimes  L_j(\mu^{(1)})\big)).
\end{align*}
Since $M_{d,j}\cong \GGI^d_{\la^\op} (M_{\la_n,j}\boxtimes \dots\boxtimes M_{\la_1,j})$, the multiset of the composition factors of $M_{d,j}$ is the disjoint union of the multisets of the composition factors of the modules  $\GGI^d_{\la^\op} \big(L_j(\mu^{(n)})\boxtimes\cdots\boxtimes  L_j(\mu^{(1)})\big)$ as $L_j(\mu^{(r)})$ runs over the multiset of the composition factors of $M_{\la_r,j}$ for $r=1,\dots,n$. 
Lemmas~\ref{LProjCover} and \ref{LHomGM} now imply that for exactly one tuple $(\mu^{(1)},\dots,\mu^{(n)})$ we have that the above Hom-space is $1$-dimensional and for all other tuples, it is trivial. But by Corollary~\ref{CTensorcd}, we have that $L_j((d))$ is a composition factor of $\GGI^d_{\la^\op} L_j\big((\la_n)\big)\boxtimes\dots L_j\big((\la_1)\big)$, so the only tuple $(\mu^{(1)},\dots,\mu^{(n)})$ for which the above Hom-space is non-trivial must be of the form $\big((\la_1),\dots,(\la_n)\big)$. 
\end{proof}

\subsection{  The modules $\Y_j^\la$}
To compare $\Di_j^\la$ with $\G_j^\la$ we introduce the intermediate modules $\Y_j^\la$. 
First, we consider the graded $C_d$-subsupermodule of $M_{d,j}$
\begin{equation}\label{EYdj}
\Y_{d,j}:=C_d\ggi^{d,j}\M_{d,j}.
\end{equation} 
More generally, for $\la\in\Comp(n,d)$, we consider the graded 
$C_\la$-subsupermodule of $M_{\la,j}$
\begin{equation}\label{EYLa}
\Y_{\la,j}:=C_\la\ggi^{\la,j}\M_{\la,j}\simeq \Y_{\la_1,j} \boxtimes\dots\boxtimes \Y_{\la_n,j}
\end{equation}
and the graded $C_d$-supermodule
\begin{equation}\label{EUpper}
\Y_j^\la:=\GGI_\la^d\Y_{\la,j}.
\end{equation}

By Lemma~\ref{LMWtSpacesNew}, $\ggi^{d,j}\M_{d,j}$ has basis  $\{\lgath_{d,j} v_{d,j}\}$. Since $\lgath_{d,j} v_{d,j}\in\Di_{d,j}$ by Lemma~\ref{LTrivU}, we deduce that 
\begin{equation}\label{EYZ}
\Y_{d,j}\subseteq \Di_{d,j}. 
\end{equation}
So, more generally,  
\begin{equation}\label{EYZLa}
\Y_{\la,j}\subseteq \Di_{\la,j}. 
\end{equation}
We will eventually prove that $\Y_{d,j}= \Di_{d,j}$ and $\Y_{\la,j}= \Di_{\la,j}$.

Recall from Lemma~\ref{LHomGMLa} that $\Hom_{C_\la}(\G_{\la,j},\M_{\la,j})\cong\k$, so we can always pick a  homomorphism 
$$\varphi_\la\in \Hom_{C_\la}(\G_{\la,j},\M_{\la,j})$$  which spans $\Hom_{C_\la}(\G_{\la,j},\M_{\la,j})$ as a $\k$-module.

\begin{Lemma} \label{LTraceLa}
Let $\la\in\Comp(d)$. Then $\Y_{\la,j}$ is the image of\, $\varphi_\la$, where $\varphi_\la$ is a homomorphism which spans 
$\Hom_{C_\la}(\G_{\la,j},\M_{\la,j})$. 
\end{Lemma}
\begin {proof}
We have  
$$\Hom_{C_\la}(\G_{\la,j},\M_{\la,j})\cong \Hom_{C_\la}(C_\la\ggi^{\la,j},\M_{\la,j})\cong \ggi^{\la,j}\M_{\la,j}
$$
which implies the result.
\end{proof}

\begin{Lemma} \label{LGGRY}
Let $\k=\F$ and $\la\in\Comp(d)$. Then $\GGR^d_\la\,\Y_{d,j}\simeq \Y_{\la,j}$. 
\end{Lemma}
\begin{proof}
By definition, $\Y_{d,j}\subseteq\M_{d,j}$, so using Lemmas~\ref{LggIndResSymGroup}(i) and \ref{LAnotherAdj},
\begin{align*}
0&\neq\Hom_{C_d}(\Y_{d,j},M_{d,j})
\\&\simeq \Hom_{C_d}(\Y_{d,j},\GGI_{\la^\op}^d\,\M_{\la^\op,j})\\&\simeq \Hom_{C_d}(\GGR^d_\la\,\Y_{d,j},\M_{\la,j}).
\end{align*}
In particular, $\GGR^d_\la\,\Y_{d,j}\neq 0$. 

By Lemma~\ref{LTraceLa}, if $0\neq \varphi\in\Hom_{C_d}(\G_{d,j},\M_{d,j})$, then $\Im\varphi=\Y_{d,j}$. Since 
$\Y_{d,j}\subseteq\Di_{d,j}$ by (\ref{EYZ}), we consider $\varphi$ as a homomorphism $\G_{d,j}\to \Di_{d,j}$. Applying the functor $\GGR^d_\la$ to $\varphi$, we get a homomorphism 
$\bar \varphi:\GGR^d_\la\,\G_{d,j}\to\GGR^d_\la\,\Di_{d,j}$, 
or, in view of Proposition~\ref{PResGGModule} and Lemma~\ref{L5.3.1}(ii), a homomorphism $\bar \varphi:
\G_{\la,j}\to\Di_{\la,j}$. As the functor $\GGR^d_\la$ is exact,  $\Im\bar \varphi=\GGR^d_\la\Y_{d,j}$, which is non-zero by the previous paragraph. Now by Lemma~\ref{LTraceLa}, we have $\Im\bar \varphi=\Y_{\la,j}$, completing the proof.
\end{proof}

\subsection{  Projectivity of $\Di_j^\la$}
Let $\la\in\Comp(d)$. By Lemma~\ref{LTraceLa}, we have $\Y_{\la,j}=\Im \varphi_\la$ for a homomorphism $\varphi_\la$ which spans 
$\Hom_{C_\la}(\G_{\la,j},\M_{\la,j})\cong\k$. So we can consider $\varphi_\la$ as a surjective homomorphism $\G_{\la,j}\to\Y_{\la,j}$. Applying the exact functor $\GGI_\la^d$, we get the surjective homomorphism of graded $C_d$-supermodules 
\begin{equation}\label{EfUpperLa}
\phi^\la:\G^\la_j=\GGI_\la^d\,\G_{\la,j}\onto\GGI_\la^d\,\Y_{\la,j}=\Y^\la_j.
\end{equation}

\begin{Lemma} \label{L5.5.2} 
Let $\la\in\Comp(d)$. For any $g\in\Hom_{C_d}(\G^\la_j, \M_{d,j})$ there exists a unique $\bar g\in\Hom_{C_d}(\Y^\la_j, \M_{d,j})$ such that $g=\bar g\circ \phi^\la$. In particular, for any submodule $V\subseteq \M_{d,j}$, we have $\Hom_{C_d}(\G^\la_j, V)\cong \Hom_{C_d}(\Y^\la_j, V)$. 
\end{Lemma}
\begin{proof}
As $\phi^\la$ is surjective, the uniqueness of $\bar g$ is clear. 

Since $\GGI_\la^d$ is left adjoint to $\GGR_\la^d$, we have the natural (counit of adjunction) homomorphism
$$\eps:\GGI_\la^d\,\GGR_\la^d\,\M_{d,j}\to\M_{d,j}$$ 
which, for any graded $C_\la$-supermodule $W$, yields the adjunction isomorphism
$$
\al_W:\Hom_{C_\la}(W,\GGR^d_\la\,\M_{d,j})\iso \Hom_{C_d}(\GGI^d_\la\,W,\M_{d,j}), \ h\mapsto \eps\circ\GGI^d_\la(h).
$$
Taking $W=\G_{\la,j}$, let 
$$g_\la:=\al_{\G_{\la,j}}^{-1}(g)\in \Hom_{C_\la}(\G_{\la,j},\GGR^d_\la\,\M_{d,j}).
$$ 
By Lemma~\ref{LggIndResSymGroupImSch}(ii), we have $\GGR^d_\la\,\M_{d,j}\simeq \M_{\la,j}\otimes_{\k\Si_\la}\k\Si_d$, so using Lemma~\ref{LTraceLa}, there  exists a homomorphism $\bar g_\la:\Y_{\la,j}\to  \GGR^d_\la\,\M_{d,j}$ such that $g_\la=\bar g_\la\circ \varphi_\la$. Hence 
$$\GGI^d_\la(g_\la)=\GGI^d_\la(\bar g_\la)\circ \GGI^d_\la(\varphi_\la)=\GGI^d_\la(\bar g_\la)\circ \phi^\la.
$$
Therefore 
$$
g=\eps\circ \GGI^d_\la(g_\la)=\eps\circ\GGI^d_\la(\bar g_\la)\circ \phi^\la,
$$ 
and we can take $\bar g=\eps\circ\GGI^d_\la(\bar g_\la)$.
\end{proof}

\begin{Lemma} \label{L5.5.3}
Let $\k=\F$. For any $\la,\mu\in\Comp(d)$, we have 
\begin{align}
&\dim\Hom_{C_d}(\Y^\la_j,\Y^\mu_j)=|{}^\la\D^\mu|,
\label{5.5.2}
\\
&\dim\Hom_{C_d}(\G^\la_j,\Y^\mu_j)=|{}^\la\D^\mu|,
\label{5.5.3}
\\
&\dim\Hom_{C_d}(\G^\la_j,\Di^\mu_j)=|{}^\la\D^\mu|.
\label{5.5.4}
\end{align}
\end{Lemma}
\begin{proof}
We first prove (\ref{5.5.3}). By the adjointness of $\GGI_\la^d $ and $\GGR_\la^d $, 
\begin{align*}
\Hom_{C_d}(\G^\la_j,\Y^\mu_j)&=\Hom_{C_d}(\GGI_\la^d\, \G_{\la,j},\GGI_\mu^d\, \Y_{\mu,j})
\\
&\simeq \Hom_{C_\la}( \G_{\la,j},\GGR_\la^d\,\GGI_\mu^d\, \Y_{\mu,j}).
\end{align*}
Since $\G_{\la,j}$ is projective,  Theorem~\ref{TGGMackey} and Lemma~\ref{LGGRY} give
\begin{align*}
&\dim\Hom_{C_\la}( \G_{\la,j},\GGR_\la^d\,\GGI_\mu^d\, \Y_{\mu,j})
\\
=\,&\sum_{x\in{}^\la\D^\mu}\dim\Hom_{C_\la}(\G_{\la,j},\GGI^\la_{\la\cap x\mu}\,{}^x(\GGR^\mu_{x^{-1}\la\cap \mu}\,\Y_{\mu,j})
\\=\,&\sum_{x\in{}^\la\D^\mu}\dim\Hom_{C_\la}(\G_{\la,j},\GGI^\la_{\la\cap x\mu}\,{}^x(\Y_{x^{-1}\la\cap \mu,j})
\\=\,&\sum_{x\in{}^\la\D^\mu}\dim\Hom_{C_\la}(\G_{\la,j},\GGI^\la_{\la\cap x\mu}\,\Y_{\la\cap x\mu,j})
\end{align*}
To complete the proof of (\ref{5.5.3}), we show that each dimension in the sum above equals $1$. Note that the composition $\la\cap x\mu$ is a refinement of $\la$. Denote by $\nu$ the composition obtained by from $\la\cap x\mu$ by taking the parts of this refinement within each part $\la_r$ of $\la$ in the opposite order. Using Lemma~\ref{LAnotherAdj}, Proposition~\ref{PResGGModule} and Lemma~\ref{LTraceLa},
we get
\begin{align*}
\dim\Hom_{C_\la}(\G_{\la,j},\GGI^\la_{\la\cap x\mu}\,\Y_{\la\cap x\mu,j})
&=
\dim\Hom_{C_\nu}(\GGR^\la_{\nu}\,\G_{\la,j},\Y_{\nu,j})
\\&=
\dim \Hom_{C_\la}(\G_{\nu,j},\Y_{\nu,j})=1.
\end{align*}

The proof for (\ref{5.5.4}) is similar to that for (\ref{5.5.3}), using Lemma~\ref{L5.3.1}(ii) instead of Lemma~\ref{LGGRY}. 

For  (\ref{5.5.2}), as $\Y_{\mu,j}\subseteq M_{\mu,j}$ and $\GGI_\mu^d$ is exact, $Y^\mu_j=\GGI_\mu^d \Y_{\mu,j}$ is isomorphic to a submodule of $\M_{d,j}\simeq \GGI_\mu^d M_{\mu,j}$, so (\ref{5.5.2}) follows from (\ref{5.5.3}) by Lemma~\ref{L5.5.2}. 
\end{proof}

\begin{Theorem} \label{TProjGen} 
Let $\k=\F$. 
\begin{enumerate}
\item[{\rm (i)}] We have $\Di_{d,j}=\Y_{d,j}$, so $\Di_{d,j}$ can be characterized as the image of any non-zero homomorphism from $\G_{d,j}$ to $\M_{d,j}$.
\item[{\rm (ii)}] For each $\la\in\Comp(d)$, we have that\, $\Di_j^\la$ is a projective graded supermodule over $\IS_{d,j}$. \item[{\rm (iii)}] If $n\geq d$ then $\bigoplus_{\la\in\Comp(n,d)}\Di_j^\la$ is a projective generator for $\IS_{d,j}$. 
\end{enumerate} 
\end{Theorem}
\begin{proof}
Fix $n\geq d$ and set 
\begin{align*}
\textstyle
\Di:=
\bigoplus_{\la\in\Comp(n,d)} \Di^\la_j,
\quad 
\Y:=\bigoplus_{\la\in\Comp(n,d)}\Y^\la_j,
\quad 
\G:=\bigoplus_{\la\in\Comp(n,d)}\G^\la_j.
\end{align*}

Let $\la\in\Par(d)$. By (\ref{EGUpperLa}), we have $C_n\ggi^{\la,j}\simeq \G^\la_j$. As $n\geq d$, we can consider $\la$ as an element of $\Comp(n,d)$, and so $\G^\la_j$ is a direct summand of $\G$. By (\ref{EFChL}), we have $\Hom_{C_n}(C_n\ggi^{\la,j},L_j(\la))\simeq \ggi^{\la,j}L_j(\la)\neq 0$, and so, by  Theorem~\ref{LAmountAmount}(ii), every composition factor of $M_{d,j}$ appears in the head of $\G$. Hence $\Hom_{C_d}(\G,X)\neq 0$ for any $X\in\mod{C_d}$ all of whose composition factors are composition factors of $M_{d,j}$. 

On the other hand, for each $\la\in\Comp(n,d)$, recalling (\ref{EYZLa}) and 
applying the exact functor $\GGI_\la^d$ to the inclusions $\Y_{\la,j}\subseteq\Di_{\la,j}\subseteq \M_{\la,j}$, we get 
$$
\Y^\la_j\subseteq\Di^\la_j\subseteq 
\GGI_\la^d \,\M_{\la,j}\simeq
\M_{d,j}.
$$ 
In particular, $\Y\subseteq \Ga$ and all composition factors of $\Ga/Y$ are composition factors of $M_{d,j}$. 
By Lemma~\ref{L5.5.3}, we have 
$
\dim\Hom_{C_d}(\G,\Y)=\dim\Hom_{C_d}(\G,\Di). 
$ 
As $\G$ is projective this implies $\Hom_{C_d}(\G,\Di/\Y)=0$, so $\Di/\Y=0$ by the previous paragraph. 
This proves (i).

Now, $\Y=\Di$ is a quotient of $\G$, see (\ref{EfUpperLa}), and  so  by Lemma~\ref{L5.5.3}, 
$$\dim\Hom_{C_d}(\Di,\Di)=\dim\Hom_{C_d}(\G,\Di).$$  This verifies the conditions of Lemma~\ref{LSchub}, so $\Di$ is
projective when considered as a graded $C_d/\Ann_{C_d}(\Di)$-supermodule. Hence so is every summand $\Di^\la_j$ of $\Di$. 
Since all summands of $\Di$ are isomorhic to submodules of $\M_{d,j}$, and one of the summands 
$\Di^{\om_d}_j \simeq \M_{d,j}$, we have $\Ann_{C_d}(\Di)=\Ann_{C_d}(\M_{d,j})$ and $C_d/\Ann_{C_d}(\Di)=C_d/\Ann_{C_d}(\M_{d,j})=\IS_{d,j}$, proving (ii). 


By Theorem~\ref{TSchurEnd3}, we have 
$\End_{\IS_{d,j}}(\Di)\cong S(n,d)$. For $n\geq d$, the number of irreducible representations of $S(n,d)$ is equal to $|\Par(d)|$. So by Fitting's Lemma~\cite[1.4]{Landrock}, the graded $\IS_{d,j}$-supermodule $\Di$ has the same amount of non-isomorphic irreducible modules appearing in its head. Taking into account Theorem~\ref{LAmountAmount}(ii), every composition factor of $\M_{d,j}$ appears in the head of $\Di$. 
Since $\M_{d,j}$ is a faithful $\IS_{d,j}$-module, every irreducible $\IS_{d,j}$-module appears as some composition factor of $\M_{d,j}$. Thus every irreducible $\IS_{d,j}$-module appears in the head of $\Di$, so $\Di$ is a projective generator, proving (iii). 
\end{proof}

Recall Lemma~\ref{LBarParabolic}. In particular, we have a subalgebra $\IS_{\la,j}\subseteq \bar\ggi_\la\IS_{d,j}\bar\ggi_\la$. 

\begin{Corollary} \label{CHCIC} 
Let $\k=\F$ and $\la\in\Comp(d)$. 
\begin{enumerate}
\item[{\rm (i)}] The functor $\GGI_\la^d$ 
restricts to a functor $\mod{\IS_{\la,j}}\to\mod{\IS_{d,j}}$. Moreover, for $V\in\mod{\IS_{\la,j}}$, we have $\GGI_\la^dV\simeq \IS_{d,j}\bar\ggi_\la\otimes_{\IS_{\la,j}}V$.  
\item[{\rm (ii)}] The functor $\GGR_\la^d$ restricts to a functor $\mod{\IS_{d,j}}\to\mod{\IS_{\la,j}}$. 
\end{enumerate}
\end{Corollary}
\begin{proof}
We prove (i), the proof of (ii) being similar. As the functor $\GGI_\la^d$ is exact, it suffices to check the first statement on projective graded $\IS_{\la,j}$-supermodules. By Theorem~\ref{TProjGen}, every indecomposable projective graded  $\IS_{\la,j}$-supermodule is a submodule of $\M_{\la,j}$,  so we just need to check that $\GGI_\la^d\M_{\la,j}$ is a  $\IS_{d,j}$-module. But  $\GGI_\la^d\M_{\la,j}\simeq\M_{d,j}$. The second statement follows from the first using adjointness of induction and restriction. 
\end{proof}

\begin{Corollary}\label{CJantzen}
Let $\la\in\Comp(d)$, $W \in\mod{\IS_{\la,j}}$ and $V \in\mod{\IS_{d,j}}$. Then  
$$
\Ext^1_{\IS_{\la,j}}(W,\GGR_\la^d V)\cong \Ext^1_{\IS_{d,j}}(\GGI_\la^d W, V).
$$
\end{Corollary}
\begin{proof}
To prove the claim, the adjoint functor property gives us an isomorphism
of functors $\Hom_{\IS_{\la,j}} (W, -) \circ \GGR_\la^d	\cong \Hom_{\IS_{d,j}}	(\GGI_\la^d W, -).$ 
Since $\GGR_\la^d$ is exact and sends injectives to injectives (being adjoint to the exact functor $\GGI_\la^d$), an application of \cite[I.4.1(3)]{Jantzen} completes the proof of the claim. 
\end{proof}

\begin{Corollary} \label{LDualityIS} 
Let $\k=\F$, $\la\in\Comp(d)$ and $V\in \mod{\IS_{\la,j}}$. Considering $V$ as a graded $C_\la$-supermodule via inflation, we have the graded $C_\la$-supermodule $V^\sharp$. Then $V^\sharp$ factors through to a  $\IS_{\la,j}$-module. 
\end{Corollary}
\begin{proof}
By Theorem~\ref{TProjGen}(iii), $V$ is a direct sum of subquotients of $\M_{\la,j}$. Since $\M_{\la,j}^\sharp\simeq \M_{\la,j}$ by Lemma~\ref{LSocleHead} and (\ref{ESharpBoxTimes}), $V^\sharp$ is also a direct sum of subquotients of $\M_{\la,j}$. Hence $V^\sharp$ factors through to a module over $\IS_{\la,j}$. 
\end{proof}

Recall from (\ref{EA}), (\ref{EAEquals}) that $\bideg(\lgath_{\la,j} v_{\la,j})=(a_{\la,j},dj\pmod{2})$. 

\begin{Corollary} \label{CZCIdem}
Let $\k=\F$, $\la\in\Comp(d)$ and set 
$$\bar\ggi^{\la,j}:=\ggi^{\la,j}+\Ann_{C_d}(\M_{d,j})\in C_d/\Ann_{C_d}(\M_{d,j})=\IS_{d,j}.
$$ 
Then $\funQ^{a_{\la,j}}\,\Uppi^{dj}\,\IS_{d,j}\bar \ggi^{\la,j}\simeq \Di_j^\la$ as  
graded $\IS_{d,j}$-supermodules.  
\end{Corollary}
\begin{proof}
Let $\la=(\la_1,\dots,\la_n)$. 
Using Corollary~\ref{CHCIC}(i), we get 
\begin{align*}
&\GGI^d_\la\big(\funQ^{a_{\la_1,j}}\,\Uppi^{\la_1j}\,\IS_{\la_1,j}\bar\ggi^{\la_1,j}\,\boxtimes\,\dots\,\boxtimes\, \funQ^{a_{\la_n,j}}\,\Uppi^{\la_1j}\,\IS_{\la_n,j}\bar\ggi^{\la_n,j}\big)
\\
\simeq\ & \funQ^{a_{\la_1,j}+\dots+a_{\la_n,j}}\,\Uppi^{\la_1j+\dots+\la_nj}\,\GGI^d_\la(
\IS_{\la,j}\bar\ggi^{\la,j})
\\
\simeq\ & \funQ^{a_{\la,j}}\,\Uppi^{dj}\,\IS_{d,j}\bar\ggi_\la\otimes_{\IS_{\la,j}}
\IS_{\la,j}\bar \ggi^{\la,j}
\\
\simeq\ & \funQ^{a_{\la,j}}\,\Uppi^{dj}\,\IS_{d,j}\bar \ggi^{\la,j}.
\end{align*}
Since by Lemma~\ref{LUpperInd}, we have $\Di_j^\la\simeq \GGI_\la^d\big(\Di_{\la_1,j}\boxtimes\dots\boxtimes\Di_{\la_n,j}\big)$, this reduces the proof of the lemma to the $n=1$ case, i.e. to proving that $\funQ^{a_{d,j}}\,\Uppi^{dj}\,\IS_{d,j}\bar \ggi^{d,j}\simeq \Di_{d,j}$. 

By Lemma~\ref{LMWtSpacesNew}, $\ggi^{d,j}\M_{d,j}$ has basis  $\{\lgath_{d,j} v_{d,j}\}$, so, by (\ref{EYdj}), there is a bidegree $(0,\0)$ surjective homomorphism 
$
\funQ^{a_{d,j}}\,\Uppi^{dj}\,\IS_{d,j}\bar\ggi_d\to \Y_{d,j},\ \bar\ggi_d\mapsto \lgath_{d,j} v_{d,j}$. But $\Y_{d,j}=\Di_{d,j}$ is projective by 
Theorem~\ref{TProjGen}. Hence $\Di_{d,j}$ is a direct summand of $\funQ^{a_{\la,j}}\IS_{d,j}\bar\ggi_d$. But $\IS_{d,j}\bar\ggi_d$  is a quotient of $\G_{d,j}$ so it is indecomposable by Lemma~\ref{LProjCover}. 
\end{proof}

\section{Morita equivalence functors}

\subsection{  The functors $\funa_{n,d,j}$ and $\funb_{n,d,j}$}
\label{SSMorISS}

In view of Theorem~\ref{TSchurEnd3} we can consider 
\begin{equation}\label{EGaJND}
\Di_j(n,d):=\bigoplus_{\la\in\Comp(n,d)}\Di_j^\la
\end{equation}
as a graded $(\IS_{d,j},S(n,d)^\op)$-bisupermodule. 
Identifying $S(n,d)^\op$ with $S(n,d)$ using the anti-automorphism $\tau$ of Lemma~\ref{LSchurAnti}, we consider $\Di_j(n,d)$ as a graded $(\IS_{d,j},S(n,d))$-bisupermodule. 
This yields the functors
\begin{align*}
\funa_{n,d,j}&:\mod{\IS_{d,j}}\to\mod{S(n,d)},\ V\mapsto \Hom_{\IS_{d,j}}(\Di_j(n,d),V),
\\
\funb_{n,d,j}&:\mod{S(n,d)}\to\mod{\IS_{d,j}},\ W\mapsto \Di_j(n,d)\otimes_{S(n,d)}W. 
\end{align*}
If $n\geq d$, {\em which we assume from now on}, and $\k=\F$, then by Theorem~\ref{TProjGen}(iii), these functors are mutually quasi-inverse (graded super)equivalences. 

Recall the $S(n,d)$-modules $\Di^\la_{n,d}$ and $\La^\la_{n,d}$ from \S\ref{SSSchurRep}. 

\begin{Lemma} \label{L6.3.1}
Let $\k=\F$ and $\la\in\Comp(n,d)$. Then 
$$
\Di_j^\la\simeq\funb_{n,d,j}(\Di^\la_{n,d})\quad{and}\quad \M_{d,j}\tty_\la\simeq\funb_{n,d,j}(\La^\la_{n,d}).
$$
\end{Lemma}
\begin{proof}
For the first claim, 
we have $\Di^\la_{n,d}\simeq S(n,d)\xi_\la$ by Lemma~\ref{1.3b}(i), so 
$$
\funb_{n,d,j}(\Di^\la_{n,d})\simeq \Di_j(n,d)\otimes_{S(n,d)}S(n,d)\xi_\la\simeq  \Di_j(n,d)\cdot\xi_\la=\tau(\xi_\la)(\Di_j(n,d))=\Di_j^\la,
$$
where the last equality comes from the definition of the action of $S(n,d)$ on $\Di_j(n,d)$ from Theorem~\ref{TSchurEnd3}.

For the second claim, we have $\La^\la_{n,d}\simeq S(n,d)\kappa(\tty_\la)$ by Lemma~\ref{1.3b}(ii), so
\begin{align*}
\funb_{n,d,j}(\La^\la_{n,d})&\simeq \Di_j(n,d)\otimes_{S(n,d)}S(n,d)\kappa(\tty_\la)
\\
&\simeq  \Di_j(n,d)\cdot\kappa(\tty_\la)
\\
&=\tau(\kappa(\tty_\la))(\Di_j(n,d)).
\end{align*}
Note using Lemma~\ref{LKappa} that 
$$
\tau(\kappa(\tty_\la))
=\tau\big(\sum_{g\in \Si_\la}\sgn(g)\xi^g_{\om_d,\om_d}\big)=\sum_{g\in \Si_\la}\sgn(g)\xi^{g^{-1}}_{\om_d,\om_d}
=\sum_{g\in \Si_\la}\sgn(g)\xi^g_{\om_d,\om_d}.
$$
Note that $\Di^{\om_d}_j=\M_{d,j}$. So by the definition of the action of $S(n,d)$ on $\Di_j(n,d)$ from Theorem~\ref{TSchurEnd3}, we now have
$$
\tau(\kappa(\tty_\la))(\Di_j(n,d))=\sum_{g\in \Si_\la}\sgn(g)\xi^g_{\om_d,\om_d}(\Di_j(n,d))\simeq\M_{d,j}\tty_\la,
$$
as required.
\end{proof}

Recall the irreducible graded $C_{d}$-supermodules $\{L_j(\la)\mid\la\in\Par(d)\}$ from Theorem~\ref{LAmountAmount}(ii). 
We consider these as irreducible graded $\IS_{d,j}$-supermodules. 
We have the irreducible (graded) $S(n,d)$-(super)modues $\{L_{n,d}(\la)\mid\la\in \Comp_+(n,d)\}$ from \S\ref{SSSchurRep}. 
Since $n\geq d$, we can identify the sets $\Par(d)$ and $\Comp_+(n,d)$. Then:

\begin{Lemma} \label{LfunbIrr}
Let $\k=\F$. For $\la\in \Comp_+(n,d)$, we have $\funb_{n,d,j}(L_{n,d}(\la))\simeq  L_j(\la)$. 
\end{Lemma}
\begin{proof}
Since $\funb_{n,d,j}$ is an equivalence, $\funb_{n,d,j}(L_{n,d}(\la))$ is irreducible, and by definition, it is a composition factor of $\Di_j(n,d)$, hence of $M_{d,j}$. By Theorem~\ref{LAmountAmount}(ii), we must have $\funb_{n,d,j}(L_{n,d}(\la))\simeq  L_j(f(\la))$ for some bijection $f$ on $\Par(d)$. We prove that $f(\la)=\la$ by induction on the dominance order on $\Par(d)$. For the smallest partition $\om_d$, we have 
$$
\funb_{n,d,j}(L_{n,d}(\om_d))\stackrel{(\ref{EExtDef})}{\simeq}\funb_{n,d,j}(\La^{(d)}_{n,d})
\stackrel{\text{Lemma \ref{L6.3.1}}}{\simeq}\M_{d,j}\tty_{d}
\stackrel{\text{Corollary \ref{CCompFactorsMy}}}{\simeq}L_j(\om_d).
$$
For an arbitrary $\la$, suppose we already know that $\funb_{n,d,j}(L_{n,d}(\mu))\simeq  L_j(\mu)$ 
for all $\mu\lhd \la$. 
By Lemma~\ref{LLaCompfactors}, the composition factors of $\funb_{n,d,j}(\La_{n,d}^{\la'})$ are of the form $L_j(\mu)$ with $\mu\lhd \la$, as well as $L_j(f(\la))$ (appearing once). On the other hand, by Lemma~\ref{L6.3.1}, we have $\funb_{n,d,j}(\La_{n,d}^{\la'})\simeq \M_{d,j}\tty_{\la'}$, whose composition factors, by Corollary~\ref{CCompFactorsMy}, are of the form $L_j(\mu)$ with $\mu\lhd \la$, as well as $L_j(\la)$ (appearing once). It follows that $f(\la)=\la$. 
\end{proof}

Recall from \S\ref{SSSchurRep} that $S(n,d)$ is a quasi-hereditary algebra with weight poset $\Comp_+(n,d)$ partially ordered by the dominance order $\unlhd$. As $n \geq d$, we can identify $\Comp_+(n,d)$ with the set $\Par(d)$ of partitions of $d$. So for $\la\in\Par(d)$, the algebra $S(n,d)$ has the standard module $\De_{n,d}(\la)$ and the costandard module $\nabla_{n,d}(\la)$. For all $\la\in\Par(d)$, define the (graded) $\IS_{d,j}$-supermodules
\begin{align*}
\De_j(\la):=\funb_{n,d,j}(\De_{n,d}(\la))\quad\text{and}\quad  
\nabla_j(\la):=\funb_{n,d,j}(\nabla_{n,d}(\la)).
\end{align*}

We need to explain why these definitions are independent of the choice of $n\geq d$. Take $N \geq n\geq d$ and identify $S(n,d)$  with the idempotent truncation $\xi(N,n;d)S(N,d)\xi(N,n;d)$  using (\ref{ENnIso}). Recalling the Morita equivalence from Lemma~\ref{LNnEquiv}, it suffices to prove the isomorphism of functors 
\begin{equation}\label{EIsoFunctors}
\funb_{N,d,j}\circ\Ind_{S(n,d)}^{S(N,d)}\simeq \funb_{n,d,j}. 
\end{equation}
To see this, note that $\Di_j(n,d)\simeq \Di_j(N,d)\cdot \xi(N,n;d)$. Now, for $W\in\mod{S(n,d)}$, we have functorial isomorphisms
\begin{align*}
\funb_{N,d,j}\circ\Ind_{S(n,d)}^{S(N,d)}(W)
&=\Di_j(N,d)\otimes_{S(N,d)}\otimes S(N,d)\xi(N,n;d)\otimes_{S(n,d)}(W)
\\
&\simeq \Di_j(N,d)\xi(N,n;d)\otimes_{S(n,d)}(W)
\\&\simeq \Di_j(n,d)\otimes_{S(n,d)}(W)
\\&=\funb_{n,d,j}(W).
\end{align*}

Finally, note using Lemmas~\ref{LSharpHom},\,\ref{LGGIrrSelfD} and general theory of quasi-hereditary algebras (see for example \cite[(1.2)]{CPS}) that 
\begin{equation}\label{EDeNaDuals}
\De_j(\la)^\sharp\simeq \nabla_j(\la)\qquad(\la\in\Par(d)). 
\end{equation}

\begin{Lemma} \label{L6.3.4}
Let $\k=\F$. We have that $\dim\IS_{d,j}=\sum_{\la\in\Par(d)}(\dim \De_j(\la))^2$. 
\end{Lemma}
\begin{proof}
This is a general fact on finite-dimensional quasi-hereditary algebras with duality. To give more details, let $\De_j^\op(\la)=\nabla_j(\la)^*$ be the right $\IS_{d,j}$-module with the action given by $(f\cdot s)(v)=f(sv)$ for $f\in \nabla_j(\la)^*,\, s\in\IS_{d,j},\,v\in \nabla_j(\la)$. 
Note by (\ref{EDeNaDuals}) that $\dim\De_j^\op(\la)=\dim\De_j(\la)$. 
Finally, as an $(\IS_{d,j},\IS_{d,j})$-bimodule, $\IS_{d,j}$ has a filtration with subquotients $\{\De_j(\la)\otimes\De_j^\op(\la)\mid\la\in\Par(d)\}$, see e.g. \cite[(2.20) and Proposition 3.5]{KMuth}.  
\end{proof}

Recall the right $S(n,d)$-modules $\tilde\De_{n,d}(\la)$ from \S\ref{SSSchurRep} and Lemma~\ref{1.2e}.

\begin{Lemma} \label{LZFiltration} 
Let $\k=F$. 
As a graded $(\IS_{d,j},S(n,d))$-bisuprmodule, $\Di_j(n,d)$ has a filtration with subquotients $\De_j(\la)\otimes\tilde\De_{n,d}(\la)$ appearing once for each $\la\in\Par(d)$ and ordered in any way refining the dominance order so that factors corresponding to more dominant $\la$ appear lower in the filtration. 
\end{Lemma}
\begin{proof}
The functor $\Di_j(n,d)\otimes_{S(n,d)}-$ can be viewed as an exact functor from the category of $(S(n,d),S(n,d))$-bimodules to the category of $(\IS_{d,j},S(n,d))$-bimodules. Then
\begin{align*}
\Di_j(n,d)\otimes_{S(n,d)}(\De_{n,d}(\la)\otimes\tilde\De_{n,d}(\la))&\simeq
(\Di_j(n,d)\otimes_{S(n,d)}\De_{n,d}(\la))\otimes\tilde\De_{n,d}(\la)
\\&\simeq \De_j(\la)\otimes\tilde\De_{n,d}(\la).
\end{align*}
So applying $\Di_j(n,d)\otimes_{S(n,d)}-$ to the filtration of Lemma~\ref{1.2e} yields the result. 
\end{proof}

We conclude this subsection with alternative descriptions of the standard modules $\De_j(\la)$ for $\IS_{d,j}$ without reference to 
the Morita equivalence $\IS_{d,j}\sim_\Mor S(n,d)$. This will be needed in \S\ref{SSInt}. 
Recall the element $u_\la$ from Lemma~\ref{L2.2.1}.

\begin{Theorem} \label{TDeAlternative} 
Let $\k=\F$ and $\la\in\Par(d)$. 
\begin{enumerate}
\item[{\rm (i)}] We have that $\Hom_{\IS_{d,j}}(\Di_j^\la,\M_{d,j}\tty_{\la'})\cong \F$, and the image of any non-zero homomorphism from $\Hom_{\IS_{d,j}}(\Di_j^\la,\M_{d,j}\tty_{\la'})$ is isomorphic to $\De_j(\la)$.
\item[{\rm (ii)}] We have that $\De_j(\la)\simeq \Di_j^\la u_{\la'}\ttx_{\la'}$. 
\end{enumerate}
\end{Theorem}
\begin{proof}
By Lemma~\ref{L6.3.1}, we have $\Di_j^\la\simeq \funb_{n,d,j}(\Di^\la_{n,d})$ and $\M_{d,j}\tty_{\la'}\simeq \funb_{n,d,j}(\La^{\la'}_{n,d})$. Since $\De_j(\la)=\funb_{n,d,j}(\De_{n,d}(\la)$ and $\funb_{n,d,j}$ is an equivalence, (i) now follows from Lemma~\ref{1.3d}. 

We have $\Di_j^\la\supseteq \M_{d,j}\ttx_\la$ by (\ref{EMGa}). 
Moreover, $\M_{d,j}$ is a faithful right $\F\Si_d$-module by Theorem~\ref{TEndMd}. Since $\ttx_\la u_{\la'}\tty_{\la'}\neq 0$ in $\F\Si_d$ by Lemma~\ref{L2.2.1} (with $\la=\la'$), we conclude that $\M_{d,j}\ttx_\la u_{\la'}\tty_{\la'}\neq 0$. Hence $\Di_j^\la u_{\la'}\tty_{\la'}\neq 0$. Finally, observe that $\Di_j^\la u_{\la'}\tty_{\la'}$ is both a homomorphic image of $\Di_j^\la$ and a submodule of $\M_{d,j}\tty_{\la'}$. So (ii) follows from (i). 
\end{proof}

\subsection{Induction and Morita equivalence}
Fix $n\geq d$. We now prove that the Morita equivalence 
$\IS_{d,j}\sim_\Mor S(n,d)$ of \S\ref{SSMorISS} 
intertwines Gelfand-Graev induction and the tensor product  (\ref{ESchurTensProd}) for the classical Schur algebras.

\begin{Theorem} \label{TMoritaGGI}
Let $\k=\F$, $m\in\N_+$, $\ud=(d_1,\dots,d_m)\in\Comp(m,d)$, and\, $W_t\in\mod{S(n,d_t)}$ for $t=1,\dots,m$. Then there is a functorial isomorphism 
$$
\GGI_{d_1,\dots,d_m}^d\big((\funb_{n,d_1,j}W_1)\boxtimes\dots\boxtimes (\funb_{n,d_m,j}W_m)\big)\simeq 
\funb_{n,d,j}(W_1\otimes\dots\otimes W_m).
$$
\end{Theorem}
\begin{proof}
Choose $\un=(n_1,\dots,n_m)\in\Comp(m,n)$ with $n_t\geq d_t$ for  $t=1,\dots,m$, and recall the idempotents $\xi(\un,\ud)\in S(n,d)$ from (\ref{EIdUnUd}) and the functors
$\Res^{S(n,d)}_{S(\un,\ud)}$ and $\Ind^{S(n,d)}_{S(\un,\ud)}$ 
from (\ref{EResIndLevi}). 
For $t=1,\dots,m$, let $U_t:=\Res^{S(n,d_t)}_{S(n_t,d_t)}W_t$. By Lemma~\ref{LNnEquiv}, we have $W_t\simeq \Ind^{S(n,d_t)}_{S(n_t,d_t)}U_t$.

Recall the notation (\ref{EUNUD}). So, for $\la\in\Comp(\un,\ud)$ we have that 
\begin{align*}
\la^1&:=(\la_1,\dots,\la_{n_1}) \in\Comp(n_1,d_1),\\ 
 \la^2&:=(\la_{n_1+1},\dots,\la_{n_1+n_2})\in\Comp(n_2,d_2),\\ &\vdots\\ 
  \la^m&:=(\la_{n_1+\dots+n_{m-1}+1},\dots,\la_{n})\in\Comp(n_m,d_m).
\end{align*}
Note that any composition in $\Comp(\un,\ud)$ is a refinement of the composition $\ud$, and consider the set of triples 
$$
\Om:=\{(\la,\mu,g)\mid \la,\mu\in \Comp(\un,\ud),\,g\in{}^\la\D^\mu_\ud\}. 
$$
For $(\la,\mu,g)\in\Om$, we have $\la^t,\mu^t\in\Comp(n_t,d_t)$ for $t=1,\dots,m$, and $g=(g_1,\dots,g_m)\in\Si_\ud=\Si_{d_1}\times\dots\times\Si_{d_m}$, with each $g_t\in{}^{\la^t}\D^{\mu^t}_{d_t}$.  So we have a basis
$$
\{\phi^g_{\la,\mu}:=\phi^{g_1}_{\la^1,\mu^1}\otimes\dots\otimes \phi^{g_1}_{\la^1,\mu^1}\mid (\la,\mu,g)\in\Om\}
$$
of $S(\un,\ud)=S(n_1,d_1)\otimes\dots\otimes S(n_m,d_m)$. 
Recalling the Levi subalgebra embedding 
$
S(\un,\ud)\subseteq \xi(\un,\ud)S(n,d)\xi(\un,\ud)$
from (\ref{ELeviEmb}), note that under this embedding the basis element $\phi^g_{\la,\mu}\in S(\un,\ud)$ maps to the element of $S(n,d)$ with the same name (which makes sense since $\Si_\ud\leq \Si_d$ and $\Comp(\un,\ud)\subseteq \Comp(n,d)$).

Define the graded $(\IS_{\ud,j},S(\un,\ud))$-bisupermodule 
$$
\textstyle\Di_{\un,\ud,j}:=\bigoplus_{\la\in \Comp(\un,\ud)}\GGI_\la^\ud\,\Di_{\la,j}\simeq 
\Di(n_1,d_1)\boxtimes\dots\boxtimes  \Di(n_m,d_m).
$$
Then $\Di_j(\un,\ud):=\GGI_\ud^d\, \Di_{\un,\ud,j}$ is a graded $(\IS_{d,j},S(\un,\ud))$-bisupermodule in a natural way, and by the transitivity of Gelfand-Graev induction we have 
$\Di_j(\un,\ud)\simeq \bigoplus_{\la\in \Comp(\un,\ud)}\Di_j^\la.$
So $\Di_j(\un,\ud)$ can be identified with the direct summand $\Di_j(n,d)\xi(\un,\ud)$ of $\Di_j(n,d)$. Identifying $\xi(\un,\ud)S(n,d)\xi(\un,\ud)$ with $\End_{\IS_{d,j}}(\Di_j(n,d)\xi(\un,\ud))^\sop$, we have a Levi subalgebra 
$$
S(\un,\ud)\subseteq \xi(\un,\ud)S(n,d)\xi(\un,\ud)=\End_{\IS_{d,j}}(\Di_j(n,d)\xi(\un,\ud))^\sop.
$$
In this way, we consider $\Di_j(n,d)\xi(\un,\ud)$ as a graded $(\IS_{d,j},S(\un,\ud))$-bisupermodule. Moreover, it is clear that we have an isomorphism 
$\Di_j(\un,\ud)\simeq \Di_j(n,d)\xi(\un,\ud)$ of graded $(\IS_{d,j},S(\un,\ud))$-bisupermodules. 

We now have 
\begin{align*}
&\funb_{n,d,j}(W_1\otimes\dots\otimes W_m)
\\
\simeq\ \ \ \ \ \  &\funb_{n,d,j}\big((\Ind^{S(n,d_1)}_{S(n_1,d_1)}U_1)\otimes\dots\otimes \Ind^{S(n,d_m)}_{S(n_m,d_m)}U_m\big)
\\
\stackrel{\text{Lemma~\ref{1.5d}}}{\simeq}
\ &\funb_{n,d,j}\big(\Ind^{S(n,d)}_{S(\un,\ud)}(U_1\boxtimes\dots\boxtimes U_m)\big)
\\
\simeq
\ \ \ \ \ \ &\Di_j(n,d)\otimes_{S(n,d)}S(n,d)\xi(\un,\ud)\otimes_{S(\un,\ud)}(U_1\boxtimes\dots\boxtimes U_m)
\\
\simeq 
\ \ \ \ \ \ &\Di_j(n,d)\xi(\un,\ud)\otimes_{S(\un,\ud)}(U_1\boxtimes\dots\boxtimes U_m)
\\
\simeq 
\ \ \ \ \ \ &\Di_j(\un,\ud)\otimes_{S(\un,\ud)}(U_1\boxtimes\dots\boxtimes U_m)
\\
\simeq 
\ \ \ \ \ \ &(\GGI_\ud^d\,\Di_{\un,\ud,j})\otimes_{S(\un,\ud)}(U_1\boxtimes\dots\boxtimes U_m)
\\
\simeq 
\ \ \ \ \ \ &\GGI_\ud^d\big(\Di_{\un,\ud,j}\otimes_{S(\un,\ud)}(U_1\boxtimes\dots\boxtimes U_m)\big)
\\
\simeq 
\ \ \ \ \ \ &\GGI_\ud^d\big(\funb_{n_1,d_1,j}(U_1)\boxtimes\dots\boxtimes \funb_{n_m,d_m,j}(U_m)\big)
\\
\stackrel{(\ref{EIsoFunctors})}{\simeq} 
\ \ \ \,&\GGI_\ud^d\big(\funb_{n,d_1,j}(\Ind^{S(n,d_1)}_{S(n_1,d_1)}U_1)\boxtimes\dots\boxtimes \funb_{n_m,d_m,j}(\Ind^{S(n,d_m)}_{S(n_m,d_m)}U_m)\big)
\\
\simeq 
\ \ \ \ \ \ &\GGI_\ud^d\big(\funb_{n,d_1,j}(W_1)\boxtimes\dots\boxtimes \funb_{n_m,d_m,j}(W_m)\big),
\end{align*}
as required. 
\end{proof}

As a first application we obtain `commutativity' of the Gelfand-Graev induction in the following sense:

\begin{Corollary} 
Let $\k=\F$, $c,d\in\N$,\, $U\in\mod{\IS_{c,j}}$ and\, $W\in\mod{\IS_{d,j}}$. Then $$\GGI_{c,d}^{c+d}(U\boxtimes W)\simeq   \GGI_{d,c}^{c+d}(W\boxtimes U).$$ 
\end{Corollary}
\begin{proof}
For $n\geq c+d$, we have using the theorem, 
\begin{align*}
\GGI_{c,d}^{c+d}(U\boxtimes W)&\simeq \GGI_{c,d}^{c+d}\big(\funb_{n,c,j}(\funa_{n,c,j}(U))\boxtimes \funb_{n,d,j}(\funa_{n,d,j}(W))\big)
\\
&\simeq
\funb_{n,c+d,j}\big(\funa_{n,c,j}(U)\otimes \funa_{n,d,j}(W)\big)
\\
&\simeq
\funb_{n,c+d,j}\big(\funa_{n,d,j}(W)\otimes \funa_{n,c,j}(U)\big)
\\
&\simeq
\GGI_{d,c}^{c+d}\big(\funb_{n,d,j}(\funa_{n,d,j}(W))\boxtimes (\funb_{n,c,j}(\funa_{n,c,j}(U))\big)
\\
&\simeq
\GGI_{d,c}^{c+d}(W\boxtimes U),
\end{align*}
as desired.
\end{proof}

As a second application, we show 
that Gelfand-Graev induction and restriction preserve the property of possessing (co)standard filtrations:

\begin{Theorem} \label{4.2f} 
Let $\k=\F$ and $\ud=(d_1,\dots,d_m)\in\Comp(d)$. Then: 
\begin{enumerate}
\item[{\rm (i)}] The functor $\GGI_\ud^d$	sends $\IS_{\ud,j}$-modules with standard (resp. costandard) filtrations to $\IS_{d,j}$-modules with standard (resp. costandard) filtrations.
\item[{\rm (ii)}] The functor $\GGR_\ud^d$	sends $\IS_{d,j}$-modules with standard (resp. costandard) filtrations to $\IS_{\ud,j}$-modules with standard (resp. costandard) filtrations. 
\end{enumerate}
\end{Theorem}
\begin{proof}
(i) For standard filtrations, it suffices to check that for arbitrary $\la^1\in\Par(d_1),\dots, \la^m\in\Par(d_m)$ the graded supermodule   
$\GGI_\ud^d(\De_j(\la^1)\boxtimes\dots\boxtimes\De_j(\la^m))$ 
has a  standard filtration. 
Let $n\geq d$. 
By  Theorem~\ref{TMoritaGGI}, we have 
$$
\GGI_\ud^d(\De_j(\la^1)\boxtimes\dots\boxtimes\De_j(\la^m))\cong \funb_{n,d,j}(\De_{n,d_1}(\la^1)\otimes\dots\otimes\De_{n,d_m}(\la^m)). 
$$
So the result on standard filtrations follows since $\De_{n,d_1}(\la^1)\otimes\dots\otimes\De_{n,d_m}(\la^m)$ has a standard filtration as an $S(n,d)$-module by Lemma~\ref{1.3c}.

The argument for costandard 
filtrations is similar or can be deduced from the result on standard filtrations using (\ref{EDeNaDuals}) and Lemma~\ref{LSharpIndComm}

(ii) We prove the result for costandard filtrations; the analogous result for standard 
filtrations follows using (\ref{EDeNaDuals}) and Lemma~\ref{LSharpResComm}. 
Take $V \in\mod{\IS_{d,j}}$ with a costandard filtration. Using the cohomological criterion for costandard filtrations \cite[A2.2(iii)]{Donkin}, we need to show that $\Ext^1_{\IS_{\ud,j}}(W,\GGR_\ud^d V)=0$ for all $W \in\mod{\IS_{\ud,j}}$ with a standard filtration. For such $W$, by (i) and the cohomological criterion for costandard filtrations, we have $\Ext^1_{\IS_{d,j}}(\GGI_\ud^d W, V)=0$. So the result follows from Corollary~\ref{CJantzen}. 
\end{proof}

\subsection{Double centralizer properties}
Throughout the subsection we assume that $\k=\F$ and $n\geq d$. In particular, we identify $\Par(d)=\Comp_+(n,d)$.
Recall the theory of tilting modules and Ringel duality from \S\ref{SSTilt}, in particular the tilting modules $T_{n,d}(\la)$ for the classical Schur algebra $S(n,d)$. 
For  $\la\in \Par(d)$, define 
\begin{equation}\label{ETilt}
T_j(\la) := \funb_{n,d,j}(T_{n,d}(\la)).
\end{equation}
Since $\be_{n,d}$ is an equivalence, $\{T_j(\la) \mid \la\in\Par(d)\}$ are the indecomposable tilting modules for $\IS_{d,j}$ (independent of the choice of $n\geq d$ for example thanks to  (\ref{EIsoFunctors})).

\begin{Lemma} \label{4.5b} 
Let $\k=\F$. The indecomposable tilting modules $\{T_j(\la) \mid \la\in\Par(d)\}$  are precisely the indecomposable summands of\, $\bigoplus_{\nu\in\Comp(d)}\M_{d,j}\tty_\nu$. Moreover, for $\la\in\Par(d)$, the module\, $T_j(\la)$ occurs exactly once as a summand of $\M_{d,j}\tty_{\la'}$, and if\, $T_j(\mu)$ is a summand of $\M_{d,j}\tty_{\la'}$ for some $\mu\in\Par(d)$, then $\mu\unlhd\la$.
\end{Lemma}
\begin{proof}
By Lemma~\ref{LDonkinLa}, the indecomposable module 
$T_{n,d}(\la)$ occurs exactly once as a summand of $\La_{n,d}^{\la'}$, and if $T_{n,d}(\mu)$ is a summand of $\La_{n,d}^{\la'}$  then $\mu\unlhd\la$.
Now the result follows on applying the functor $\funb_{n,d,j}$ and using Lemma~\ref{L6.3.1}.
\end{proof}

\begin{Corollary} \label{4.5c} 
Let $\k=\F$. For $\la\in\Par(d)$, the graded supermodule $T_j(\la)$ is the unique
indecomposable summand of $\M_{d,j}\tty_{\la'}$ containing a graded subsupermodule isomorphic to $\De_j(\la)$.
\end{Corollary}
\begin{proof}
By Theorem~\ref{TDeAlternative}(i), $\M_{d,j}\tty_{\la'}$ has a unique graded subsupermodule isomorphic to $\De_j(\la)$. By Lemma~\ref{4.5b}, $\M_{d,j}\tty_{\la'}$ has a unique summand isomorphic to $T_j(\la)$, and $\Hom_{\IS_{d,j}}(\De_j(\la),V) = 0$ 
for any other summand $V$ of $\M_{d,j}\tty_{\la'}$. 
\end{proof}

\begin{Theorem} \label{4.5d} 
Let $\k=\F$ and $n\geq d$. The graded $\IS_{d,j}$-supermodule 
\begin{equation}\label{ETJND}
T_j(n,d) :=\bigoplus_{\nu\in \Comp(n,d)} \M_{d,j}\tty_\nu
\end{equation}
 is a full tilting module. 
Moreover, the Ringel dual $\IS_{d,j}'$ of $\IS_{d,j}$ relative to $T_j(n,d)$ is precisely the algebra $S(n,d)^\op$ where $S(n,d)$ acts on $T_j(n,d)$ as in Theorem~\ref{TSchurEnd2}.
\end{Theorem}
\begin{proof}
By Lemma~\ref{4.5b}, $T_j(n,d)$ is a full tilting module. The second statement is a restatement of Theorem~\ref{TSchurEnd2}.
\end{proof}


In the following theorem, $T_j(n,d)$ is as defined in (\ref{ETJND}) and the $S(n,d)$-action on $T_j(n,d)$ is as in Theorem~\ref{TSchurEnd2}, while the $\F\Si_d$-action on $\M_{d,j}$ is as in Theorem~\ref{TEndMd}.

\begin{Theorem} \label{4.5e} 
{\bf (Double Centralizer Properties)} 
Let $\k=\F$.   
\begin{enumerate}
\item[{\rm (i)}] If $n\geq d$ then 
$$\End_{\IS_{d,j}}\big(T_j(n,d)\big)\cong S(n,d)\quad {and}\quad \End_{S(n,d)}\big(T_j(n,d))\cong \IS_{d,j}.
$$
\item[{\rm (ii)}] $\End_{\IS_{d,j}}(\M_{d,j})^\sop\cong \F\Si_d$  and $\End_{\F\Si_d}(\M_{d,j})\cong \IS_{d,j}$. 
\end{enumerate}
\end{Theorem}
\begin{proof}
(i) follows from Theorem~\ref{4.5d} and Lemma~\ref{4.5a}.

(ii) By Theorem~\ref{TEndMd}, we already have $\End_{\IS_{d,j}}(\M_{d,j})^\sop\cong F\Si_d$. 

Let $e:=\xi_{\om_d}\in S(n,d)$, see (\ref{EIdSchur}). 
By Lemma~\ref{LTiltSoc}, the composition factors of the socle and the head of $T_{n,d}(\la)$ belong to the head of the projective $S(n,d)$-module $S(n,d)e$.  

Now $T_j:=T_j(n,d)$ is a full tilting module for $S(n,d)=\End_{\IS_{d,j}}(T_j)$ by Lemma~\ref{4.5a} and Theorem~\ref{4.5d}. So, by the previous paragraph, every composition factor of the socle and the head of $T_j$ belongs to the head of $S(n,d)e$. By Lemma~\ref{L3.1c}, we deduce that $\End_{S(n,d)}(T_j) \cong \End_{eS(n,d)e}(e T_j)$. Switching to right actions via anti-automorphism $\tau$ of Lemma~\ref{LSchurAnti}, and using (i), we have now shown that
$\IS_{d,j} \cong \End_{eS(n,d)e}(T_je).$
So, to prove (ii), it suffices to show that $\End_{\F\Si_d}(\M_{d,j}) \cong \End_{eS(n,d)e}(T_je)$. 

As left $\IS_{d,j}$-modules, $\M_{d,j} \cong T_je$. Recall the isomorphism $\kappa$ from  Lemma~\ref{LKappa} and  the explicit action of $S(n,d)$ on $T_j$ from Theorem~\ref{TSchurEnd2}. Now it is easy to see that the $(\IS_{d,j} , eS(n,d) e)$-bimodule $T _je$ is isomorphic to the $(\IS_{d,j} , \F\Si_d)$-bimodule $\M_{d,j}$, if we identify $\F\Si_d$ with $eS(n,d)e$ via the isomoprhism $g\mapsto \sgn(g)\kappa(g)$ for all $g\in \Si_d$. In view of this, we have 
$\End_{\F\Si_d}(\M_{d,j}) \cong \End_{eS(n,d)e}(T_je)$.  
\end{proof}

\section{Base change}\label{SSInt}
In this section, we assume that the ground field $\F$ is an $\O$-module. In fact, it will suffice to consider only the following two cases: 

$\bullet$\, $\F=\K$, where $\K$ is the (characteristic $0$) field of fractions $\K$ of $\O$; 

$\bullet$\, $\F=\O/\m$ for a maximal ideal $\m$ of $\O$.

The algebras $S(n,d)$, $\IS_{d,j}$, etc. as well as the modules $\M_{d,j},\Di_{d,j}$, etc. were all defined for $\k=\O$ or $\F$. 
To specify the ground ring we will now use the notation like $S(n,d)_\F$, $S(n,d)_\O$, $\M_{d,j,\F}$, $\M_{d,j,\O}$, $\M_{d,j,\K}$, etc. 
It is often possible to see that the base change $\F\otimes_\O X_\O$ for the  object $X_\O$ over $\O$ is isomorphic to the corresponding object $X_\F$ over $\F$. For example, we clearly have $\F\otimes_\O R_{d\de,\O}\cong R_{d\de,\F}$ and $\F\otimes_\O C_{d,\O}\cong C_{d,\F}$. 
What is often more difficult to see is that $X_\O$ is $\O$-free, and so a priori the (graded) dimension of $X_\F$ might depend on the characteristic of $\F$. For example, even though $R_{d\de,\O}$ is $\O$-free, we will not be able to prove the same for its quotient $\hat C_{d,\O}$ or the idempotent truncation $C_{d,\O}$ of $\hat C_{d,\O}$. However, we will establish the freeness for the quotient $\IS_{d,j,\O}$ of $C_{d,\O}$. Our main result in this direction is Theorem~\ref{T6.4.4} below. 

In what follows, we identify $\F\otimes_\O C_{d,\O}= C_{d,\F}$, so given a graded $C_{d,\O}$-supermodule $V_\O$, the base change $\F\otimes_\O V_\O$ is a grade $C_{d,\F}$-supermodule.

\subsection{Base change for modules}

\begin{Lemma} \label{LMO}
As a graded $\O$-supermodule, $\M_{d,j,\O}$ is finite rank free, and\, $\F\otimes_\O\M_{d,j,\O}\simeq \M_{d,j,\F}$ as graded $C_{d,\F}$-supermodules. 
\end{Lemma}
\begin{proof}
This comes from Lemma~\ref{CBasisIndGG}.
\end{proof}

We can identify $\O$ with a subring of its field of fractions $\K$. 
Given a finite dimensional graded $\K$-superspace $V_\K$, a {\em full rank lattice} in $V_\K$ is a free graded $\O$-subsupermodule $V_\O$ of $V_\K$ such that $\rank_\O V_\O=\dim_\K V_\K$. In this case  $V_\K\simeq\K\otimes_\O V_\O$. Conversely, if $V_\O$ is a finite rank free graded $\O$-supermodule then $V_\O$ can be identified with a full rank lattice in $\K\otimes_\O V_\O$.  The following result is standard.

\begin{Lemma} \label{LCA} 
Suppose $V_\O$ is a  full rank lattice in a finite dimensional graded $\K$-superspace $V_\K$, $W_\K$ is a graded $\K$-subsuperspace of $V_\K$ and $W_\O:=V_\O\cap W_\K$. Then:
\begin{enumerate}
\item[{\rm (i)}] $W_\O$ is a  full rank lattice in $W_\K$;
\item[{\rm (ii)}] $W_\O$ is a direct summand of the graded $\O$-supermodule $V_\O$. In particular, the natural map\, $\F\otimes_\O W_{\O}\to \F\otimes_\O V_{\O}$ is injective. 
\end{enumerate}  
\end{Lemma}

\begin{Lemma} \label{LZO} 
As a graded $\O$-supermodule, $\Di_{d,j,\O}$ is finite rank free,  and\, $\F\otimes_\O\Di_{d,j,\O}\simeq \Di_{d,j,\F}$ as graded $C_{d,\F}$-supermodules. Moreover, 
$$\Di_{d,j,\O}=\Y_{d,j,\O}:=C_{d,\O}\ggi^{d,j}\M_{d,j,\O}.$$ 
\end{Lemma}
\begin{proof}
By Lemma~\ref{LMO}, we can identify $\M_{d,j,\O}$ with a full rank lattice in $\M_{d,j,\K}$.  Then by definition, $\Di_{d,j,\O}=\Di_{d,j,\K}\cap\M_{d,j,\O}$.  So by Lemma~\ref{LCA}, $\Di_{d,j,\O}$ is a full rank lattice in $\Di_{d,j,\K}$ and the natural map $\iota:\F\otimes_\O\Di_{d,j,\O}\to \F\otimes_\O\M_{d,j,\O}=\M_{d,j,\F}$ is injective. Since the action of $\Si_d$ on $\M_{d,j}$ is compatible with base change, we have $\Im\iota\subseteq \Di_{d,j,\F}$. 
On the other hand, the natural homomorphism $\F\otimes_\O\Y_{d,j,\O}\to \M_{d,j,\F}$ has image $\Y_{d,j,\F}$. By (\ref{EYZ}), we have $\Y_{d,j,\O}\subseteq \Di_{d,j,\O}$, so $\Y_{d,j,\F}\subseteq \Im\iota$. But $\Y_{d,j,\F}=\Di_{d,j,\F}$ by Theorem~\ref{TProjGen}(i), so $\Im\iota=\Di_{d,j,\F}$. 

Finally, the embedding $\Y_{d,j,\O}\to \Di_{d,j,\O}$ has to be an isomorphism since otherwise for some $\F$, the induced map $\F\otimes_\O\Y_{d,j,\O}\to \F\otimes_\O\Di_{d,j,\O}$ is not surjective, and so the composition $\F\otimes_\O\Y_{d,j,\O}\to \F\otimes_\O\Di_{d,j,\O}\iso\Di_{d,j,\F}=\Y_{d,j,\F}$ is not surjective, giving a contradiction. 
\end{proof}

Since induction commutes with base change, we deduce:

\begin{Corollary} \label{C6.4.2} 
Let $\la\in\Comp(d)$. Then\, $\Di^\la_{j,\O}$ is finite rank free as a graded $\O$-supermodule, and $\F\otimes_\O\Di^\la_{j,\O}\simeq\Di^\la_{j,\F}$ as graded $C_{d,\F}$-supermodules. 
\end{Corollary}

Let $\la\in\Par(d)$. By Theorem~\ref{TDeAlternative}(ii), we have 
$
\De_j(\la)_\F:=\Di_{j,\F}^\la u_{\la'}\ttx_{\la'}.
$
So we define the corresponding `standard module over $\O$' as follows:
$$
\De_j(\la)_\O:=\Di_{j,\O}^\la u_{\la'}\ttx_{\la'}.
$$

\begin{Theorem} \label{T6.4.3} 
Let $\la\in\Comp(d)$. Then $\De_j(\la)_\O$ is finite rank free as a graded $\O$-supermodule, and $\F\otimes_\O\De_j(\la)_\O\simeq\De_j(\la)_\F$ as graded $C_{d,\F}$-supermodules. 
\end{Theorem}
\begin{proof}
By Corollary~\ref{C6.4.2},  we have that 
$\Di_{j,\O}^\la$ is finite rank free as a graded $\O$-supermodule, and  the embedding $\De_j(\la)_\O\to\Di^\la_{j,\O}$ induces a homomorphism of graded $C_{d,\F}$-supermodules $\iota: \F\otimes_\O\De_j(\la)_\O\to\F\otimes_\O\Di^\la_{j,\O}=\Di^\la_{j,\F}$ whose image is $\De_j(\la)_\F$. In particular, $\dim_\F (\F\otimes_\O\De_j(\la)_\O)\geq \dim_\F\De_j(\la)_\F$, and to establish that $\iota$ is an isomorphism, it suffices to prove that $\dim_\F (\F\otimes_\O\De_j(\la)_\O)= \dim_\F\De_j(\la)_\F$. 

Let $n\geq d$. By (\ref{EGaJND}) and Lemma~\ref{LZFiltration}, we have
\begin{equation}\label{EDimZDeDeF}
\begin{split}
\sum_{\mu\in\Comp(n,d)}\dim_\F\Di^\mu_{j,\F}
&=\dim_\F \Di_j(n,d)_{\F}
\\
&=\sum_{\nu\in\Par(d)}(\dim_\F\De_j(\nu)_\F)(\dim_\F\De_{n,d}(\nu)_\F),
\end{split}
\end{equation}
and
\begin{equation}
\label{EDimZDeDeK}
\begin{split}
\sum_{\mu\in\Comp(n,d)}\dim_\F\Di^\mu_{j,\K}
&=\dim_\K \Di_j(n,d)_{\K}
\\
&=\sum_{\nu\in\Par(d)}(\dim_\K\De_j(\nu)_\K)(\dim_\K\De_{n,d}(\nu)_\K).
\end{split}
\end{equation}
By Corollary~\ref{C6.4.2}, the left hand sides in (\ref{EDimZDeDeF}) and (\ref{EDimZDeDeK}) are equal to each other. Moreover, it is well known that $\dim_\K\De_{n,d}(\nu)_\K=\dim_\F\De_{n,d}(\nu)_\F$ for all $\nu\in\Par(d)$, see for example \cite[5.4f]{Green}. So it suffices to prove that $\dim_\K\De_j(\nu)_\K=\dim_\F (\F\otimes_\O\De_j(\nu)_\O)$ for all $\nu\in\Par(d)$. 
But the natural map $\K\otimes_\O  \De_j(\nu)_\O\to \K\otimes_\O\Di^\nu_{j,\O}=\Di^\nu_{j,\K}$ is injective and has image $\Di_{j,\K}^\nu u_{\nu'}\ttx_{\nu'}=\De_j(\nu)_\K$, so $\De_j(\nu)_\O$ is a full rank lattice in $\De_j(\nu)_\K$, which implies $\dim_\K\De_j(\nu)_\K=\dim_\F (\F\otimes_\O\De_j(\nu)_\O)$. 
\end{proof}

\subsection{Base change for endomorphism algebras}

\begin{Theorem} \label{T6.4.4} Let 
$\Di_j(n,d)_{\O}:=\bigoplus_{\la\in\Comp(n,d)}\Di_{j,\O}^\la$. 
\begin{enumerate}
\item[{\rm (i)}] As a graded $\O$-supermodule, $\IS_{d,j,\O}$ is finite rank free, and we have an isomorphism of graded superalgebras $\F\otimes_\O\IS_{d,j,\O}\cong \IS_{d,j,\F}$. 
\item[{\rm (ii)}] For any $\la\in\Comp(d)$, the graded $\IS_{d,j,\O}$-supermodule $\Di_{j,\O}^\la$ is projective. 
\item[{\rm (iii)}] We have an isomorphism of graded superalgebras $$\End_{\IS_{d,j,\O}}(\Di_j(n,d)_{\O})\cong S(n,d)_\O.
$$ 
\item[{\rm (iv)}] If $n\geq d$ then $\Di_j(n,d)_{\O}$ is a projective generator for $\IS_{d,j,\O}$, and $\IS_{d,j,\O}$ is graded Morita superequivalent to $S(n,d)_\O$. 
\end{enumerate}

\end{Theorem}
\begin{proof}
(i) By definition, $\IS_{d,j,\O}$ is the $\O$-submodule of $\End_\O(\M_{d,j,\O})$ consisting of the endomorphisms coming from 
$C_{d,\O}$ acting on $\M_{d,j,\O}$. 
Taking into account Lemma~\ref{LMO}, $\IS_{d,j,\O}$ is finite rank free as a graded $\O$-supermodule. Also, since $\K\otimes_\O C_{d,\O}\cong C_{d,\K}$, we have that $\IS_{d,j,\O}$ is a full rank lattice in $\IS_{d,j,\K}$. In particular, $\dim_\F (\F\otimes_\O\IS_{d,j,\O})=\rank_\O\IS_{d,j,\O}= \dim_\K \IS_{d,j,\K}$.

Moreover, 
the natural inclusion
$\IS_{d,j,\O}\into \End_\O(\M_{d,j,\O})$ yields a map $$\F\otimes_\O\IS_{d,j,\O} \to \F\otimes_\O\End_\O(\M_{d,j,\O})\simeq\End_\F(\M_{d,j,\F})$$ whose image is $\IS_{d,j,\F}$. 
Thus we have a surjective homomorphism $\F\otimes_\O\IS_{d,j,\O}\to \IS_{d,j,\F}$ of graded superalgebras. 
This map must be an isomorphism since the two algebras have the same dimension---indeed, $\dim_\F (\F\otimes_\O\IS_{d,j,\O})=\dim_\K \IS_{d,j,\K}$ by the previous paragraph, and $\dim_\K \IS_{d,j,\K}= \dim \IS_{d,j,\F}$ by  Lemma~\ref{L6.3.4} and Theorem~\ref{T6.4.3}.
 
 (ii) By Corollary~\ref{C6.4.2} and Theorem~\ref{TProjGen}(ii), 
 for any field $\F$ which is an $\O$-module, 
 we have that $\F\otimes_\O\Di^\la_{j,\O}\cong \Di^\la_{j,\F}$ is a projective graded $\IS_{d,j,\F}$-supermodule. Therefore, in view of the Universal Coefficients Theorem, $\Di^\nu_{j,\O}$ is a projective graded $\IS_{d,j,\O}$-supermodule. 
 
(iii) For $\k\in\{\O,\F\}$, denote $E_\k:= \End_{\IS_{d,j,\k}}(\Di_j(n,d)_{\k}).$ 
By Theorem~\ref{TSchurEnd3}, we have $E_\F \cong S(n,d)_{\F}$. 
By \cite[Lemma~14.5]{Landrock}, we have a natural embedding 
$\F\otimes_\O E_\O \into E_\F$ which is an isomorphism for $\F=\K$. 
Moreover, $E_\O$ is a $\O$-free, so  
\begin{align*}
\dim_\F(\F\otimes_\O E_\O)&=\dim_\K(\K\otimes_\O E_\O)=\dim_\K E_\K\\&=\dim_\K S(n,d)_{\K}=\dim_\F S(n,d)_{\F}=\dim_\F E_\F.
\end{align*}
By dimensions, the natural embedding $\F\otimes_\O E_\O \to E_\F$ must be an isomorphism. 

By Theorem~\ref{TSchurEnd3}, the basis element $\xi_{\la,\mu}^g$ of $E_\K= S(n,d)_\K$ acts as zero on all summands except $\Di^\mu_{j,\K}$, on which the action is induced by the right multiplication by $X_{\mu,\la}^g$. By definition,  $\Di^\nu_{j,\O}= \Di^\nu_{j,\K}\cap \M_{d,j,\O}$ for all $\nu$. Also, $X_{\mu,\la}^g\in \O\Si_d$, therefore the right multiplication by $X_{\mu,\la}^g$ stabilizes $\M_{d,j,\O}$. Hence $\Di^\mu_{j,\O} X_{\mu,\la}^g\subseteq \Di^\la_{j,\O}$, so each $\xi^g_{\la,\mu} \in E_\K$ restricts to give a well-defined element of $E_\O$. We have constructed an isomorphic copy $S_\O$ of $S(n,d)_{\O}$ in $E_\O$, namely, the $\O$-span of the standard basis elements $\xi^g_{\la,\mu}\in  S(n,d)_{\K}$. 

It remains to show that $S_\O=E_\O$. We have a short exact sequence of $\O$-modules: 
$
0\to S_\O\to E_\O\to Q_\O\to 0
$, and we need to prove that $Q_\O=0$, for which it suffices to prove that $\F\otimes_\O Q_\O=0$ for any $\F$. Tensoring with $\F$, we have an exact sequence
$$
\F\otimes_\O S_\O \stackrel{\iota}{\to} E_\F\to\F\otimes_\O Q_\O\to 0.
$$
The map $\iota$ sends $1\otimes \xi^g_{\la,\mu}$ to the corresponding endomorphism $\upzeta^g_{\la,\mu}$ defined as in Theorem~\ref{TSchurEnd3}. Hence, $\iota$ is injective, so $\iota$ is an isomorphism by dimensions, and $\F\otimes_\O Q_\O= 0$, as required. 

(iv) By (ii), $\Di_j(n,d)_\O$ is a projective graded $\IS_{d,j,\O}$-supermodule. For $n\geq d$, it is a projective generator, because this is so on tensoring with $\F$, using (i) and Theorem~\ref{TProjGen}(iii). 
\end{proof}


\begin{Corollary} \label{CHCICInt} 
Let $\la\in\Comp(d)$. 
\begin{enumerate}
\item[{\rm (i)}] The functor $\GGI_\la^d$ restricts to a functor $\mod{\IS_{\la,j,\O}}\to\mod{\IS_{d,j,\O}}$. Furthermore, for $V\in\mod{\IS_{\la,j,\O}}$, we have $\GGI_\la^dV\simeq \IS_{d,j,\O}\bar\ggi_\la\otimes_{\IS_{\la,j,\O}}V$.  
\item[{\rm (ii)}] The functor $\GGR_\la^d$ restricts to a functor $\mod{\IS_{d,j,\O}}\to\mod{\IS_{\la,j,\O}}$. 
\end{enumerate}
\end{Corollary}
\begin{proof}
The proof of Corollary~\ref{CHCIC} now generalizes using the fact that by Theorem~\ref{T6.4.4}, every indecomposable projective graded  $\IS_{\la,j,d}$-supermodule is a submodule of $\M_{\la,j,\O}$. 
\end{proof}

\begin{Lemma} \label{LProjCoverO}
The graded $C_{d,\O}$-module $\G_{d,j}$ is indecomposable.
\end{Lemma}
\begin{proof}
We have $\G_{d,j,\F}=\F\otimes\G_{d,j,\O}$ is indecomposable by Lemma~\ref{LProjCover}, so $\G_{d,j,\O}$ is also indecomposable.
\end{proof}

\begin{Corollary} \label{CZCIdemInt}
Let $\la\in\Comp(d)$ and set 
$$\bar\ggi^{\la,j}:=\ggi^{\la,j}+\Ann_{C_{d,\O}}(\M_{d,j,\O})\in C_{d,\O}/\Ann_{C_{d,\O}}(\M_{d,j,\O})=\IS_{d,j,\O}.
$$ 
Then $\funQ^{a_{\la,j}}\,\Uppi^{dj}\,\IS_{d,j,\O}\bar \ggi^{\la,j}\simeq \Di^\la_{j,\O}$ as graded $\IS_{d,j,\O}$-supermodules.  
\end{Corollary}
\begin{proof}
The same proof as for Corollary~\ref{CZCIdem} works, using Corollary~\ref{CHCICInt}(i) instead of Corollary~\ref{CHCIC}(i) to  reduce to the special case $\la=(d)$, Theorem~\ref{T6.4.4} instead of Theorem~\ref{TProjGen}, and Lemma~\ref{LProjCoverO} instead of Lemma~\ref{LProjCover}. 
\end{proof} 

Corollary~\ref{CZCIdemInt} and Theorem~\ref{T6.4.4}(iii) now imply:

\begin{Corollary} 
We have an isomorphism of graded superalgebras $$\textstyle\End_{\IS_{d,j,\k}}\Big(\bigoplus_{\la\in\Comp(n,d)} \funQ^{a_{\la,j}}\,\Uppi^{dj}\,\IS_{d,j,\k}\bar \ggi^{\la,j}\Big)\cong S(n,d)_\k.
$$ 
\end{Corollary}


\chapter{Regrading imaginary Schur-Weyl duality}


Recall from \S\ref{SSRegradingCd} the regrading $\zC_d$ of $C_d$. By Theorem~\ref{TNonNeg}, $\zC_d$ is non-negatively graded. For $\nu\in\Comp(n,d)$, recalling (\ref{EzCParIdentify}), we have a parabolic subalgebra 
$$\zC_\nu=\zC_{\nu_1}\otimes\dots\otimes\zC_{\nu_n}\subseteq \ggis_\nu\zC_d\ggis_\nu$$ and the functors 
$$\GGIS_\nu^d:\mod{\zC_\nu}\to\mod{\zC_d}\quad\text{and}\quad 
\GGRS_\nu^d:\mod{\zC_d}\to\mod{\zC_\nu}$$
defined in (\ref{EGGIRS}). 

For $j\in J$, identifying $\zC_1$ with $H_1(\Zig_\ell)=\Zig_\ell[\zz]$ via the isomorphism of Theorem~\ref{THCIsoZ} and recalling the graded $\Zig_\ell[\zz]$-supermodules $\zL_j=\k \zv_j$ from \S\ref{SSAffZig}, we obtain the graded $\zC_1$-supermodule $\zL_j$. Recalling the $C_1$-module $L_j$ from \S\ref{SSL} and the functors $\funh_d$ from 
\S\ref{SSFunctorsRegr}, note that $\zL_j\simeq \funh_1 \LL_j$. 
Moreover, by (\ref{EDeDe+}), we have
\begin{equation}\label{ELC+}
\zL_j\simeq \zC_1\ggis^j/(\zC_1\ggis^j)^{>0}.
\end{equation}

\section{Regraded imaginary tensor spaces}

\subsection{  Mixed imaginary tensor space $\zM_\bj$}
\label{SSMBJ}
Let $\bj\in J^d$. The corresponding (mixed) imaginary tensor space is 
$$
\zM_\bj:=\GGIS_{\om_d}^d ( \zL_{j_1}\boxtimes\dots\boxtimes \zL_{j_d})=\zC_d\ggis_{\om_d}\otimes_{\zC_{\om_d}}( \zL_{j_1}\boxtimes\dots\boxtimes \zL_{j_d})
$$
with generator 
$$
\zv_{\bj}:=\ggis_{\om_d}\otimes \zv_{j_1}\otimes\dots\otimes \zv_{j_n}.
$$

We consider the graded left $\zC_d$-subsupermodule of  $\zC_d\ggis^\bj$ generated by $
(\zC_{\om_d}\ggis^\bj)^{>0}$: 
\begin{equation}\label{ENBJ}
\zN_\bj:=\zC_d(\zC_{\om_d}\ggis^\bj)^{>0}. 
\end{equation}

\begin{Lemma} \label{LTensSpaceAsQuotient} 
For $\bj=j_1\cdots j_d\in J^d$ we have an isomorphism of graded $\zC_d$-supermodules 
$$
\zC_d\ggis^\bj/\zN_\bj\iso \zM_\bj=\zC_d\ggis_{\om_d}\otimes_{\zC_{\om_d}}( \zL_{j_1}\boxtimes\dots\boxtimes \zL_{j_d}),\ \ggis^\bj+\zN_\bj\mapsto \ggis_{\om_d}\otimes \zv_{j_1}\otimes\dots\otimes \zv_{j_d}. 
$$
\end{Lemma}
\begin{proof}
Since the algebra $\zC_1$ is non-negatively graded, we have for $\zC_{\om_d}=\zC_1^{\otimes d}$:
$$
\zC_{\om_d}^{>0}=\sum_{r=1}^d\zC_1^{\otimes (r-1)}\otimes \zC_1^{>0}\otimes \zC_1^{\otimes (d-r)}.
$$
So 
$$
(\zC_{\om_d}\ggis^{\bj})^{>0}=\sum_{r=1}^d\zC_1\ggis^{j_1}\otimes\dots\otimes \zC_1\ggis^{j_{r-1}}\otimes (\zC_1\ggis^{j_r})^{>0}\otimes \zC_1\ggis^{j_{r+1}}\otimes\dots\otimes \zC_1\ggis^{j_{d}},
$$
and
$$
\zC_{\om_d}\ggis^{\bj}/(\zC_{\om_d}\ggis^{\bj})^{>0}\simeq 
(\zC_1\ggis^{j_1}/(\zC_1\ggis^{j_1})^{>0})\boxtimes\dots\boxtimes (\zC_1\ggis^{j_d}/(\zC_1\ggis^{j_d})^{>0}).
$$
Now, using (\ref{ELC+}) and exactness of Gelfand-Graev induction, we obtain 
\begin{align*}
\GGIS_{\om_d}^d (\zL_{j_1}\boxtimes\dots\boxtimes \zL_{j_d})
&\simeq
\GGIS_{\om_d}^d \big((\zC_1\ggis^{j_1}/(\zC_1\ggis^{j_1})^{>0})\boxtimes\dots\boxtimes (\zC_1\ggis^{j_d}/(\zC_1\ggis^{j_d})^{>0})\big)
\\&\simeq
\GGIS_{\om_d}^d \big(\zC_{\om_d}\ggis^{\bj}/(\zC_{\om_d}\ggis^{\bj})^{>0}\big)
\\&\simeq
\big(\GGIS_{\om_d}^d \zC_{\om_d}\ggis^{\bj}\big)/\big(\GGIS_{\om_d}^d(\zC_{\om_d}\ggis^{\bj})^{>0}\big)
\\&\simeq
\big(\zC_d\ggis_{\om_d}\otimes_{\zC_{\om_d}} \zC_{\om_d}\ggis^{\bj}\big)/\big(\zC_d\ggis_{\om_d}\otimes_{\zC_{\om_d}}(\zC_{\om_d}\ggis^{\bj})^{>0}\big)
\\&\simeq
(\zC_d\ggis^{\bj})/(\zC_d (\zC_{\om_d}\ggis^{\bj})^{>0})
\\&=\zC_d\ggis^\bj/\zN_\bj,
\end{align*}
which implies the claim. 
\end{proof}

\subsection{  Regraded imaginary tensor space $\zM_{d,j}$}
\label{SSRegrMJ}
Let $j\in J$. As a special case of the construction of \S\ref{SSMBJ}, 
we consider the graded $\zC_d$-supermodule 
$$
\zM_{d,j}:=\zM_{j^d}=\GGIS_{\om_d}^d\,\zL_j^{\boxtimes d}
=\zC_d\ggis_{\om_d}\otimes_{\zC_{\om_d}}\zL_j^{\boxtimes d}.
$$
with generator $\zv_{d,j}:=\zv_{j^d}=\ggis_{\om_d}\otimes \zv_j^{\otimes d}$.

By Lemma~\ref{LTensSpaceAsQuotient}, 
\begin{equation}\label{EMN}
\zM_{d,j}\simeq \zC_d\ggis^{j^d}/\zN_{j^d}.
\end{equation}
We identify $\zM_{d,j}= \zC_d\ggis^{j^d}/\zN_{j^d}$ using this isomorphism.

There is a natural right action of the graded superalgebra $\ggis^{j^d}\zC_d\ggis^{j^d}$ on  $\zC_d\ggis^{j^d}$ via right multiplication. Identifying $\ggis^{j^d}\zC_d\ggis^{j^d}$ with $\ze^{j^d}H_d(\Zig_\ell)\ze^{j^d}$ via the isomorphism of Theorem~\ref{THCIsoZ}, we get the right action of $\ze^{j^d}H_d(\Zig_\ell)\ze^{j^d}$ on $\zC_d\ggis^{j^d}$. 
Recalling Theorem~\ref{TAffBasis}, we have an embedding 
$$\k\Si_d\into \ze^{j^d}H_d(\Zig_\ell)\ze^{j^d},\ w\mapsto \ze^{j^d}w\ze^{j^d}=w\ze^{j^d}.$$
This gives the right action of $\k\Si_d$ on $\zC_d\ggis^{j^d}$. Taking into account the formula for the image of $s_r$ under the isomorphism of Theorem~\ref{THCIsoZ}, we have for this action:
$$
(\zc\ggis^{j^d})\cdot s_r=(-1)^{j+1}\zc\ggis^{j^d}\dot \zs_r=(-1)^{j+1}\zc\dot \zs_r\ggis^{j^d}\qquad(\zc\in\zC_d,\ 1\leq r<d).
$$
By Lemma~\ref{LPositiveTau}, recalling (\ref{EMN}), this action factors through to the right action of $\k\Si_d$ on the module $\zM_{d,j}= \zC_d\ggis^{j^d}/\zN_{j^d}$ such that 
\begin{equation}\label{EActions_rRegrOneColor}
(\zc\zv_{d,j})\cdot s_r=(-1)^{j+1}\zc\dot \zs_r\zv_{d,j}
\qquad(\zc\in\zC_d,\ 1\leq r<d),
\end{equation}
hence
\begin{equation}\label{EActionWRegrOneColor}
(\zc\zv_{d,j})\cdot w=(-1)^{\ttl(w)(j+1)}\zc\dot \zw\zv_{d,j}
\qquad(\zc\in\zC_d,\ w\in\Si_d),
\end{equation}
making $\zM_{d,j}$ into a graded $(\zC_d,\k\Si_d)$-bisupermodule.

Recalling the graded $C_d$-supermodule $\M_{d,j}$ from (\ref{EM}) and using 
 $\zL_j\simeq \funh_1 \LL_j$, we deduce using Lemma~\ref{LGGISEq} that 
 \begin{equation}\label{EzMM}
 \zM_{d,j}\simeq \funh_d \M_{d,j}.
 \end{equation} 
We identify
$\zM_{d,j}$ with $\funh_d \M_{d,j}$ via the isomorphism of (\ref{EzMM}). 
Recalling the terminology of \S\ref{SSFunctorsRegr}, 
we can now speak of the element of $\zM_{d,j}$ corresponding to an element of $\M_{d,j}$ under the functor $\funh_d$. 
Then the functor $\funh_d$ intertwines the $\Si_d$-action on $\zM_{d,j}$ just defined and the $\Si_d$-action on $\M_{d,j}$ from \S\ref{SSSymAct} in the following sense:

\begin{Lemma} \label{LActionCorresponds}
Let $m\in M_{d,j}$, $w\in\Si_d$ and $\zm$ be the elements of $\zM_{d,j}$ corresponding to $m$ under the functor $\funh_d$. Then $\zm\cdot w$ corresponds to  $m\cdot w$ under the functor $\funh_d$.
\end{Lemma}
\begin{proof}
This follows from 
(\ref{EActionWRegrOneColor}), (\ref{EActions_r}) and 
(\ref{ECVCorresponds}). 
\end{proof}

\begin{Proposition}
The right action of $\k\Si_d$ on $\zM_{d,j}$ is faithful and induces the graded superalgebra isomorphism $\k\Si_d\cong \End_{\zC_d}(\zM_{d,j})^\sop$. 
\end{Proposition}
\begin{proof}
This follows from Lemma~\ref{LActionCorresponds} and Theorem~\ref{TEndMd}. 
\end{proof}

\begin{Proposition}\label{PzMPerm}
We have: 

\begin{enumerate}
\item[{\rm (i)}] 
For each $\la\in \Comp(d)$, 
there is a bidegree $(0,\0)$ isomorphism of right $\k\Si_d$-modules
$$
\ggis^{\la,j} \zM_{d,j}\iso \k_{\Si_\la}\otimes_{\k\Si_\la}\k\Si_d
$$
which maps $\lgathz_{\la,j} \zv_{d,j}$ to  $1\otimes 1_{\Si_d}$. 

\item[{\rm (ii)}] $\zM_{d,j}=\bigoplus_{\la\in\EC(d)}\ggis^{\la,j} \zM_{d,j}$. 
\end{enumerate}
\end{Proposition}
\begin{proof}
(i) follows from Lemmas~\ref{LActionCorresponds} and \ref{LBasisIndGGNew}. 

(ii) follows from (\ref{EzMM}) and  Lemma~\ref{LMWtSpaces}.
\end{proof}

Recalling the general set-up of Example~\ref{ExProgenerator}, we consider the graded superalgebra 
\begin{equation}\label{E030924_5}
\zC_j(n,d):=\End_{\zC_{d}}\Big(\bigoplus_{\la\in\Comp(n,d)}\zC_{d}\ggis^{\la,j}\Big)^\sop
\simeq\bigoplus_{\la,\mu\in\Comp(n,d)} \ggis^{\mu,j}\zC_{d}\ggis^{\la,j}.
\end{equation}
We identify $\zC_j(n,d)$ with $\bigoplus_{\la,\mu\in\Comp(n,d)} \ggis^{\mu,j}\zC_{d}\ggis^{\la,j}$ via the isomorphism above. 
Then the action of $\zC_d$ on $\zM_{d,j}$ yields the action of the graded superalgebra $\zC_j(n,d)=\bigoplus_{\la,\mu\in\Comp(n,d)} \ggis^{\mu,j}\zC_{d}\ggis^{\la,j}$ on
$$
\zM_j(n,d):=
\bigoplus_{\la\in\Comp(n,d)} \ggis^{\la,j}\zM_{d,j}.
$$
In view of Proposition~\ref{PzMPerm}, we have a graded $(\zC_j(n,d),\k\Si_d)$-bisupermodule structure on $\zM_j(n,d)$, and as, a graded right $\k\Si_d$-supermodule, 
$$
\zM_j(n,d)\simeq\bigoplus_{\la\in\Comp(n,d)}\k_{\Si_\la}\otimes_{\k\Si_\la}\k\Si_d.
$$
So 
$$
\End_{\k\Si_d}(\zM_j(n,d))\cong S(n,d),
$$
the classical Schur algebra, see (\ref{E050924}). 
The natural graded superalgebra homomorphism $\phi:\zC_j(n,d)\to \End_{\k\Si_d}(\zM_j(n,d))$ restricts to the graded superalgebra homomorphism
$$
\bar\phi:\zC_j(n,d)^0\to \End_{\k\Si_d}(\zM_j(n,d)).
$$

\begin{Theorem} \label{T050924_4}
Let $j\in J$. Then $\bar\phi$ is an isomorphism. In particular, $\zC_j(n,d)^0\cong S(n,d)$. 
\end{Theorem}
\begin{proof}
If $\k=\F$, then by Theorem~\ref{4.5e}(ii), the natural homomorphism $\zC_d\to\End_{\k\Si_d}(\zM_{d,j})$ is surjective, which implies that $\phi$ is surjective. But $S(n,d)$ is concentrated in degree $0$, so we deduce that $\bar\phi$ must be surjective. On the other hand, by Lemmas~\ref{LSchurBasis} and \ref{LDoubleCosets}, we have $\dim\zC_j(n,d)^0\leq \dim S(n,d)$, so $\phi$ is an isomorphism by dimensions. 

The result over $\k=\O$ follows from that over the quotient fields $\O/\m$ upon extension of scalars, since $S(n,d)$ is torsion-free as an $\O$-module. 
\end{proof}

\subsection{The algebra $\zC(n,d)$}
Recall from \S\ref{SSWtSpaces} and (\ref{EBeta}) how we consider the multicompositions from $\Comp^J(n,d)$ as colored compositions from $\Comp^\col(n\ell,d)$ via the embedding $\ttb:\Comp^J(n,d)\to \Comp^\col(n\ell,d)$. For $\bla\in \Comp^J(n,d)$, we have the idempotents  
$\ggis^\bla\in\zC_d$ and the composition $\ud(\bla)\in\Comp(J,d)$, see (\ref{EBlaNot}).  

We consider the 
graded superalgebra
$$
\zC(n,d):=\End_{\zC_{d}}\Big(\bigoplus_{\bla\in\Comp^J(n,d)}\zC_{d}\ggis^\bla\Big)^\sop
\simeq\bigoplus_{\bla,\bmu\in\Comp^J(n,d)} \ggis^\bmu\zC_{d}\ggis^\bla,
$$
where we have used the isomorphism (\ref{EEndAei}).

If $n\geq d$, 
this 
is just a Morita equivalent version of $\zC_{d}$ (which is sometimes a little bit more convenient to work with):

\begin{Lemma} \label{LEXMor} 
If $n\geq d$ then the graded superalgebras $\zC(n,d)$ and $\zC_{d}$ are graded Morita superequivalent.  
\end{Lemma}
\begin{proof}
We have to show that $\zP:=\bigoplus_{\bla\in\Comp^J(n,d)}\zC_{d}\ggis^\bla$ is a projective generator for $\zC_d$. If $\k=\O$, it suffices to prove this on tensoring with $\F$. So we may assume that $\k=\F$. Then it suffices to see that $\Hom_{\zC_d}(\zP,\zL)\neq 0$ for any $\zL\in\Irr(\zC_d)$. Recalling (\ref{EFunH}), by Theorem~\ref{LAmountAmount}(i), we may assume that $\zL=\zL(\bla):=\funh_d(L(\bla))$  for some $\bla\in\Par_J(d)$. Since $n\geq d$, we may assume that $\bla\in\Comp^J(n,d)$, and then 
$\Hom_{\zC_d}(\zC_d\ggis^\bla,\zL(\bla))\simeq \ggis^\bla\zL(\bla)\neq 0$ by Corollary~\ref{CCFMjAmountNew} and (\ref{ELBla}).
\end{proof}

For a composition $\ud\in\Comp(J,d)$, we also have the graded superalgebra 
$$
\zC(n,\ud):=\End_{\zC_{d}}\Big(\bigoplus_{\substack{\bla\in\Comp^J(n,d),\\ \ud(\bla)=\ud}}\zC_{d}\ggis^\bla\Big)^\sop
\simeq\bigoplus_{\substack{\bla,\bmu\in\Comp^J(n,d),\\\ud(\bla)=\ud(\bmu)=\ud}} \ggis^\bmu\zC_{d}\ggis^\bla.
$$

For degree $0$ components, from Lemma~\ref{LBiWeightSpaceTriv}, we immediately have:

\begin{Lemma} \label{L050924_3} 
We have the decomposition of $\zC(n,d)^0$ as a direct sum of graded superalgebras: 
$$\zC(n,d)^0=\bigoplus_{\ud\in\Comp(J,d)}\zC(n,\ud)^0.$$
\end{Lemma}

For $j\in J$, recall the graded superalgebra $\zC_j(n,d)$ from (\ref{E030924_5}).

\begin{Lemma} \label{L050924_2} 
Let $\ud\in\Comp(J,d)$. Then 
\begin{align*}
\zC(n,\ud)^0=\bigoplus_{\substack{\bla,\bmu\in\Comp^J(n,d),\\\ud(\bla)=\ud(\bmu)=\ud}} \ggis^\bmu\zC_{\ud}^0\ggis^\bla
\,\cong\, \zC_0(n,d_0)^0\otimes\dots\otimes \zC_{\ell-1}(n,d_{\ell-1})^0.
\end{align*}
\end{Lemma}
\begin{proof}
This follows immediately from Proposition~\ref{PDoubleCosetsGen}.
\end{proof}

\begin{Corollary} \label{C100924} 
We have 
$$\zC(n,d)^0\cong \bigoplus_{\ud=(d_0,\dots,d_{\ell-1})\in\Comp(J,d)}S(n,d_0)\otimes\dots\otimes S(n,d_{\ell-1}).$$
\end{Corollary}
\begin{proof}
This follows from Lemmas~\ref{L050924_3}, \ref{L050924_2} and Theorem~\ref{T050924_4}.  
\end{proof}

\section{Regrading divided powers}

\subsection{Mixed imaginary tensor spaces $\zM^{\la,\bi}$}
Given a colored composition $(\la,\bi)\in\Comp^\col(n,d)$ as in \S\ref{SSPar}, we have the corresponding (mixed) imaginary tensor space 
$$
\zM^{\la,\bi}:=\GGIS_\la^d ( \zM_{\la_1,i_1}\boxtimes\dots\boxtimes \zM_{\la_n,i_n})=\zC_d\ggis_\la\otimes_{\zC_\la}( \zM_{\la_1,i_1}\boxtimes\dots\boxtimes \zM_{\la_n,i_n})
$$
with generator 
$
\zv_{\la,\bi}:=\ggis_\la\otimes \zv_{\la_1,i_1}\otimes\dots\otimes \zv_{\la_n,i_n}.
$

By transitivity of the Gelfand-Graev induction, we have an isomorphism of graded $\zC_d$-supermodules
$
\zM^{\la,\bi}\simeq \zM_{i_1^{\la_1}\cdots i_n^{\la_n}}
$
under which $\zv_{\la,\bi}$ corresponds to $\zv_{i_1^{\la_1}\cdots i_n^{\la_n}}$. So the mixed tensor spaces $
\zM^{\la,\bi}$ are no different from the mixed tensor spaces $\zM_\bj$ defined in \S\ref{SSMBJ}, but the notation $
\zM^{\la,\bi}$ is convenient for dealing with the right action of $\Si_\la$ on $\zM^{\la,\bi}$. 

By the $n=1$ case considered in \S\ref{SSRegrMJ}, for $t=1,\dots,n$, we have the right action of 
$\Si_{\la_t}$ on $\zM_{\la_t,i_t}$ which commutes with the left $\zC_{\la_t}$-action. Hence we have the action of 
$\Si_\la=\Si_{\la_1}\times\dots\times \Si_{\la_n}$ on $\zM_{\la_1,i_1}\boxtimes\dots\boxtimes \zM_{\la_n,i_n}$ which commutes with the left $\zC_\la$-action. This induces the right  action of $\Si_\la$ on $\zM^{\la,\bi}$ which commutes with the left $\zC_d$-action. 

\subsection{Mixed divided powers}
Define
$$
\Diz^{\la,\bi}:=\{\zm\in\zM^{\la,\bi}\mid \zm\cdot w=w\ \,\text{for all}\ \,w\in\Si_\la\}. 
$$

As a special case, we define the graded $\zC_d$-supermodules 
\begin{align*}
\Diz_{d,j}:=\{\zm\in \zM_{d,j}\mid \zm\cdot w=\zm \ \text{for all}\ w\in\Si_d\}.
\end{align*}
Comparing with (\ref{EGaDJ}), it follows immediately from Lemma~\ref{LActionCorresponds} that  $\Diz_{d,j}\simeq \funh_d\Di_{d,j}$.

\begin{Lemma} \label{L7.15}  
We have $\Diz_{d,j}=\zC_d\ggis^{d,j} \zM_{d,j}$. 
\end{Lemma}
\begin{proof}
Since $\Diz_{d,j}\simeq \funh_d\Di_{d,j}$, it suffices to prove that 
$\Di_{d,j}=C_d\ggi^{d,j} \M_{d,j}$. 
When $\k=\F$, this follows from (\ref{EYdj}) and  Theorem~\ref{TProjGen}(i). When $\k=\O$ this follows from  Lemma~\ref{LZO}.  
\end{proof}

\begin{Lemma} \label{LZMixed}
Let $(\la,\bi)\in\Comp^\col(n,d)$. Then:
\begin{enumerate}
\item[{\rm (i)}] $\Diz^{\la,\bi}\simeq \GGIS_\la^d\big(\Diz_{\la_1,i_1}\boxtimes\dots\boxtimes \Diz_{\la_n,i_n}\big)$; 
\item[{\rm (ii)}] $\Diz^{\la,\bi}=\zC_d\big(\ggis^{\la,\bi}\otimes \zM_{\la_1,i_1}\boxtimes\dots\boxtimes \zM_{\la_n,i_n}\big)$. 
\end{enumerate} 
\end{Lemma}
\begin{proof}
(i) By Corollary~\ref{CFGProj}, $\zC_d\ggis_\la$ is a finitely generated projective right graded $\zC_\la$-supermodule, so 
there is a right graded $\zC_\la$-supermodule $\zW$ such that 
$\zC_d\ggis_\la\oplus\zW$ is a free right graded $\zC_\la$-supermodule. 
It is easy to see using the freeness that then the $\Si_\la$-invariants of $(\zC_d\ggis_\la\oplus\zW)\otimes_{\zC_\la}(\zM_{\la_1,i_1}\boxtimes\dots\boxtimes \zM_{\la_n,i_n})$
are exactly 
\begin{align*}
&(\zC_d\ggis_\la\oplus\zW)\otimes_{\zC_\la}(\Diz_{\la_1,i_1}\boxtimes\dots\boxtimes \Diz_{\la_n,i_n})
\\
\simeq\, &\big(\GGIS_\la^d (\Diz_{\la_1,i_1}\boxtimes\dots\boxtimes \Diz_{\la_n,i_n})\big)
\,\oplus\, \big(\zW\otimes _{\zC_\la}(\Diz_{\la_1,i_1}\boxtimes\dots\boxtimes \Diz_{\la_n,i_n})\big).
\end{align*} 
But  
\begin{align*}
&(\zC_d\ggis_\la\oplus\zW)\otimes_{\zC_\la}(\zM_{\la_1,i_1}\boxtimes\dots\boxtimes \zM_{\la_n,i_n})
\\
\simeq\, &\big(\GGIS_\la^d (\zM_{\la_1,i_1}\boxtimes\dots\boxtimes \zM_{\la_n,i_n})\big)\,\oplus\, \big(\zW\otimes _{\zC_\la}(\zM_{\la_1,i_1}\boxtimes\dots\boxtimes \zM_{\la_n,i_n})\big),
\end{align*}
and it now follows that the $\Si_\la$-invariants of $\GGIS_\la^d(\zM_{\la_1,i_1}\boxtimes\dots\boxtimes \zM_{\la_n,i_n})$
are exactly $\GGIS_\la^d (\Diz_{\la_1,i_1}\boxtimes\dots\boxtimes \Diz_{\la_n,i_n})$. 

(ii) By (i) and Lemma~\ref{L7.15}, we have 
\begin{align*}
\Diz^{\la,\bi}&\simeq
\GGIS_\la^d\big(\Diz_{\la_1,i_1}\boxtimes\dots\boxtimes \Diz_{\la_n,i_n}\big)
\\
&=\zC_d\ggis_\la\otimes_{\zC_\la}\big((\zC_{\la_1}\ggis^{\la_1,i_1}\zM_{\la_1,i_1})\boxtimes\dots\boxtimes (\zC_{\la_n}\ggis^{\la_n,i_n}\zM_{\la_n,i_n})\big)
\\
&= 
\zC_d\big(\ggis^{\la,\bi}\otimes \zM_{\la_1,i_1}\boxtimes\dots\boxtimes \zM_{\la_n,i_n}\big),
\end{align*}
as required. 
\end{proof}

For $(\la,\bi)\in\Comp^\col(n,d)$, we have the word $i_1^{\la_1}\cdots i_n^{\la_n}\in J^d$. 
By Lemma~\ref{LTensSpaceAsQuotient}, we have an isomorphism of 
graded $\zC_d$-supermodules 
\begin{equation}\label{ETildeMMIso}
\zM^{\la,\bi}\iso \zC_d\ggis^{i_1^{\la_1}\cdots i_n^{\la_n}}/\zN_{i_1^{\la_1}\cdots i_n^{\la_n}},\ 
\zv_{\la,\bi}\mapsto \ggis^{i_1^{\la_1}\cdots i_n^{\la_n}}+\zN_{i_1^{\la_1}\cdots i_n^{\la_n}}.
\end{equation}
The graded superalgebra $\ggis^{i_1^{\la_1}\cdots i_n^{\la_n}}\zC_d\ggis^{i_1^{\la_1}\cdots i_n^{\la_n}}$ acts on $\zC_d\ggis^{i_1^{\la_1}\cdots i_n^{\la_n}}$ on the right by right multiplication. Identifying $\ggis^{i_1^{\la_1}\cdots i_n^{\la_n}}\zC_d\ggis^{i_1^{\la_1}\cdots i_n^{\la_n}}$ with $\ze^{i_1^{\la_1}\cdots i_n^{\la_n}}H_d(\Zig_\ell)\ze^{i_1^{\la_1}\cdots i_n^{\la_n}}$ via the isomorphism of Theorem~\ref{THCIsoZ}, we get the right action of $\ze^{i_1^{\la_1}\cdots i_n^{\la_n}}H_d(\Zig_\ell)\ze^{i_1^{\la_1}\cdots i_n^{\la_n}}$ on $\zC_d\ggis^{i_1^{\la_1}\cdots i_n^{\la_n}}$. 
Recalling Theorem~\ref{TAffBasis} and the defining relations of $H_d(\Zig_\ell)$, we have an embedding 
$$\k\Si_\la\into \ze^{i_1^{\la_1}\cdots i_n^{\la_n}}H_d(\Zig_\ell)\ze^{i_1^{\la_1}\cdots i_n^{\la_n}},\ w\mapsto \ze^{i_1^{\la_1}\cdots i_n^{\la_n}}w\,\ze^{i_1^{\la_1}\cdots i_n^{\la_n}}=w\,\ze^{i_1^{\la_1}\cdots i_n^{\la_n}}.$$
This gives the right action of $\k\Si_\la$ on $\zC_d\ggis^{i_1^{\la_1}\cdots i_n^{\la_n}}$. Taking into account the formula for the image of $s_r$ under the isomorphism of Theorem~\ref{THCIsoZ}, we have for this action:
$$
(\zc\ggis^{i_1^{\la_1}\cdots i_n^{\la_n}})\cdot s_r=(-1)^{i_t+1}\zc\ggis^{i_1^{\la_1}\cdots i_n^{\la_n}}\dot \zs_r=(-1)^{i_t+1}\zc\dot \zs_r\ggis^{i_1^{\la_1}\cdots i_n^{\la_n}}\qquad(\zc\in\zC_d)
$$
if $s_r$ belongs to the $t$th component $\Si_{\la_t}$ of $\Si_{\la_1}\times\dots\times\Si_{\la_n}=\Si_\la$.
By Lemma~\ref{LPositiveTau}, this action factors through to the right action of $\k\Si_\la$ on the module $\zM_{\la,\bi}\simeq \zC_d\ggis^{i_1^{\la_1}\cdots i_n^{\la_n}}/\zN_{i_1^{\la_1}\cdots i_n^{\la_n}}$ such that 
\begin{equation}\label{EActions_rRegr}
(\zc\zv_{\la,\bi})\cdot s_r=(-1)^{j_t+1}\zc\dot \zs_r\zv_{\la,\bi}
\qquad(\zc\in\zC_d)
\end{equation}
if $s_r$ belongs to the $t$th component $\Si_{\la_t}$ of $\Si_{\la_1}\times\dots\times\Si_{\la_n}=\Si_\la$.

Define the left graded $\zC_d$-supermodule
$$
\hat\Diz^{\la,\bi}:=\{\zm\in \zC_d\ggis^{i_1^{\la_1}\cdots i_n^{\la_n}}\mid \zm\cdot w=\zm\ \text{for all}\ w\in\Si_\la\}.
$$

\begin{Lemma} \label{LGathhatGa}
Let $(\la,\bi)\in\Comp^\col(n,d)$. Then $\lgathz_{\la,\bi}\in\ggis^{\la,\bi}\hat\Diz^{\la,\bi}$. 
\end{Lemma}
\begin{proof}
By (\ref{E290923}), we have $\lgathz_{\la,\bi}\in\ggis^{\la,\bi}\zC_d\ggis^{i_1^{\la_1}\cdots i_n^{\la_n}}$. Moreover, $\lgathz_{\la,\bi}\cdot w=\lgathz_{\la,\bi}$ for all $w\in\Si_\la$ by Proposition~\ref{PUpsilon}(ii) and (\ref{EActions_rRegr}). 
\end{proof}

Define the left graded $\zC_d$-supermodule
$$
\tilde\Diz^{\la,\bi}:=\{\zm\in \zC_d\ggis^{i_1^{\la_1}\cdots i_n^{\la_n}}\mid \zm\cdot w-\zm\in\zN_{i_1^{\la_1}\cdots i_n^{\la_n}}\ \text{for all}\ w\in\Si_\la\}.
$$

\begin{Lemma} \label{L7.24}
Let $(\la,\bi)\in\Comp^\col(n,d)$. Then 
$$
\tilde\Diz^{\la,\bi}=\zC_d\ggis^{\la,\bi}\zC_\la\ggis^{i_1^{\la_1}\cdots i_n^{\la_n}}+\zN_{i_1^{\la_1}\cdots i_n^{\la_n}}.
$$ 
\end{Lemma}
\begin{proof}
Throughout the proof we identify $\zM^{\la,\bi}$ with $\zC_d\ggis^{i_1^{\la_1}\cdots i_n^{\la_n}}/\zN_{i_1^{\la_1}\cdots i_n^{\la_n}}$ via the isomorphism (\ref{ETildeMMIso}) so 
$$
\ggis_\la\otimes \zM_{\la_1,i_1}\boxtimes\dots\boxtimes \zM_{\la_n,i_n}=(\zC_\la\ggis^{i_1^{\la_1}\cdots i_n^{\la_n}}+\zN_{i_1^{\la_1}\cdots i_n^{\la_n}})/\zN_{i_1^{\la_1}\cdots i_n^{\la_n}}.
$$
Hence by Lemma~\ref{LZMixed}(ii), under this identification, we have
$$
\Diz^{\la,\bi}=(\zC_d\ggis^{\la,\bi}\zC_\la\ggis^{i_1^{\la_1}\cdots i_n^{\la_n}}+\zN_{i_1^{\la_1}\cdots i_n^{\la_n}})/\zN_{i_1^{\la_1}\cdots i_n^{\la_n}}.
$$
By definition, $\tilde\Diz^{\la,\bi}$ is the preimage of $\Diz^{\la,\bi}$ under the canonical surjection $$\zC_d\ggis^{i_1^{\la_1}\cdots i_n^{\la_n}}\onto \zC_d\ggis^{i_1^{\la_1}\cdots i_n^{\la_n}}/\zN_{i_1^{\la_1}\cdots i_n^{\la_n}}=  \zM^{\la,\bi}.$$ 
So the result follows. 
\end{proof}

\subsection{  On the idempotent truncation $\ggis^{\la,\bi}\zC_d^0\ggis_{\om_d}$}
Let $(\la,\bi)\in\Comp^\col(n,d)$. 
By Lemmas~\ref{LDegreeG} and \ref{LGathhatGa}, 
$$
\lgath_{\la,\bi}\in \ggis^{\la,\bi}(\hat\Diz^{\la,\bi})^0\subseteq \ggis^{\la,\bi}(\tilde\Diz^{\la,\bi})^0.
$$

\begin{Lemma} \label{L7.35Ver}
Let $(\la,\bi),(\mu,\bj)\in\Comp^\col(n,d)$ and\, $\zv\in \ggis^{\mu,\bj}(\tilde\Diz^{\la,\bi})^0$. Then $\zv=\zc\lgathz_{\la,\bi}$ for some $\zc\in\ggis^{\mu,\bj}\zC_d^0\ggis^{\la,\bi}$. 
\end{Lemma}
\begin{proof}
By Lemma~\ref{L7.24}, we have 
$$
\zv\in \ggis^{\mu,\bj}\zC_d\ggis^{\la,\bi}\zC_\la\ggis^{i_1^{\la_1}\cdots i_n^{\la_n}}+\ggis^{\mu,\bj}\zN_{i_1^{\la_1}\cdots i_n^{\la_n}}.
$$ 
Since all algebras are non-negatively graded, it follows from the definition (\ref{ENBJ}) that $\zN_{i_1^{\la_1}\cdots i_n^{\la_n}}^0=0$, and taking into account the assumption $\deg(\zv)=0$, we actually have 
$$
\zv\in (\ggis^{\mu,\bj}\zC_d\ggis^{\la,\bi})^0(\ggis^{\la,\bi}\zC_\la\ggis^{i_1^{\la_1}\cdots i_n^{\la_n}})^0.
$$ 
But 
$(\ggis^{\la,\bi}\zC_\la\ggis^{i_1^{\la_1}\cdots i_n^{\la_n}})^0
$ 
is spanned by $\lgathz_{\la,\bi}$ thanks to Lemma~\ref{LGBasis}. 
\end{proof}

\begin{Corollary}\label{L7.35}
Let $(\la,\bi),(\mu,\bj)\in\Comp^\col(n,d)$. If $\zv\in \ggis^{\mu,\bj}(\hat\Diz^{\la,\bi})^0$ then $\zv=\zc\lgathz_{\la,\bi}$ for some $\zc\in\ggis^{\mu,\bj}\zC_d^0\ggis^{\la,\bi}$.
\end{Corollary}
\begin{proof}
Follows from Lemma~\ref{L7.35Ver} using $\hat\Diz^{\la,\bi}\subseteq \tilde\Diz^{\la,\bi}$.
\end{proof}

\end{document}